\documentclass[reqno,english,11pt]{amsart}

\usepackage{tcolorbox}

\usepackage{mathtools}

\usepackage[margin=1in]{geometry}

\usepackage[latin9]{inputenc}
\usepackage{amsbsy}
\usepackage{amstext}
\usepackage{amsthm}
\usepackage{amssymb}
\usepackage{bm}
\usepackage{xcolor}
\usepackage{amsfonts,euscript,mathrsfs,color,amsmath,latexsym}
\numberwithin{equation}{section}
\numberwithin{figure}{section}
\theoremstyle{plain}
\newtheorem{thm}{\protect\theoremname}[section]
\theoremstyle{remark}
\newtheorem{rem}[thm]{\protect\remarkname}
\theoremstyle{definition}
\newtheorem{defn}{\protect\definitionname}[section]

\newtheorem{lemma}{Lemma}[section]
\newtheorem{proposition}{Proposition}[section]

\newtheorem{cor}{Corollary}[section]

\usepackage{hyperref}

\newcommand{\norm}[1]{\left\lVert#1\right\rVert}

\makeatother

\usepackage{babel}
\providecommand{\definitionname}{Definition}
\providecommand{\remarkname}{Remark}
\providecommand{\theoremname}{Theorem}

\begin{document}
	
	\title	[]{A New Structure for the 2D water wave equation: Energy stability and Global well-posedness}

	\author{Qingtang Su}
\author{Siwei Wang}
\address[Su]{Morningside Center of Mathematics, Academy of Mathematics and Systems Sciences,
Chinese Academy of Sciences (CAS), Beijing, 100080, Peoples Republic of China}
\email{suqingtang@amss.ac.cn}
\address[Wang]{Academy of Mathematics and Systems Sciences,
Chinese Academy of Sciences (CAS), Beijing, 100080, Peoples Republic of China}
\email{wangsiwei22@mails.ucas.ac.cn}

\begin{abstract}
We study the two-dimensional gravity water waves with a one-dimensional interface with small initial data. Our main contributions include the development of two novel localization lemmas and a Transition-of-Derivatives method, which enable us to reformulate the water wave system into the following simplified structure:
$$(D_t^2-iA\partial_{\alpha})\theta=i\frac{t}{\alpha}|D_t^2\zeta|^2D_t\theta+R$$
where $R$ behaves well in the energy estimate. As a key consequence, we derive the uniform bound
\begin{equation}\label{abstract:uniform}
    \sup_{t\geq 0}\Big(\norm{D_t\zeta(\cdot,t)}_{H^{s+1/2}}+\norm{\zeta_{\alpha}(\cdot,t)-1}_{H^s}\Big)\leq C\epsilon,
\end{equation}
which enhances existing global uniform energy estimates for 2D water waves by imposing less restrictive constraints on the low-frequency components of the initial data.

\end{abstract}

	\maketitle
	
\setcounter{tocdepth}{1}

\tableofcontents

 \section{Introduction}

 \subsection{Background}
The goal of this paper is to study the structure of the two dimensional water waves and its long time behavior. When surface tension and viscosity are neglected, and the water depth is assumed to be infinite, the motion of irrotational and incompressible two-dimensional water waves is governed by the following equations:
\begin{equation}\label{euler}
\begin{cases}
v_t+v\cdot \nabla v=-\nabla P-(0,1)\ \text{in}\ \Omega(t)\\
div~v=0, ~curl~v=0\ \text{in}\ \Omega(t)\\
P\Big|_{\Sigma(t)}\equiv 0\\
(1,v) \text{ is tangent to the free surface } (t, \Sigma(t)),
\end{cases}
\end{equation}
where $v$ is the fluid velocity, $P$ is the fluid pressure. 

From ${\rm div\,} v=0$ and ${\rm curl\,} v=0$, it follows that $\bar{v}$ is holomorphic in $\Omega(t)$. Consequently, $v$ is fully determined by its boundary value on $\Sigma(t)$. Let the interface $\Sigma(t)$ be parametrized by $z=z(\alpha,t)$, where $\alpha\in\mathbb R$ serves as the Lagrangian coordinate, chosen such that $z_t(\alpha,t)=v(z(\alpha,t),t)$. 

Thus, we have $v_t+v\cdot\nabla v\Big|_{\Sigma(t)}=z_{tt}$. Since $P(z(\alpha,t),t)\equiv 0$, the gradient of the pressure along the interface can be expressed as
$$\nabla P\Big|_{\Sigma(t)}=-iaz_{\alpha},$$
where $a:=-\frac{\partial P}{\partial n}\frac1{|z_\alpha|}$ is a real-valued function. 

Therefore, the motion of $\Sigma(t)$ can be written as
	\begin{equation}
	z_{tt}-iaz_{\alpha}=-i.
	\end{equation}
	Since $\bar{z}_t$ represents the boundary value of $\bar{v}$, 
	the water wave equations (\ref{euler}) can be equivalently expressed as:
	\begin{equation}\label{system_boundaryold}
	\begin{cases}
	z_{tt}-iaz_{\alpha}=-i\\
	\bar{z}_t~\text{is}~\text{holomorphic},
	\end{cases}
	\end{equation}
	where the holomorphicity of $\bar{z}_t$ implies the existence of a bounded holomorphic function $\mathcal{V}(\cdot,t)$ on $\Omega(t)$, satisfying $\mathcal{V}(x+iy,t)\rightarrow 0$ as $y\rightarrow -\infty$, such that $\bar{z}_t(\alpha,t)=\mathcal{V}(z(\alpha,t),t)$.

 The Taylor sign condition (often referred to as Rayleigh-Taylor sign conditio) $-\frac{\partial P}{\partial\Vec{n}}>0$ on the pressure is a crucial stability criterion for the water waves problem. When this condition fails, the system is generally unstable; see \cite{beale1993growth, birkhoff1962helmholtz, taylor1950instability,ebin1987equations, su2023transition} for examples. In the irrotational case, Nalimov \cite{Nalimov}, Yosihara \cite{Yosihara} and Craig \cite{Craig} proved local well-posedness for 2d water waves equation for small initial data. In the irrotational case without a bottom, Wu \cite{Wu1997, Wu1999} demonstrated the validity of the Taylor sign condition, a result crucial to establishing the first local-in-time existence theorems for large data in Sobolev spaces. Since then, a lot of interesting local well-posedness results were obtained, see for example \cite{alazard2014cauchy, ambrose2005zero, christodoulou2000motion, coutand2007well, iguchi2001well, lannes2005well, lindblad2005well, ogawa2002free, shatah2006geometry, zhang2008free}, and the references therein. See also \cite{Wu1, Wu2, Wu3, wu2019wellposedness} for water waves with non-smooth interfaces. For the formation of splash singularities, see for example \cite{castro2012finite, castro2013finite, coutand2014finite, coutand2016impossibility}.
For cases with non-trivial vorticity or a bottom, the Taylor sign condition may fail and must be assumed for the initial data.   Regarding the local-in-time well-posedness with regular vorticity, see \cite{iguchi1999free, ogawa2002free, ogawa2003incompressible, christodoulou2000motion, christodoulou2000motion, lindblad2005well, zhang2008free} and \cite{su2020long} for water waves with point vortices. 

To study the long time behavior of the water waves, we need a deeper understanding of its structure. It is known since the work of Dyachenko and Zakharov \cite{dyachenko1994free} that for the weakly nonlinear 2d
 infinite depth water waves, there are no 3-wave interactions and all of the 4-wave interaction coefficients
 vanish on the non-trivial resonant manifold. This means that, in terms of energy, there are no cubic terms in the nonlinearity. Dyachenko and Zakharov's work is formal. A rigorous justification of Dyachenko and Zakharov's observation is given by Berti, Feola, and Pusateri in \cite{berti2023birkhoff}.

In the irrotational case, in \cite{Wu2009}, Wu found an explicit cubic structure for the 2d water waves:
$$(D_t^2-iA\partial_{\alpha})\theta=-2\left[D_t\zeta, \mathcal{H}\frac{1}{\zeta_{\alpha}}+\bar{\mathcal{H}}\frac{1}{\bar{\zeta}_{\alpha}}\right]\partial_{\alpha}D_t\zeta+\frac{1}{ \pi i}\int \Big(\frac{D_t\zeta(\alpha)-D_t\zeta(\beta)}{\zeta(\alpha,t)-\zeta(\beta,t)}\Big)^2\partial_{\beta}(\zeta-\bar{\zeta})\,d\beta$$
(Here $\mathcal{H}$ is defined as the Hilbert transform associated to $\zeta$. We refer the readers to the relevant part in \cite{Wu2009} or Chapter 2; such symbol will appear frequently in the following discussion). This cubic structure confirms the absence of three-wave interactions, and using this, Wu is able to prove the almost global well-posedness of the 2d gravity water waves with small and localized initial data. See \S \ref{intro:subsub wu structure} for more details.
 
In \cite{wu2020quartic}, directly in the physical space, Wu constructed a sequence of energy functionals $E_j(t)$ which are
 explicit in the Riemann mapping variable and involve material derivatives of order $j$ of the solutions for the
 2d water wave equation, so that
 $\frac{dE_j(t)}{dt}$ is quintic or higher order. See also \cite{HunterTataruIfrim1,HunterTataruIfrim2, wang2018global, wang2019global, wang2020global, zheng2022long, ai2022two}. We also mention \cite{ifrim2015two, bieri2017motion, ginsberg2018lifespan, ginsberg2024long} and \cite{su2020long} for some investigations in the rotational case. 

In \cite{Totz2012}, Totz and Wu rigorously justify the NLS approximation of the 2d water waves. They are able to construct water waves of the form
$$\zeta(\alpha,t)=\alpha+\epsilon B(X,T)e^{ik\alpha+i\omega t}+O(\epsilon^2),$$
where $\zeta$ the modified Lagrangian coordinates of the free surface, $B$ solves the 1d cubic NLS, $X=\epsilon(\alpha-\omega' t)$, and $T=\epsilon^2 t$, $k$ and $\omega$ are constants, satisfying $\omega^2=k$. They show that this approximation is valid at least on a time scale of order $O(\epsilon^{-2})$. If we insert this ansatz back into the water wave system, then the water wave system can be reformulated in the following form:
\begin{equation}\label{structure: NLS}
    (D_t^2-iA\partial_{\alpha})\theta=2i k^{3/2}|D_t\zeta|^2 D_t\zeta +O(\epsilon^4):=G_1+O(\epsilon^4).
\end{equation}
It's remarkable that the leading order term $G_1$ in \eqref{structure: NLS} is simple and explicit. Moreover, if we define the natural energy associated to \eqref{structure: NLS}
$$E(t):=\int \frac{1}{A}|D_t\zeta|^2-i\theta_{\alpha}\bar{\theta},$$
then the leading order term $G_1$ disappears in the energy estimates. 

Based on this observation, in this paper, we show that such a nice structure is true in a more general setting. In particular, if the initial data is smooth, small, and localized, then we have 

$$(D_t^2-iA\partial_{\alpha})\theta=i\frac{t}{\alpha}|D_t^2\zeta|^2D_t\theta+R$$
where
\[
(I-\mathcal{H})R=O(\epsilon^3t^{-1-\delta})_{L^2}.
\]
As a key consequence, we derive the uniform bound
\begin{equation}\label{abstract:uniform}
    \sup_{t\geq 0}\Big(\norm{D_t\zeta(\cdot,t)}_{H^{s+1/2}}+\norm{\zeta_{\alpha}(\cdot,t)-1}_{H^s}\Big)\leq C\epsilon,
\end{equation}


\subsection{Main theorem}

Denote 
$$\xi_0=z(\alpha,0)-\alpha, \quad v_0=z_t(\alpha,0).$$
Our main theorem is as follows.
\begin{thm}[Main theorem]\label{main}
  Let $s\in \mathbb{N}$ and $s\geq 10$. Let $\epsilon_0>0$ be sufficiently small. For all $0<\epsilon\leq\epsilon_0$, assume
    \begin{equation}
        \norm{v_0}_{H^{s+1/2}}+\norm{\xi_0}_{\dot{H}^{1/2}}+\norm{\partial_{\alpha} \xi_0}_{H^{s}}+\norm{\alpha\partial_{\alpha}v_0}_{H^{s-1/2}}+\norm{\alpha\partial_{\alpha}\partial_{\alpha}\xi_0}_{H^{s-1}}\leq \epsilon
    \end{equation}
\begin{itemize}
    \item [(1)] We can reformulate the water wave equation in the following form
\begin{equation}
    (D_t^2-iA\partial_{\alpha})\theta=i\frac{t}{\alpha}|D_t^2\zeta|^2D_t\theta+R,
\end{equation}
where
\[
(I-\mathcal{H})R=O(\epsilon^3t^{-1-\delta})_{L^2}.
\]
\item [(2)] The water wave system \eqref{system_boundaryold} is globally wellposed and moreover, there exists a constant $C$ depends on $s$ only, such that 
\begin{equation}
    \sup_{t\geq 0}(\norm{z_t}_{H^{s+1/2}}+\norm{\xi}_{\dot{H}^{1/2}}+\norm{z_{\alpha}-1}_{H^s}+\norm{z_{tt}}_{H^s})\leq C\epsilon,
\end{equation}
and for all $t\geq 1$,
\begin{equation}
   \norm{z_{\alpha}(\cdot,t)-1}_{W^{s-2,\infty}}+ \norm{z_{tt}(\cdot,t)}_{W^{s-2,\infty}}+\norm{\partial_{\alpha}z_t}_{W^{s-3,\infty}}\leq C\epsilon t^{-1/2}.
\end{equation}
\end{itemize}

\end{thm}

\begin{rem}
We emphasize that:
   \begin{itemize}
       \item This result implies the nonlinear stability of the equilibrium state.
       \item Compared to \cite{wang2018global} and \cite{ai2022two}, we assume $z_t\in H^{s+1/2}$, whereas \cite{wang2018global} essentially assumes $|\partial_{\alpha}|^{-3/10}z_t\in H^{s'}$ and \cite{ai2022two} assumes $|\partial_{\alpha}|^{-1/4}z_t\in H^{s'}$. Our assumptions on the initial data are less restrictive for low frequencies. 
       \item Notably, we prove an almost conservation law: let $\mathcal{E}_s(t)$ be an appropriately defined $s$-th order energy, with $\mathcal{E}_s(0)\sim \epsilon^2$. We have
       \begin{equation}
           \sup_{t\geq 0}|\mathcal{E}_s(t)-\mathcal{E}_s(0)|\lesssim \epsilon^{4}.
       \end{equation}
       
       \item The assumption that $s\geq 10$ is certainly not optimal. In fact, $s\geq 4$ would suffice. Reducing the value of $s$ is not the primary focus of this paper.

       \end{itemize} 
\end{rem}

\subsection{The difficulty and the main idea} 

\subsubsection{The main difficulty} The water wave equations are nonlocal, quasilinear, and governed by a free boundary, making their structure inherently complex and difficult to fully understand. In \cite{Wu2009}, Wu uncovered that the water wave system possesses a cubic structure, which she leveraged to prove the almost global well-posedness of the 2D water wave equations. Building on this insight, Ionescu and Pusateri established global well-posedness by combining Wu's cubic structure with the technique of modified scattering. Around the same time, Alazard and Delort employed the paradifferential calculus to obtain the delicate $L^2$ to $L^{\infty}$ decay estimates and achieve global well-posedness. To elucidate the underlying ideas, we consider a simplified toy model. Let $f\in C^{\infty}$ be given. Suppose we are given a PDE system:

\begin{equation}\label{toy-model}
    \begin{cases}
        u_{tt}+|\partial_{\alpha}|u=f(|u_t|^2)u_t\\
    \norm{u_t(t=0)}_{H^s}+\norm{|\partial_{\alpha}|^{1/2}u(t=0)}_{H^s}=\epsilon\ll 1.
    \end{cases}
\end{equation}
Suppose $u_t$ has pointwise decay $\norm{u_t}_{\infty}\leq \epsilon (1+t)^{-1/2}$. Define the energy
$$E(t)=\norm{u_t(\cdot,t)}_{L^2}^2+\norm{\Lambda u(\cdot,t)}_{L^2}^2.$$
Then one has
$$\frac{d E(t)}{dt}=2 Re \int O(|u_t|^2)|u_t|^2 \leq C\epsilon^2 (1+t)^{-1}E(t).$$
From which we obtain
\begin{equation}\label{bound-toy}
    E(t)\leq CE(0)t^{C\epsilon^2}.
\end{equation}

This toy model illustrates the idea of Wu\cite{Wu2009}, Ionescu and Pusateri\cite{Ionescu2015}, among others. They reformulated the water wave equations into a structure analogous to \eqref{toy-model}, enabling them to establish that the energy of the solution satisfies a bound of the form \eqref{bound-toy}. However, since their analysis does not fully exploit the intricate structure of the nonlinearity, they did not demonstrate that $\sup_{t\geq 0}E(t)<\infty$. 

\subsubsection{Our idea}
Our main observation is that the nonlinearity of the water wave equation exhibits an additional structure. Specifically, if $f(|u_t|^2)=i|u_t|^2$, then $E'(t)=0$. So $E(t)=E(0)$. For the water wave equations, we can reformulate them into the following form:
\begin{equation}\label{toy-model-new}
    \begin{cases}
        u_{tt}+|\partial_{\alpha}|u=i\frac{t^2}{\alpha^2}|u_t|^2 u_t+O(\epsilon^2 t^{-1-\delta})_{L^{\infty}}u_t\\
    \norm{u_t(t=0)}_{H^s}+\norm{|\partial_{\alpha}|^{1/2}u(t=0)}_{H^s}=\epsilon\ll 1.
    \end{cases}
\end{equation}
Let $E(t)$ be the energy as defined above. Then,
$$\frac{dE(t)}{dt}=2Re\int \bar{u}_t\Big(i\frac{t^2}{\alpha^2}|u_t|^2 u_t+O(\epsilon^2 t^{-1-\delta})_{L^{\infty}}u_t\Big)=2Re\int \bar{u}_t O(t^{-1-\delta})_{L^{\infty}}u_t\lesssim \epsilon^2 t^{-1-\delta} E(t).$$
From this, we conclude that $\sup_{t\geq 0}E(t)\lesssim E(0)$,  thereby proving the nonlinear stability of the equilibrium.

It is noteworthy that the leading term in the nonlinearity is local. This emerges naturally from our "localization lemma," which effectively reduces nonlocal terms to local ones.

The key elements in establishing $\sup_{t\geq 0}E(t)\lesssim E(0)$ are as follows.
\begin{itemize}
\item [(1)] Cubic Nonlinearity: The nonlinearity of the system consists of cubic and higher-order terms.

\item [(2)] Good Structure: The leading-order term in the nonlinearity cancels out in the energy estimates, simplifying the analysis.

\item [(3)] Decay estimates: $\norm{u_t(\cdot,t)}_{L^{\infty}}\lesssim \epsilon t^{-1/2}$, ensuring sufficient decay over time. 
\end{itemize}
If the water wave equations satisfy these three conditions, a similar approach can be applied to establish their nonlinear stability.

In the following, we detail how to verify these three properties for the water wave equations.

\subsubsection{Change of coordinates}\label{intro:subsub wu structure}
In \cite{Wu2009}, Wu introduced the modified Lagrangian coordinates to demonstrate that the water wave equations possess a cubic structure. We briefly outline her approach below.

Let $\kappa(\cdot,t):\mathbb{R}\rightarrow \mathbb{R}$ be a diffeomorphism. Denote $\zeta:=z\circ \kappa^{-1}$. The diffeomorphism $\kappa$ is chosen such that $\bar{\zeta}-\alpha$ is holomorphic in the sense described earlier.

By composing the equation \eqref{system_boundaryold} with the diffeomorphism $\kappa^{-1}$, we find that 	$\zeta$ satisfies
 
	\begin{equation}\label{system_boundary}
	\begin{cases}
	(D_t^2-iA\partial_{\alpha})\zeta=-i\\
	D_t\bar{\zeta}, \quad \bar{\zeta}-\alpha~\text{are}~\text{holomorphic},
	\end{cases}
	\end{equation}
where the following notations are used:
 \begin{equation}\label{eq:DAintro}
	D_{t}=\partial_t+b\partial_{\alpha},\,\,
	b=(\partial_{t}\kappa)\circ \kappa^{-1},\,\,
	A:=(a\kappa_{\alpha})\circ \kappa^{-1}.
	\end{equation}

	This coordinate system (which we call \emph{Wu's coordinate}) was first introduced in \cite{Wu2009} to establish the almost global well-posedness of 2D water waves with small, localized data. It has since been employed in works such as \cite{Wu2011, Totz2012, Totz2015, su2020partial, su2020long, Ionescu2015, chen2023nonlinear},  to study water wave problems on longer time scales.
    
Once the existence of such coordinates is established, the water wave system can be directly analyzed in this framework without reverting to the original Lagrangian coordinates. In \S \ref{section:structure}, by deriving expressions for $b$ and $A$, we formulate the water wave system in terms of the $\zeta$ variables and avoid explicitly using the change of variables $\kappa$. By Corollary \ref{control of change of coordinates}, $\norm{\kappa_{\alpha}(\cdot,t)-\kappa_{\alpha}(\cdot,0)}_{\infty}\lesssim \epsilon^2$, so the change of variables can be uniformly controlled over time.

\vspace*{2ex}
$\bullet$ Cubic structure:  In Wu's Lagrangian coordinates, Wu \cite{Wu2009} finds the following cubic structure:
    \begin{equation}\label{intro:cubic}
        (D_t^2-iA\partial_{\alpha})(I-\mathcal{H})(\zeta-\alpha)=G,
    \end{equation}
    where $D_t=\partial_t+b\partial_{\alpha}$ with $b$ and $A-1$ quadratic, and 
    $$G=-2\left[D_t\zeta, \mathcal{H}\frac{1}{\zeta_{\alpha}}+\bar{\mathcal{H}}\frac{1}{\bar{\zeta}_{\alpha}}\right]\partial_{\alpha}D_t\zeta+\frac{1}{\pi i}\int \Big(\frac{D_t\zeta(\alpha,t)-D_t\zeta(\beta,t)}{\zeta(\alpha,t)-\zeta(\beta,t)}\Big)^2(\bar{\zeta}_{\beta}-1)d\beta:=G_1+G_2.$$

\vspace*{2ex}

$\bullet$ The energy functional

Let $\theta=(I-\mathcal{H})(\zeta-\alpha)$. Define energy
$$E_0^\theta(t)=\int \frac{1}{A}|D_t\theta|^2d\alpha +i\int\theta\partial_\alpha\bar{\theta}d\alpha
$$
Then we have 
		\begin{equation}\label{eqn:basic}
			\frac{dE_0^\theta}{dt}=Re\left\{\int\frac{2D_t\bar{\theta}}{A}Gd\alpha\right\}-\int\frac{|D_t\theta|^2}{A}\left(\frac{a_t}{a}\circ\kappa^{-1}\right)d\alpha
		\end{equation}
We will demonstrate that $\frac{a_t}{a}\circ\kappa^{-1}$ decays at a rate of $t^{-1-\delta}$ (refer to Proposition \ref{norm of at/a}). Thus, it remains to address $G$.

\subsubsection{The good structure of $G$}
Our key observation is the following decomposition for the nonlinearities (see Proposition \ref{key prop}). 
\begin{proposition}\label{prop:intro}
    We have 
    \begin{equation}
        \norm{(I-\mathcal{H})G_1}_{L^2}\lesssim \epsilon^3 t^{-1-\delta},
    \end{equation}
    and
    \begin{equation}
        G_2=i\frac{t}{\alpha}|D_t^2\zeta|^2D_t\theta+e, \quad with\quad \norm{e}_{L^2}\lesssim \epsilon^3 t^{-1-\delta}.
    \end{equation}
\end{proposition}

The key idea in proving Proposition \ref{prop:intro} is a \textbf{localization lemma}, which stems from the \textbf{profile decomposition} described below.

Denote  
$$I(f,g,h)=\frac{1}{\pi i}\int \frac{(f(\alpha)-f(\beta))(g(\alpha)-g(\beta))}{(\alpha-\beta)^2}\bar{h}_{\beta}d\beta.$$
Where $f, g, h$ are anti-holomorphic. For the linear water wave equations $u_{tt}-iu_{\alpha}=0$, the leading order term of $u(\alpha,t)$ is $\frac{1}{\sqrt{t}}e^{-it^2/4\alpha}\phi(\frac{t^2}{\alpha^2})$. While $u$ and its derivatives have decay $t^{-1/2}$ in $L^{\infty}$,  the profile $U:= \frac{1}{\sqrt{t}}\phi(\frac{t^2}{\alpha^2})$ satisfies
$$\norm{\partial_t U}_{L^2}\lesssim t^{-1}, \quad \norm{\partial_{\alpha}U}_{L^2}\lesssim t^{-1}.$$

This motivates us to make the profile decomposition
$$f(\alpha,t)=e^{-it^2/4\alpha}F(\alpha,t), \quad g(\alpha,t)=e^{-it^2/4\alpha}G(\alpha,t), \quad h(\alpha,t)=e^{-it^2/4\alpha}H(\alpha,t).$$
Assume a priori that for some small $\delta>0$
\begin{equation}\label{intro:derivativegood}
\norm{F_{\alpha}}_{L^2}+\norm{G_{\alpha}}_{L^2}+\norm{H_{\alpha}}_{L^2}\lesssim \epsilon t^{-1+\delta} 
\end{equation}
Then we have 
\begin{equation}
\begin{split}
I(f,g,h)(\alpha,t)=&FGH(\alpha,t) \frac{1}{\pi i}\int \frac{(e^{-it^2/4\alpha}-e^{-it^2/4\beta})^2}{(\alpha-\beta)^2} e^{it^2/4\beta}d\beta+O(t^{-1-\delta})_{L^2}\\
=&\frac{it^2}{2\alpha^2}FGH(\alpha,t) e^{-it^2/4\alpha}+O(t^{-1-\delta})_{L^2}\\
=&\frac{it^2}{2\alpha^2}fgh(\alpha,t)+O(t^{-1-\delta})_{L^2},
\end{split}
\end{equation}
where $O(t^{-1-\delta})_{L^2}$ means the $L^2$ norm of this term decays like $t^{-1-\delta}$. Refer to Lemma \ref{decouple} and Lemma \ref{decouple-2} for detailed explanations and proofs. By applying these localization lemmas to the nonlinearities $G=G_1+G_2$, we derive Proposition \ref{prop:intro}.
\begin{rem}\label{rem:whyL0s-1}
To rigorously justify \eqref{intro:derivativegood}, note that
\begin{equation}
    F_{\alpha}=e^{it^2/4\alpha}\Big(\frac{1}{\alpha}L_0f-\frac{t}{2\alpha^2}\Omega_0f\Big).
\end{equation}
So it suffices to guarantee that $\norm{L_0f}_{L^2}+\norm{\Omega_0f}_{L^2}\lesssim \epsilon t^{\delta}$. Therefore, in order to apply the localization lemma to $\partial_{\alpha}^s I(f,g,h)$, we need to control $\norm{L_0f}_{H^{s-1}}$, $\norm{\Omega_0f}_{H^{s-1}}$, among others. This is why we need to control $L_0D_t\zeta$ and $L_0D_t^2\zeta$ up to $s-1$ derivatives.
\end{rem}

Applying Proposition \ref{prop:intro} to \eqref{eqn:basic}, we obtain
\begin{align*}
    \frac{dE_0^{\theta}}{dt}=Re\int \frac{2D_t\bar{\theta}}{A}\Big(i\frac{t}{\alpha}|D_t^2\zeta|^2D_t\theta\Big)+Re\int \frac{2D_t\bar{\theta}}{A} e-\int\frac{|D_t\theta|^2}{A}\left(\frac{a_t}{a}\circ\kappa^{-1}\right)d\alpha
\end{align*}

Note that the first integral vanishes. Since $\norm{\frac{a_t}{a}\circ\kappa^{-1}}_{L^{\infty}}\lesssim t^{-1-\delta}$ (see Proposition \ref{norm of at/a}), we obtain
$$\frac{dE_0^\theta}{dt}\lesssim \epsilon^3 t^{-1-\delta}E_0^\theta,$$
from which it follows that
\begin{equation}\label{almost conserved law}
\sup_{t\geq 0}|E_0^\theta(t)-E_0^\theta(0)|\lesssim \epsilon^4.
\end{equation}
Therefore, the energy $E_0^{\theta}(t)$ is almost conserved and, in particular, is uniformly bounded over time. We apply the same method to show that the higher-order energies are also almost conserved. See Corollary \ref{cor:almost conservation} for more details.

\begin{rem}
    We compare our structure with the previous works.
    \begin{itemize}
\item In \cite{alazard2015global}, directly in physical space, and in conjunction with a modified scattering approach, Alazard and Delort derive an equation of the form $\partial_t W=O(t^{-1-\delta})$ which allows them to obtain sharp decay estimates. This enables them to bound the lower-order energies uniformly in time. However, they are unable to achieve the same for the highest-order energy.

\item In \cite{Ionescu2015}, Ionescu and Pusateri took the advantages of Wu's structure in \cite{Wu2009}, the normal form, and the modified scattering to prove the global well-posedness. They use Wu's cubic structure to derive energy estimates, and the sharp decay estimate is obtained by using the normal form together with the modified scattering. The lower-order energies are bounded uniformly in this article. However, they are unable to achieve the same for the highest-order energy.

\item In \cite{deng2022wave}, Deng et al. established a novel quartic energy bound through an innovative proof. By combining this result with decay estimates from \cite{Ionescu2015}, they successfully derived uniform energy estimates at the highest derivative order. Their analysis assumed $h\in H^s$,  a condition that particularly constrains low-frequency components.

        \item In contrast, we employ a different method to handle the difficulties mentioned above. First, our work uses a formulation in the physical space that is more explicit and simpler to manipulate. Additionally, we develop a novel technique, which we term the "transition of derivatives" to avoid the loss of derivatives (see \S \ref{section:transition} for a more detailed discussion) and to lower the requirement for the low frequency part of the initial data. Also, such a technique, combining with localization lemma, helps us writing out the further structure of the nonlinearity in the equation, which enables us to establish the uniform-in-time bounds for the highest-order energy and overcoming the limitations of earlier methods.


\end{itemize}
\end{rem}


\subsubsection{The decay estimates} The preceding discussion relies on the a priori assumption that 
\begin{equation}\label{decay:apriori}
    \norm{\zeta_{\alpha}(\cdot,t)-1}_{W^{s-2,\infty}}+\norm{D_t^2\zeta(\cdot,t)}_{W^{s-2,\infty}}+\norm{\partial_{\alpha}D_t\zeta}_{W^{s-3,\infty}}\lesssim \epsilon t^{-1/2}.
\end{equation}

To rigorously establish the decay estimate \eqref{decay:apriori}, we divide the spatial domain into two regions and analyze them separately:

 (1) $|\alpha|$ is far away from $t$. In this region, the solution exhibits better decay. Specifically, when $|\alpha|$ is sufficiently large or small compared to $t$, we can use the fact that the solution decays faster. To be precise, for regions where $|\alpha|\ge t^{1+\mu}$ or $|\alpha|\le t^{1-\mu}$, following Proposition 3.1 in \cite{Wu2009}, we have 
\begin{equation}\label{region:alpha>t}
\alpha^2\partial_{\alpha}(e^{it^2/4\alpha}f)=e^{it^2/4\alpha}\Big(\alpha L_0f-\frac{t}{2}\Omega_0f\Big),
\end{equation}
and
 \begin{equation}\label{region:alpha<t}
        \frac{t^2i}{4}f=-\alpha L_0f+\frac{t}{2}\Omega_0f+\alpha^2\partial_{\alpha}f.
    \end{equation}
\begin{itemize}
    \item [(I)] If $|\alpha|\geq t$, by \eqref{region:alpha>t}, we have
    \begin{equation}\label{region:ggt}
        |f(\alpha,t)|\leq \frac{2}{|\alpha|^{1/2}}\norm{L_0f(\cdot,t)}_{L^2}+\frac{t}{|\alpha|^{3/2}}\norm{\Omega_0f(\cdot,t)}_{L^2},
        \end{equation}
        where $L_0=\frac{1}{2}t\partial_t+\alpha\partial_{\alpha}$ and $\Omega_0=\alpha\partial_t+\frac{1}{2}ti$.
    \item [(II)] For $|\alpha|\leq t$, by \eqref{region:alpha<t},
    \begin{equation}\label{region:llt}
        |f(\alpha,t)|\leq \frac{4|\alpha|}{t^2}\norm{L_0f}_{H^1}+\frac{4}{t}\norm{\Omega_0f}_{H^1}+\frac{8|\alpha|^{3/2}}{t^2}\norm{L_0\partial_{\alpha}f}_{L^2}+\frac{|\alpha|^{1/2}}{t}\norm{\Omega_0\partial_{\alpha}f}_{L^2}.
    \end{equation}
\end{itemize}
If we had
\begin{align*}
\sup_{t\in [0,T]}(\norm{L_0f}_{H^1}+\norm{\Omega_0f}_{H^1})\leq C
\end{align*}
for some $C$ independent of $T$, then we deduce the decay estimate $|f(\alpha,t)|\lesssim t^{-1/2}$. However, for $f=D_t\zeta, D_t^2\zeta$ and $\zeta_{\alpha}-1$, we cannot bound $L_0f$ and $\Omega_0f$ uniformly in time. Instead, 
$$\norm{L_0f}_{H^1}+\norm{\Omega_0f}_{H^1}\lesssim \epsilon t^{\delta}$$
for some $0<\delta\ll 1$. Therefore, for fixed $\mu>2\delta$ and for $|\alpha|\geq t^{1+\mu}$ or $|\alpha|\leq t^{1-\mu}$, by \eqref{region:ggt} and \eqref{region:llt}, we have 
$$|f(\alpha,t)|\lesssim t^{-1/2-\mu/2+\delta_0}.$$

\vspace*{2ex}
 
 (2) For regions $t^{1-\mu}\lesssim |\alpha|\lesssim t^{1+\mu}$, the estimates from \eqref{region:ggt} and \eqref{region:llt} do not provide the sharp $t^{-1/2}$ decay. To handle this, we take advantage of the structure of the solution, particularly focusing on the asymptotic behavior. 

Using the localization lemma, we reformulate the equation:
\begin{equation}\label{decay:reformulate}
(\partial_t^2-i\partial_{\alpha})D_t\theta=-3i\left(\frac{t}{2\alpha}\right)^3|D_t\zeta|^2D_t\zeta+O(t^{-5/4-\delta})_{H^{s-1}}.
\end{equation}

For $\nu$ fixed. Define a functional
$$\tilde{E}(t):=\int D_t^2\theta \bar{u}-D_t\theta \partial_t\bar{u}, \quad u=\chi\left(\frac{\alpha-\nu t}{\sqrt{t}}\right)e^{-it^2/4\alpha}.$$
On one hand, we have 
\begin{equation}\label{decay nu 1}
    \tilde{E}(t)=\sqrt{t}e^{it^2/4\alpha}D_t^2\theta(\nu t,t)+O(\epsilon t^{-1/40}).
\end{equation}
On the other hand, we have
\begin{equation}\label{decay nu 2}
\frac{d}{dt}\tilde{E}(t)=\frac{3i|D_t^2\zeta(\nu t,t)|^2}{\nu}\tilde{E}(t)+O(\epsilon^3t^{-1-\delta}).
\end{equation}
So we obtain
\begin{equation}\label{decay nu 3}
        \norm{\tilde{E}(t)}_{L^\infty}\le C\epsilon^2.
\end{equation}
Combining \eqref{decay nu 1}-\eqref{decay nu 2}-\eqref{decay nu 3}, we obtain 
\begin{equation}
    |D_t^2\theta(\nu t, t)|\lesssim \epsilon t^{-1/2},
\end{equation}
for $t^{-1/5}\lesssim |\nu|\lesssim t^{1/5}$. Therefore, we obtain the desired decay of $D_t^2\theta$ for $\alpha$ near $t$. The decay estimates for $\partial_{\alpha}D_t\theta$ can be established using a similar approach.

\begin{rem}
    Similar functional $\tilde{E}(t)$ was used in \cite{ai2022two}. However, because we have reformulated the water wave equations in a more efficient manner, the estimate for $\tilde{E}(t)$ is significantly simpler in our approach.
    

\end{rem}

\subsubsection{Control of vector fields and the Transition-of -derivatives method}\label{section:transition} To close the decay estimates and the localization lemma, it remains to control $L_0D_t\zeta$, $L_0D_t^2\zeta$ and $L_0(\zeta_{\alpha}-1)$. By Remark \ref{rem:whyL0s-1}, we need to control $\norm{L_0f}_{H^{s-1}}$ for $f=D_t\zeta$ and $D_t^2\zeta$.
The estimates for $\Omega_0$ can be derived from the formula
\begin{equation}
    \Omega_0\partial_{\alpha}f=L_0\partial_tf-\frac{t}{2}(\partial_t^2-i\partial_{\alpha})f.
\end{equation}
We direct readers to Corollary \ref{cor:Omega_0} for further details. To estimate $L_0D_t\theta$, we apply $L_0$ on both sides of the cubic equation \eqref{intro:cubic} and obtain
\begin{equation}
\begin{split}
   & (D_t^2-iA\partial_{\alpha})L_0\theta=L_0G+G\\
    &-\Big\{L_0D_tb-(L_0b-\frac{1}{2}b)b_{\alpha}\Big\}\partial_{\alpha}\theta+\Big(L_0b-\frac{1}{2}b\Big)\Big(D_t\partial_{\alpha}+\partial_{\alpha}D_t\Big)\theta-i(L_0A)\partial_{\alpha}\theta.
\end{split}
\end{equation}
To estimate $\norm{L_0D_t^2\theta}_{H^{s-1}}$, we must estimate $\norm{L_0b\partial_{\alpha}D_t^2\theta}_{H^{s-1}}$. The most challenging aspect of this arises from the term $\norm{L_0b\partial_{\alpha}^s D_t^2\theta}_{L^2}$, which presents significant difficulties in the analysis. To illustrate this, we can approximately treat $b=D_t\zeta (\bar{\zeta}_{\alpha}-1)$. Thus, we have the expression
$$L_0b \partial_{\alpha}^s D_t^2\theta=(\bar{\zeta}_{\alpha}-1) \partial_{\alpha}^sD_t\theta L_0D_t\zeta+ D_t\zeta\partial_{\alpha}^sD_t\theta L_0(\bar{\zeta}_{\alpha}-1).$$
The more problematic term is the second one, since we have to estimate $\partial_{\alpha}^sD_t\theta$ in $L^2$ (without any decay), and we also lack decay estimates for $L_0(\bar{\zeta}_{\alpha}-1)$. Furthermore, the decay for 
$D_t\zeta$ is only of order $t^{-1/4}$, as shown in Lemma \ref{decay of D_tzeta}. Therefore, it appears that
$$\norm{D_t\zeta\partial_{\alpha}^sD_t^2\theta L_0(\bar{\zeta}_{\alpha}-1)}_{L^2}\lesssim \norm{D_t\zeta}_{L^{\infty}}\norm{L_0(\bar{\zeta}_{\alpha}-1)}_{L^{\infty}}\norm{D_t^2\zeta}_{H^s}\lesssim \epsilon t^{-1/4}\norm{L_0(\zeta_{\alpha}-1)}_{H^1}\norm{D_t^2\zeta}_{H^s},$$
which leads to insufficient decay. 

To address this issue, we employ a Transition-of-Derivatives method, which we outline below. Using \eqref{region:alpha>t} and \eqref{region:alpha<t}, we derive the following lemma (refer to Lemma \ref{transition-lemma}).
\begin{lemma}\label{transition-derivatives}
    We have
    \begin{equation}
     f(\alpha,t)=-\frac{2\alpha}{it}\partial_tf(\alpha,t)+\frac{2}{it}\Omega_0 f(\alpha,t),
     \end{equation}
    \begin{equation}
     \partial_t f(\alpha,t)=-\frac{it}{2\alpha}f(\alpha,t)+\frac{1}{\alpha}\Omega_0 f(\alpha,t),
     \end{equation}
     \begin{equation}\label{raise the derivative}
         f(\alpha,t)=-\frac{4\alpha}{it^2}L_0f+\frac{2}{it}\Omega_0 f+\frac{4\alpha^2}{it^2}\partial_{\alpha}f,
     \end{equation}
          \begin{equation}
         \partial_{\alpha}f(\alpha,t)=\frac{it^2}{4\alpha^2}f+\frac{1}{\alpha}L_0f-\frac{t}{2\alpha^2}\Omega_0 f,
     \end{equation}
\end{lemma}
\begin{rem}
This lemma allows us to transit the derivatives in the following way:
\begin{itemize}
    \item For a single function $f$, the lemma is especially useful in the region where $|\alpha|$ is close to $t$, since $\partial_{\alpha}f=\frac{it^2}{4\alpha^2}f$ plus a fast decay error term $\frac{1}{\alpha}L_0f-\frac{t}{2\alpha^2}\Omega_0 f$, provided that $L_0f$ and $\Omega_0f$ grow slowly. A similar argument can be applied to $\partial_tf$. This capability is particularly advantageous as it enables us to relax the assumptions on the low-frequency components of $f$. It can also be used to relax the assumptions on high-frequency too. We hope to explore this feature in future works.

    \item For products like $f\partial_{\alpha}g$, as long as we can control $L_0f, \Omega_0f$ and $L_0g, \Omega_0g$, we can transit the derivatives of $g$ to $f$, incurring only a negligible cost. See Lemma \ref{lemma:transit derivatives}, Lemma \ref{lemma:transit2} for further details. Using these lemmas, we obtain the estimates
    $$\norm{D_t\zeta \partial_{\alpha}^sD_t^2\zeta}_{L^{\infty}(|\alpha|\geq t^{3/4})}\lesssim \epsilon^2 t^{-1};$$
    for $|\alpha|<t^{3/4}$, we apply (\ref{raise the derivative}) to obtain
    \[
    \norm{D_t\zeta}_{L^\infty(|\alpha|<t^{3/4})}\le C\epsilon t^{-1+\delta_0}\ln t.
    \]
    Combining them together yields
    \[
    \norm{D_t\zeta\partial_{\alpha}^sD_t^2\theta}_{L^2}\le C\epsilon^2t^{-1+\delta_0}\ln t
    \]
    which is sufficient to close the highest order energy estimates for $L_0D_t^2\theta$. For other orders, there holds a sharper estimate since Lemma \ref{lemma:transit2} holds for all $\alpha\in\mathbb{R}$ in such case.
\end{itemize}
\end{rem}




\subsection{Outline of the paper} In \S 2, we present fundamental analytical tools and key formulas that will be used throughout the paper. In \S 3, we formulate a main proposition and demonstrate how it leads to the proof of the main theorem. In \S 4, under the bootstrap assumptions, we establish estimates for relevant quantities that are crucial for both the energy and decay estimates. In \S 5, we provide the energy estimates. In \S 6, we extend these estimates to include the energy associated with vector fields. In \S 7, we prove the decay estimates for $\alpha$ near $t$. Finally, in the appendix, we compile essential tools from harmonic analysis for reference.

\subsection{Notation and convention}\label{notation} Assume $f$ a function on the boundary of $\Omega(t)$. By saying $f$ holomorphic,  we mean $f$ is the boundary value of a holomorphic function in $\Omega(t)$. 
The commutator $[A,B]=AB-BA$. By $f=O(A)_{L^2}$, we mean $\norm{f}_{L^2}\leq CA$ for some constant $C>0$. By $f=O(A)$, we mean $\norm{f}_{L^{\infty}}\leq CA$ for some constant $C>0$. Throughout this paper, if not specified, a constant $C$ always means a constant that depends on $s$ only.

\section{Preliminaries}
\subsection{The Hilbert transform and the characterization of holomorphic functions}
\begin{defn}[Hilbert transform]
Assume that $z(\alpha)$ satisfies
\begin{equation}\label{chordchordarcarc}
    \beta_0|\alpha-\beta|\leq |z(\alpha)-z(\beta)|\leq \beta_1|\alpha-\beta|, \quad \quad \forall \alpha,\beta\in \mathbb{R},
\end{equation}
where $0<\beta_0<\beta_1<\infty$ are two absolute constants.
We define the Hilbert transform associated to a curve $z(\alpha)$ as 
\begin{equation}
\mathfrak{H}f(\alpha):=\frac{1}{\pi i}p.v.\int_{-\infty}^{\infty}\frac{z_{\beta}(\beta)}{z(\alpha)-z(\beta)}f(\beta)d\beta.
\end{equation}
We use the notation $\mathcal{H}$ to denote the Hilbert transform associated with $\zeta$:
$$\mathcal{H}f(\alpha,t):=\frac{1}{\pi i}p.v.\int_{-\infty}^{\infty}\frac{\zeta_{\beta}}{\zeta(\alpha,t)-\zeta(\beta,t)}f(\beta,t)d\beta.$$
The standard Hilbert transform is the Hilbert transform associated with $z(\alpha)=\alpha$, which is denoted by
\begin{equation}
\mathbb{H}f(\alpha):=\frac{1}{\pi i}p.v.\int_{-\infty}^{\infty}\frac{1}{\alpha-\beta}f(\beta)d\beta.
\end{equation}
\end{defn}
It's well-known (see \cite{david1984operateurs}, Theorem 6) that if $\zeta(\alpha)$ satisfies (\ref{chordchordarcarc}), then $\mathcal{H}$ is bounded on $L^2$.
\begin{lemma}\label{boundednesshilbert}
Assume that $\zeta(\alpha)$ satisfies (\ref{chordchordarcarc}), then 
\begin{equation}
    \|\mathcal{H}f\|_{L^2}\leq C\|f\|_{L^2},
\end{equation}
for some constant that depends on $\beta_0$ and $\beta_1$ only.
\end{lemma}
We can use the Hilbert transform to characterize the boundary value of holomorphic functions. 
\begin{lemma}\label{holomorphic}
Let $f\in L^2(\mathbb{R})$. Then $f$ is the boundary value of a holomorphic function in $\Omega(t)$ if and only if $(I-\mathcal{H})f=0$. $f$ is the boundary value of a holomorphic function in $\Omega(t)^c$ if and only if $(I+\mathcal{H})f=0$. 
\end{lemma}

\subsection{Water waves in Wu's coordinates}
As discussed in the introduction, we formulate the water waves by the following
	\begin{equation}\label{system_newvariables}
	\begin{cases}
	(D_t^2-iA\partial_{\alpha})\zeta=-i\\
	D_t\bar{\zeta}, ~~\bar{\zeta}(\alpha,t)-\alpha\quad \text{holomorphic}
	\end{cases}
	\end{equation}
where
\begin{equation}\label{eq:Dt}
    D_{t}:=\partial_t+b\partial_{\alpha}
\end{equation}for some real-valued function $b$.

By Lemma \ref{holomorphic}, due to the holomorphic conditions in \eqref{system_newvariables}, we have
	\begin{equation}\label{holo:Dtzeta}
	    (I-\mathcal{H})D_t\bar{\zeta}=0,
	\end{equation} and 
	\begin{equation}\label{holo:barzeta}
	    (I-\mathcal{H})(\bar{\zeta}(\alpha,t)-\alpha)=0.
	\end{equation}

\subsection{Boundary value of the velocity potential} Note that we have Bernoulli equation (define $\phi$ as the velocity potential, i.e. $\nabla\phi =v$.)
\begin{equation}
    \phi_t+\frac{1}{2}|v|^2+y=P.
\end{equation}

Define $\psi=\phi(z(\alpha,t),t)$ and $Q=(I-\mathcal{H})(\psi\circ\kappa^{-1})$, then we have
\begin{proposition}\label{formula for derivatives of Q}
    \begin{equation}\label{formula of Q_alpha}
        \begin{aligned}
            \bar{Q}_\alpha
            &=D_t\bar{\zeta}+(\zeta_\alpha-1)D_t\bar{\zeta}-\frac{1}{2}\left(\zeta_\alpha\mathcal{H}\frac{1}{\zeta_\alpha}+\bar{\zeta}_\alpha\bar{\mathcal{H}}\frac{1}{\bar{\zeta}_\alpha}\right)(\zeta_\alpha D_t\bar{\zeta}),
        \end{aligned}
    \end{equation}
    and
    \begin{equation}\label{formula of Dt Q}
        D_tQ=-[D_t\zeta,\mathcal{H}]\frac{Re\{\bar{\zeta}_\alpha D_t\zeta\}}{\zeta_\alpha}+\frac{1}{2}(I-\mathcal{H})|D_t\zeta|^2-\frac{1}{2i}(I-\mathcal{H})(\zeta-\bar{\zeta}).
    \end{equation}
\end{proposition}

\begin{proof}
    We first prove (\ref{formula of Dt Q}). Recall that
    \[
    \phi_t+\frac{1}{2}|z_t|^2+y=P,
    \]
    and we restrict it to the boundary to obtain
    \[
    \psi_t-\frac{1}{2}|z_t|^2+y=0,
    \]
    since the pressure vanishes on the boundary. This implies
    \[
    D_t(\psi\circ\kappa^{-1})-\frac{1}{2}|D_t\zeta|^2+\eta=0
    \]
    where $\eta=Im\ \zeta$. Apply both sides by $I-\mathcal{H}$ to get
    \[
    (I-\mathcal{H})D_t(\psi\circ\kappa^{-1})=\frac{1}{2}(I-\mathcal{H})|D_t\zeta|^2-\frac{1}{2i}(I-\mathcal{H})(\zeta-\bar{\zeta});
    \]
    we further deduce that
    \begin{align*}
        (I-\mathcal{H})D_t(\psi\circ\kappa^{-1})
        &=D_tQ+[D_t,\mathcal{H}](\psi\circ\kappa^{-1}),\\
        &=D_tQ+[D_t\zeta,\mathcal{H}]\frac{\partial_\alpha(\psi\circ\kappa^{-1})}{\zeta_\alpha}.
    \end{align*}
    Since $\nabla\psi(\zeta,t)=v(\zeta,t)=D_t\zeta$, we have
    \begin{equation}\label{formula for psi_alpha}
        \partial_\alpha(\psi\circ\kappa^{-1})=\partial_\alpha\psi(\zeta(\alpha,t),t)=\nabla\psi(\zeta,t)\cdot\partial_\alpha\zeta=Re\left\{\bar{\zeta}_\alpha D_t\zeta\right\},
    \end{equation}
    and (\ref{formula of Dt Q}) is thus proved.

    Next we derive (\ref{formula of Q_alpha}). According to (\ref{formula for psi_alpha}) and $\partial_\alpha\mathcal{H}f=\zeta_\alpha\mathcal{H}\frac{\partial_\alpha f}{\zeta_\alpha}$, there holds
    \begin{align*}
        Q_\alpha
        =\partial_\alpha(I-\mathcal{H})(\psi\circ\kappa^{-1})
        =\left(I-\zeta_\alpha\mathcal{H}\frac{1}{\zeta_\alpha}\right)\partial_\alpha(\psi\circ\kappa^{-1})
        =\left(I-\zeta_\alpha\mathcal{H}\frac{1}{\zeta_\alpha}\right)Re\{\zeta_\alpha D_t\bar{\zeta}\},
    \end{align*}
    and thus
    \[
    Re\{\zeta_\alpha D_t\bar{\zeta}\}
    =Q_\alpha+\zeta_\alpha\mathcal{H}\frac{1}{\zeta_\alpha}Re\{\zeta_\alpha D_t\bar{\zeta}\}
    =ReQ_\alpha+\frac{1}{2}\left(\zeta_\alpha\mathcal{H}\frac{1}{\zeta_\alpha}+\bar{\zeta}_\alpha\bar{\mathcal{H}}\frac{1}{\bar{\zeta}_\alpha}\right)Re\{\zeta_\alpha D_t\bar{\zeta}\}.
    \]
    Also, since $(I-\mathcal{H})D_t\bar{\zeta}=0$, there holds
    \[
    \left(I-\zeta_\alpha\mathcal{H}\frac{1}{\zeta_\alpha}\right)(\zeta_\alpha D_t\bar{\zeta})=0,
    \]
    hence
    \[
    Q_\alpha=-i\left(I-\zeta_\alpha\mathcal{H}\frac{1}{\zeta_\alpha}\right)Im\{\zeta_\alpha D_t\bar{\zeta}\}.
    \]
    Similarly, from such formula we derive that
    \[
    Im\{\zeta_\alpha D_t\bar{\zeta}\}
    =iQ_\alpha+\zeta_\alpha\mathcal{H}\frac{1}{\zeta_\alpha}Im\{\zeta_\alpha D_t\bar{\zeta}\}
    =Re\{iQ_\alpha\}+\frac{1}{2}\left(\zeta_\alpha\mathcal{H}\frac{1}{\zeta_\alpha}+\bar{\zeta}_\alpha\bar{\mathcal{H}}\frac{1}{\bar{\zeta}_\alpha}\right)Im\{\zeta_\alpha D_t\bar{\zeta}\}.
    \]
    Combining them yields (\ref{formula of Q_alpha}).
\end{proof}

\subsection{Formulae for important quantities}\label{section:structure}
To obtain a closed system in \eqref{system_newvariables}, we need formulae for quantities such as $b$, $A$, $D_tb$, $\frac{a_t}{a}\circ \kappa^{-1}$ in terms of the unknown $\zeta$. We record the formulae in the following and refer the readers to section 2 of \cite{Wu2009}.

\subsubsection{Formula for the quantities $b$ and $D_tb$}
We have 
\begin{equation}\label{forb}
\begin{split}
(I-\mathcal{H})b
=& -[D_t\zeta, \mathcal{H}]\frac{\bar{\zeta}_{\alpha}-1}{\zeta_{\alpha}}.
\end{split}
\end{equation}
And
\begin{equation}
\begin{split}
(I-\mathcal{H})D_t b=&[D_t\zeta,\mathcal{H}]\frac{\partial_{\alpha}b}{\zeta_{\alpha}}-[D_t^2\zeta,\mathcal{H}]\frac{\bar{\zeta}_{\alpha}-1}{\zeta_{\alpha}}-[D_t\zeta,\mathcal{H}]\frac{\partial_{\alpha}D_t\bar{\zeta}}{\zeta_{\alpha}}\\
&+\frac{1}{\pi i}\int \Big(\frac{D_t\zeta(\alpha,t)-D_t\zeta(\beta,t)}{\zeta(\alpha,t)-\zeta(\beta,t)}\Big)^2(\bar{\zeta}_{\beta}(\beta,t)-1)d\beta.
\end{split}
\end{equation}
So $b$ and $D_tb$ is quadratic.

\subsubsection{The quantity $A$} We have
\begin{equation}
\begin{split}
(I-\mathcal{H})A=&1+i[D_t\zeta,\mathcal{H}]\frac{\partial_{\alpha}D_t\bar{\zeta}}{\zeta_{\alpha}}+i[D_t^2\zeta,\mathcal{H}]\frac{\bar{\zeta}_{\alpha}-1}{\zeta_{\alpha}}.
\end{split}
\end{equation}
So $A-1$ is quadratic.

\subsubsection{The quantity $\frac{a_t}{a}\circ\kappa^{-1}$} We have
\begin{equation}\label{atanew}
\begin{split}
&(I-\mathcal{H})\frac{a_t}{a}\circ\kappa^{-1}A\bar{\zeta}_{\alpha}\\=&2i[D_t^2\zeta, \mathcal{H}]\frac{\partial_{\alpha}D_t\bar{\zeta}}{\zeta_{\alpha}}+2i[D_t\zeta, \mathcal{H}] \frac{\partial_{\alpha}D_t^2\bar{\zeta}}{\zeta_{\alpha}}-\frac{1}{\pi }\int \Big(\frac{D_t\zeta(\alpha,t)-D_t\zeta(\beta,t)}{\zeta(\alpha,t)-\zeta(\beta,t)}\Big)^2 (D_t\bar{\zeta})_{\beta}d\beta.
\end{split}
\end{equation}
So $\frac{a_t}{a}\circ\kappa^{-1}$ is quadratic.

	\subsection{Cubic structure}
To obtain the long-time existence of $\zeta$, we need a cubic structure for the system \eqref{system_newvariables}.

Setting $\theta:=(I-\mathcal{H})(\zeta-\bar{\zeta})$. Since $(I-\mathcal{H})(\bar{\zeta}-\alpha)=0$, then one has
\begin{equation}
    \theta=(I-\mathcal{H})(\zeta-\alpha).
\end{equation}
In \cite{Wu2009}, Wu derived the following cubic structure.
	\begin{equation}\label{cubic}
	\begin{split}
	(D_t^2-iA\partial_{\alpha})\theta=&-2[D_t\zeta, \mathcal{H}\frac{1}{\zeta_{\alpha}}+\bar{\mathcal{H}}\frac{1}{\bar{\zeta}_{\alpha}}]\partial_{\alpha}D_t\zeta+\frac{1}{ \pi i}\int \Big(\frac{D_t\zeta(\alpha)-D_t\zeta(\beta)}{\zeta(\alpha,t)-\zeta(\beta,t)}\Big)^2\partial_{\beta}(\zeta-\bar{\zeta})\,d\beta\\
	:=& G_1+G_2,
	\end{split}
	\end{equation}
where $G_1$ and $G_2$ are cubic or higher-power nonlinearities. 

\subsection{Vector fields and equations for vector fields} We need to use the following vector fields.
\begin{defn}
    Let $$L_0=\frac{1}{2}t\partial_t+\alpha\partial_{\alpha}, \quad \Omega_0=\alpha\partial_t+\frac{1}{2}ti.$$
Define
$$\tilde{\Omega}_0:=\bar{\Omega}_0=\alpha\partial_t-\frac{1}{2}ti.$$
\end{defn}
\begin{rem}
    We make the following remarks.
    \begin{itemize}
        \item [(1)] We use the formula $\Omega_0\partial_{\alpha}=L_0\partial_t-\frac{t}{2}(\partial_t^2-i\partial_{\alpha})$ to derive the estimates for $\Omega_0$;

        \item [(2)] If $f$ is almost anti-holomorphic, then we control $\Omega_0 f$. If $f$ is almost holomorphic, then we control $\tilde{\Omega}_0f$. 
    \end{itemize}
\end{rem}

Let $\mathcal{P}:=D_t^2-iA\partial_{\alpha}$. We have 
\begin{equation}
\begin{split}
    [L_0, \mathcal{P}]=& -\mathcal{P}+\Big\{L_0D_tb-(L_0b-\frac{1}{2}b)b_{\alpha}\Big\}\partial_{\alpha}\\
    &+\Big(L_0b-\frac{1}{2}b\Big)\Big(D_t\partial_{\alpha}+\partial_{\alpha}D_t\Big)-i(L_0A)\partial_{\alpha}.
\end{split}
\end{equation}

\subsection{Energy functional}
Define
\begin{equation}
E_k^{\theta}:=\int \frac{1}{A}|D_t\theta_k|^2+i\theta_k\partial_\alpha\bar{\theta}_kd\alpha.
\end{equation}
where $\theta_k=\partial_\alpha^k\theta$. By Wu's basic energy lemma (lemma 4.1, \cite{Wu2009}), we have
\begin{equation}\label{energy1}
\frac{d}{dt}E_k^{\theta}=Re\int\frac{2}{A}D_t\bar{\theta}_kG_k^\theta-\int \frac{1}{A}\frac{a_t}{a}\circ\kappa^{-1}|D_t\theta_k|^2
\end{equation}

\noindent Define
\begin{equation}
E_k^{\sigma}:=\int \frac{1}{A}|D_t\sigma_k|^2+i\sigma_k\partial_\alpha\bar{\sigma}_kd\alpha.
\end{equation}
where $\sigma_k=\partial_\alpha^k\sigma=\partial_\alpha^kD_t\theta$.
Then we have 
\begin{equation}\label{energy2}
\frac{d}{dt}E_k^{\sigma}=Re\int\frac{2}{A}D_t\bar{\sigma}_k G_k^\sigma-\int \frac{1}{A}\frac{a_t}{a}\circ\kappa^{-1}|D_t\sigma_k|^2
\end{equation}
Define
\begin{equation}
\mathcal{E}_s:=\sum_{k=0}^s(E_k^{\theta}+E_k^{\sigma}).
\end{equation}
Define
\begin{equation}\label{energy3}
E_{L_0\theta}^k=\int \frac{1}{A}|D_t L_0\theta_k|^2+iL_0\theta_k\partial_\alpha L_0\bar{\theta}_kd\alpha
\end{equation}
Define
\begin{equation}\label{energy4}
E_{L_0\sigma}^k=\int \frac{1}{A}|D_t L_0\sigma_k|^2+iL_0\sigma_k\partial_\alpha L_0\bar{\sigma}_kd\alpha
\end{equation}
Define
\begin{equation}
    \mathcal{E}_s^{L_0}:=\sum_{k=0}^{s-1}(E_{L_0\theta}^{k}+E_{L_0\sigma}^k).
\end{equation}

\section{Main proposition and the proof of the Main Theorem}

\subsection{Bootstrap assumptions}

We make the following bootstrap assumptions. Given $T\gg 1$, $0<\delta_0\ll 1$ sufficiently small. On $[0,T]$, assume
\begin{enumerate}
    \item 
\begin{equation}\label{bootstrap-1}
    \sup_{t\in [0,T]}\Big(\norm{\Lambda(\zeta-\alpha)}_{H^s}+\norm{D_t^2\zeta}_{H^{s}}+\norm{D_t\zeta}_{H^{s+1/2}}\Big)\leq 100\epsilon.
\end{equation}

\item  For all $t\in [0,T]$,
\begin{equation}\label{bootstrap-2}
    \Big(\norm{L_0D_t^2\zeta}_{H^{s-2}} +\norm{L_0D_t\zeta}_{H^{s-3/2}}\Big)\leq 100\epsilon (1+t)^{\delta_0},
\end{equation}

\begin{equation}\label{bootstrap-2'}
    \Big(\norm{L_0D_t^2\zeta}_{\dot{H}^{s-1}}+\norm{L_0D_t\zeta}_{\dot{H}^{s-1}}\Big)\leq 100\epsilon (1+t)^{2\delta_0},
\end{equation}

and
\begin{equation}
     \norm{\Omega_0\theta_{\alpha}}_{H^{s-2}}+\norm{\partial_{\alpha}\Omega_0  Q_{\alpha}}_{H^{s-3}}\leq 100\epsilon (1+t)^{\delta_0},
\end{equation}

\begin{equation}
     \norm{\Omega_0\theta_{\alpha}}_{\dot{H}^{s-1}}+\norm{\Omega_0  Q_{\alpha}}_{\dot{H}^{s-1}}\leq 100\epsilon (1+t)^{2\delta_0}.
\end{equation}

\item For all $t\in [0,T]$, 
\begin{equation}\label{bootstrap-3}
    \norm{(\partial_{\alpha}D_t\zeta, \zeta_{\alpha}-1, D_t^2\zeta)}_{W^{s-3,\infty}\times W^{s-2,\infty}\times W^{s-2,\infty}}\leq 100\epsilon t^{-1/2}.
\end{equation}
\end{enumerate}

\subsection{Main proposition}

\begin{proposition}\label{main-prop}
    Under the bootstrap assumptions \eqref{bootstrap-1}-\eqref{bootstrap-2}-\eqref{bootstrap-3}, we have the improved bounds
    \begin{itemize}

\item [(1)] 
\begin{equation}\label{improve-1}
    \sup_{t\in [0,T+1]}\Big(\norm{\Lambda(\zeta-\alpha)}_{H^s}+\norm{D_t^2\zeta}_{H^{s}}+\norm{D_t\zeta}_{H^{s+1/2}}\Big)\leq 10\epsilon.
\end{equation}

\item [(2)] For all $t\in [0,T+1]$,
\begin{equation}\label{improve-2}
    \Big( \norm{L_0D_t^2\zeta}_{H^{s-2}} +\norm{L_0D_t\zeta}_{H^{s-2}}\Big)\leq 10\epsilon (1+t)^{\delta_0},
\end{equation}
\begin{equation}\label{improve-2'}
    \Big( \norm{L_0D_t^2\zeta}_{\dot{H}^{s-1}} +\norm{L_0D_t\zeta}_{\dot{H}^{s-1}}\Big)\leq 10\epsilon (1+t)^{2\delta_0},
\end{equation}

and
\begin{equation}
     \norm{\Omega_0\theta_{\alpha}}_{H^{s-2}}+\norm{\partial_{\alpha}\Omega_0  Q_{\alpha}}_{H^{s-3}}\leq 10\epsilon (1+t)^{\delta_0},
\end{equation}

\begin{equation}
     \norm{\Omega_0\theta_{\alpha}}_{\dot{H}^{s-1}}+\norm{\Omega_0  Q_{\alpha}}_{\dot{H}^{s-1}}\leq 10\epsilon (1+t)^{2\delta_0}.
\end{equation}

\item [(3)] For all $t\in [0,T+1]$, 
\begin{equation}\label{improve-3}
    \norm{(\partial_{\alpha}D_t\zeta, \zeta_{\alpha}-1, D_t^2\zeta)}_{W^{s-3,\infty}\times W^{s-2,\infty}\times W^{s-2,\infty}}\leq 10\epsilon t^{-1/2}
\end{equation}
\end{itemize}
\end{proposition}

\begin{proof}
    The proposition follows from Corollary \ref{sec 5:energy bound}, Corollary \ref{sec 6:L0bound}, Corollary \ref{cor:Omega_0}, Corollary \ref{sec 7:decay bound}, Lemma \ref{bouns decay for small or big alpha}, and the local wellposedness of the water wave system.
\end{proof}

\subsection{Proof of the main theorem}
Given the main proposition, we prove the main theorem.
\begin{proof}[Proof of Theorem \ref{main}]
Denote
\begin{equation}
    \mathcal{T}:=\{T\in [0,\infty): \forall t\in [0,T], \eqref{bootstrap-1}-\eqref{bootstrap-2}-\eqref{bootstrap-2'}-\eqref{bootstrap-3}~hold\}.
\end{equation}
Since $0\in \mathcal{T}$, so $\mathcal{T}\neq \emptyset$. Clearly, $\mathcal{T}$ is a closed set. Let $T\in \mathcal{T}$, by Proposition \ref{main-prop}, $\mathcal{T}$ is also an open set in $[0,\infty)$. So $\mathcal{T}=[0,\infty)$. This completes the proof of Theorem \ref{main}.
\end{proof}

\section{Estimates for relevant quantities}
In this section, if not specified, we will always assume the bootstrap assumptions \eqref{bootstrap-1}-\eqref{bootstrap-2}-\eqref{bootstrap-2'}-\eqref{bootstrap-3}.

First, we show that $\zeta(\alpha,t)-\zeta(\beta,t)$ behaves like $\alpha-\beta$ in the following sense.
\begin{lemma}\label{nearly alpha}
    We have 
    \begin{equation}
        \sup_{\alpha,\beta\in \mathbb{R}}\Big|\frac{\zeta(\alpha,t)-\zeta(\beta,t)}{\alpha-\beta}-1\Big|\leq C\epsilon (1+t)^{-1/2}.
    \end{equation}
\end{lemma}
\begin{proof}
    We have
    \begin{align*}
        \Big|\frac{\zeta(\alpha,t)-\zeta(\beta,t)}{\alpha-\beta}-1\Big|=&\Big| \frac{\zeta(\alpha,t)-\alpha-(\zeta(\beta,t)-\beta)}{\alpha-\beta}\Big|\leq \norm{\zeta_{\alpha}-1}_{\infty}\leq C\epsilon t^{-1/2}.
    \end{align*}
\end{proof}

\begin{lemma}\label{decay of D_tzeta}
    Assume the bootstrap assumptions. There holds
    \begin{equation}\label{equ: decay dtzeta}
        \norm{D_t\zeta}_{L^\infty}\le C\epsilon t^{-\frac{1}{4}},
    \end{equation}
    \begin{equation}\label{equ: decay second highest}
        \norm{\partial_\alpha^{s-1}D_t\zeta}_{L^\infty}\le C\epsilon t^{-\frac{1}{4}}\ln t.
    \end{equation}
    and
    \begin{equation}\label{equ: decayhighest}
       \norm{\partial_{\alpha}^{s-1}D_t^2\zeta}_{L^{\infty}} +\norm{\partial_\alpha^s\zeta}_{L^\infty}\le C\epsilon t^{-\frac{1}{6}}\ln t.
    \end{equation}
\end{lemma}

\begin{proof}
Define $f_L:=P_{\leq t^{-1/2}}f$ and $f_H:=f-f_L$. On one hand, according to Bernstein's inequality (Lemma \ref{Bernstein}),
    \[
    \norm{(D_t\zeta)_L}_{L^\infty}\le Ct^{-\frac{1}{4}}\norm{(D_t\zeta)_L}_{L^2}\le C\epsilon t^{-\frac{1}{4}};
    \]
    on the other hand, we have
    \[
    \norm{(D_t\zeta)_H}_{L^\infty}\le Ct^{\frac{1}{4}}\norm{\partial_\alpha(D_t\zeta)_H}_{L^\infty}\le C\epsilon t^{-\frac{1}{4}}.
    \]
    This proves \eqref{equ: decay dtzeta}. By Bernstein's inequality, and the bootstrap assumptions,
    \begin{align*}
\norm{\partial_\alpha^{s-1}D_t\zeta}_{L^\infty}\leq &\ \norm{P_{\leq t^{1/4}}\partial_\alpha^{s-1}D_t\zeta}_{L^\infty}+\norm{P_{\geq t^{1/4}}\partial_\alpha^{s-1}D_t\zeta}_{L^\infty}\\
\leq &\  t^{1/4}\norm{P_{\leq t^{1/4}}\partial_\alpha^{s-2}D_t\zeta}_{L^\infty}+\sum_{j\geq 0: 2^j\leq t^2}\norm{P_{2^j t^{1/4}}\partial_\alpha^{s-1}D_t\zeta}_{L^\infty}+\norm{P_{\geq t^{9/4}}\partial_\alpha^{s-1}D_t\zeta}_{L^\infty}\\
\leq &\  Ct^{1/4-1/2}\epsilon+\sum_{j\geq 0: 2^j\leq t^2}2^{-j}t^{-1/4}\norm{\partial_{\alpha}^sD_t\zeta}_{\dot{H}^{1/2}}+C\epsilon t^{-1}\norm{\partial_{\alpha}^s D_t\zeta}_{L^2}\\
\leq &\  C\epsilon t^{-1/4}\ln t.
    \end{align*}
    This completes the proof of \eqref{equ: decay second highest}. The proof for \eqref{equ: decay second highest} is similar to that of \eqref{equ: decayhighest}. So the proof of the lemma is completed.
\end{proof}

Also, we record some formla that we will use frequently here.
\begin{lemma}\label{transition-lemma}
    We have
    \begin{equation}\label{equ: lemma transition sec 4 1}
     f(\alpha,t)=-\frac{2\alpha}{it}\partial_tf(\alpha,t)+\frac{2}{it}\Omega_0 f(\alpha,t),
     \end{equation}
    \begin{equation}\label{equ: lemma transition sec 4 2}
     \partial_t f(\alpha,t)=-\frac{it}{2\alpha}f(\alpha,t)+\frac{1}{\alpha}\Omega_0 f(\alpha,t),
     \end{equation}
     \begin{equation}\label{equ:increase a derivative}
         f(\alpha,t)=-\frac{4\alpha}{it^2}L_0f+\frac{2}{it}\Omega_0 f+\frac{4\alpha^2}{it^2}\partial_{\alpha}f,
     \end{equation}
          \begin{equation}\label{equ:increase a derivative 2}
         \partial_{\alpha}f(\alpha,t)=\frac{it^2}{4\alpha^2}f+\frac{1}{\alpha}L_0f-\frac{t}{2\alpha^2}\Omega_0 f,
     \end{equation}
\end{lemma}
\begin{proof}
Let $F(\alpha,t):=e^{it^2/4\alpha}f(\alpha,t)$, then 
\begin{align*}
\partial_t F=\partial_t(e^{it^2/4\alpha}f)=&\frac{it}{2\alpha}e^{it^2/4\alpha}f+e^{it^2/4\alpha}f_t=\frac{1}{\alpha}e^{it^2/4\alpha}(\alpha\partial_t+\frac{it}{2})f=\frac{1}{\alpha}e^{it^2/4\alpha}\Omega_0f.
\end{align*}
Therefore,
\begin{align*}
    f_t=\partial_t(e^{-it^2/4\alpha}F)=-\frac{it}{2\alpha}f+e^{-it^2/4\alpha}F_t=-\frac{it}{2\alpha}f+\frac{1}{\alpha}\Omega_0f,
\end{align*}
which proves \eqref{equ: lemma transition sec 4 2}. \eqref{equ: lemma transition sec 4 1} is equivalent to \eqref{equ: lemma transition sec 4 2}. \eqref{equ:increase a derivative} and \eqref{equ:increase a derivative 2} follow in a similar way. 
\end{proof}
Lemma \ref{transition-lemma} allows us to increase or decrease the derivatives of $f$ provided that we can control $\Omega_0 f$ and $L_0f$. We can use the bootstrap assumptions \eqref{bootstrap-2}-\eqref{bootstrap-3} to estimate $\norm{L_0D_t\zeta}_{H^{s-1/2}}$, $\norm{L_0(\zeta_{\alpha}-1)}_{H^{s-1}}$, $\norm{\Omega_0 \theta_{\alpha}}_{H^{s-1}}$, etc. We cannot directly control $\Omega_0 Q_{\alpha}$ in $L^2$. Instead, we have the following. 

\begin{lemma}
For any fixed $0<k\le 2$, we have
\begin{equation}\label{Omega_0D_t zeta high}
\norm{\Omega_0 P_{\geq t^{-k}} Q_{\alpha}}_{L^{\infty}}\leq C\epsilon t^{\delta_0}\ln t. 
\end{equation}
\begin{equation}\label{L0D_t zeta high}
\norm{L_0 P_{\geq t^{-k}}Q_{\alpha}}_{L^2}\leq C\epsilon t^{\delta_0}.
\end{equation}
\end{lemma}
\begin{proof}
We have 
    $$\Omega_0 P_{\geq t^{-k}} Q_{\alpha}=P_{\geq t^{-k}}\Omega_0 Q_{\alpha}+[\Omega_0, P_{\geq t^{-k}}]Q_{\alpha}=P_{\geq t^{-k}}\Omega_0 Q_{\alpha}+[\alpha\partial_t, P_{\geq t^{-k}}]Q_{\alpha}$$
By the bootstrap assumptions and the Bernstein inequality (Lemma \ref{Bernstein}), we have
\begin{align*}
    \norm{P_{\geq t^{-k}}\Omega_0 Q_{\alpha}}_{L^{\infty}}\leq &\  \sum_{t^{-k}\le 2^j\le t^{k}}\norm{P_{2^j}\Omega_0 Q_{\alpha}}_{L^{\infty}}+\sum_{2^j>t^k}\norm{P_{2^j}\Omega_0 Q_{\alpha}}_{L^{\infty}}\\
    \leq &\  \sum_{t^{-k}\le 2^j\leq t^k}\norm{\Lambda \Omega_0 Q_{\alpha}}_{L^2}+\sum_{2^j>t^k}2^{-j/2}\norm{\partial_{\alpha}\Omega_0 Q_{\alpha}}_{L^{2}}\\
    \leq &\  C\ln t\norm{\Lambda \Omega_0 Q_{\alpha}}_{L^2}+t^{-\frac{k}{2}}\norm{\partial_{\alpha}\Omega_0 Q_{\alpha}}_{L^{2}}\\
    \leq &\  C\epsilon t^{\delta_0}\ln t.
\end{align*}
For $[\alpha\partial_t, P_{\geq t^{-k}}]Q_{\alpha}$, note that
\begin{align*}
    [\alpha\partial_t, P_{\geq t^{-k}}]Q_{\alpha}=[\alpha,P_{\ge t^{-k}}]Q_{t\alpha}+\alpha[\partial_t,P_{\ge t^{-k}}]Q_\alpha;
\end{align*}
We have
\begin{align*}
    [\alpha,P_{\ge t^{-k}}]Q_{t\alpha}=\sum_{2^j\ge t^{-k}}[\alpha,P_{2^j}]Q_{t\alpha}=\sum_{2^j\ge t^{-k}}i2^{-j}\tilde{P}_{2^j}Q_{t\alpha}
\end{align*}
where $\tilde{P}$ is the Littlewood-Paley operator associated to $\phi'$ and $\phi$ is the notation which is used in Appendix \ref{sec-bernstein}.
By the Bernstein inequality, for $t^{-k}\le 2^j\leq t^k$,
\begin{equation}
    \norm{[\alpha, P_{2^j}]Q_{t\alpha}}_{L^{\infty}}\leq 2^{-j}\norm{\tilde{P}_{2^j}Q_{t\alpha}}_{L^{\infty}}\leq C\norm{\Lambda \tilde{P}_{2^j}Q_t}_{L^2}\leq C\norm{\Lambda Q_t}_{L^2}\leq C\epsilon.
\end{equation}
For $2^j\geq t^k$, by the Bernstein inequality again, we have
\begin{equation}
    \norm{[\alpha, P_{2^j}]Q_{t\alpha}}_{L^{\infty}}\leq 2^{-\frac{j}{2}}\norm{\partial_\alpha Q_{t\alpha}}_{L^2}\le C\epsilon 2^{-\frac{j}{2}}.
\end{equation}
So,
\begin{equation}
    \norm{ [\alpha, P_{\geq t^{-k}}]Q_{t\alpha}}_{L^{\infty}}\leq \sum_{t^{-k}\le 2^j\le t^k} \norm{[\alpha, P_{2^j}]Q_{t\alpha}}_{L^{\infty}}+\sum_{2^j>t^k} \norm{[\alpha, P_{2^j}]Q_{t\alpha}}_{L^{\infty}}\leq C\epsilon \ln t.
\end{equation}
For $\alpha[\partial_t,P_{\ge t^{-k}}]Q_\alpha$, we have
\begin{align*}
\widehat{[\partial_t,P_{\ge t^{-k}}]f}(\xi)=\left(\partial_t((1-\phi(t^k\xi))\hat{f}(\xi))-(1-\phi(t^k\xi))\partial_t\hat{f}(\xi)\right)
=-kt^{k-1}\xi\phi'(t^k\xi)\hat{f}(\xi).
\end{align*}
Define $\psi(\xi)=\xi\phi'(\xi)$ (note that $\psi$ is supported on an annulus rather than a disk) and let $P^*$ be the corresponding Littlewood-Paley operator, then
\[
\alpha[\partial_t,P_{\ge t^{-k}}]Q_\alpha=-\alpha kt^{-1}P^*_{\ge t^{-k}}Q_\alpha=-\alpha kt^{-1}\sum_{2^j\ge t^{-k}}P^*_{2^j}Q_\alpha.
\]
Set $\chi=F^{-1}[\psi]$, then
\begin{align*}
    \alpha kt^{-1}P^*_{2^j}Q_\alpha
    &=\alpha kt^{-1}2^j\int\chi(2^j(\alpha-\beta))Q_\beta d\beta\\
    &=kt^{-1}2^j\int(\alpha-\beta)\chi(2^j(\alpha-\beta))Q_\beta d\beta+kt^{-1}2^j\int\chi(2^j(\alpha-\beta))\beta Q_\beta d\beta\\
    &\overset{\text{def}}{=}I_1+I_2.
\end{align*}
Let $\lambda(x)=x\chi(x)$ and we derive
\[
I_1=kt^{-1}\int\lambda(2^j(\alpha-\beta))Q_\beta d\beta
=kt^{-1}2^{-j}P^{**}_{2^j}Q_\alpha
\]
where $\widehat{P^{**}_{N}f}=F[\lambda](\frac{\xi}{N})\hat{f}(\xi)$. Now we have
\[
|I_1|\le kt^{-1}2^{-j}\norm{P^{**}_{2^j}Q_\alpha}_{L^\infty}
\le kt^{-1}2^{-\frac{j}{2}}\norm{Q_\alpha}_{L^2}\le C\epsilon t^{-1}2^{-\frac{j}{2}}
\]
which implies
\[
\left|\alpha kt^{-1}\sum_{2^j\ge t^{-k}}P_{2^j}^*Q_\alpha\right|
\le C\epsilon t^{-1}\sum_{2^j\ge t^{-k}}2^{-\frac{j}{2}}
\le C\epsilon t^{-1+\frac{k}{2}}
\le C\epsilon
\]
since $k\le 2$.

In addition, note that
\[
I_2=kt^{-1}P^*_{2^j}(\alpha Q_\alpha)=-kt^{-1}P^*_{2^j}L_0Q+\frac{k}{2}P^*_{2^j}Q_t:=I_{21}+I_{22}
\]
and we have, according to Bernstein's inequality,
\[
|I_{21}|\le Ct^{-1}\norm{P_{2^j}^*L_0Q}_{L^\infty}\le Ct^{-1}2^{-\frac{j}{2}}\norm{\partial_\alpha L_0Q}_{L^2}\le C\epsilon t^{-1+\delta_0}2^{-\frac{j}{2}}
\]
which implies
\[
\sum_{2^j\ge t^{-k}}|I_{21}|\le C\epsilon t^{-1+\delta_0+\frac{k}{2}}\le C\epsilon t^{\delta_0}.
\]
Finally, for $I_{22}$, we have
\[
|I_{22}|\le C\norm{P^*_{2^j}Q_t}_{L^\infty}\le C\norm{\Lambda Q_t}_{L^2}\le C\epsilon
\]
and
\[
|I_{22}|\le C2^{-j}\norm{\partial_\alpha Q_t}_{L^\infty}\le C\epsilon 2^{-j}.
\]
Combining them yields
\[
\sum_{2^j\ge t^{-k}}|I_{22}|\le \sum_{t^{-k}\le 2^j\le t^k}|I_{22}|+\sum_{2^j>t^k}|I_{22}|
\le C\epsilon\ln t+C\epsilon t^{-k}
\le C\epsilon\ln t.
\]
In summary,
\[
\norm{\alpha[\partial_t,P_{\ge t^{-k}}]Q_\alpha}_{L^\infty}
=\norm{\alpha kt^{-1}\sum_{2^j\ge t^{-k}}P^*_{2^j}Q_\alpha}
\le C\epsilon t^{\delta_0}\ln t
\]
which completes the proof for \eqref{Omega_0D_t zeta high}. For \eqref{L0D_t zeta high}, direct calculation gives 
$[L_0, P_j]=P^*_{2^j},$
hence
\begin{align*}
    \norm{L_0P_{\geq t^{-k}}Q_{\alpha}}_{L^2}^2\leq &\  C\sum_{2^j\ge t^{-k}}\norm{L_0P_{2^j}Q_{\alpha}}_{L^2}^2\\
    \leq &\  C\left(\sum_{2^j\ge t^{-k}}\norm{P_{2^j} L_0 Q_{\alpha}}_{L^2}^2 + \sum_{2^j\ge t^{-k}}\norm{P^*_{2^j} Q_{\alpha}}_{L^2}^2\right)\\
    \leq &\  C(\norm{L_0Q_{\alpha}}_{L^2}^2+\norm{Q_{\alpha}}_{L^2}^2)\\
    \leq &\  C\epsilon^2 t^{2\delta_0},
\end{align*}
which implies \eqref{L0D_t zeta high}.
\end{proof}

Since quantities such as $b$ and $A-1$ are represented in terms of $(I-\mathcal{H})b$ and $(I-\mathcal{H})(A-1)$, we need to control the former by the latter.

First, we have the following lemma.

\begin{lemma}\label{curve-flat hilbert}
Assume the bootstrap assumptions. We have
\begin{equation}
    \norm{(\mathcal{H}-\mathbb{H})f}_{H^s}+\norm{\left(\mathcal{H}\frac{1}{\zeta_{\alpha}}-\mathbb{H}\right)f}_{H^s}\leq C\epsilon t^{-1/2}\norm{f}_{H^s},
\end{equation}
For any $2\leq q<\infty$, there holds
\begin{equation}
    \norm{(\mathcal{H}-\mathbb{H})f}_{W^{s-2,\infty}}+\norm{\left(\mathcal{H}\frac{1}{\zeta_{\alpha}}-\mathbb{H}\right)f}_{W^{s-2,\infty}}\leq C(q-1)(\epsilon t^{-1/2})^{1-2/q}\norm{f}_{W^{s-2,\infty}}.
\end{equation}
\end{lemma}
\begin{proof}
We provide details for the estimates for $(\mathcal{H}-\mathbb{H})f$ only. The counterparts for $(\mathcal{H}\frac{1}{\zeta_{\alpha}}-\mathbb{H})f$ are similar. Indeed,
            \[
            (\mathcal{H}-\mathbb{H})f=\frac{1}{\pi i}\int\left(\frac{\zeta_\beta}{\zeta(\alpha)-\zeta(\beta)}-\frac{1}{\alpha-\beta}\right)f(\beta)d\beta
            \]
            and it's clear that
            \begin{align*}
                \left|\int_{|\alpha-\beta|\ge 1}\left(\frac{\zeta_\beta}{\zeta(\alpha)-\zeta(\beta)}-\frac{1}{\alpha-\beta}\right)f(\beta)d\beta\right|
                &\le\norm{f}_{L^\infty}\int_{|\alpha-\beta|\ge 1}\left|\frac{\zeta_\beta-1}{\zeta(\alpha)-\zeta(\beta)}+\frac{\alpha-\zeta(\alpha)-(\beta-\zeta(\beta))}{(\zeta(\alpha)-\zeta(\beta))(\alpha-\beta)}d\beta\right|,\\
                &\le C\norm{f}_{L^\infty}\norm{\zeta_\alpha-1}_{L^q}\norm{\frac{1}{\alpha-\beta}1_{|\alpha-\beta|\geq 1}}_{L^{q'}},\\
                &\le C(q-1)(\epsilon t^{-1/2})^{1-2/q}\norm{f}_{W^{s-2,\infty}}.
            \end{align*}
            On the other hand, we apply Taylor's expansion to get
            \begin{align*}
                &\left|\int_{|\alpha-\beta|<1}\left(\frac{\zeta_\beta}{\zeta(\alpha)-\zeta(\beta)}-\frac{1}{\alpha-\beta}\right)f(\beta)d\beta\right|\\
                =&\ \left|\int_{|\alpha-\beta|<1}\left(\frac{\zeta_\beta(\alpha-\beta)-(\zeta(\alpha)-\zeta(\beta))}{(\zeta(\alpha)-\zeta(\beta))(\alpha-\beta)}\right)f(\beta)d\beta\right|,\\
                =&\ \left|\int_{|\alpha-\beta|<1}\left(\frac{\zeta_\beta(\alpha-\beta)-\zeta_\alpha(\alpha-\beta)-\zeta''(\xi)(\alpha-\beta)^2}{(\zeta(\alpha)-\zeta(\beta))(\alpha-\beta)}\right)f(\beta)d\beta\right|,\\
                =&\ \left|\int_{|\alpha-\beta|<1}\left(\frac{(\zeta''(\eta)-\zeta''(\xi))(\alpha-\beta)^2}{(\zeta(\alpha)-\zeta(\beta))(\alpha-\beta)}\right)f(\beta)d\beta\right|,\\
                \le&\ C\norm{\zeta_{\alpha\alpha}}_{L^\infty}\norm{f}_{L^\infty},\\
                \le&\ C\epsilon t^{-\frac{1}{2}}\norm{f}_{L^\infty}.
            \end{align*}
For higher derivatives, integration by parts, we have
$$(\mathcal{H}-\mathbb{H})f(\alpha,t)=\frac{1}{\pi i}\int \left(\ln \frac{\zeta(\alpha,t)-\zeta(\beta,t)}{\alpha-\beta}\right) f_{\beta}(\beta,t)d\beta.$$
Writing 
$$\partial_{\alpha}^{s-2}=((\partial_{\alpha}+\partial_{\beta})-\partial_{\beta})^{s-2}=(-\partial_{\beta})^{s-2}+\sum_{1\leq k\leq s-2}C_{s-2,k}(\partial_{\alpha}+\partial_{\beta})^{s-2},$$
we have
\begin{align*}
    \partial_{\alpha}^{s-2}(\mathcal{H}-\mathbb{H})f(\alpha,t)=& \frac{1}{\pi i}\int \Big\{ (-\partial_{\beta})^{s-2}\Big(\ln \frac{\zeta(\alpha,t)-\zeta(\beta,t)}{\alpha-\beta}\Big)\Big\} f_{\beta}(\beta,t)d\beta\\
    &+\frac{1}{\pi i}\sum_{1\leq k\leq s-2}\int \left\{(\partial_{\alpha}+\partial_{\beta})^k (-\partial_{\beta})^{s-2-k}\ln \frac{\zeta(\alpha,t)-\zeta(\beta,t)}{\alpha-\beta}  \right\}f_{\beta}\\
    =& \frac{1}{\pi i}\int \left\{ (-\partial_{\beta})\left(\ln \frac{\zeta(\alpha,t)-\zeta(\beta,t)}{\alpha-\beta}\right)\right\} \partial_{\beta}^{s-2}f(\beta,t)d\beta\\
    &+\frac{1}{\pi i}\sum_{1\leq k\leq s-2}\int \left\{(\partial_{\alpha}+\partial_{\beta})^k \ln \frac{\zeta(\alpha,t)-\zeta(\beta,t)}{\alpha-\beta}  \right\}(-\partial_{\beta})^{s-1-k}f,\\
    :=&I+II.
\end{align*}
Note that $I=(\mathcal{H}-\mathbb{H})\partial_{\alpha}^{s-2}f$. By the argument above, we have 
\begin{equation}
    |I|\leq C(q-1)(\epsilon t^{-1/2})^{1-2/q}\norm{\partial_{\alpha}^{s-2}f}_{L^{\infty}}.
\end{equation}
Note that $(\partial_{\alpha}+\partial_{\beta})\frac{1}{\alpha-\beta}=0$. A similar argument as above gives
\begin{equation}
    |II|\leq C(q-1)(\epsilon t^{-1/2})^{1-2/q}\norm{f}_{W^{s-2,\infty}}.
\end{equation}
This completes the proof of the lemma.
\end{proof}

Also, we have the following $L^{\infty}$ estimates of the Hilbert transform.
\begin{lemma}\label{L infty of hilbert}
We have
\begin{equation}
    \norm{\mathcal{H}f}_{L^{\infty}}\leq Ct^{-2}\norm{f}_{H^1}+\norm{f}_{L^{\infty}}\ln t.
\end{equation}
\end{lemma}
\begin{proof}
We have
\begin{align*}
\mathcal{H}f(\alpha,t)=& \frac{1}{\pi i}p.v.\int_{-\infty}^{\infty}\frac{\zeta_{\beta}}{\zeta(\alpha,t)-\zeta(\beta,t)}f(\beta,t)d\beta\\
=& \frac{1}{\pi i}p.v.\int_{|\zeta(\alpha,t)-\zeta(\beta,t)|\leq t^{-4}}\frac{\zeta_{\beta}}{\zeta(\alpha,t)-\zeta(\beta,t)}f(\beta,t)d\beta\\
&+\frac{1}{\pi i}p.v.\int_{t^{-4}|\leq \zeta(\alpha,t)-\zeta(\beta,t)|\leq t^{4}}\frac{\zeta_{\beta}}{\zeta(\alpha,t)-\zeta(\beta,t)}f(\beta,t)d\beta\\
&+ \frac{1}{\pi i}p.v.\int_{|\zeta(\alpha,t)-\zeta(\beta,t)|\geq t^{4}}\frac{\zeta_{\beta}}{\zeta(\alpha,t)-\zeta(\beta,t)}f(\beta,t)d\beta\\
:=& I+II+III.
\end{align*}
By the definition of principal values, the Cauchy-Schwartz inequality, following by the Hardy's inequality, we have 
\begin{align*}
    I=& \frac{1}{\pi i}\int_{|\zeta(\alpha,t)-\zeta(\beta,t)|\leq t^{-4}}\frac{\zeta_{\beta}}{\zeta(\alpha,t)-\zeta(\beta,t)}\Big(f(\beta,t)-f(\alpha,t)\Big)d\beta\\
    \leq &\  \frac{1}{\pi}\Big(\int \frac{|f(\alpha,t)-f(\beta,t)|^2}{(\alpha-\beta)^2}d\beta\Big)^{1/2}\Big(\int_{|\zeta(\alpha,t)-\zeta(\beta,t)|\leq t^{-4}} d\beta\Big)^{1/2}\norm{\zeta_{\alpha}}_{L^{\infty}}\\
    \leq &\  Ct^{-2}\norm{f}_{\dot{H}^1}.
\end{align*}
For $II$, using Lemma \ref{nearly alpha}, we have 
\begin{align*}
    |II|\leq \int_{t^{-4}\leq |\zeta(\alpha,t)-\zeta(\beta,t)|\leq t^4}\frac{1}{|\zeta(\alpha,t)-\zeta(\beta,t)|}d\beta \norm{f}_{L^{\infty}}\leq C\norm{f}_{L^{\infty}}\ln t.
\end{align*}
For $III$, by the Cauchy-Schwartz inequality, we obtain
\begin{align*}
    |III|\leq C\norm{f}_{L^2}\int_{|\zeta(\alpha,t)-\zeta(\beta,t)|\geq t^4}\frac{1}{|\zeta(\alpha,t)-\zeta(\beta,t)|}d\beta\leq Ct^{-2}\norm{f}_{L^2}.
\end{align*}
This completes the proof of the Lemma.

\end{proof}

\begin{lemma}\label{control f by using (I-H)f}
            Let $f$ be a real-valued function. Suppose that $(I-\mathcal{H})f=g$ or $(I-\mathcal{H})f\bar{\zeta}_\alpha=g$, then for $0\le k\le s$,
            \begin{equation}
                \norm{f}_{H^k}\le C\norm{g}_{H^k};
            \end{equation}
            furthermore, for $0\le l\le s-2$,
            \begin{equation}\label{infinite norm of f and (I-H)f}
                \norm{f}_{W^{{l},\infty}}\le C\norm{g}_{W^{{l},\infty}}.
            \end{equation}
        \end{lemma}
        \begin{proof}
            We shall only prove (\ref{infinite norm of f and (I-H)f}) for $s=0$ as others are similar to it. Note that
            \[
            f=Re\left\{g+(\mathcal{H}-\mathbb{H})f\right\}
            \]
            and it suffices to control $\norm{(\mathcal{H}-\mathbb{H})f}_{L^\infty}$. By Lemma \ref{curve-flat hilbert} with $q=2$, we have
            \begin{align*}
                \norm{(\mathcal{H}-\mathbb{H})f}_{L^{\infty}}\leq C\epsilon \norm{f}_{L^{\infty}}.
            \end{align*}
            So we conclude that
            \[
            \norm{f}_{L^\infty}\le\norm{g}_{L^\infty}+C\epsilon\norm{f}_{L^\infty}
            \]
            which implies (\ref{infinite norm of f and (I-H)f}) (for $s=0$ and sufficiently small $\epsilon>0$).
        \end{proof}

\begin{proposition}\label{good-replace}
    Assume the bootstrap assumptions. We have
    \begin{equation}
        \norm{D_t\zeta-Q_{\alpha}}_{H^s}\leq C\epsilon^2 t^{-1/2}, \quad \norm{2iD_t^2\zeta+\theta_{\alpha}}_{H^s}\leq C\epsilon^2 t^{-1/2},
    \end{equation}
and
\begin{equation}
    \norm{D_t\zeta-Q_{\alpha}}_{W^{s-2,\infty}}\leq C\epsilon^2 t^{-\frac{2}{3}}\ln t, \quad \norm{2iD_t^2\zeta+\theta_{\alpha}}_{W^{s-2,\infty}}\leq C\epsilon^2 t^{-\frac{2}{3}}\ln t,
\end{equation}
\end{proposition}
\begin{proof}
    We first prove the estimate for $\theta$. Note that
    \[
    \partial_\alpha\theta=2\partial_\alpha(\zeta-\alpha)+\partial_\alpha(\mathcal{H}+\bar{\mathcal{H}})(\zeta-\alpha)=-2iD_t^2\zeta+(A-1)\zeta_\alpha+\partial_\alpha(\mathcal{H}+\bar{\mathcal{H}})(\zeta-\alpha)
    \]
    and it's easy to verify $\norm{A-1}_{H^s}\le C\epsilon^2t^{-\frac{1}{2}}$ and $\norm{A-1}_{W^{s-1,\infty}}\le C\epsilon^2t^{-\frac{3}{4}}\ln t$ using Lemma \ref{singular}. For $\partial_\alpha(\mathcal{H}+\bar{\mathcal{H}})(\zeta-\alpha)$, write
    \begin{align*}
        \partial_\alpha(\bar{\mathcal{H}}+\mathcal{H})(\zeta-\alpha)
        &=(\bar{\mathcal{H}}+\mathcal{H})(\zeta_\alpha-1)+[\partial_\alpha,\bar{\mathcal{H}}+\mathcal{H}](\zeta-\alpha),\\
        &=\frac{2}{\pi}\int\frac{Im\left\{\zeta_\beta(\bar{\zeta}(\alpha)-\bar{\zeta}(\beta))\right\}}{|\zeta(\alpha)-\zeta(\beta)|^2}(\zeta_\beta-1)d\beta+[\zeta_\alpha,\mathcal{H}]\frac{\zeta_\alpha-1}{\zeta_\alpha}+[\bar{\zeta}_\alpha,\bar{\mathcal{H}}]\frac{\zeta_\alpha-1}{\bar{\zeta}_\alpha},
    \end{align*}
    and, by Lemma \ref{singular}, we have
    \[
    \norm{\partial_\alpha(\bar{\mathcal{H}}+\mathcal{H})(\zeta-\alpha)}_{H^s}
    \le C\epsilon^2t^{-1/2}.
    \]
    Also, since $\norm{\partial_\alpha^s\zeta}_{L^\infty}\le C\epsilon t^{-\frac{1}{6}}\ln t$, we derive that
    \begin{align*}
    &\norm{\frac{2}{\pi i}\int\frac{Im\{\zeta_\beta(\partial_\alpha^{s-2}\bar{\zeta}-\partial_\beta^{s-2}\bar{\zeta})\}}{|\zeta(\alpha)-\zeta(\beta)|^2}(\zeta_\beta-1)d\beta}_{L^\infty}\\
    \le&\ C\norm{\partial_\alpha^{s-1}\bar{\zeta}}_{W^{1,\infty}}\norm{\zeta_\alpha-1}_{W^{1,\infty}}+C\norm{\partial_\alpha^{s-1}\bar{\zeta}}_{L^\infty}\norm{\zeta_\alpha-1}_{L^\infty}\ln t\\
    &+Ct^{-\frac{1}{2}}\norm{\partial_\alpha^{s-1}\bar{\zeta}}_{L^\infty}\norm{\zeta_\alpha-1}_{L^2}\\
    \le&\ C\epsilon^2t^{-\frac{2}{3}}\ln t.
    \end{align*}
    This gives the conclusion. For the estimate of $Q_\alpha-D_t\zeta$, we write (\ref{formula of Q_alpha}) as
    \begin{align*}
        \bar{Q}_\alpha-D_t\bar{\zeta}
        &=(\zeta_\alpha-1)D_t\bar{\zeta}-\frac{1}{2}\bar{\zeta}_\alpha\left(\mathcal{H}\frac{1}{\zeta_\alpha}+\bar{\mathcal{H}}\frac{1}{\bar{\zeta}_\alpha}\right)(\zeta_\alpha D_t\bar{\zeta})-\frac{1}{2}(\zeta_\alpha-\bar{\zeta}_\alpha)\mathcal{H}D_t\bar{\zeta},\\
        &=(\zeta_\alpha-1)D_t\bar{\zeta}-\frac{1}{2}\bar{\zeta}_\alpha\left(\mathcal{H}\frac{1}{\zeta_\alpha}+\bar{\mathcal{H}}\frac{1}{\bar{\zeta}_\alpha}\right)(\zeta_\alpha D_t\bar{\zeta})-iIm\{\zeta_\alpha-1\}D_t\bar{\zeta}
    \end{align*}
    and we first deal with the $L^2$ and $L^\infty$ cases. It's easy to verify that
    \begin{align*}
        \norm{D_t\zeta-Q_\alpha}_{L^2}
        &\le C\epsilon^2t^{-\frac{1}{2}}+C\norm{\left(\mathcal{H}\frac{1}{\zeta_\alpha}+\bar{\mathcal{H}}\frac{1}{\bar{\zeta}_\alpha}\right)(\zeta_\alpha D_t\bar{\zeta})}_{L^2},\\
        &\le C\epsilon^2t^{-\frac{1}{2}}+C\norm{\int\frac{Im\{\zeta(\alpha)-\zeta(\beta)\}}{|\zeta(\alpha)-\zeta(\beta)|^2}\zeta_\beta D_t\bar{\zeta}(\beta)d\beta}_{L^2},\\
        &\le C\epsilon^2t^{-\frac{1}{2}}
    \end{align*}
    and
    \[
    \norm{D_t\zeta-Q_\alpha}_{L^\infty}\le C\epsilon^2t^{-\frac{3}{4}}\ln t
    \]
    by applying the estimate in Lemma \ref{singular}.
    
    We treat the $s$ order derivatives as follows: for $\partial_\alpha^s(\zeta_\alpha-1)D_t\bar{\zeta}$, it suffices to consider the case where all derivatives apply to $\zeta_\alpha-1$. Actually, there holds
    \begin{align*}
        \norm{\partial_\alpha^s(\zeta_\alpha-1)D_t\bar{\zeta}}_{L^2}
        &\le\sum_{j=0}^sC_{j,s}\norm{\partial_\alpha^j(\zeta_\alpha-1)\partial_\alpha^{s-j}D_t\bar{\zeta}}_{L^2}\\
        &\le C\sum_{j=0}^{s-1}\norm{\partial_\alpha^j(\zeta_\alpha-1)}_{L^2}\norm{\partial_\alpha^{s-j}D_t\bar{\zeta}}_{L^\infty}+\norm{\partial_\alpha^s(\zeta_\alpha-1)D_t\bar{\zeta}}_{L^2}\\
        &\le C\epsilon^2t^{-\frac{1}{2}}+\norm{\partial_\alpha^s(\zeta_\alpha-1)\bar{Q}_\alpha}_{L^2}+\norm{\partial_\alpha^s(\zeta_\alpha-1)}_{L^2}\norm{\bar{Q}_\alpha-D_t\bar{\zeta}}_{L^\infty}\\
        &\le C\epsilon^2t^{-\frac{1}{2}}+\norm{\partial_\alpha^s(\zeta_\alpha-1)P_{\ge t^{-2}}\bar{Q}_\alpha}_{L^2}.
    \end{align*}
    Here we have used Bernstein's inequality to prove $\norm{P_{\le t^{-2}}\bar{Q}_\alpha}_{L^\infty}\le C\epsilon t^{-1}$ in the last step. For the second term, we have
    \begin{align*}
        \norm{\partial_\alpha^s(\zeta_\alpha-1)P_{\ge t^{-2}}\bar{Q}_\alpha}_{L^2(|\alpha|\le t)}
        &\le C\epsilon\norm{P_{\ge t^{-2}}\bar{Q}_\alpha}_{L^\infty(|\alpha|\le t)}\\
        &\le C\epsilon\norm{\frac{4\alpha^2}{it^2}\partial_\alpha P_{\ge t^{-2}}\bar{Q}_\alpha-\frac{4\alpha}{it^2}L_0P_{\ge t^{-2}}\bar{Q}_\alpha+\frac{2}{it}\Omega_0P_{\ge t^{-2}}\bar{Q}_\alpha}_{L^\infty(|\alpha|\le t)}\\
        &\le C\epsilon^2 t^{-\frac{1}{2}}.
    \end{align*}
    Here we have used \eqref{Omega_0D_t zeta high} and \eqref{L0D_t zeta high}. On the other hand, note that according to the $H^s$ estimate for $\theta_\alpha+2iD_t^2\zeta$, we have
    \begin{align*}
        \norm{\partial_\alpha^s(\zeta_\alpha-1)P_{\ge t^{-2}}\bar{Q}_\alpha}_{L^2(|\alpha|>t)}
        &\le\frac{1}{2}\norm{\partial_\alpha^s\theta_\alpha P_{\ge t^{-2}}\bar{Q}_\alpha}_{L^2(|\alpha|>t)}\\
        &\ \ \ +\frac{1}{2}\norm{2\partial_\alpha^s(\zeta_\alpha-1)-\partial_\alpha^s\theta_\alpha}_{L^2}\norm{P_{\ge t^{-2}}\bar{Q}_\alpha}_{L^\infty}\\
        &\le C\epsilon\norm{\partial_\alpha^s\theta_\alpha}_{L^\infty(|\alpha|>t)}+C\epsilon^2t^{-\frac{3}{4}}\\
        &\le C\epsilon^2 t^{-\frac{1}{2}}.
    \end{align*}
    Here we have used the formula
    \[
    \partial_\alpha f=\frac{it^2}{4\alpha^2}f+\frac{1}{\alpha}L_0f-\frac{t}{2\alpha^2}\Omega_0f
    \]
    twice to lower the derivatives on $\theta$ and the estimate $\norm{\partial_\alpha^{s-2}\theta_\alpha}_{L^\infty}\le C\epsilon t^{-\frac{1}{2}}$. In summary, we have
    \[
    \norm{\partial_\alpha^s(\zeta_\alpha-1)D_t\bar{\zeta}}_{L^2}\le C\epsilon^2t^{-\frac{1}{2}}.
    \]
    Similarly, $\norm{Im\{\zeta_\alpha-1\}D_t\bar{\zeta}}_{H^s}\le C\epsilon^2t^{-\frac{1}{2}}$ and we are left with
    \[
    \norm{\int\frac{Im\{\zeta(\alpha)-\zeta(\beta)\}}{|\zeta(\alpha)-\zeta(\beta)|^2}\zeta_\beta D_t\bar{\zeta}(\beta)d\beta}_{H^s}.
    \]
    With the experience above, we know that it suffices to consider
    \[
    \norm{\int\frac{Im\{\partial_\alpha^s\zeta-\partial_\beta^s\zeta\}}{|\zeta(\alpha)-\zeta(\beta)|^2}\zeta_\beta D_t\bar{\zeta}(\beta)d\beta}_{L^2}.
    \]
    We split the integral into two parts: for $|\alpha-\beta|\ge \sqrt{t}$, there holds
    \begin{align*}
        \norm{\int_{|\alpha-\beta|\ge\sqrt{t}}\frac{Im\{\partial_\alpha^s\zeta-\partial_\beta^s\zeta\}}{|\zeta(\alpha)-\zeta(\beta)|^2}\zeta_\beta D_t\bar{\zeta}(\beta)d\beta}_{L^2}
        &\le C\norm{\partial_\alpha^s\zeta}_{L^2}\int_{|\alpha-\beta|\ge \sqrt{t}}\frac{|D_t\bar{\zeta}(\beta)|}{(\alpha-\beta)^2}d\beta\\
        &\ \ \ +C\norm{\int_{|\alpha-\beta|\ge\sqrt{t}}\frac{|\partial_\beta^s\zeta D_t\bar{\zeta}|}{(\alpha-\beta)^2}d\beta}_{L^2}\\
        &\le C\epsilon^2t^{-\frac{1}{2}}+C\norm{\frac{1}{\alpha^2}}_{L^1(|\alpha|\ge\sqrt{t})}\norm{\partial_\alpha^s\zeta D_t\bar{\zeta}}_{L^2}\\
        &\le C\epsilon^2t^{-\frac{1}{2}}.
    \end{align*}
    We have used the Young's inequality here. For $|\alpha-\beta|<\sqrt{t}$, note that $|\alpha|\le t$ implies $|\beta|\le 2t$ and we derive
    \[
    \norm{\int_{|\alpha-\beta|<\sqrt{t}}\frac{Im\{\partial_\alpha^s\zeta-\partial_\beta^s\zeta\}}{|\zeta(\alpha)-\zeta(\beta)|^2}\zeta_\beta D_t\bar{\zeta}(\beta)d\beta}_{L^2(|\alpha|\le t)}
    \le C\norm{\partial_\alpha^{s+1}\zeta}_{L^2(|\alpha|\le 2t)}\norm{D_t\bar{\zeta}}_{L^\infty(|\alpha|\le 2t)};
    \]
    by the relevant part in the discussion for $\norm{\partial_\alpha^s(\zeta_\alpha-1)D_t\bar{\zeta}}_{L^2}$ we know that $\norm{D_t\bar{\zeta}}_{L^\infty(|\alpha|\le 2t)}\le C\epsilon t^{-\frac{1}{2}}$. This gives
    \[
    \norm{\int_{|\alpha-\beta|<\sqrt{t}}\frac{Im\{\partial_\alpha^s\zeta-\partial_\beta^s\zeta\}}{|\zeta(\alpha)-\zeta(\beta)|^2}\zeta_\beta D_t\bar{\zeta}(\beta)d\beta}_{L^2(|\alpha|\le t)}\le C\epsilon^2t^{-\frac{1}{2}}.
    \]
    For $|\alpha|>t$, there holds $|\beta|>t/2$ and we have
    \begin{align*}
        &\norm{\int_{|\alpha-\beta|<\sqrt{t}}\frac{Im\{\partial_\alpha^s\zeta-\partial_\beta^s\zeta\}}{|\zeta(\alpha)-\zeta(\beta)|^2}\zeta_\beta D_t\bar{\zeta}(\beta)d\beta}_{L^2(|\alpha|>t)}\\
        \le&\ \frac{1}{2}\norm{\int_{|\alpha-\beta|<\sqrt{t}}\frac{Im\{\partial_\alpha^{s-1}[2(\zeta_\alpha-1)-\theta_\alpha]-\partial_\beta^{s-1}[2(\zeta_\beta-1)-\theta_\beta]\}}{|\zeta(\alpha)-\zeta(\beta)|^2}\zeta_\beta D_t\bar{\zeta}(\beta)d\beta}_{L^2(|\alpha|>t)}\\
        &+\frac{1}{2}\norm{\int_{|\alpha-\beta|<\sqrt{t}}\frac{Im\{\partial_\alpha^{s-1}\theta_\alpha-\partial_\beta^{s-1}\theta_\beta\}}{|\zeta(\alpha)-\zeta(\beta)|^2}\zeta_\beta D_t\bar{\zeta}(\beta)d\beta}_{L^2(|\alpha|>t)}\\
        \le&\ C\norm{2(\zeta_\alpha-1)-\theta_\alpha}_{H^s}\norm{D_t\bar{\zeta}}_{L^\infty}
        +C\norm{\partial_\alpha^s\theta_\alpha}_{L^\infty(|\alpha|>t/2)}\norm{D_t\bar{\zeta}}_{L^2}\\
        \le&\ C\epsilon^2t^{-\frac{1}{2}}.
    \end{align*}
    Now we conclude that
    \[
    \norm{\int\frac{Im\{\partial_\alpha^s\zeta-\partial_\beta^s\zeta\}}{|\zeta(\alpha)-\zeta(\beta)|^2}\zeta_\beta D_t\bar{\zeta}(\beta)d\beta}_{L^2}\le C\epsilon^2t^{-\frac{1}{2}}
    \]
    which implies $\norm{D_t\zeta-Q_\alpha}_{H^s}\le C\epsilon^2t^{-\frac{1}{2}}$. The proof for $W^{s-2,\infty}$ is quite similar and we shall write out the only difference:
    \begin{align*}
        &\norm{\int_{|\alpha-\beta|<\sqrt{t}}\frac{Im\{\partial_\alpha^{s-3}[2(\zeta_\alpha-1)-\theta_\alpha]-\partial_\beta^{s-3}[2(\zeta_\beta-1)-\theta_\beta]\}}{|\zeta(\alpha)-\zeta(\beta)|^2}\zeta_\beta D_t\bar{\zeta}(\beta)d\beta}_{L^\infty(|\alpha|>t)}\\
        \le&\ C\norm{2(\zeta_\alpha-1)-\theta_\alpha}_{W^{s-1,\infty}}\norm{D_t\bar{\zeta}}_{W^{1,\infty}}+C\epsilon^2t^{-\frac{2}{3}}\ln t\\
        \le&\ C\epsilon t^{-\frac{1}{4}}\norm{2(\zeta_\alpha-1)-\theta_\alpha}_{H^s}+C\epsilon^2t^{-\frac{2}{3}}\ln t\\
        \le&\ C\epsilon^2t^{-\frac{2}{3}}\ln t.
    \end{align*}
    This completes the proof.
\end{proof}

\begin{rem}
    The advantage of $Q_\alpha$ and $\theta_\alpha$  lies in the fact that we have control over $\Lambda\Omega_0 Q_\alpha$ and $\Omega_0\theta_\alpha$, whereas direct estimates for $\Lambda\Omega_0 D_t\zeta$ and $\Omega_0 D_t^2\zeta$ are not available.
    As indicated in Proposition \ref{good-replace}, whenever necessary, we may substitute $D_t\zeta$ and $D_t^2\zeta$ with $Q_{\alpha}$ and $\theta_{\alpha}$ respectively, and the error term behaves better than the leading term. For simplicity, we will continue to use the notation $D_t\zeta$ and $D_t^2\zeta$, even when these substitutions are implicitly applied.

    Also, the $W^{s-2,\infty}$ estimate in Proposition \ref{good-replace} can be improved given $|\alpha|$ is near $t$.
\end{rem}

Proposition \ref{good-replace} allows us to apply Lemma \ref{transition-lemma} to $D_t\zeta$ to derive the decay for $\norm{D_t\zeta}_{L^\infty(|\alpha|\le Ct^{1+\mu})}$.
\begin{cor}\label{decay for Dt zeta near t}
    Suppose that $\mu>-\frac{1}{2}+\delta_0$. We have
    \begin{equation}
        \norm{D_t\zeta}_{L^\infty(|\alpha|\le Ct^{1+\mu})}\le C\epsilon t^{-\frac{1}{2}+\mu}.
    \end{equation}
    And for $1\leq k\leq s$,
            \begin{equation}
        \norm{\partial_{\alpha}^k (\zeta-\alpha)}_{L^\infty(|\alpha|\ge Ct^{1-\mu})}\le C\epsilon t^{-\frac{1}{2}+\mu}.
    \end{equation}
        \begin{equation}
        \norm{\partial_{\alpha}^k D_t\zeta}_{L^\infty(|\alpha|\ge Ct^{1-\mu})}\le C\epsilon t^{-\frac{1}{2}+\mu}.
    \end{equation}
\end{cor}

\begin{proof}
    Write
    \[
    D_t\zeta=Q_\alpha+O(\epsilon t^{-\frac{1}{2}})=P_{\ge t^{-2}}Q_\alpha+P_{<t^{-2}}Q_\alpha
    \]
    and note that
    \[
    P_{\ge t^{-2}}Q_\alpha=-\frac{2\alpha}{it}\partial_tP_{\ge t^{-2}}Q_\alpha+\frac{2}{it}\Omega_0(P_{\ge t^{-2}}Q_\alpha)
    \]
    hence, according to (\ref{Omega_0D_t zeta high}),
    \[
    \norm{P_{\ge t^{-2}}Q_\alpha}_{L^\infty}
    \le C\epsilon t^{-\frac{1}{2}+\mu}+C\epsilon t^{-1+\delta_0}\ln t
    \le C\epsilon t^{-\frac{1}{2}+\mu}.
    \]
    Also, there holds
    \[
    \norm{P_{<t^{-2}}Q_\alpha}_{L^\infty}
    \le Ct^{-1}\norm{P_{<t^{-2}}Q_\alpha}_{L^2}\le C\epsilon t^{-1}.
    \]
    This completes the proof.
\end{proof}

Define $\tilde{\sigma}=D_t\theta$ and we will need the following expansion of $D_t\theta$ and $D_t\tilde{\sigma}$ in Chapter 5:
\begin{lemma}\label{remainder of Dt theta}
    There holds
    \begin{equation}\label{expansion of Dt theta}
        D_t\theta=D_t(I-\mathcal{H})(\zeta-\bar{\zeta})=2D_t\zeta-2b-(\bar{\mathcal{H}}+\mathcal{H})D_t(\zeta-\alpha)-[D_t,\bar{\mathcal{H}}+\mathcal{H}](\zeta-\alpha).
    \end{equation}
    and thus, if we assume Lemma \ref{norms of b} for a while,
    \[
    D_t\bar{\theta}=2D_t\bar{\zeta}+O(\epsilon^2t^{-\frac{1}{2}})_{H^s}.
    \]
\end{lemma}

\begin{proof}
     First, we know that $\norm{b}_{H^s}\le C\epsilon^2t^{-\frac{1}{2}}$ from Lemma \ref{norms of b}; next, straight calculation yields
    \[
    (\bar{\mathcal{H}}+\mathcal{H})f=\left(\bar{\mathcal{H}}\frac{1}{\bar{\zeta}_\alpha}+\mathcal{H}\frac{1}{\zeta_\alpha}\right)f+\bar{\mathcal{H}}\left(1-\frac{1}{\bar{\zeta}_\alpha}\right)f+\mathcal{H}\left(1-\frac{1}{\zeta_\alpha}\right)f.
    \]
    And it's enough to consider $f=D_t\bar{\zeta}$ since $D_t\alpha=b$ behaves better. First we have
    \[
    \left\|\left(\bar{\mathcal{H}}\frac{1}{\bar{\zeta}_\alpha}+\mathcal{H}\frac{1}{\zeta_\alpha}\right)D_t\zeta\right\|_{H^s}
    =C\left\|\int\frac{\eta(\alpha)-\eta(\beta)}{|\zeta(\alpha)-\zeta(\beta)|^2}D_t\bar{\zeta}d\beta\right\|_{H^s}
    \le C\epsilon^2t^{-\frac{1}{2}}.
    \]
    Also, by the $L^2$ boundedness of $\mathcal{H}$ and the bootstrap assumptions, 
    \begin{align*}
        \left\|\bar{\mathcal{H}}\left(1-\frac{1}{\bar{\zeta}_\alpha}\right)D_t\bar{\zeta}\right\|_{H^s}
       \le C\norm{(1-\frac{1}{\bar{\zeta}_{\alpha}})D_t\bar{\zeta}}_{H^s}\leq C\epsilon^2t^{-\frac{1}{2}}.
    \end{align*}
    Similar process applies to $\mathcal{H}\left(1-\frac{1}{\zeta_\alpha}\right)D_t\bar{\zeta}$, so we conclude that $\|(\bar{\mathcal{H}}+\mathcal{H})D_t(\bar{\zeta}-\alpha)\|_{H^s}\le C\epsilon^2t^{-\frac{1}{2}}$.
		
For the last term in (\ref{expansion of Dt theta}), we have
\[
    [D_t,\mathcal{H}](\bar{\zeta}-\alpha)
    =[D_t\zeta,\mathcal{H}]\frac{\bar{\zeta}_\alpha-1}{\zeta_\alpha},\ [D_t,\bar{\mathcal{H}}](\bar{\zeta}-\alpha)=[D_t\bar{\zeta},\bar{\mathcal{H}}]\frac{\bar{\zeta}_\alpha-1}{\bar{\zeta}_\alpha}
\]
and thus
\[
    \|[D_t,\mathcal{H}](\bar{\zeta}-\alpha)\|_{H^s}=\norm{(I-\mathcal{H})b}_{H^s}\le C\epsilon^2t^{-\frac{1}{2}}.
\]
Also, since both $D_t\bar{\zeta}$ and $\frac{\bar{\zeta}_\alpha-1}{\bar{\zeta}_\alpha}$ are holomorphic in $\Omega(t)$, we have
\begin{align*}
    [D_t\bar{\zeta},\bar{\mathcal{H}}]\frac{\bar{\zeta}_\alpha-1}{\bar{\zeta}_\alpha}
&=[D_t\bar{\zeta},\mathcal{H}]\frac{\bar{\zeta}_\alpha-1}{\zeta_\alpha}-\left[D_t\bar{\zeta},\bar{\mathcal{H}}\frac{1}{\bar{\zeta}_\alpha}+\mathcal{H}\frac{1}{\zeta_\alpha}\right](\bar{\zeta}_\alpha-1),\\
&=-\left[D_t\bar{\zeta},\bar{\mathcal{H}}\frac{1}{\bar{\zeta}_\alpha}+\mathcal{H}\frac{1}{\zeta_\alpha}\right](\bar{\zeta}_\alpha-1)
\end{align*}
and thus
\[
\norm{[D_t,\bar{\mathcal{H}}](\bar{\zeta}-\alpha)}_{H^s}=\norm{[D_t\bar{\zeta},\bar{\mathcal{H}}]\frac{\bar{\zeta}_\alpha-1}{\bar{\zeta}_\alpha}}_{H^s}\le C\epsilon^3t^{-\frac{3}{4}}.
\]
This completes the proof.
\end{proof}

\begin{lemma}\label{structure of Dt sigma}
    If we assume Lemma \ref{norms of b} for a while, there holds
    \[
    D_t\tilde{\sigma}=2D_t^2\zeta+e
    \]
    where $\norm{e}_{H^{s}}\le C\epsilon^2t^{-\frac{1}{2}}$.
\end{lemma}

\begin{lemma}\label{bouns decay for small or big alpha}
    Assume the bootstrap assumptions. We have
    \begin{equation}
        \norm{(\partial_{\alpha}D_t\zeta, D_t^2\zeta, \zeta_{\alpha}-1)}_{W^{s-3,\infty}\times W^{s-2,\infty}\times W^{s-2,\infty}(E_t)}\leq C \epsilon t^{-1/2-1/10+\delta_0},
    \end{equation}
    where $E_t:=\{\alpha: |\alpha|\geq t^{6/5}~or~|\alpha|\leq t^{4/5}\}$.
\end{lemma}

\begin{proof}
    We prove $ \norm{\partial_{\alpha}D_t\zeta}_{W^{s-3,\infty}(E_t)}\leq C\epsilon t^{-1/2-1/10+\delta_0}$. The estimates for $D_t^2\zeta$ and $\zeta_{\alpha}-1$ are similar. By Proposition \ref{good-replace}, it suffices to prove $ \norm{\partial_{\alpha}Q_{\alpha}}_{W^{s-3,\infty}(E_t)}\leq C\epsilon t^{-1/2-1/10+\delta_0}$. For $|\alpha|\leq t^{4/5}$, by Lemma \ref{transition-lemma}, we have
    \begin{align*}
        \partial_{\alpha}Q_{\alpha}=-\frac{4\alpha}{it^2}L_0\partial_{\alpha}Q_{\alpha}+\frac{2}{it}\Omega_0 \partial_{\alpha}Q_{\alpha}+\frac{4\alpha^2}{it^2}\partial_{\alpha}\partial_{\alpha}Q_{\alpha}
    \end{align*}
    Clearly, by the bootstrap assumptions \eqref{bootstrap-1}-\eqref{bootstrap-2'} and Sobolev embedding,
    \begin{align*}
        \norm{-\frac{4\alpha}{it^2}L_0\partial_{\alpha}Q_{\alpha}+\frac{2}{it}\Omega_0 \partial_{\alpha}Q_{\alpha}}_{W^{s-3,\infty}(|\alpha|\leq  t^{4/5})}\leq &\  Ct^{-1}\Big(\norm{L_0 \partial_{\alpha}Q_{\alpha}}_{H^{s-2}}+\norm{\Omega_0\partial_{\alpha}Q_{\alpha}}_{H^{s-2}}\Big)\\
        \leq &\  C\epsilon t^{-1+2\delta_0}.
    \end{align*}
    By the Lemma \ref{decay of D_tzeta} and Proposition \ref{good-replace}, we have 
    \begin{align*}
        \norm{\frac{4\alpha^2}{it^2}\partial_{\alpha}\partial_{\alpha}Q_{\alpha}}_{W^{s-3,\infty}(|\alpha|\leq t^{4/5})}\leq C\epsilon \frac{t^{8/5}}{t^2}\norm{\partial_{\alpha}^2Q_{\alpha}}_{W^{s-3,\infty}}
        \le C\epsilon^2t^{-\frac{2}{5}-\frac{1}{4}}\ln t
        = C\epsilon^2t^{-\frac{1}{2}-\frac{3}{10}}\ln t.
    \end{align*}

For $|\alpha|\geq t^{6/5}$, we use the following (see (3.6) of \cite{Wu2009})
\begin{equation}
    |f(\alpha,t)|\leq \frac{2}{|\alpha|^{1/2}}\norm{L_0f}_2+\frac{t}{|\alpha|^{3/2}}\norm{\Omega_0f}_2
\end{equation}
to obtain
\begin{align*}
    \norm{\partial_{\alpha}Q_{\alpha}}_{W^{s-3,\infty}(|\alpha|\geq t^{6/5})} \leq \left.\left(\frac{2}{|\alpha|^{1/2}}\norm{L_0 \partial_{\alpha}Q_{\alpha}}_2+\frac{t}{|\alpha|^{3/2}}\norm{\Omega_0 \partial_{\alpha}Q_{\alpha}}_2\right)\right|_{|\alpha|\geq t^{6/5}}\leq C\epsilon t^{-1/2-1/10+\delta_0}.
\end{align*}
This completes the proof.
\end{proof}

Before addressing the estimates for $b$ and $A-1$, we present a useful lemma that will play a key role in several proofs, particularly in Chapter 6. 
\begin{lemma}\label{lemma:transit derivatives}
For any $\alpha,\beta\in \mathbb{R}$ such that 
$$\frac{1}{2}|\alpha|\leq |\beta|\leq 2|\alpha|,$$
we have
    \begin{equation}
        |D_t\zeta(\alpha,t)D_t^3(\beta,t)|+|D_t\zeta(\alpha,t) D_t^3\bar{\zeta}(\beta,t)|\le C\epsilon^2t^{-1}.
    \end{equation}
\end{lemma}

\begin{proof}
    Without loss of generality we may assume that $\widehat{D_t\zeta}$ is supported in $|\xi|\in[t^{-2},t^2]$, otherwise we can decompose $D_t\zeta$ as
    $$D_t\zeta=P_{\leq t^{-2}}D_t\zeta+P_{>t^{-2}}D_t\zeta.$$
    Using the Bernstein inequality, it's clear that 
    $$\norm{P_{\leq t^{-2}}D_t\zeta}_{\infty}\leq C\epsilon t^{-1}.$$ Thus, it suffices to estimate $\norm{(P_{>t^{-2}}D_t\zeta) D_t^3\zeta}_{\infty}$ and we will omit the symbol $P_{>t^{-2}}$ for the remainder of the proof.
    
    For $|\alpha|\leq t^{3/4}$, we have 
    \[
    \norm{D_t\zeta}_{L^\infty(|\alpha|\le t^{3/4})}
    =\norm{\frac{4\alpha^2}{it^2}\partial_\alpha D_t\zeta-\frac{4\alpha}{it^2}L_0D_t\zeta+\frac{2}{it}\Omega_0D_t\zeta}_{L^\infty}
    \le C\epsilon t^{-1+\delta_0}.
    \]
    This implies
    \[
    \sup_{|\alpha|\leq t^{3/4}, \frac{|\alpha|}{2}\leq |\beta|\leq 2|\alpha|}\Big| D_t\zeta(\alpha,t) D_t^3\zeta(\beta,t)\Big|\le C\epsilon^2t^{-1}.
    \]
    For $|\alpha|>t^{3/4}$, we proceed as follows.
    \begin{align*}
        D_t\zeta(\alpha,t) D_t^3\zeta(\beta,t)
        &=D_t\zeta(\alpha,t) \partial_tD_t^2\zeta(\beta,t)+D_t\zeta(\alpha,t) b\partial_{\alpha}D_t^2\zeta(\beta,t),\\
        &=D_t\zeta(\alpha,t)\left(-\frac{it}{2\beta}D_t^2\zeta(\beta,t)+\frac{1}{\beta}\Omega_0D_t^2\zeta(\beta,t)\right)+o(\epsilon^4t^{-1}),\\
        &=\frac{\alpha}{\beta}\left(\partial_tD_t\zeta(\alpha,t)-\frac{1}{\alpha}\Omega_0D_t\zeta\right)D_t^2\zeta(\beta,t)+\frac{1}{\beta}D_t\zeta(\alpha,t)\Omega_0D_t^2\zeta(\beta,t)+o(\epsilon^4t^{-1}),\\
        &=\frac{\alpha}{\beta}D_t^2\zeta(\alpha,t)D_t^2\zeta(\beta,t)-\frac{1}{\alpha}\Omega_0D_t\zeta(\alpha,t) D_t^2\zeta(\beta,t)+\frac{1}{\beta}D_t\zeta(\alpha,t)\Omega_0D_t^2\zeta(\beta,t)+o(\epsilon^4t^{-1})
    \end{align*}
    which implies (recall that $\norm{\Omega_0P_{>t^{-2}}D_t\zeta}_{L^\infty}\le C\epsilon t^{\delta_0}\ln t$)
    \[
    \sup_{|\alpha|\geq t^{3/4}, \frac{|\alpha|}{2}\leq |\beta|\leq 2|\alpha|}|D_t\zeta(\alpha,t)D_t^3\zeta(\beta,t)|\le C\epsilon^2t^{-1}.
    \]
    The same arguments apply to $\norm{D_t\zeta D_t^3\bar{\zeta}}_{L^\infty}$. This completes the proof.
\end{proof}

\begin{cor}\label{use in chapter 6}
    \begin{equation}
        \norm{D_t\zeta \partial_\alpha D_t\zeta}_{L^\infty}+\norm{D_t\zeta \partial_\alpha D_t\bar{\zeta}}_{L^\infty}\le C\epsilon^2t^{-1}.
    \end{equation}
\end{cor}

\begin{proof}
    By taking $D_t$ derivative on both sides of the equation $(D_t^2-iA\partial_\alpha)\zeta=-i$, one gets
    \[
    D_t^3\zeta
    =i\partial_\alpha D_t\zeta+i(D_tA-b_\alpha)\zeta_\alpha+i(A-1)\partial_\alpha D_t\zeta
    =i\partial_\alpha D_t\zeta+O(\epsilon^3t^{-1}).
    \]  
    Here we have invoked a simple estimate for $b$ and $A-1$, which can be derived by applying Lemma \ref{singular}. A similar argument applies to $\norm{D_t\zeta \partial_\alpha D_t\bar{\zeta}}_{L^\infty}$. This completes the proof.
\end{proof}

Using the same argument as before, we have the following.
\begin{lemma}\label{lemma:transit2}
For any $\alpha,\beta\in \mathbb{R}$ such that 
$$\frac{1}{2}|\alpha|\leq |\beta|\leq 2|\alpha|,$$
we have for $1\leq k\leq s-1$,
    \begin{equation}
|D_t\zeta(\alpha,t)\partial_{\alpha}^kD_t\zeta(\beta,t)|+|D_t\zeta(\alpha,t) \partial_{\beta}^kD_t\bar{\zeta}(\beta,t)|\le C\epsilon^2t^{-1}
    \end{equation}
    and
    \begin{equation}
        |D_t\zeta(\alpha,t)\partial_{\alpha}^{k+1}\zeta(\beta,t)|+|D_t\zeta(\alpha,t) \partial_{\beta}^{k+1}\bar{\zeta}(\beta,t)|\le C\epsilon^2t^{-1}
    \end{equation}
\end{lemma}

\begin{proof}
    It's enough to note that $\norm{\partial_\alpha^kD_t\zeta}_{L^\infty}\le C\epsilon t^{-\frac{1}{4}}\ln t$ and $\norm{\partial_\alpha^{k+1}\zeta}_{L^\infty}\le C\epsilon t^{-\frac{1}{6}}\ln t$ (by Lemma \ref{decay of D_tzeta}) for any $1\le k\le s-1$.
\end{proof}

\begin{rem}
    We emphasize that Lemma \ref{lemma:transit2} doesn't work if $k=s$, since we cannot estimate $\norm{\partial_\alpha^sD_t\zeta}_{L^\infty}$ or $\norm{\partial_\alpha^{s+1}\zeta}_{L^\infty}$; however, by examining the proof of Lemma \ref{lemma:transit derivatives}, one finds that Lemma \ref{lemma:transit2} holds even for $k=s$ if $|\alpha|\ge t^{3/4}$.
\end{rem}

\subsection{The quantities $b$, $A$, etc}

\begin{lemma}\label{norms of b}
    We have, for any $\delta>\delta_0$,
    \begin{equation}\label{L2 norm of b}
        \norm{b(\cdot,t)}_{L^2}\leq C\epsilon^2 t^{-1/2},
    \end{equation}
    \begin{equation}\label{infinity norm of b}
    	\norm{b(\cdot,t)}_{L^{\infty}}\le C\epsilon^2 t^{-\frac{3}{4}}\ln t,
    \end{equation}
    \begin{equation}\label{H s-1 norm of b_alpha}
        \norm{\partial_\alpha b(\cdot,t)}_{H^{s-1}}\le C\epsilon^2t^{-1+\delta}
    \end{equation}
        \begin{equation}\label{W s-2 norm of b_alpha}
        \norm{\partial_\alpha b(\cdot,t)}_{W^{s-2,\infty}}\le C\epsilon^2t^{-7/6+\delta}
    \end{equation}
            \begin{equation}\label{W s-2 norm of b_alpha FOR LARGE ALPHA}
        \norm{\partial_\alpha b(\cdot,t)}_{W^{s-2,\infty}(|\alpha|\geq t)}\le C\epsilon^2t^{-3/2+\delta}
    \end{equation}
    and
    \begin{equation}\label{W s-3 norm of b_alpha}
        \norm{\partial_\alpha b(\cdot,t)}_{W^{s-3,\infty}}\le C\epsilon^2t^{-3/2+\delta}.
    \end{equation}
\end{lemma}

\begin{proof}
        For (\ref{L2 norm of b}), by Lemma \ref{control f by using (I-H)f}, we have
        \[
        \norm{b}_{L^2}\le C\norm{(I-\mathcal{H})b}_{L^2}
        =C\norm{(I-\mathcal{H})D_t\zeta\frac{\bar{\zeta}_\alpha-1}{\zeta_\alpha}}_{L^2}
        \le C\epsilon^2t^{-\frac{1}{2}}.
        \]
	Next we prove (\ref{infinity norm of b}). By applying Lemma \ref{control f by using (I-H)f}, followed by Lemma \ref{L infty of hilbert}, Lemma \ref{decay of D_tzeta}, and using the bootstrap assumptions, we obtain 
	\[
    \begin{aligned}
	\norm{b}_{L^\infty}\le &\ C\norm{(I-\mathcal{H})b}_{L^\infty}\leq \norm{(I-\mathcal{H})D_t\zeta\frac{\bar{\zeta}_{\alpha}-1}{\zeta_{\alpha}}}_{L^{\infty}}\\
    \leq &\ C\norm{D_t\zeta\frac{\bar{\zeta}_{\alpha}-1}{\zeta_{\alpha}}}_{L^{\infty}}+Ct^{-2}\norm{D_t\zeta\frac{\bar{\zeta}_{\alpha}-1}{\zeta_{\alpha}}}_{H^1}+C\norm{D_t\zeta\frac{\bar{\zeta}_{\alpha}-1}{\zeta_{\alpha}}}_{L^{\infty}}\ln t\\
    \leq &\ C\epsilon^2t^{-\frac{3}{4}}+Ct^{-2}\epsilon^2 t^{-1/2}+C\epsilon^2 t^{-3/4}\ln t\\
    \leq &\ C\epsilon^2 t^{-3/4}\ln t.
    \end{aligned}
	\]
    For (\ref{H s-1 norm of b_alpha}), we only prove the $s-1$ order as the other cases follow similarly and are easier to handle. A straightforward computation yields:
    \[
    \partial_\alpha^{s-1}(I-\mathcal{H})D_t\zeta\frac{\bar{\zeta}_\alpha-1}{\zeta_\alpha}
    =(I-\mathcal{H})\partial_\alpha^{s-1}\left(D_t\zeta\frac{\bar{\zeta}_\alpha-1}{\zeta_\alpha}\right)-[\partial_\alpha^{s-1},\mathcal{H}]D_t\zeta\frac{\bar{\zeta}_\alpha-1}{\zeta_\alpha}.
    \]
    It's easy to verify that
    \[
    [\partial_\alpha^{s-1},\mathcal{H}]D_t\zeta\frac{\bar{\zeta}_\alpha-1}{\zeta_\alpha}
    =\sum_{j=1}^{s-1}\partial_\alpha^{s-1-j}[\zeta_\alpha,\mathcal{H}]\frac{1}{\zeta_\alpha}\partial_\alpha^{j}\left(D_t\zeta\frac{\bar{\zeta}_\alpha-1}{\zeta_\alpha}\right)
    \]
    and thus
    \[
    \norm{[\partial_\alpha^{s-1},\mathcal{H}]D_t\zeta\frac{\bar{\zeta}_\alpha-1}{\zeta_\alpha}}_{L^2}\le C\epsilon^3t^{-1}.
    \]
    Let $(D_t\zeta)_L=P_{<t^{-2}}D_t\zeta$ and $(D_t\zeta)_H=P_{\ge t^{-2}}D_t\zeta$, then
    \[
    \norm{(D_t\zeta)_L}_{L^\infty}\le t^{-1}\norm{(D_t\zeta)_L}_{L^2}\le C\epsilon t^{-1}
    \]
    which is sufficient. For the high frequency part, define
    \[
    F=e^{it^2/4\alpha}(D_t\zeta)_H,\ G=e^{-it^2/4\alpha}\partial_\alpha^{s-1}\bar{\zeta}.
    \]
    We only need to verify
    \[
    \norm{(I-\mathcal{H})\partial_\alpha(FG)}_{L^2}\le C\epsilon^2t^{-1+\delta_0}
    \]
   since this case represents the most extreme scenario. For $|\alpha|>t^{1/2+2\delta_0}$, we note that
    \[
    \partial_\alpha F=e^{\frac{it^2}{4\alpha}}\left(\frac{1}{\alpha} L_0(D_t\zeta)_H-\frac{t}{2\alpha^2}\Omega_0(D_t\zeta)_H\right),\ \partial_\alpha G=e^{-\frac{it^2}{4\alpha}}\left(\frac{1}{\alpha}L_0\partial_\alpha^{s-1}\bar{\zeta}+\frac{t}{2\alpha^2}\tilde{\Omega}_0\partial_\alpha^{s-1}\bar{\zeta}\right),
    \]
    which implies
    \[
    G\partial_\alpha F=\partial_\alpha^{s-1}\bar{\zeta}\left(\frac{1}{\alpha}L_0(D_t\zeta)_H-\frac{t}{2\alpha^2}\Omega_0(D_t\zeta)_H\right).
    \]
    Observe that
    \begin{align*}
        -\frac{t}{2\alpha^2}\partial_\alpha^{s-1}\bar{\zeta}\Omega_0(D_t\zeta)_H
        &=-\frac{2}{t}\left(-\partial_\alpha^s\bar{\zeta}+\frac{1}{\alpha}L_0\partial_\alpha^{s-1}\bar{\zeta}-\frac{t}{2\alpha^2}\tilde{\Omega}_0\partial_\alpha^{s-1}\bar{\zeta}\right)\Omega_0(D_t\zeta)_H\\
        &=O(\epsilon^2 t^{-1+\delta_0}\ln t)_{L^2}.
    \end{align*}
    similarly,
    \[
    \norm{\frac{1}{\alpha}\partial_\alpha^{s-1}\bar{\zeta}L_0(D_t\zeta)_H}_{L^2(|\alpha|>t^{1/2+2\delta_0})}\le C\epsilon^2t^{-1+\delta_0}.
    \]
    This leads to $\norm{G\partial_\alpha F}_{L^2(|\alpha|>t^{1/2+2\delta_0})}\le C\epsilon^2t^{-1+\delta_0}\ln t$. Applying a similar argument to $F\partial_\alpha G$ gives $$\norm{F\partial_\alpha G}_{L^2(|\alpha|>t^{1/2+2\delta_0})}\le C\epsilon^2t^{-1+\delta_0}\ln t.$$ 
    Therefore,
    \[
    \norm{\partial_\alpha(FG)}_{L^2(|\alpha|>t^{1/2+2\delta_0})}\le C\epsilon^2t^{-1+\delta_0}\ln t.
    \]
    For $|\alpha|<t^{1/2+2\delta_0}$, using \eqref{equ:increase a derivative}, we have
    \[
    \norm{(D_t\zeta)_H}_{L^\infty(|\alpha|\le t^{1/2+2\delta_0})}\le \norm{\frac{4\alpha^2}{it^2}\partial_\alpha(D_t\zeta)_H}_{L^\infty(|\alpha|\le t^{1/2+2\delta_0})}+C\epsilon t^{-1+\delta_0}\ln t\le C\epsilon t^{-1+\delta_0}\ln t
    \]
    and thus
    \begin{align*}
        \norm{\partial_\alpha(FG)}_{L^2(|\alpha|<t^{1/2+2\delta_0})}
        &\le\norm{(D_t\zeta)_H\partial_\alpha^{s}\bar{\zeta}}_{L^2(|\alpha|<t^{1/2+2\delta_0})}+\norm{\partial_\alpha(D_t\zeta)_H\partial_\alpha^{s-1}\bar{\zeta}}_{L^2(|\alpha|<t^{1/2+2\delta_0})}\\
        &\le C\epsilon^2t^{-1+\delta_0}\ln t.
    \end{align*}
    Finally,
    \[
    \norm{(I-\mathcal{H})\partial_\alpha(FG)}_{L^2}\le C\norm{\partial_\alpha(FG)}_{L^2}\le C\epsilon^2t^{-1+\delta_0}\ln t,
    \]
    and (\ref{H s-1 norm of b_alpha}) is thus proved.

    For (\ref{W s-2 norm of b_alpha}), we still only need to consider $\partial_\alpha(FG)$. Note that, according to the calculation above,
    \begin{align*}
        \norm{G\partial_\alpha F}_{L^\infty(|\alpha|>t^{5/6})}
        &\le C\norm{\partial_\alpha^{s-1}\bar{\zeta}}_{L^\infty}\norm{\frac{1}{\alpha}L_0(D_t\zeta)_H-\frac{t}{2\alpha^2}\Omega_0(D_t\zeta)_H}_{L^\infty(|\alpha|>t^{5/6})},\\
        &\le C\epsilon^2 t^{-\frac{7}{6}+\delta_0}\ln t
    \end{align*}
    also, by using
    \[
    \frac{t}{2\alpha^2}(D_t\zeta)_H=\frac{2}{t}\left(-\partial_\alpha(D_t\zeta)_H+\frac{1}{\alpha}L_0(D_t\zeta)_H+\frac{t}{2\alpha^2}\Omega_0(D_t\zeta)_H\right)
    \]
    one finds $\norm{F\partial_\alpha G}_{L^\infty(|\alpha|>t^{5/6})}\le C\epsilon^2t^{-\frac{7}{6}}$. i.e. $\norm{\partial_\alpha(FG)}_{L^\infty(|\alpha|>t^{5/6})}\le C\epsilon^2t^{-\frac{7}{6}+\delta_0}\ln t$.

    For $|\alpha|\le t^{5/6}$, we raise the derivative of $(D_t\zeta)_H$ twice to obtain
    \[
    \norm{(D_t\zeta)_H}_{L^\infty(|\alpha|\le t^{5/6})}\le C\epsilon t^{-1+\delta_0}\ln t
    \]
    and use Lemma \ref{decay of D_tzeta} to get $\norm{\partial_\alpha^s\bar{\zeta}}_{L^\infty}\le C\epsilon t^{-1/6}\ln t$. Hence we may conclude that $\norm{\partial_\alpha(FG)}_{L^\infty}\le C\epsilon^2t^{-7/6+\delta_0}(\ln t)^2$. Now (\ref{W s-2 norm of b_alpha}) follows from Lemma \ref{L infty of hilbert}. Similar process works for (\ref{W s-3 norm of b_alpha}), since in that case we have $\norm{\partial_\alpha^{s-1}\bar{\zeta}}_{L^\infty}\le C\epsilon t^{-1/2}$. Finally, (\ref{W s-2 norm of b_alpha FOR LARGE ALPHA}) can be verified by slightly modifying the proof.
\end{proof}

\begin{proposition}\label{norms of A-1}
    We have 
\begin{equation}\label{L2 norm of A-1}
    \norm{A-1}_{H^s}\leq C\epsilon^2 t^{-1+\delta_0}\ln t,
\end{equation}
and
    \begin{equation}\label{W^{k,infty} of A-1}
        \norm{A-1}_{W^{s-2,\infty}}\le C\epsilon^2 t^{-1-\frac{1}{4}+\delta}
    \end{equation}
    for any $\delta>\delta_0$.
\end{proposition}

\begin{proof}
    First, we consider the case where there is no derivative. For (\ref{W^{k,infty} of A-1}), recall that
    \[
    (I-\mathcal{H})(A-1)=i(I-\mathcal{H})\left(D_t^2\zeta\frac{\bar{\zeta}_\alpha-1}{\zeta_\alpha}+D_t\zeta\frac{\partial_\alpha D_t\bar{\zeta}}{\zeta_\alpha}\right), 
    \]
    which can be rewritten as
    \begin{align*}
         (I-\mathcal{H})(A-1)=& i(I-\mathcal{H})\left( \frac{1}{2}\partial_t(Q_{\alpha} \bar{\theta}_{\alpha})+R\right),
    \end{align*}
    where
    \begin{align*}
        R=& D_t^2\zeta(\bar{\zeta}_{\alpha}-1)\left(\frac{1}{\zeta_{\alpha}}-1\right)+D_t\zeta \partial_{\alpha}D_t\bar{\zeta}\left(\frac{1}{\zeta_{\alpha}}-1\right)+D_t\zeta\frac{b_{\alpha}\bar{\zeta}_{\alpha}}{\zeta_{\alpha}}+b\partial_{\alpha}(D_t\zeta (\bar{\zeta}_{\alpha}-1))\\
        &+\partial_t\left((D_t\zeta (\bar{\zeta}_{\alpha}-1)- \frac{1}{2}Q_{\alpha} \bar{\theta}_{\alpha}\right).
    \end{align*}
    Due to Lemma \ref{decay of D_tzeta}, the bootstrap assumptions, and Corollary \ref{norms of b}, we have 
    \begin{equation}
        \norm{D_t^2\zeta(\bar{\zeta}_{\alpha}-1)\left(\frac{1}{\zeta_{\alpha}}-1\right)+D_t\zeta \partial_{\alpha}D_t\bar{\zeta}\left(\frac{1}{\zeta_{\alpha}}-1\right)+D_t\zeta\frac{b_{\alpha}\bar{\zeta}_{\alpha}}{\zeta_{\alpha}}+b\partial_{\alpha}(D_t\zeta (\bar{\zeta}_{\alpha}-1))}_{L^2}\leq C\epsilon^3 t^{-1}, 
    \end{equation}
    and
    \begin{equation}
       \norm{D_t^2\zeta(\bar{\zeta}_{\alpha}-1)\left(\frac{1}{\zeta_{\alpha}}-1\right)+D_t\zeta \partial_{\alpha}D_t\bar{\zeta}\left(\frac{1}{\zeta_{\alpha}}-1\right)+D_t\zeta\frac{b_{\alpha}\bar{\zeta}_{\alpha}}{\zeta_{\alpha}}+b\partial_{\alpha}(D_t\zeta (\bar{\zeta}_{\alpha}-1))}_{L^{\infty}}\leq C\epsilon^3 t^{-\frac{3}{2}}\ln t.
    \end{equation}
    Also, the estimate for $\partial_t((D_t\zeta(\bar{\zeta}_\alpha-1)-\frac{1}{2}Q_\alpha\bar{\theta}_\alpha))$ can be achieved by using a simple interpolation
    \[
    D_t\zeta(\bar{\zeta}_\alpha-1)-\frac{1}{2}Q_\alpha\bar{\theta}_\alpha
    =\frac{1}{2}(D_t\zeta-Q_\alpha)\bar{\theta}_\alpha+D_t\zeta\left((\bar{\zeta}_\alpha-1)-\frac{1}{2}\bar{\theta}_\alpha\right)
    \]
    and applying Lemma \ref{good-replace}. This completes the estimate for $R$.
    
    Now we move to the main term. Let $(Q_{\alpha})_L=P_{<t^{-2}}Q_\alpha$ and $(Q_{\alpha})_H=Q_\alpha-(Q_\alpha)_L$. By Bernstein inequality,  
    \begin{align*}
        \norm{\partial_t(Q_{\alpha})_L (\bar{\zeta}_{\alpha}-1)}_{L^2}
        \leq &\ \norm{(Q_{\alpha t})_L \bar{\theta}_{\alpha}}_{L^2}+\norm{(Q_{\alpha})_L \partial_t \bar{\theta}_{\alpha}}_{L^2}\\
        \leq &\  Ct^{-1}\norm{Q_{\alpha t}}_{L^2}\norm{\theta_{\alpha}}_{L^2}+t^{-1}\norm{Q_{\alpha}}_{L^{2}}\norm{\partial_t\theta_{\alpha}}_{L^2},\\
        \leq &\  C\epsilon^2 t^{-1}, 
    \end{align*}
Similarly,
    \begin{align*}
       \norm{\partial_t(Q_{\alpha})_L (\bar{\zeta}_{\alpha}-1)}_{L^\infty}\leq C\epsilon^2 t^{-3/2}
    \end{align*}
 Next, define $(Q_{\alpha})_H=Fe^{-it^2/4\alpha}$ and $\bar{\theta}_{\alpha}=Ge^{it^2/4\alpha}$. Then, it follows that
 $$\partial_t((Q_{\alpha})_H\bar{\theta}_{\alpha})=\partial_t(FG).$$
 Note that
 \[
 \partial_tF=\frac{1}{\alpha}e^{\frac{it^2}{4\alpha}}\Omega_0(Q_\alpha)_H,\ \partial_tG=\frac{1}{\alpha}e^{-\frac{it^2}{4\alpha}}\tilde{\Omega}_0\theta_\alpha.
 \]
 Repeating the relevant proof for $b$ gives
$$\norm{\partial_t(FG)}_{L^2}\le C\epsilon^2t^{-1+\delta_0}\ln t,\quad \norm{\partial_t(FG)}_{L^\infty}\le C\epsilon^2t^{-\frac{3}{2}+\delta},\quad \forall \delta>\delta_0.$$

Thus, by choosing $\delta=1/10$ and using Lemma \ref{L infty of hilbert}, we obtain
\begin{equation}
\norm{(I-\mathcal{H})(A-1)}_{\infty}\leq C\epsilon^2 t^{-1-\frac{2}{5}},
\end{equation}
and
\begin{equation}
\norm{(I-\mathcal{H})(A-1)}_{2}\leq C\epsilon^2 t^{-1+\delta_0}\ln t.
\end{equation}
Using the same argument, we obtain
\begin{equation}
\norm{(I-\mathcal{H})(A-1)}_{W^{s-3,\infty}}\leq C\epsilon^2 t^{-1-\frac{2}{5}},
\end{equation}
and
\begin{equation}
\norm{(I-\mathcal{H})(A-1)}_{H^s}\leq C\epsilon^2 t^{-1+\delta_0}\ln t.
\end{equation}
Note that this argument encounters a slight issue when applied to 
$\norm{\partial_\alpha^{s-2}(I-\mathcal{H})(A-1)}_{L^\infty}$, since a term of highest order cannot provide the necessary $t^{-1/2}$ decay. We address this case in detail as follows. Since $A-1$ is real-valued, by Lemma \ref{control f by using (I-H)f},
\begin{align*}
    \norm{\partial_{\alpha}^{s-2}(A-1)}_{L^{\infty}}\leq C\norm{(I-\mathcal{H})(A-1)}_{W^{s-2,\infty}}.
\end{align*}
We have already obtain
$$\norm{(I-\mathcal{H})(A-1)}_{W^{s-3,\infty}}\leq C\epsilon^2 t^{-3/2+\delta_0}.$$
For $\norm{\partial_{\alpha}^{s-2}(I-\mathcal{H})(A-1)}_{L^{\infty}}$, note that
\begin{align*}
    &\partial_{\alpha}^{s-2}(I-\mathcal{H})(A-1)
    = \partial_{\alpha}^{s-2}(I-\mathcal{H})\left(D_t^2\zeta\frac{\bar{\zeta}_\alpha-1}{\zeta_\alpha}+D_t\zeta\frac{\partial_\alpha D_t\bar{\zeta}}{\zeta_\alpha}\right)\\
    =&(I-\mathcal{H})\partial_{\alpha}^{s-2}\left(D_t^2\zeta\frac{\bar{\zeta}_\alpha-1}{\zeta_\alpha}+D_t\zeta\frac{\partial_\alpha D_t\bar{\zeta}}{\zeta_\alpha}\right)-[\partial_{\alpha}^{s-2},\mathcal{H}]\left(D_t^2\zeta\frac{\bar{\zeta}_\alpha-1}{\zeta_\alpha}+D_t\zeta\frac{\partial_\alpha D_t\bar{\zeta}}{\zeta_\alpha}\right).
\end{align*}
It's easy to obtain
\begin{equation}
    \norm{[\partial_{\alpha}^{s-2},\mathcal{H}]\left(D_t^2\zeta\frac{\bar{\zeta}_\alpha-1}{\zeta_\alpha}+D_t\zeta\frac{\partial_\alpha D_t\bar{\zeta}}{\zeta_\alpha}\right)}_{L^{\infty}}\leq C\epsilon^2 t^{-3/2+\delta_0}.
\end{equation}
By Lemma \ref{L infty of hilbert}, it suffices to estimate
$$\norm{\partial_{\alpha}^{s-2}\left(D_t^2\zeta\frac{\bar{\zeta}_\alpha-1}{\zeta_\alpha}+D_t\zeta\frac{\partial_\alpha D_t\bar{\zeta}}{\zeta_\alpha}\right)}_{L^{\infty}}.$$
We focus on estimating the term
\[
\norm{D_t^2\zeta\frac{\partial_\alpha^{s-2}(\bar{\zeta}_\alpha-1)}{\zeta_\alpha}+(D_t\zeta)_H\frac{\partial_\alpha^{s-1}D_t\bar{\zeta}}{\zeta_\alpha}}_{L^{\infty}}
\]
as this is the most challenging term that arises in the estimates.

For $|\alpha|\ge t^{5/6}$, we have 
\begin{align*}
    D_t^2\zeta\partial_\alpha^{s-2}(\bar{\zeta}_\alpha-1)+(D_t\zeta)_H\partial_\alpha^{s-1}D_t\bar{\zeta}
    &=D_t(D_t\zeta)_L\partial_\alpha^{s-1}\bar{\zeta}+D_t((D_t\zeta)_H\partial_\alpha^{s-1}\bar{\zeta})-(D_t\zeta)_H[\partial_\alpha^{s-1},D_t]\bar{\zeta}.
\end{align*}
Define $F=e^{it^2/4\alpha}(D_t\zeta)_H,\ G=e^{-it^2/4\alpha}\partial_\alpha^{s-1}\bar{\zeta}$, then we have
\[
D_t(FG)=\partial_t(FG)+b\partial_\alpha(FG);
\]
note that
\[
\norm{\partial_tF}_{L^\infty(|\alpha|\ge t^{5/6})}=\norm{e^{\frac{it^2}{4\alpha}}\frac{1}{\alpha}\Omega_0(D_t\zeta)_H}_{L^\infty(|\alpha|\ge t^{5/6})}\le C\epsilon t^{-\frac{5}{6}+\delta_0}\ln t,
\]
and
\[\ \norm{\partial_tG}_{L^\infty(|\alpha|\ge t^{5/6})}\le C\epsilon t^{-\frac{5}{6}+\delta_0}.\]
Thus, we get
$$\norm{\partial_t(FG)}_{L^\infty(|\alpha|\ge t^{5/6})}\le C\epsilon^2t^{-\frac{1}{2}-\frac{5}{6}+\delta_0}\ln t\le C\epsilon^2t^{-1-\frac{1}{3}+\delta_0}\ln t.$$
Similarly, we have
\[
\norm{b\partial_\alpha(FG)}_{L^\infty(|\alpha|\ge t^{5/6})}\le C\epsilon^2t^{-1-\frac{1}{3}-\delta'}
\]
and thus $\norm{D_t(FG)}_{L^\infty(|\alpha|\ge t^{5/6})}\le C\epsilon^2t^{-1-\frac{1}{3}+2\delta_0}$.

For $|\alpha|<t^{5/6}$, simply note that
\[
\norm{(D_t\zeta)_H}_{L^\infty(|\alpha|<t^{5/6})}\le C\epsilon t^{-1+\delta_0}\ln t
\]
(raise its derivative twice) and use Lemma \ref{decay of D_tzeta} to get $\norm{\partial_\alpha^{s-1}D_t\bar{\zeta}}_{L^\infty}\le C\epsilon t^{-1/4}\ln t$, which implies
\[
\norm{D_t^2\zeta\frac{\partial_\alpha^{s-2}(\bar{\zeta}_\alpha-1)}{\zeta_\alpha}+(D_t\zeta)_H\frac{\partial_\alpha^{s-1}D_t\bar{\zeta}}{\zeta_\alpha}}_{L^\infty(|\alpha|<t^{5/6})}\le C\epsilon t^{-1-\frac{1}{4}+\delta_0}(\ln t)^2.
\]

Now we consider the error terms. For the commutator, note that
\[
[\partial_\alpha^{s-1},D_t]\bar{\zeta}=[\partial_\alpha^{s-1},b\partial_\alpha]\bar{\zeta}=\sum_{j=1}^{s-1}C_{j,k}\partial_\alpha^jb\partial_\alpha^{s-j}\bar{\zeta}
\]
where $C_{j,k}=k!/j!(k-j)!$.
Note that $\norm{(D_t\zeta)_H}_{L^\infty(|\alpha|<t)}\le C\epsilon t^{-1/2}$, so we invoke (\ref{W s-3 norm of b_alpha}) and (\ref{W s-2 norm of b_alpha FOR LARGE ALPHA}) to obtain
\begin{align*}
    \norm{(D_t\zeta)_H[\partial_\alpha^{s-1},D_t]\bar{\zeta}}_{L^\infty}
    &\le C\epsilon^4t^{-\frac{1}{4}-\frac{3}{2}+\delta-\frac{1}{2}}+\norm{(D_t\zeta)_H\partial_\alpha^{s-1}b\partial_\alpha\bar{\zeta}}_{L^\infty(|\alpha|\ge t)}\\
    &\ \ \ +\norm{(D_t\zeta)_H\partial_\alpha^{s-1}b\partial_\alpha\bar{\zeta}}_{L^\infty(|\alpha|<t)},\\
    &\le C\epsilon^3t^{-2-\frac{1}{4}+\delta}+C\epsilon^3t^{-\frac{1}{4}-\frac{3}{2}+\delta}+C\epsilon^3t^{-\frac{1}{2}-\frac{7}{6}+\delta},\\
    &\le C\epsilon^3t^{-\frac{3}{2}}.
\end{align*}
For the low frequency term, there holds
\[
[D_t,P_{<t^{-2}}]=2t^{-1}\tilde{P}_{<t^{-2}},
\]
where $\tilde{P}_{<t^{-2}}$ is another frequency cut-off, which leads to
\[
\norm{D_t(D_t\zeta)_L\partial_\alpha^{s-1}\bar{\zeta}}_{L^\infty}\le C\epsilon^2t^{-\frac{3}{2}}.
\]
Now we have successfully obtained the desired estimate for the extreme term, and thus the proof is completed.
\end{proof}

\begin{rem}\label{decay for A-1 near t}
    Proposition \ref{norms of A-1} implies that there exists $C>0$ such that $1-C\epsilon^2\le A\le 1+C\epsilon^2$, which will be frequently used without explanation.
    
    Also, one may infer from the proof that
    \[
    \norm{A-1}_{W^{s-2,\infty}(|\alpha|\ge t^{4/5})}\le C\epsilon^2t^{-1-\frac{3}{10}+\delta_0}\ln t=C\epsilon^2t^{-\frac{5}{4}-\frac{1}{20}+\delta_0}\ln t
    \]
    which is important in Chapter 7.
\end{rem}

Also, the result of Proposition \ref{norms of A-1} allows us to analyze the asymptotic behavior of $b$:

\begin{cor}\label{structure of b}
    Suppose that $Ct^{1-\mu}\le|\alpha|\le Ct^{1+\mu}$, then
    \begin{equation}
        b=-\frac{2\alpha}{t}|D_t^2\zeta|^2+O(\epsilon^2t^{-1-\delta})
    \end{equation}
    for some $\delta>0$.
\end{cor}

\begin{proof}
    Note that
    \[
    (I-\mathbb{H})b=(I-\mathcal{H})b+(\mathbb{H}-\mathcal{H})b
    \]
    and we assert that
    \[
    \norm{(\mathbb{H}-\mathcal{H})b}_{L^\infty}\le C\epsilon^3t^{-1-\delta}
    \]
    for certain $\delta>0$. First, it's clear that
            \begin{align*}
                &\left|\int_{|\alpha-\beta|<1}\left(\frac{\zeta_\beta}{\zeta(\alpha)-\zeta(\beta)}-\frac{1}{\alpha-\beta}\right)b(\beta)d\beta\right|\\
                =&\ \left|\int_{|\alpha-\beta|<1}\left(\frac{\zeta_\beta(\alpha-\beta)-(\zeta(\alpha)-\zeta(\beta))}{(\zeta(\alpha)-\zeta(\beta))(\alpha-\beta)}\right)b(\beta)d\beta\right|,\\
                =&\ \left|\int_{|\alpha-\beta|<1}\left(\frac{\zeta_\beta(\alpha-\beta)-\zeta_\alpha(\alpha-\beta)-\zeta''(\xi)(\alpha-\beta)^2}{(\zeta(\alpha)-\zeta(\beta))(\alpha-\beta)}\right)b(\beta)d\beta\right|,\\
                =&\ \left|\int_{|\alpha-\beta|<1}\left(\frac{(\zeta''(\eta)-\zeta''(\xi))(\alpha-\beta)^2}{(\zeta(\alpha)-\zeta(\beta))(\alpha-\beta)}\right)b(\beta)d\beta\right|,\\
                \le&\ C\norm{\zeta_{\alpha\alpha}}_{L^\infty}\norm{b}_{L^\infty},\\
                \le&\ C\epsilon^3t^{-1-\delta}.
            \end{align*}
    On the other hand,
    \begin{align*}
      & \left|\int_{|\alpha-\beta|\ge   1}\left(\frac{\zeta_\beta}{\zeta(\alpha)-\zeta(\beta)}-\frac{1}{\alpha-\beta}\right)b(\beta)d\beta\right|\\
      \le&\ \norm{b}_{L^\infty}\left(\int_{1\le|\alpha-\beta|\le t^2}+\int_{|\alpha-\beta|\ge t^2}\right)\left|\frac{\zeta_\beta-1}{\zeta(\alpha)-\zeta(\beta)}+\frac{\alpha-\zeta(\alpha)-(\beta-\zeta(\beta))}{(\zeta(\alpha)-\zeta(\beta))(\alpha-\beta)}\right|d\beta,\\
      \le&\ C\epsilon^2t^{-\frac{1}{2}-\delta}\left(C\norm{\zeta_\alpha-1}_{L^\infty}\ln t+Ct^{-2}\norm{\zeta_\alpha-1}_{L^2}\right),\\
      \le&\ C\epsilon^3t^{-1-\delta'}.
    \end{align*}
    So we have
    \begin{equation}
        (I-\mathbb{H})b
        =(I-\mathcal{H})b+O(\epsilon^3t^{-1-\delta})
        =(I-\mathcal{H})D_t\zeta(\bar{\zeta}_\alpha-1)+O(\epsilon^3t^{-1-\delta}).
    \end{equation}
    Since $D_t^2\zeta-iA\zeta_\alpha=-i$, we have $\bar{\zeta}_\alpha-1=iD_t^2\bar{\zeta}-(A-1)\bar{\zeta}_\alpha$ and thus
    \begin{align*}
        (I-\mathcal{H})D_t\zeta(\bar{\zeta}_\alpha-1)
        &=(I-\mathcal{H})iD_t\zeta D_t^2\bar{\zeta}-(I-\mathcal{H})D_t\zeta(A-1)\zeta_\alpha.
    \end{align*}
    It's clear that
    \[
    \norm{D_t\zeta(A-1)\zeta_\alpha}_{L^\infty}
    \le C\epsilon\norm{A-1}_{L^\infty}
    \le C\epsilon^3t^{-1-\delta}
    \]
    and
    \[
    \norm{\partial_\alpha(D_t\zeta(A-1)\zeta_\alpha)}_{L^\infty}
    \le C\epsilon^3t^{-1-\delta}.
    \]
    Hence, using the same technique in the last part of Proposition \ref{norms of A-1}, we get
    \[
    \norm{(I-\mathcal{H})(D_t\zeta(A-1)\zeta_\alpha)}_{L^\infty}\le C\epsilon^3t^{-1-\delta}.
    \]

    Now we have
    \begin{equation}
        (I-\mathbb{H})b=(I-\mathcal{H})iD_t\zeta D_t^2\bar{\zeta}+O(\epsilon^3t^{-1-\delta}).
    \end{equation}
    Let $(D_t\zeta)_L$ and $(D_t\zeta)_H$ be as before. Clearly, by Bernstein inequality,
$$\norm{(I-\mathcal{H})((D_t\zeta)_L D_t^2\bar{\zeta})}_{\infty}\leq C\norm{(D_t\zeta)_L D_t^2\bar{\zeta}}_{W^{1,\infty}}\leq C\norm{(D_t\zeta)_L}_{W^{1,\infty}}\norm{D_t^2\zeta}_{W^{1,\infty}}\leq C\epsilon^2 t^{-3/2}.$$
Using
\begin{equation}
    (D_t\zeta)_H(\alpha,t)=\frac{2\alpha}{it}\partial_t (D_t\zeta)_H+\frac{1}{\alpha}\Omega_0 (D_t\zeta)_H,
\end{equation}
and
\begin{equation}
    \norm{\frac{1}{\alpha}\Omega_0 (D_t\zeta)_H}_{L^{\infty}(|\alpha|\geq t^{1-\mu})}\leq Ct^{-1+\mu}\norm{\Omega_0(D_t\zeta)_H}_{\infty}\leq Ct^{-1+\mu+\delta_0}\ln t.
\end{equation}
So 
\begin{align*}
(I-\mathbb{H})b=-(I-\mathcal{H})\frac{2\alpha}{t}\partial_t(D_t\zeta)_H D_t^2\bar{\zeta}+O(\epsilon^2 t^{-3/2}),
\end{align*}
which can be rewritten as
\begin{align*}
(I-\mathbb{H})b=-(I-\mathcal{H})\frac{2\alpha}{t} |D_t^2\zeta|^2-(I-\mathcal{H})\Big(b\partial_{\alpha}(D_t\zeta)_H D_t^2\bar{\zeta}+(D_t\zeta)_L D_t^2\bar{\zeta}\Big)+O(\epsilon^2 t^{-3/2}).
\end{align*}
Note that
\begin{align*}
    \norm{(I-\mathcal{H})\Big(b\partial_{\alpha}(D_t\zeta)_H D_t^2\bar{\zeta}+(D_t\zeta)_L D_t^2\bar{\zeta}\Big)}_{\infty}\leq C\epsilon^2 t^{-3/2},
\end{align*}
    so we get
    \[
    b=-\frac{2\alpha}{t}|D_t^2\zeta|^2+O(\epsilon^2t^{-3/2})
    \]
    as desired.
\end{proof}

\begin{rem}
    One may infer that $\norm{b}_{L^\infty(Ct^{1-\mu}\le|\alpha|\le Ct^{1+\mu})}\le C\epsilon^2t^{-1+\mu}$.
\end{rem}

Using Corollary \ref{structure of b}, we obtain the uniform control of change of variables.
\begin{cor}\label{control of change of coordinates}
Let $T$ be as in the bootstrap assumptions. We have
\begin{equation}
    \sup_{t\in [0,T]}\norm{\kappa_{\alpha}(\cdot,t)-\kappa_{\alpha}(\cdot,0)}_{\infty}\leq C\epsilon^2.
\end{equation}
\end{cor}
\begin{proof}
    Recall that $b=\kappa_t\circ \kappa^{-1}$. So we have $\kappa_t(\alpha,t)=b(\kappa(\alpha,t),t)$. Therefore
    $$\partial_t \kappa_{\alpha}=b_{\alpha}\circ\kappa \kappa_{\alpha}.$$
    By Corollary \ref{structure of b}, we have 
    \begin{align*}
        |\kappa_{\alpha}(\alpha,t)-\kappa_{\alpha}(\alpha,0)|=&\Big|\int_0^t b_{\alpha}\circ\kappa(\alpha,\tau)\kappa_{\alpha}(\alpha,\tau)d\tau\Big|\\
        \leq &\  C\epsilon^2\int_0^t (1+\tau)^{-1-\delta'}d\tau \sup_{\tau\in [0,t]} \Big(\norm{\kappa_{\alpha}(\cdot,\tau)-\kappa_{\alpha}(\cdot,0)}_{\infty}+\norm{\kappa_{\alpha}(\cdot,0)}_{\infty}\Big),
    \end{align*}
    which implies
    \begin{align*}
        \sup_{t\in [0,T]}\norm{\kappa_{\alpha}(\cdot,t)-\kappa_{\alpha}(\cdot,0)}_{\infty}\leq C\epsilon^2\Big( \sup_{t\in [0,T]}\norm{\kappa_{\alpha}(\cdot,t)-\kappa_{\alpha}(\cdot,0)}_{\infty}+\norm{\kappa_{\alpha}(\cdot,0)}_{\infty}\Big)
    \end{align*}
    and this completes the proof.
\end{proof}

Also, the technique of Proposition \ref{norms of A-1} straightly applies to the estimate of $\frac{a_t}{a}\circ \kappa^{-1}$.
\begin{proposition}\label{norm of at/a}
We have 
\begin{equation} 
    \norm{\frac{a_t}{a}\circ \kappa^{-1}}_{W^{s-2,\infty}}\le C\epsilon^2t^{-1-\delta},
\end{equation}
\begin{equation}
    \norm{\frac{a_t}{a}\circ\kappa^{-1}}_{H^s}\le C\epsilon^2 t^{-1+\delta_0}\ln t.
\end{equation}
\end{proposition}

\subsection{Localization lemma} 
Singular integrals are non-local operators, which can be inconvenient in certain situations. To address this, we establish two localization lemmas for trilinear singular integrals. Consider the following expression:
$$I(f,g,h)(\alpha,t):=\frac{1}{\pi i}\int \frac{(f(\alpha,t)-f(\beta,t))(g(\alpha,t)-g(\beta,t))}{(\alpha-\beta)^2}h(\beta,t)d\beta.$$
We consider two cases:

\noindent \textbf{Case I:} $f$ and $g$ almost antiholomorphic, and $h$ almost holomorphic. In this case, we establish Lemma \ref{decouple}.

\noindent \textbf{Case II:} $f$ and $h$ almost antiholomorphic, and $g$ almost holomorphic. In this case, we establish Lemma \ref{decouple-2}.

\subsubsection{Case I.}
For our purposes, we take $f$, $g$, and $h$ as follows:
\begin{itemize}
\item [(1)] $f=P_{\geq t^{-2}}Q_{\alpha}$, or $f=\partial_{\alpha}^jQ_{\alpha}$, $j=1, \cdots, s$.

\item [(2)] $g=P_{\geq t^{-2}}Q_{\alpha}$, or $g=\partial_{\alpha}^kQ_{\alpha}$, $k=1, \cdots, s$. 

\item [(3)] $h=\partial_{\alpha}^{l+1}\bar{\theta}$ or $h=\partial_{\alpha}^{l+1}(\bar{Q}_{\alpha})$, where $0\leq l\leq s$.
In addition, we require $j+k+l\leq s$.
\end{itemize}
Let's recall the bounds for such functions $f$, $g$, and $h$ as follows:
\begin{itemize}
\item [(1)] $\Omega_0 P_{\ge t^{-2}}Q_\alpha$ is not assumed to be in $ L^2$. We can only estimate $\Omega_0 P_{\ge t^{-2}}Q_\alpha$ in $L^{\infty}$. By Lemma \ref{Omega_0D_t zeta high}, we have
\begin{equation}\label{equ:localization vector bound 1}
    \norm{\Omega_0P_{\ge t^{-2}}Q_\alpha}_{L^{\infty}}\leq C\epsilon t^{\delta_0}\ln t.
\end{equation}

\item [(2)] By the bootstrap assumptions,
\begin{equation}\label{equ:localization vector bound 2}
    \norm{\Omega_0 \partial_{\alpha}^jQ_{\alpha}}_{L^2}\leq C\epsilon t^{2\delta_0},\quad \norm{\Omega_0 \partial_{\alpha}^j\theta_{\alpha}}_{L^2}\leq C\epsilon t^{2\delta_0}, \quad j=1,\cdots, s-1,
\end{equation}
and
\begin{equation}\label{equ:localization vector bound 3}
    \norm{L_0 \partial_{\alpha}^jQ_{\alpha}}_{L^2}\leq C\epsilon t^{2\delta_0},\quad \norm{L_0 \partial_{\alpha}^j\theta_{\alpha}}_{L^2}\leq C\epsilon t^{2\delta_0}, \quad j=0,\cdots, s-1.
\end{equation}
\end{itemize}
We now proceed with the analysis.

\begin{lemma}\label{decouple}
    Under the bootstrap assumptions, define $S(t)=\{\alpha\in \mathbb{R}\ |\ t^{4/5}\le|\alpha|\le t^{6/5}\}$. Then, we have
\begin{equation}
    1_{S(t)}(\alpha)I(f,g,h)(\alpha,t)=2\frac{it^2}{4\alpha^2}fgh+e,
\end{equation}
where
\begin{equation}
    \norm{e}_{L^2(S(t))}\le C\epsilon^3 t^{-\frac{5}{4}-\delta}
\end{equation}
for some $\delta>0$. Specifically, we can choose $\delta=\frac{1}{16}$.
\end{lemma}

\begin{proof}
 Recall that $f=P_{\ge t^{-2}}Q_\alpha$ or $\partial_{\alpha}^kQ_{\alpha}$ for $1\leq k\leq s$. Without loss of generality, we assume $k\leq s-2$. Otherwise, we replace $f$ by $g$. 
Decompose $I(f,g,h)$ as follows:
\begin{align*}
    I(f,g,h)
    &=\left(\frac{1}{\pi i}\int_{|\alpha-\beta|\ge t^{9/16}}+\frac{1}{\pi i}\int_{|\alpha-\beta|\le t^{9/16}}\right)\frac{(f(\alpha)-f(\beta))(g(\alpha)-g(\beta))}{(\alpha-\beta)^2}h(\beta)d\beta,\\
    &=I_1(f,g,h)+I_2(f,g,h).
\end{align*}

\noindent \textbf{Case 1:} $f=\partial_\alpha^kQ_\alpha$ for $1\le k\le s-2$.

We define $F(\alpha,t)=e^{\frac{it^2}{4\alpha}}f(\alpha,t)$, $G=e^{\frac{it^2}{4\alpha}}g(\alpha,t)$, $H(\alpha,t)=e^{-\frac{it^2}{4\alpha}}h(\alpha,t)$ and $E_{\alpha,t}=\{\beta\in \mathbb{R}\ |\ |\alpha-\beta|\le t^{9/16}\}$. We now decompose $I_2$ as
\begin{align*}
    I_2(f,g,h)
    &=\frac{1}{\pi i}\int_{E_{\alpha,t}} \frac{(e^{-\frac{it^2}{4\alpha}}F(\alpha,t)-e^{-\frac{it^2}{4\beta}}F(\beta,t))(e^{-\frac{it^2}{4\alpha}}G(\alpha,t)-e^{-\frac{it^2}{4\beta}}G(\beta,t))}{(\alpha-\beta)^2}e^{\frac{it^2}{4\beta}}H(\beta,t)d\beta\\
    &=I_{m}+I_{e},
\end{align*}
where the main term
$$I_{m}:= FGH\frac{1}{\pi i}\int_{E_{\alpha,t}}\left(\frac{e^{-\frac{it^2}{4\alpha}}-e^{-\frac{it^2}{4\beta}}}{\alpha-\beta}\right)^2e^{\frac{it^2}{4\beta}}d\beta,$$
and the error term
\begin{align*}
    I_{e}=&H(\alpha,t)\frac{1}{\pi i}\int_{E_{\alpha,t}} \frac{(G(\alpha,t)-G(\beta,t))(e^{-\frac{it^2}{4\alpha}}F(\alpha,t)-e^{-\frac{it^2}{4\beta}}F(\beta,t))}{(\alpha-\beta)^2}e^{-\frac{it^2}{4\beta}}e^{\frac{it^2}{4\beta}}d\beta\\
    &+G(\alpha,t)H(\alpha,t)\frac{1}{\pi i}\int_{E_{\alpha,t}} \frac{(e^{-\frac{it^2}{4\alpha}}-e^{-\frac{it^2}{4\beta}})(F(\alpha,t)-F(\beta,t))}{(\alpha-\beta)^2}e^{-\frac{it^2}{4\beta}}e^{\frac{it^2}{4\beta}}d\beta\\
&+\frac{1}{\pi i}\int_{E_{\alpha,t}} \frac{(e^{-\frac{it^2}{4\alpha}}F(\alpha,t)-e^{-\frac{it^2}{4\beta}}F(\beta,t))(e^{-\frac{it^2}{4\alpha}}G(\alpha,t)-e^{-\frac{it^2}{4\beta}}G(\beta,t))}{(\alpha-\beta)^2}e^{\frac{it^2}{4\beta}}\Big(H(\beta,t)-H(\alpha,t)\Big)d\beta\\
    :=& I_{e1}+I_{e2}+I_{e3}.
\end{align*}
We first handle the main term. We claim that
\[
I_m=\frac{it^2}{2\alpha^2}e^{-\frac{it^2}{4\alpha}}FGH+O(\epsilon^3t^{-\frac{5}{4}-\frac{5}{16}})_{H^s}=\frac{it^2}{2\alpha^2}fgh+O(\epsilon^3t^{-\frac{5}{4}-\frac{1}{16}})_{H^s}.
\]
First, observe that
\[
\norm{\frac{1}{\pi i}\int_{R-E_{\alpha,t}}\left(\frac{e^{-\frac{it^2}{4\alpha}}-e^{-\frac{it^2}{4\beta}}}{\alpha-\beta}\right)^2e^{\frac{it^2}{4\beta}}d\beta}_{L^\infty(S(t))}\le C\int_{t^{\frac{9}{16}}}^\infty\frac{dr}{r^2}\le Ct^{-\frac{9}{16}},
\]
which implies
\begin{align*}
    \norm{FGH\frac{1}{\pi i}\int_{R-E_{\alpha,t}}\left(\frac{e^{-\frac{it^2}{4\alpha}}-e^{-\frac{it^2}{4\beta}}}{\alpha-\beta}\right)^2e^{\frac{it^2}{4\beta}}d\beta}_{L^2(S(t))}
    \le Ct^{-\frac{9}{16}}\norm{fgh}_{L^2(S(t))}
    \le C\epsilon^3t^{-\frac{5}{4}-\frac{1}{16}}.
\end{align*}
Now we have
\[
I_m=\frac{1}{\pi i}FGH\int\left(\frac{e^{-\frac{it^2}{4\alpha}}-e^{-\frac{it^2}{4\beta}}}{\alpha-\beta}\right)^2e^{\frac{it^2}{4\beta}}d\beta+O(\epsilon^3t^{-\frac{5}{4}-\frac{1}{16}})_{L^2}.
\]

The integral is computed straightly. Set $\xi=\frac{t^2}{4\alpha},\ \eta=\frac{t^2}{4\beta}$ (note that they are real) and we derive
\begin{align*}
    \frac{1}{\pi i}\int\left(\frac{e^{-\frac{it^2}{4\alpha}}-e^{-\frac{it^2}{4\beta}}}{\alpha-\beta}\right)^2e^{\frac{it^2}{4\beta}}d\beta
    &=-\frac{4\xi^2}{t^2}\frac{1}{\pi i}\int\left(\frac{e^{-i\xi}-e^{-i\eta}}{\xi-\eta}\right)^2e^{i\eta}d\eta.
\end{align*}
Recall that the integral is of Cauchy principal value (which implies $\int\frac{1}{\xi-\eta}d\eta=0$), and observe that it isn't a singular integral near $\eta=\xi$, so we apply integration by parts to obtain
\begin{align*}
    -\frac{4\xi^2}{t^2}\frac{1}{\pi i}\int\left(\frac{e^{-i\xi}-e^{-i\eta}}{\xi-\eta}\right)^2e^{i\eta}d\eta
    &=\frac{4\xi^2}{t^2}\frac{1}{\pi i}\int\frac{e^{i\eta}}{\xi-\eta}\left(-2i(e^{-i\xi}-e^{-i\eta})e^{-i\eta}+i(e^{-i\xi}-e^{-i\eta})^2\right)d\eta,\\
    &=-\frac{8i\xi^2}{t^2}\frac{1}{\pi i}\int\frac{e^{-i\xi}-e^{-i\eta}}{\xi-\eta}d\eta+\frac{4i\xi^2}{t^2}\frac{1}{\pi i}\int\frac{(e^{-i\xi}-e^{-i\eta})^2}{\xi-\eta}e^{i\eta}d\eta,\\
    &=\frac{8i\xi^2}{t^2}\mathbb{H}e^{-i\xi}+\frac{4i\xi^2}{t^2}e^{-2i\xi}\mathbb{H}e^{i\xi}+\frac{4i\xi^2}{t^2}\mathbb{H}e^{-i\xi},\\
    &=\frac{8i\xi^2}{t^2}e^{-i\xi},\\
    &=\frac{it^2}{2\alpha^2}e^{-\frac{it^2}{4\alpha}}.
\end{align*}
We have used $\mathbb{H}e^{-i\xi}=e^{-i\xi}$ and $\mathbb{H}e^{i\xi}=-e^{i\xi}$ here. This proves our assertion.

Next, we consider $L^2$ norms of error terms including $I_1(f,g,h)$ and $I_{e_j}\ (j=1,2,3)$. For $I_1(f,g,h)$, if $f$ has the most derivatives, we write
\[
I_1(f,g,h)
=\frac{f(\alpha)}{\pi i}\int_{|\alpha-\beta|\ge t^{9/16}}\frac{g(\alpha)-g(\beta)}{(\alpha-\beta)^2}h(\beta)d\beta
+\frac{1}{\pi i}\int_{|\alpha-\beta|\ge t^{9/16}}\frac{g(\alpha)-g(\beta)}{(\alpha-\beta)^2}f(\beta)h(\beta)d\beta
\]
and we have
\begin{equation}
\norm{I_1(f,g,h)}_{L^2}\leq Ct^{-9/16}\norm{f}_{L^2}\norm{g}_{L^{\infty}}\norm{h}_{L^{\infty}}\le C\epsilon^3t^{-\frac{5}{4}-\frac{1}{16}}.
\end{equation}
The parallel argument holds when $g$ has the most derivatives. When most derivatives are in $h$ (say, $h=\partial_\alpha^k \tilde{h}$ where $k\ge [s/2]$), one first uses integration by parts to obtain
\begin{align*}
    I_1(f,g,h)
    &=\frac{1}{\pi i}\int_{|\alpha-\beta|\ge t^{9/16}}\left(\frac{f(\alpha)-f(\beta)}{(\alpha-\beta)^2}g'+\frac{g(\alpha)-g(\beta)}{(\alpha-\beta)^2}f'\right)\partial_\beta^{k-1}\tilde{h}d\beta\\
    &\ \ \ +\frac{2}{\pi i}\int_{|\alpha-\beta|\ge t^{9/16}}\frac{(f(\alpha)-f(\beta))(g(\alpha)-g(\beta))}{(\alpha-\beta)^3}\partial_\beta^{k-1}\tilde{h}d\beta\\
    &\ \ \ +\frac{1}{\pi i}\left.\frac{(f(\alpha)-f(\beta))(g(\alpha)-g(\beta))}{(\alpha-\beta)^2}\partial_\beta^{k-1}\tilde{h}\right|_{|\alpha-\beta|=t^{9/16}}.
\end{align*}
Then, we apply Young's inequality to each integral and estimate the boundary term by
\begin{align*}
    \norm{\frac{1}{\pi i}\left.\frac{(f(\alpha)-f(\beta))(g(\alpha)-g(\beta))}{(\alpha-\beta)^2}\partial_\beta^{k-1}\tilde{h}(\beta)\right|_{|\alpha-\beta|=t^{9/16}}}_{L^2}
    &\le Ct^{-\frac{9}{16}}\norm{f}_{L^\infty}\norm{\partial_\alpha g}_{L^\infty}\norm{\partial_\alpha^{k-1}\tilde{h}}_{L^2}\\
    &\le C\epsilon^3t^{-\frac{5}{4}-\frac{1}{16}}.
\end{align*}
This gives the desired result.

Next we estimate $\norm{I_{e_j}}_{L^2(S(t))}$ for $j=1,2,3$. Note that $\forall\beta\in E_{\alpha,t}$, $t^{\frac{4}{5}}\lesssim|\alpha|\lesssim t^{\frac{6}{5}}$ implies $t^{\frac{4}{5}}\lesssim|\beta|\lesssim t^{\frac{6}{5}}$. For $I_{e_2}$, we have
\begin{align*}
\norm{I_{e_2}}_{L^2(S(t))}
\le C\norm{GH}_{L^\infty}\norm{\partial_\alpha F}_{L^2(t^{4/5}\lesssim|\alpha|\lesssim t^{6/5})}
\le C\epsilon^3t^{-\frac{5}{4}-\frac{1}{10}+2\delta_0},
\end{align*}
since
\[
\norm{\partial_\alpha F}_{L^2(t^{4/5}\lesssim|\alpha|\lesssim t^{6/5})}
=\norm{\frac{1}{\alpha^2}e^{\frac{it^2}{4\alpha}}\left(\alpha L_0f-\frac{t}{2}\Omega_0f\right)}_{L^2(t^{4/5}\lesssim|\alpha|\lesssim t^{6/5})}
\le C\epsilon t^{-\frac{3}{5}+2\delta_0}.
\]
Such an estimate applies to $I_{e_1},I_{e_3}$\footnote{Actually, this requires $g,h$ are not of highest order derivatives since we only control the $s$ order vector field. Hence such proof need a slightly modification for such case, see the remark below Lemma \ref{decouple-2}}, too. So we derive that
\begin{align*}
   \sum_{j=1}^3 \norm{I_{ej}}_{L^2(S(t))}\leq C\epsilon^3 t^{-\frac{5}{4}-\frac{1}{10}+2\delta_0}
\end{align*}
and we conclude that $\norm{e}_{L^2(S(t))}\le C\epsilon^3t^{-\frac{5}{4}-\frac{1}{16}}$, given $\delta_0$ sufficiently small.

\noindent \textbf{Case 2:} $f=P_{\ge t^{-2}}Q_\alpha$.
The method for case 1 still applies to each term except
\[
\tilde{I}=GH\frac{1}{\pi i}\int_{E_{\alpha,t}}\frac{(e^{-\frac{it^2}{4\alpha}}-e^{-\frac{it^2}{4\beta}})(f(\alpha)-f(\beta))}{(\alpha-\beta)^2}e^{\frac{it^2}{4\beta}}d\beta,
\]
since if we still define $F=e^{it^2/4\alpha}f$, then $\partial_\alpha F$ can't be estimated in $L^2$, due to the restriction on $\Omega_0P_HQ_\alpha$. However, by the formula
\[
f(\alpha)=-\frac{2\alpha}{it}\partial_tf+\frac{2}{it}\Omega_0f,
\]
we have
\begin{align*}
    \tilde{I}
    &=-GH\frac{1}{\pi i}\int_{E_{\alpha,t}}\frac{(e^{-\frac{it^2}{4\alpha}}-e^{-\frac{it^2}{4\beta}})\left(\frac{2\alpha}{it}\partial_tf(\alpha)-\frac{2\beta}{it}\partial_tf(\beta)\right)}{(\alpha-\beta)^2}e^{\frac{it^2}{4\beta}}d\beta\\
    &\ \ \ +\frac{2GH}{it}\frac{1}{\pi i}\int_{E_{\alpha,t}}\frac{(e^{-\frac{it^2}{4\alpha}}-e^{-\frac{it^2}{4\beta}})\left(\Omega_0f(\alpha)-\Omega_0f(\beta)\right)}{(\alpha-\beta)^2}e^{\frac{it^2}{4\beta}}d\beta
\end{align*}
Observe that
\begin{align*}
    \norm{\frac{1}{\pi i}\frac{2GH}{it}\int_{E_{\alpha,t}}\frac{(\Omega_0f(\alpha)-\Omega_0f(\beta))(e^{-\frac{it^2}{4\alpha}}-e^{-\frac{it^2}{4\beta}})}{(\alpha-\beta)^2}e^{\frac{it^2}{4\beta}}d\beta}_{L^2(S(t))}
    \le C\epsilon^3t^{-\frac{3}{2}+2\delta_0}.
\end{align*}
Define $F=e^{it^2/4\alpha}D_t^2\zeta$, then
\begin{align*}
    \tilde{I}
    &=-GH\frac{1}{\pi i}\int_{E_{\alpha,t}}\frac{(e^{-\frac{it^2}{4\alpha}}-e^{-\frac{it^2}{4\beta}})\left(\frac{2\alpha}{it}e^{-\frac{it^2}{4\alpha}}F(\alpha)-\frac{2\beta}{it}e^{-\frac{it^2}{4\beta}}F(\beta)\right)}{(\alpha-\beta)^2}e^{\frac{it^2}{4\beta}}d\beta
    +O(\epsilon^3t^{-\frac{5}{4}-\delta})_{L^2}\\
    &=-2\frac{it^2}{4\alpha^2}e^{-\frac{it^2}{4\alpha}}\frac{2\alpha}{it}FGH-GH\frac{1}{\pi i}\int_{E_{\alpha,t}}\frac{(e^{-\frac{it^2}{4\alpha}}-e^{-\frac{it^2}{4\beta}})\left(\frac{2\alpha}{it}F(\alpha)-\frac{2\beta}{it}F(\beta)\right)}{(\alpha-\beta)^2}d\beta+O(\epsilon^3t^{-\frac{5}{4}-\delta})_{L^2}.
\end{align*}
Note that
\[
\partial_\alpha\left(\frac{2\alpha}{it}F(\alpha)\right)
=\frac{2}{it}F(\alpha)+\frac{2\alpha}{it}\frac{1}{\alpha^2}e^{\frac{it^2}{4\alpha}}\left(\alpha L_0D_t^2\zeta-\frac{t}{2}\Omega_0D_t^2\zeta\right)
\]
which implies
\[
\norm{\partial_\alpha\left(\frac{2\alpha}{it}F(\alpha)\right)}_{L^2(t^{4/5}\lesssim|\alpha|\lesssim t^{6/5})}\le C\epsilon t^{-\frac{4}{5}+2\delta_0}.
\]
and thus
\[
\tilde{I}=-2\frac{it^2}{4\alpha^2}\frac{2\alpha}{it}D_t^2\zeta gh+O(\epsilon^3t^{-\frac{5}{4}-\delta})_{L^2}.
\]
Finally, we have
\begin{align*}
    -\frac{2\alpha}{it}D_t^2\zeta
    &=-\frac{2\alpha}{it}b\partial_\alpha (D_t\zeta)_H-\frac{2\alpha}{it}\partial_t(D_t\zeta)_H+O(\epsilon t^{-1})\\
    &=(D_t\zeta)_H-\frac{2}{it}\Omega_0(D_t\zeta)_H+O(\epsilon t^{-1})\\
    &=(D_t\zeta)_H+O(\epsilon t^{-1+\delta_0}\ln t),
\end{align*}
which implies the desired result.
\end{proof}
\subsubsection{Case II.}
We also need to estimate singular integrals $I(f,g,h)$, with $f, h$ almost anti-holomorphic, while $g$ almost holomorphic. In this case, the main term of $I(f,g,h)$ vanishes. For our purposes, we take $f$, $g$, and $h$ as follows:
\begin{itemize}
\item [(1)] $f=P_{\geq t^{-2}}Q_{\alpha}$, or $f=\partial_{\alpha}^jQ_{\alpha}$, $j=1, \cdots, s$.

\item [(2)] $g=P_{\geq t^{-2}}\bar{Q}_{\alpha}$, or $g=\partial_{\alpha}^k\bar{Q}_{\alpha}$, $k=1, \cdots, s$. 

\item [(3)] $h=\partial_{\alpha}^{l+1}D_t\theta$ or $h=\partial_{\alpha}^{l+1}\theta$, where $0\leq l\leq s$.
In addition, we require $j+k+l\leq s$.
\end{itemize}

\noindent More precisely, we have the following lemma.
\begin{lemma}\label{decouple-2}
    Under the bootstrap assumptions, we have
\begin{equation}
    \norm{I(f,g,h)}_{L^2}\lesssim \epsilon^3 t^{-5/4-\delta}.
\end{equation}
\end{lemma}

\begin{proof}
	We shall only state the difference between this case and Lemma \ref{decouple}. Write
	\[
	I_2(f,g,h)=\frac{1}{\pi i}\int_{|\alpha-\beta|\le t^{9/16}}\frac{(e^{-\frac{it^2}{4\alpha}}F(\alpha)-e^{-\frac{it^2}{4\beta}}F(\beta))(e^{\frac{it^2}{4\alpha}}G(\alpha)-e^{\frac{it^2}{4\beta}}G(\beta))}{(\alpha-\beta)^2}e^{-\frac{it^2}{4\beta}}Hd\beta
	\]
	and decompose it to $I_m+I_e$ as in Lemma \ref{decouple}. We only need to calculate 
        \[
        \frac{1}{\pi i}\int\frac{(e^{-\frac{it^2}{4\alpha}}-e^{-\frac{it^2}{4\beta}})(e^{\frac{it^2}{4\alpha}}-e^{\frac{it^2}{4\beta}})}{(\alpha-\beta)^2}e^{-\frac{it^2}{4\beta}}d\beta.
	\]
        Still, let $\xi=\frac{t^2}{4\alpha},\ \eta=\frac{t^2}{4\beta}$ and we derive
        \begin{align*}
            &\frac{1}{\pi i}\int\frac{(e^{-\frac{it^2}{4\alpha}}-e^{-\frac{it^2}{4\beta}})(e^{\frac{it^2}{4\alpha}}-e^{\frac{it^2}{4\beta}})}{(\alpha-\beta)^2}e^{-\frac{it^2}{4\beta}}d\beta\\
            =&\ -\frac{4\xi^2}{t^2}\frac{1}{\pi i}\int\frac{(e^{-i\xi}-e^{-i\eta})(e^{i\xi}-e^{i\eta})}{(\xi-\eta)^2}e^{-i\eta}d\eta,\\
            =&\ -\frac{4\xi^2}{t^2}\frac{1}{\pi i}\int(2-e^{-i\xi+i\eta}-e^{-i\eta+i\xi})e^{-i\eta}d\frac{1}{\xi-\eta},\\
            =&\ -\frac{4\xi^2}{t^2}\frac{1}{\pi i}\int\frac{1}{\xi-\eta}\left((-ie^{-i\xi+i\eta}+ie^{-i\eta+i\xi})e^{-i\eta}-ie^{-i\eta}(2-e^{-i\xi+i\eta}-e^{-i\eta+i\xi})\right)d\eta,\\
            =&\ -\frac{4i\xi^2}{t^2}e^{i\xi}\mathbb{H}e^{-2i\xi}+\frac{8i\xi^2}{t^2}\mathbb{H}e^{-i\xi}-\frac{4i\xi^2}{t^2}e^{i\xi}\mathbb{H}e^{-2i\xi},\\
            =&\ 0.
        \end{align*}
        i.e. the main term vanishes, and this completes the proof.
\end{proof}

\begin{rem}\label{localization lemma for highest derivative}
    The proof of Localization Lemma needs to be modified if one of $f,g,h$ is of highest order derivative. For instance, consider the integral
    \[
    \frac{1}{\pi i}\int\frac{((D_t\zeta)_H(\alpha)-(D_t\zeta)_H(\beta))(\partial_\alpha^sD_t^2\zeta-\partial_\beta^sD_t^2\zeta)}{(\zeta(\alpha)-\zeta(\beta))^2}(\bar{\zeta}_\beta-1)d\beta,
    \]
    which will be analyzed in Chapter 5. 
    
    \noindent First, according to Lemma \ref{nearly alpha}, replacing $(\zeta(\alpha)-\zeta(\beta))^2$ by $(\alpha-\beta)^2$ does not affect the result\footnote{This type of argument will be used frequently throughout the paper, but we will not write out each instance explicitly.}.
    
    \noindent However, note that $\norm{L_0\partial_\alpha^sD_t^2\zeta}_{L^2}$ cannot be directly controlled, meaning our estimate does not apply to some remainder terms. To address this, we redefine variables as follows: let $F=e^{it^2/4\alpha}(D_t\zeta)_H,\ G=e^{it^2/4\alpha}\partial_\alpha^{s-1}D_t^2\zeta$, $H=e^{-it^2/4\alpha}(\bar{\zeta}_\alpha-1)$, then
    \[
    \partial_\alpha^sD_t^2\zeta
    =\partial_\alpha(e^{-\frac{it^2}{4\alpha}}G)
    =e^{-\frac{it^2}{4\alpha}}\partial_\alpha G+\frac{it^2}{4\alpha^2}e^{-\frac{it^2}{4\alpha}}G,
    \]
    and
    \[
    \norm{\partial_\alpha G}_{L^2}=\norm{\frac{1}{\alpha^2}e^{\frac{it^2}{4\alpha}}\left(\alpha L_0\partial_\alpha^{s-1}D_t^2\zeta-\frac{t}{2}\Omega_0\partial_\alpha^{s-1}D_t^2\zeta\right)}_{L^2},
    \]
which exhibits favorable decay. Consequently, the main term becomes
    \begin{align*}
            &\frac{1}{\pi i}\int_{E_{\alpha,t}}\frac{(e^{-\frac{it^2}{4\alpha}}F-e^{-\frac{it^2}{4\beta}}F)(\frac{it^2}{4\alpha^2}e^{-\frac{it^2}{4\alpha}}G-\frac{it^2}{4\beta^2}e^{-\frac{it^2}{4\beta}}G)}{(\zeta(\alpha)-\zeta(\beta))^2}e^{\frac{it^2}{4\beta}}H d\beta\\
        =&\ \frac{it^2}{4\alpha^2}FGH\frac{1}{\pi i}\int\left(\frac{e^{-\frac{it^2}{4\alpha}}-e^{-\frac{it^2}{4\beta}}}{\alpha-\beta}\right)^2d\beta+O(\epsilon^3t^{-\frac{5}{4}-\delta})_{L^2},\\
        =&\ 2\left(\frac{it^2}{4\alpha^2}\right)^2e^{-\frac{it^2}{4\alpha}}FGH+O(\epsilon^3t^{-\frac{5}{4}-\delta})_{L^2},\\
        =&\ 2\left(\frac{it^2}{4\alpha^2}\right)^2(D_t\zeta)_H\partial_\alpha^{s-1}D_t^2\zeta(\bar{\zeta}_\alpha-1)+O(\epsilon^3t^{-\frac{5}{4}-\delta})_{L^2}.
    \end{align*}
This uses the same decomposition method as before. Additionally, the formula
    \[
    \partial_\alpha f=\frac{it^2}{4\alpha^2}f+\frac{1}{\alpha}L_0f-\frac{t}{2\alpha^2}\Omega_0f
    \]
allows us to refine the main term further as
    \[
    2\left(\frac{it^2}{4\alpha^2}\right)^2(D_t\zeta)_H\partial_\alpha^{s-1}D_t^2\zeta(\bar{\zeta}_\alpha-1)
    =2\frac{it^2}{4\alpha^2}(D_t\zeta)_H\partial_\alpha^sD_t^2\zeta(\bar{\zeta}_\alpha-1)+O(\epsilon^3t^{-\frac{5}{4}-\delta})_{L^2}
    \]
    as desired.

    A different adjustment is required when the highest-order derivative applies to 
 $h$, i.e. $h=\partial_\alpha^{s+1}D_t\zeta$ (or $\partial_\alpha^{s+1}D_t\bar{\zeta}$). In this case, $L_0\partial_\alpha^sD_t\zeta$ cannot be controlled, so it must not appear in the final expression. For example, consider
    \[
    I=\frac{1}{\pi i}\int_{|\alpha-\beta|\le t^{9/16}}\frac{(D_t\bar{\zeta})_H(\alpha)-(D_t\bar{\zeta})_H(\beta))(D_t^2\bar{\zeta}(\alpha)-D_t^2\bar{\zeta}(\beta)}{(\zeta(\alpha)-\zeta(\beta))^2}\partial_\beta^{s+1}D_t\zeta d\beta
    \]
    and claim that
    \[
    I=2\left(\frac{it^2}{4\alpha^2}\right)^2(D_t\bar{\zeta})_HD_t^2\bar{\zeta}\partial_\alpha^sD_t\zeta+O(\epsilon^3t^{-\frac{5}{4}-\delta})_{L^2}.
    \]
    Actually, we write
    \[
    \partial_\beta^{s+1}D_t\zeta=\partial_\beta\left(\frac{it^2}{4\beta^2}\partial_\beta^{s-1}D_t\zeta+\frac{1}{\beta}L_0\partial_\beta^{s-1}D_t\zeta-\frac{t}{2\beta^2}\Omega_0\partial_\beta^{s-1}D_t\zeta\right)
    :=\partial_\beta\left(\frac{it^2}{4\beta^2}\partial_\beta^{s-1}D_t\zeta+R\right)
    \]
    Then such a case reduces to
    \[
    \tilde{I}=\frac{1}{\pi i}\int_{|\alpha-\beta|\le t^{9/16}}\frac{(D_t\bar{\zeta})_H(\alpha)-(D_t\bar{\zeta})_H(\beta))(D_t^2\bar{\zeta}(\alpha)-D_t^2\bar{\zeta}(\beta)}{(\zeta(\alpha)-\zeta(\beta))^2}\partial_\beta\left(\frac{it^2}{4\beta^2}\partial_\beta^{s-1}D_t\zeta\right)d\beta,
    \]
    since there holds
    \begin{align*}
        &\norm{\frac{1}{\pi i}\int_{|\alpha-\beta|\le t^{9/16}}\frac{(D_t\bar{\zeta})_H(\alpha)-(D_t\bar{\zeta})_H(\beta))(D_t^2\bar{\zeta}(\alpha)-D_t^2\bar{\zeta}(\beta)}{(\zeta(\alpha)-\zeta(\beta))^2}
        \partial_\beta R d\beta}_{L^2(S(t))}\\
        \le&\ C\norm{\partial_\alpha (D_t\bar{\zeta})_H}_{L^\infty}\norm{\partial_\alpha D_t^2\bar{\zeta}}_{L^\infty}\norm{R}_{L^2(t^{4/5}\lesssim|\alpha|\lesssim t^{6/5})}\\
        \le&\ C\epsilon^3t^{-\frac{5}{4}-\delta}.
    \end{align*}
    It's clear that $\tilde{I}$ can be handled by the method shown before. This completes the proof.
\end{rem}

\subsection{The structure of the water wave equations}
As a consequence of Lemma \ref{decouple} and Lemma \ref{decouple-2}, we can rewrite the water wave equations in a form $(D_t^2-iA\partial_{\alpha})\theta=G_1+G_2$ such that essentially
$$G_1+G_2=i\frac{t}{\alpha}|D_t^2\zeta|^2 D_t\zeta+R$$ where $\norm{(I-\mathcal{H})R}_{H^s}\le C\epsilon^3t^{-1-\delta}$.
More precisely, we have 
\begin{proposition}[Key proposition]\label{key prop}
    We have 
    \begin{equation}
        \norm{(I-\mathcal{H})G_1}_{L^2}\lesssim \epsilon^3 t^{-1-\delta},
    \end{equation}
    and
    \begin{equation}
        G_2=i\frac{t}{\alpha}|D_t^2\zeta|^2D_t\zeta+O(\epsilon^3t^{-1-\delta})_{L^2}.
    \end{equation}
\end{proposition}
We will not prove Proposition \ref{key prop} directly; instead, we will prove the version for its $s$-th derivative in Chapter 5 (see the relevant part of Proposition \ref{formula of dE_s/dt}). The latter is more challenging, and any difficulties encountered in the former are fully captured in the latter. In other words, the issues that arise when dealing with the former will also appear in the latter. In fact, the methods used for the latter can be directly applied to the former.

\section{Energy estimates}
Throughout this section, unless specified otherwise, we always assume the bootstrap assumptions and that $t\in [0,T]$ with the $T$ given in the bootstrap assumptions. We begin with the original equation. Recall that
	\begin{equation}\label{0 order equation}
		(D_t^2-iA\partial_\alpha)\theta=G.
	\end{equation}
	where
	\begin{equation}\label{formula for G}
		G=-2\left[D_t\zeta,\mathcal{H}\frac{1}{\zeta_\alpha}+\bar{\mathcal{H}}\frac{1}{\bar{\zeta}_\alpha}\right]\partial_\alpha D_t\bar{\zeta}+\frac{1}{\pi i}\int\left(\frac{D_t\zeta(\alpha,t)-D_t\zeta(\beta,t)}{\zeta(\alpha,t)-\zeta(\beta,t)}\right)^2(\zeta-\bar{\zeta})_\beta d\beta:=G_1+G_2.
	\end{equation}
	Define the energy functional for $\theta$ as:
	\begin{equation}
		E_0^\theta=\int\frac{|D_t\theta|^2}{A}d\alpha+i\int\theta\partial_\alpha\bar{\theta}d\alpha
	\end{equation}
We then obtain the following equality for the $t$ derivative of $E_0^\theta$:
		\begin{equation}\label{basic lemma}
			\frac{dE_0^\theta}{dt}=2Re\left\{\int\frac{D_t\bar{\theta}}{A}Gd\alpha\right\}-\int\frac{|D_t\theta|^2}{A}\left(\frac{a_t}{a}\circ\kappa^{-1}\right)d\alpha.
		\end{equation}
The proof of \eqref{basic lemma} is straightforward and can be found in detail in Lemma 4.1 of \cite{Wu2009}.

Next, we consider the equation for $\sigma=D_t\theta$. We have
 \[
 (D_t^2-iA\partial_\alpha)D_t\theta=[-iA\partial_\alpha,D_t]\theta+D_tG\overset{\text{def}}{=}\tilde{G},
 \]
 or equivalently
 \[
 (\partial_t^2-ia\partial_\alpha)(\sigma\circ\kappa)
 =[-ia\partial_\alpha,\partial_t](\theta\circ\kappa)+\partial_t(G\circ\kappa)=i\frac{a_t}{a}a\partial_{\alpha}\theta\circ\kappa+\partial_t(G\circ\kappa).
 \]
 Hence,
 \[\tilde{G}=i\frac{a_t}{a}\circ\kappa^{-1}A\theta_{\alpha}+D_tG.\]
Define the energy functional for $\sigma$ as:
 \[
 E_0^\sigma(t)=\int\frac{|D_t\sigma|^2}{A}d\alpha+i\int \sigma\partial_\alpha\bar{\sigma} d\alpha.
 \]
Using the basic energy estimate in \eqref{basic lemma}, we obtain the following expression for $\frac{d}{dt}E_0^{\sigma}$:

 \begin{equation}\label{formula of dE_{1/2}/dt}
     \begin{aligned}
         \frac{d}{dt}E_0^\sigma(t)
         &=2Re\left\{\int\frac{D_t\bar{\sigma}}{A}D_tG d\alpha\right\}-\int\frac{|D_t\sigma|^2}{A}\left(\frac{a_t}{a}\circ\kappa^{-1}\right)d\alpha\\
         &\ \ \ +Re\left\{2i\int\frac{1}{A}D_t\bar{\sigma}\left(\frac{a_t}{a}\circ\kappa^{-1}\right)A\partial_\alpha\theta d\alpha\right\}.
     \end{aligned}
 \end{equation}
Next, we consider the equation for higher-order derivatives. Recall the corresponding equations:
\begin{equation}
    \begin{aligned}
        &(D_t^2-iA\partial_\alpha)\theta_k=G_k^\theta,\\
        &(D_t^2-iA\partial_\alpha)\sigma_k=G_k^\sigma,
    \end{aligned}
\end{equation}
where $1\le k\le s$, and
\begin{equation}
    \theta_k\overset{\text{def}}{=}\partial_\alpha^k\theta,\ \sigma_k\overset{\text{def}}{=}\partial_\alpha^k\sigma.
\end{equation}
The corresponding expressions for $G_k^{\theta}$ and $G_k^{\sigma}$ are:
\begin{equation}
    G_k^\theta=\partial_\alpha^kG+[D_t^2-iA\partial_\alpha,\partial_\alpha^k]\theta, \quad \quad 
    G_k^\sigma=\partial_\alpha^k\tilde{G}+[D_t^2-iA\partial_\alpha,\partial_\alpha^k]\sigma.
\end{equation}
Define the $k$-th order energies $E_k^{\theta}$ and $E_k^{\sigma}$ as:
\begin{equation}
    E_k^\theta(t)=\int\frac{|D_t\theta_k|^2}{A}d\alpha+i\int\theta_k\partial_\alpha\bar{\theta}_kd\alpha,
\end{equation}
\begin{equation}
    E_k^\sigma(t)=\int\frac{|D_t\sigma_k|^2}{A}d\alpha+i\int \sigma_k\partial_\alpha\bar{\sigma}_k d\alpha.
\end{equation}
From the basic energy estimate \eqref{basic lemma}, we have:
\begin{equation}\label{formula for dE_k/dt}
    \frac{d}{dt}E_k^\theta(t)=Re\left\{\int\frac{2D_t\bar{\theta}_k}{A}G_k^\theta d\alpha\right\}-\int\frac{|D_t\theta_k|^2}{A}\left(\frac{a_t}{a}\circ\kappa^{-1}\right)d\alpha,
\end{equation}
and similarly for $\sigma_k$:
\begin{equation}
    \frac{d}{dt}E_k^\sigma(t)=Re\left\{\int\frac{2D_t\bar{\sigma}_k}{A}G_k^\sigma d\alpha\right\}-\int\frac{|D_t\sigma_k|^2}{A}\left(\frac{a_t}{a}\circ\kappa^{-1}\right)d\alpha.
\end{equation}
The total energy $\mathcal{E}_s$ is defined as:
\begin{equation}
    \mathcal{E}_s(t):= \sum_{k=0}^s\left(E_k^{\theta}(t)+E_k^{\sigma}(t)\right).
\end{equation}
We aim to prove the following energy estimate:

\begin{thm}\label{thm:sec 5main}
    There exists $\delta>0$ such that 
    \begin{equation}\label{energy estimate}
\frac{d\mathcal{E}_s(t)}{dt}\leq C\epsilon^2 (1+t)^{-1-\delta}\mathcal{E}_s(t)+C\epsilon^4 (1+t)^{-1-\delta}.
    \end{equation}
\end{thm}
It is straightforward to verify that \eqref{energy estimate} holds for $0\leq t\leq 1$. Therefore, we assume $t\geq 1$ in the remainder of this section.

At first, we have the following estimates.
 \begin{lemma}\label{third term}
There holds
    \begin{align*}
        \sum_{k=0}^s\left(\left| \int\frac{|D_t\theta_k|^2}{A}\left(\frac{a_t}{a}\circ\kappa^{-1}\right)d\alpha\right|+\left| \int\frac{|D_t\sigma_k|^2}{A}\left(\frac{a_t}{a}\circ\kappa^{-1}\right)d\alpha\right|\right)
     \le C\epsilon^2t^{-1-\delta}\mathcal{E}_s.
    \end{align*}

 \end{lemma}

 \begin{proof}
Using Proposition \ref{norm of at/a}, we first bound each term in the sum:
\begin{itemize}
    \item [(1)] For the first term, we have:
     \[
     \left|\int\frac{|D_t\theta_k|^2}{A}\left(\frac{a_t}{a}\circ\kappa^{-1}\right)d\alpha\right|
     \le C\norm{D_t\theta_k}_{L^2}^2\norm{\frac{a_t}{a}\circ\kappa^{-1}}_{L^\infty}\le C\epsilon^2t^{-1-\delta}E_{k}^\theta.
     \] 
     \item [(2)] For the second term, we have:
         \[
     \left|\int\frac{|D_t\sigma_k|^2}{A}\left(\frac{a_t}{a}\circ\kappa^{-1}\right)d\alpha\right|
     \le C\norm{D_t\sigma_k}_{L^2}^2\norm{\frac{a_t}{a}\circ\kappa^{-1}}_{L^\infty}\le C\epsilon^2t^{-1-\delta}E_{k}^\sigma.
     \] 
\end{itemize}
Combining these, we conclude the desired bound and complete the proof.
 \end{proof}
 To complete the proof of Theorem \ref{thm:sec 5main}, it remains to establish the inequality
 \begin{equation}\label{equ:energy-remaining}
     \sum_{k=0}^s\left( Re\left|\int \frac{2D_t\bar{\theta}_k}{A}G_k^{\theta}d\alpha\right|+Re\left|\int \frac{2D_t\bar{\sigma}_k}{A}G_k^{\sigma}d\alpha\right|\right)\leq C\epsilon^2t^{-1-\delta}\mathcal{E}_s+C\epsilon^4 t^{-1-\delta}.
 \end{equation}
The proof of \eqref{equ:energy-remaining} is divided into two parts:
 \begin{itemize}
     \item [(1)] First, we prove the $s$-order energy estimates for $\theta$:
     $$\frac{dE_s^{\theta}}{dt}\leq C\epsilon^2t^{-1-\delta}\mathcal{E}_s+C\epsilon^4 t^{-1-\delta},$$

\item [(2)] Next, we establish the $s$-order energy estimates for $\sigma$:
     $$\frac{dE_s^{\sigma}}{dt}\leq C\epsilon^2t^{-1-\delta}\mathcal{E}_s+C\epsilon^4 t^{-1-\delta}.$$
 \end{itemize}
 The primary challenge in both cases lies in addressing the low-frequency components and components with too many derivatives, particularly the insufficient decay of $\zeta-\alpha$ and $\partial_\alpha^sD_t^2\zeta$. The lower order energy estimate is omitted, as all the difficulties that arise in the low-order energy estimates also appear in the $s$-order energy estimate.

\subsection{Some useful lemmas}
First, let us briefly highlight the relationship between several norms frequently used throughout the proof and the energy functional $\mathcal{E}_s$.

\begin{lemma}\label{Hs and order s energy}
    There holds
    \[
    \norm{\zeta_\alpha-1}_{H^s}+\norm{D_t\zeta}_{H^{s+1/2}}+\norm{D_t^2\zeta}_{H^s}+\norm{\Lambda(\zeta-\alpha)}_{H^s}\le C\sqrt{\mathcal{E}_s}+C\epsilon^{\frac{3}{2}}t^{-\frac{1}{4}}\ln t.
    \]
\end{lemma}

\begin{proof}
    Recall that
    \[
    \mathcal{E}_s=\sum_{j=0}^s(E_k^\theta+E_{k}^\sigma)
    =\sum_{j=0}^{s}\left(\int\frac{|D_t\theta_k|^2}{A}d\alpha+i\int\theta_k\partial_\alpha\bar{\theta}_kd\alpha+\int\frac{|D_t\sigma_k|^2}{A}d\alpha+i\int\sigma_k\partial_\alpha\bar{\sigma}_kd\alpha\right).
    \]
    At first, note that (see Lemma \ref{structure of Dt sigma} for an example)
    \[
    D_t\sigma_k=2\partial_\alpha^kD_t^2\zeta+error
    \]
    and thus
    \[
    \norm{D_t^2\zeta}_{H^s}\le C\sqrt{\mathcal{E}_s}+C\epsilon^2 t^{-\frac{1}{2}}.
    \]
    Also, since $(\bar{\zeta}_\alpha-1)=iD_t^2\bar{\zeta}-(A-1)\bar{\zeta}_\alpha$, we have
    \[
    \norm{\zeta_\alpha-1}_{H^s}\le C\sqrt{\mathcal{E}_s}+O(\epsilon^2 t^{-\frac{1}{2}}).
    \]
    Similarly, using $\theta=2(\zeta-\alpha)+e$ we obtain
    \[
    \norm{D_t\zeta}_{H^s}\le C\sqrt{\mathcal{E}_s}+O(\epsilon^2 t^{-\frac{1}{2}}).
    \]
    When it comes to $\norm{\Lambda(\zeta-\alpha)}_{H^s}$ and $\norm{\Lambda D_t\zeta}_{H^s}$, some extra preparation is needed. Define
    \[
    \eta_\theta^k=\frac{I-\mathcal{H}}{2}\theta_k,\ r_\theta^k=\frac{I+\mathcal{H}}{2}\theta_k,\ \eta_\sigma^k=\frac{I-\mathcal{H}}{2}\sigma_k,\ r_\sigma^k=\frac{I+\mathcal{H}}{2}\sigma_k,\ 
    \]
    and
    \[
    R_\theta^k=i\int\eta_\theta^k\partial_\alpha\bar{r}_\theta^kd\alpha+i\int r_\theta^k\partial_\alpha\bar{\eta}_\theta^kd\alpha+i\int r_\theta^k\partial_\alpha\bar{r}_\theta^kd\alpha,
    \]
    \[
    R_\sigma^k=i\int\eta_\sigma^k\partial_\alpha\bar{r}_\sigma^kd\alpha+i\int r_\sigma^k\partial_\alpha\bar{\eta}_\sigma^kd\alpha+i\int r_\sigma^k\partial_\alpha\bar{r}_\sigma^kd\alpha.
    \]
    Then it's clear that
    \[
    i\int\theta_k\partial_\alpha\bar{\theta}_kd\alpha=\int|\Lambda\eta_\theta^k|^2d\alpha+R_\theta^k,
    \]
    \[
    i\int\sigma_k\partial_\alpha\bar{\sigma}_kd\alpha=\int|\Lambda\eta_\sigma^k|^2d\alpha+R_\sigma^k
    \]
    and
    \[
    R_\theta^k=i\int\eta_\theta^k\partial_\alpha\bar{r}_\theta^kd\alpha-i\int\bar{\eta}_\theta^k\partial_\alpha r_\theta^kd\alpha+i\int r_\theta^k\partial_\alpha\bar{r}_\theta^kd\alpha
    =2Re\left\{i\int\eta_\theta^k\partial_\alpha \bar{r}_\theta^kd\alpha\right\}+i\int r_\theta^k\partial_\alpha\bar{r}_\theta^kd\alpha,
    \]
    \[
    R_\sigma^k=2Re\left\{i\int\eta_\sigma^k\partial_\alpha \bar{r}_\sigma^kd\alpha\right\}+i\int r_\sigma^k\partial_\alpha \bar{r}_\sigma^kd\alpha.
    \]
    We verify that $|R_\theta^k|\le C\epsilon\mathcal{E}_s+C\epsilon^3t^{-\frac{1}{2}}\ln t,\ \forall 0\le k\le s$. For $k=0$, there holds $r_\theta^0=0$; for $k\ge 1$, since $(I+\mathcal{H})\theta=0$, we derive that
    \begin{align*}
        \partial_\alpha r_\theta^k
        =\frac{1}{2}\partial_\alpha(I+\mathcal{H})\partial_\alpha^k\theta
        =-\frac{1}{2}\partial_\alpha[\partial_\alpha^k,\mathcal{H}]\theta
        &=-\frac{1}{2}\sum_{j=1}^k\partial_\alpha^{k-j+1}[\partial_\alpha,\mathcal{H}]\partial_\alpha^{j-1}\theta,\\
        &=-\frac{1}{2}\sum_{j=1}^k\partial_\alpha^{k-j+1}[\zeta_\alpha,\mathcal{H}]\frac{\partial_\alpha^j\theta}{\zeta_\alpha}.
    \end{align*}
    Note that
    \[
    \norm{[\zeta_\alpha,\mathcal{H}]\frac{\partial_\alpha^{k+1}\theta}{\zeta_\alpha}}_{L^2}\le C\norm{\partial_\alpha^2\zeta}_{L^\infty}\norm{\partial_\alpha^k\theta}_{L^2}
    \le C\epsilon^2t^{-\frac{1}{2}}
    \]
    and, thanks to Lemma \ref{L infty of hilbert},
    \begin{align*}
        \norm{[\partial_\alpha^{k+1}\zeta,\mathcal{H}]\frac{\partial_\alpha\theta}{\zeta_\alpha}}_{L^2}
        &\le\norm{\partial_\alpha^{k+1}\zeta}_{L^2}\norm{\mathcal{H}\frac{\partial_\alpha\theta}{\zeta_\alpha}}_{L^\infty}+\norm{\mathcal{H}\partial_\alpha^{k+1}\zeta\frac{\partial_\alpha\theta}{\zeta_\alpha}}_{L^2},\\
        &\le C\epsilon^2t^{-\frac{1}{2}}\ln t.
    \end{align*}
    Hence we may conclude that $\norm{r_\theta^k}_{H^1}\le C\epsilon^2t^{-\frac{1}{2}}\ln t$. Which implies
    \[
    |R_\theta^k|\le C\norm{\eta_\theta^k}_{L^2}\norm{\partial_\alpha r_\theta^k}_{L^2}+C\norm{r_\theta^k}_{H^1}^2
    \le C\epsilon^3t^{-\frac{1}{2}}\ln t.
    \]
    For $R_\sigma^k$, we mainly consider the case $k=s$, since others can be handled by the method above. Note that
    \[
    |R_\sigma^s|\le\norm{\Lambda\eta_\sigma^s}_{L^2}\norm{\Lambda r_\sigma^s}_{L^2}+\norm{\Lambda r_\sigma^s}_{L^2}^2,
    \]
    and
    \begin{align*}
        r_\sigma^s
        =\frac{1}{2}(I+\mathcal{H})\partial_\alpha^sD_t\theta
        &=\frac{1}{2}\partial_\alpha^s(I+\mathcal{H})D_t\theta-\frac{1}{2}[\partial_\alpha^s,\mathcal{H}]D_t\theta\\
        &=-\frac{1}{2}\partial_\alpha^s[D_t,\mathcal{H}]\theta-\frac{1}{2}[\partial_\alpha^s,\mathcal{H}]D_t\theta\\
        &=-\frac{1}{2}\partial_\alpha^s[D_t\zeta,\mathcal{H}]\frac{\partial_\alpha\theta}{\zeta_\alpha}-\frac{1}{2}[\partial_\alpha^s,\mathcal{H}]D_t\theta.
    \end{align*}
    Note that, by repeating the process above, one may derive $\norm{\Lambda[\partial_\alpha^s,\mathcal{H}]D_t\theta}_{L^2}\le\norm{[\partial_\alpha^s,\mathcal{H}]D_t\theta}_{H^1}\le C\epsilon^2t^{-1/2}\ln t$. For the first term, we only consider the extreme case, i.e. all derivatives apply to $D_t\zeta$ or $\partial_\alpha\theta$. Let
    \[
    I_1=[\partial_\alpha^sD_t\zeta,\mathcal{H}]\frac{\partial_\alpha\theta}{\zeta_\alpha},\ I_2=[D_t\zeta,\mathcal{H}]\frac{\partial_\alpha^{s+1}\theta}{\zeta_\alpha}.
    \]
    Note that
    \[
    I_1=\partial_\alpha^sD_t\zeta\mathcal{H}\frac{\partial_\alpha\theta}{\zeta_\alpha}+\mathcal{H}\partial_\alpha^sD_t\zeta\frac{\partial_\alpha\theta}{\zeta_\alpha}
    \]
    and we have, according to Lemma 2.14 in \cite{Yosihara} and Lemma \ref{estimate for half derivative},
    \begin{align*}
        \norm{\partial_\alpha^sD_t\zeta\mathcal{H}\frac{\partial_\alpha\theta}{\zeta_\alpha}}_{\dot{H}^{\frac{1}{2}}}
        &\le\norm{(\Lambda\partial_\alpha^sD_t\zeta)\mathcal{H}\frac{\partial_\alpha\theta}{\zeta_\alpha}}_{L^2}+\norm{\left[\Lambda,\mathcal{H}\frac{\partial_\alpha\theta}{\zeta_\alpha}\right]\partial_\alpha^sD_t\zeta}_{L^2}\\
        &\le C\epsilon\norm{\mathcal{H}\frac{\partial_\alpha\theta}{\zeta_\alpha}}_{L^\infty}+\norm{\mathcal{H}\frac{\partial_\alpha\theta}{\zeta_\alpha}}_{H^2}\norm{\partial_\alpha^sD_t\zeta}_{H^{-\frac{1}{2}}}\\
        &\le C\epsilon^2t^{-\frac{1}{2}}\ln t+C\epsilon\sqrt{\mathcal{E}_s}.
    \end{align*}
    Also,
    \[
    \norm{\mathcal{H}\partial_\alpha^sD_t\zeta\frac{\partial_\alpha\theta}{\zeta_\alpha}}_{\dot{H}^{\frac{1}{2}}}
    \le\norm{\partial_\alpha^sD_t\zeta\frac{\partial_\alpha\theta}{\zeta_\alpha}}_{H^{\frac{1}{2}}}
    \le C\epsilon^2t^{-\frac{1}{2}}+C\epsilon\sqrt{\mathcal{E}_s}.
    \]
    In addition, observe that
    \[
    \norm{I_2}_{\dot{H}^{\frac{1}{2}}}\le\norm{I_2}_{H^1}\le C\epsilon^2t^{-\frac{1}{2}}.
    \]
    Hence we conclude that
    \[
    \norm{\Lambda r_\sigma^s}_{L^2}\le C\epsilon\sqrt{\mathcal{E}_s}+C\epsilon^2t^{-\frac{1}{2}}\ln t,
    \]
    which leads to
    \[
    |R_\sigma^s|\le C\epsilon\mathcal{E}_s+C\epsilon^3t^{-\frac{1}{2}}\ln t.
    \]
    Finally, we derive
    \[
    \norm{\Lambda\eta_\theta^k}_{L^2}^2+\norm{\Lambda\eta_\sigma^k}_{L^2}^2\le\mathcal{E}_s-R_\theta^k-R_\sigma^k
    \le C\mathcal{E}_s+C\epsilon^3t^{-\frac{1}{2}}\ln t.
    \]
    For $\norm{\Lambda(\zeta-\alpha)}_{H^s}$, note that $\forall 0\le k\le s$,
    \[
    \norm{\Lambda\theta_k}_{L^2}=\norm{\Lambda r_\theta^k+\Lambda\eta_\theta^k}_{L^2}
    \le\norm{\Lambda\eta_\theta^k}_{L^2}+C\epsilon^2t^{-\frac{1}{2}}\ln t
    \le C\sqrt{\mathcal{E}_s}+C\epsilon^{\frac{3}{2}}t^{-\frac{1}{4}}\ln t.
    \]
    This implies
    \[
    \norm{\Lambda(\zeta-\alpha)}_{H^s}\le C\sqrt{\mathcal{E}_s}+\norm{\Lambda e}_{H^s}+C\epsilon^{\frac{3}{2}}t^{-\frac{1}{4}}\ln t
    \]
    and it's easy to verify $\norm{\Lambda e}_{H^s}\le C\epsilon\sqrt{\mathcal{E}_s}+C\epsilon^2t^{-\frac{1}{2}}\ln t$ by using the estimate we have shown above. Similar estimate holds for $\norm{\Lambda D_t\zeta}_{H^s}$.
\end{proof}

\begin{lemma}\label{small or big alpha}
    Fix small $\mu>0$. Define 
    \[
    C(f,g,h)=\frac{1}{\pi i}\int\frac{(f(\alpha)-f(\beta))(g(\alpha)-g(\beta))}{(\zeta(\alpha)-\zeta(\beta))^2}\partial_\beta hd\beta.
    \]
    then there exists $C>0$ such that
    \begin{equation}
        \begin{aligned}
             \norm{C(f,g,h)}_{L^2(|\alpha|\le t^{1-\mu})}
        &\le C(\norm{\partial_\alpha f}_{W^{1,\infty}(|\alpha|\lesssim t^{1-\mu})}+\epsilon t^{-\frac{1}{2}}\norm{\partial_\alpha f}_{W^{1,\infty}})\norm{g}_{L^2}\norm{\partial_\alpha h}_{W^{1,\infty}}\ln t\\
        &\ \ \ +Ct^{-1+\mu}\norm{f}_{L^\infty}\norm{g}_{L^2}\norm{\partial_\alpha h}_{L^\infty},
        \end{aligned}
    \end{equation}
    and
    \begin{equation}
        \begin{aligned}
            \norm{C(f,g,h)}_{L^2(|\alpha|\ge t^{1+\mu})}
        &\le C(\norm{\partial_\alpha f}_{W^{1,\infty}(|\alpha|\gtrsim t^{1+\mu})}+\epsilon t^{-\frac{1}{2}}\norm{\partial_\alpha f}_{W^{1,\infty}})\norm{g}_{L^2}\norm{\partial_\alpha h}_{W^{1,\infty}}\ln t\\
        &\ \ \ +Ct^{-1-\mu}\norm{f}_{L^\infty}\norm{g}_{L^2}\norm{\partial_\alpha h}_{L^\infty}.
        \end{aligned}
    \end{equation}
\end{lemma}

\begin{proof}
    We first consider the case $|\alpha|\le t^{1-\mu}$. Write
    \begin{align*}
        C(f,g,h)
        &=\frac{1}{\pi i}\int\frac{(f(\alpha)-f(\beta))(g(\alpha)-g(\beta))}{(\alpha-\beta)^2}\partial_\beta hd\beta+R,\\
        &=\frac{1}{\pi i}\left(\int_{|\alpha-\beta|\le\frac{t^{1-\mu}}{2}}+\int_{|\alpha-\beta|\ge\frac{t^{1-\mu}}{2}}\right)\frac{(f(\alpha)-f(\beta))(g(\alpha)-g(\beta))}{(\alpha-\beta)^2}\partial_\beta hd\beta+R,\\
        &=I_1+I_2+R
    \end{align*}
    where $\norm{R}_{L^2}\le C\epsilon t^{-\frac{1}{2}}\norm{\partial_\alpha f}_{W^{1,\infty}}\norm{g}_{L^2}\norm{\partial_\alpha h}_{W^{1,\infty}}\ln t$, according to Lemma \ref{useful lemma};
    and it's clear that
    \begin{align*}
        \norm{I_1}_{L^2(|\alpha|\le t^{1-\mu})}
        &\le C\norm{g}_{L^2}\norm{\partial_\alpha f}_{W^{1,\infty}(|\alpha|\le \frac{3}{2}t^{1-\mu})}\norm{\partial_\alpha h}_{W^{1,\infty}}\ln t.
    \end{align*}
    For $I_2$, according to Young's inequality,
    \begin{align*}
       &\norm{\int_{|\alpha-\beta|\ge\frac{t^{1-\mu}}{2}}\frac{f(\alpha)-f(\beta)}{(\alpha-\beta)^2}g(\beta)\partial_\beta hd\beta}_{L^2(|\alpha|\le t^{1-\mu})}\\
       &\le\norm{f}_{L^\infty}\norm{\int_{|\alpha-\beta|\ge\frac{t^{1-\mu}}{2}}\frac{g(\beta)\partial_\beta h}{(\alpha-\beta)^2}d\beta}_{L^2}
       +\norm{\int_{|\alpha-\beta|\ge\frac{t^{1-\mu}}{2}}\frac{f(\beta)g(\beta)\partial_\beta h}{(\alpha-\beta)^2}d\beta}_{L^2},\\
       &\le C\norm{f}_{L^\infty}\norm{\frac{1}{\alpha^2}}_{L^1(|\alpha|\ge\frac{t^{1-\mu}}{2})}\norm{g\partial_\alpha h}_{L^2}
       +\norm{\frac{1}{\alpha^2}}_{L^1(|\alpha|\ge\frac{t^{1-\mu}}{2})}\norm{fg\partial_\alpha h}_{L^2},\\
       &\le Ct^{-1+\mu}\norm{f}_{L^\infty}\norm{g}_{L^2}\norm{\partial_\alpha h}_{L^\infty}.
    \end{align*}
    Also, it's obvious that
    \[
    \norm{g(\alpha)\int_{|\alpha-\beta|\ge\frac{t^{1-\mu}}{2}}\frac{f(\alpha)-f(\beta)}{(\alpha-\beta)^2}\partial_\beta hd\beta}_{L^2(|\alpha|\le t^{1-\mu})}
    \le Ct^{-1+\mu}\norm{f}_{L^\infty}\norm{g}_{L^2}\norm{\partial_\alpha h}_{L^\infty}.
    \]
    Hence we conclude that
    \begin{align*}
        \norm{C(f,g,h)}_{L^2(|\alpha|\le t^{1-\mu})}
        &\le C(\norm{\partial_\alpha f}_{W^{1,\infty}(|\alpha|\lesssim t^{1-\mu})}+\epsilon t^{-\frac{1}{2}}\norm{\partial_\alpha f}_{W^{1,\infty}})\norm{g}_{L^2}\norm{\partial_\alpha h}_{W^{1,\infty}}\ln t\\
        &\ \ \ +Ct^{-1+\mu}\norm{f}_{L^\infty}\norm{g}_{L^2}\norm{\partial_\alpha h}_{L^\infty}.
    \end{align*}
    Observe that the same proof works for $|\alpha|\ge t^{1+\mu}$, and this completes the proof.
\end{proof}
\subsection{Estimate $E_s^\theta$}
\begin{lemma}\label{calculate for k order commutator}
    There holds
    \begin{equation}
        \left|\int\frac{D_t\bar{\theta}_k}{A}[D_t^2-iA\partial_\alpha,\partial_\alpha^k]\theta\right|\le C\epsilon^2t^{-1-\delta}\mathcal{E}_s,\ \forall 1\le k\le s.
    \end{equation}
\end{lemma}

\begin{proof}
    Straight calculation yields
    \[
    [D_t^2,\partial_\alpha^k]=D_t[D_t,\partial_\alpha^k]+[D_t,\partial_\alpha^k]D_t.
    \]
    Note that
    \[
    [D_t,\partial_\alpha^k]
    =[b\partial_\alpha,\partial_\alpha^k]
    =-\sum_{j=1}^kC_{j,k}\partial_\alpha^jb\partial_\alpha^{k-j+1},\quad C_{j,k}:=\frac{k!}{j!(k-j)!},
    \]
    hence there holds
    \begin{align*}
        \left|\int\frac{D_t\bar{\theta}_k}{A}[D_t^2,\partial_\alpha^k]\theta\right|
        &\le C\norm{D_t\bar{\theta}_k}_{L^2}\left(\norm{D_t\sum_{j=1}^kC_{j,k}(\partial_\alpha^jb)\theta_{k-j+1}}_{L^2}+\norm{\sum_{j=1}^kC_{j,k}(\partial_\alpha^jb)\sigma_{k-j+1}}_{L^2}\right).
    \end{align*}
    Using Lemma \ref{norms of b}, we derive that
    \[
    \norm{D_t\partial_\alpha^jb}_{L^2}\le C\epsilon^2t^{-1+\delta_0}\ln t,\ \forall 1\le j\le s.
    \]
    If $k-j+1\le s-2$, we have
    \[
    \norm{D_t\partial_\alpha^jb\theta_{k-j+1}}_{L^2}\le\norm{D_t\partial_\alpha^jb}_{L^2}\norm{\theta_{k-j+1}}_{L^\infty}\le C\epsilon^2t^{-1-\delta}\sqrt{\mathcal{E}_s}.
    \]
    Otherwise we have
    \[
    \norm{D_t\partial_\alpha^jb\theta_{k-j+1}}_{L^2}\le\norm{D_t\partial_\alpha^j}_{L^2}\le\norm{D_t\partial_\alpha^jb}_{L^\infty}\norm{\theta_{k-j+1}}_{L^2}\le C\epsilon^2t^{-1-\delta}\sqrt{\mathcal{E}_s}.
    \]
    Similar estimates hold for the other part.

    So far we have
    \begin{equation}
        \left|\int\frac{D_t\bar{\theta}_k}{A}[D_t^2,\partial_\alpha^k]\theta\right|\le C\epsilon^2t^{-1-\delta}\mathcal{E}_s;
    \end{equation}
    for the other half, we compute
    \[
    [A\partial_\alpha,\partial_\alpha^k]=-\sum_{j=1}^kC_{j,k}\partial_\alpha^jA\partial_\alpha^{k-j+1}
    \]
    and the proof is thus finished by applying Proposition \ref{norms of A-1}.
\end{proof}

\begin{proposition}\label{formula of dE_s/dt}
There exists $\delta>0$ such that
\begin{equation}
    \left|\frac{dE_s^\theta}{dt}\right|\le C\epsilon^2t^{-1-\delta}\mathcal{E}_s+O(\epsilon^4t^{-1-\delta}).
\end{equation}
\end{proposition}
	
\begin{proof}
    We have
    \[
    \frac{dE_s^\theta}{dt}=Re\left\{\int\frac{2D_t\bar{\theta}_s}{A}G_s^\theta d\alpha\right\}-\int\frac{|D_t\theta_s|^2}{A}\left(\frac{a_t}{a}\circ\kappa^{-1}\right)d\alpha
    \]
    where $G_s^\theta=\partial_\alpha^s(G_1+G_2)+[D_t^2-iA\partial_\alpha,\partial_\alpha^s]\theta$. The commutator is handled by Lemma \ref{calculate for k order commutator}, while $\frac{a_t}{a}\circ\kappa^{-1}$ is handled by Lemma \ref{third term}. It suffices to deal with $\partial_\alpha^s(G_1+G_2)$.
    
Decompose $G_1,G_2$ as (let $f(\alpha):=\zeta-\alpha$, and $\eta:=Im(\zeta-\alpha)$)
\[
\begin{aligned}
    G_1=-2\left[D_t\bar{\zeta},\mathcal{H}\frac{1}{\zeta_\alpha}+\bar{\mathcal{H}}\frac{1}{\bar{\zeta}_\alpha}\right]\partial_\alpha D_t\zeta
    &=\frac{4}{\pi}\int\frac{(\eta(\alpha)-\eta(\beta))(D_t\bar{\zeta}(\alpha)-D_t\bar{\zeta}(\beta))}{|\zeta(\alpha)-\zeta(\beta)|^2}\partial_\beta D_t\zeta d\beta,\\
    &=\frac{4}{\pi}\int\frac{(\eta(\alpha)-\eta(\beta))(D_t\bar{\zeta}(\alpha)-D_t\bar{\zeta}(\beta))}{(\zeta(\alpha)-\zeta(\beta))^2}\partial_\beta D_t\zeta d\beta\\
    &\ \ \ +C\int\frac{(\eta(\alpha)-\eta(\beta))^2(D_t\bar{\zeta}(\alpha)-D_t\bar{\zeta}(\beta))}{(\zeta(\alpha)-\zeta(\beta))|\zeta(\alpha)-\zeta(\beta)|^2}\partial_\beta D_t\zeta d\beta,\\
    &=\frac{2}{\pi i}\int\frac{(f(\alpha)-f(\beta))(D_t\bar{\zeta}(\alpha)-D_t\bar{\zeta}(\beta)}{(\zeta(\alpha)-\zeta(\beta))^2}\partial_\beta D_t\zeta d\beta\\
    &\ \ \ -\frac{2}{\pi i}\int\frac{(\bar{f}(\alpha)-\bar{f}(\beta))(D_t\bar{\zeta}(\alpha)-D_t\bar{\zeta}(\beta)}{(\zeta(\alpha)-\zeta(\beta))^2}\partial_\beta D_t\zeta d\beta\\
    &\ \ \ +C\int\frac{(\eta(\alpha)-\eta(\beta))^2(D_t\bar{\zeta}(\alpha)-D_t\bar{\zeta}(\beta))}{(\zeta(\alpha)-\zeta(\beta))|\zeta(\alpha)-\zeta(\beta)|^2}\partial_\beta D_t\zeta d\beta
\end{aligned}
\]
\[
    G_2=\frac{1}{\pi i}\int\left(\frac{D_t\zeta(\alpha)-D_t\zeta(\beta)}{\zeta(\alpha)-\zeta(\beta)}\right)^2(\zeta_{\beta}-1) d\beta-\frac{1}{\pi i}\int\left(\frac{D_t\zeta(\alpha)-D_t\zeta(\beta)}{\zeta(\alpha)-\zeta(\beta)}\right)^2\overline{(\zeta_{\beta}-1)}d\beta
\]
and we only consider those terms in $\partial_\alpha^sG_1,\ \partial_\alpha^sG_2$, as they are most challenging to handle, and their treatment essentially encompasses all the methods we will use:
\[
G_{11}=\frac{2}{\pi i}\int\frac{(f(\alpha)-f(\beta))(\partial_\alpha^sD_t\bar{\zeta}(\alpha)-\partial_\beta^sD_t\bar{\zeta}(\beta)}{(\zeta(\alpha)-\zeta(\beta))^2}\partial_\beta D_t\zeta d\beta,
\]
\[
G_{12}=\frac{2}{\pi i}\int\frac{(\bar{f}(\alpha)-\bar{f}(\beta))(\partial_\alpha^sD_t\bar{\zeta}(\alpha)-\partial_\beta^sD_t\bar{\zeta}(\beta)}{(\zeta(\alpha)-\zeta(\beta))^2}\partial_\beta D_t\zeta d\beta,
\]
\[
G_{21}=\frac{1}{\pi i}\int\frac{(D_t\zeta(\alpha)-D_t\zeta(\beta))(\partial_\alpha^sD_t\zeta-\partial_\beta^sD_t\zeta)}{(\zeta(\alpha)-\zeta(\beta))^2}(\zeta_\beta-1)d\beta,
\]
\[
G_{22}=\frac{1}{\pi i}\int\frac{(D_t\zeta(\alpha)-D_t\zeta(\beta))(\partial_\alpha^sD_t\zeta-\partial_\beta^sD_t\zeta)}{(\zeta(\alpha)-\zeta(\beta))^2}(\bar{\zeta}_\beta-1)d\beta.
\]
We divide the rest of the proof into two parts:
\begin{itemize}
    \item [(1)] We prove that
    \[
    G_{21}=O(\epsilon^4t^{-1-\delta})_{L^2},
    \]
    \[
   G_{22}=\frac{4\alpha}{it}\partial_\alpha^sD_t\zeta|\partial_\alpha D_t\zeta|^2+O(\epsilon^3t^{-1-\delta})_{L^2},
    \]
    and
    \begin{equation}
        \frac{2}{A}D_t\bar{\theta}_sG_2=\frac{16\alpha}{iAt}|D_t^2\zeta|^2|\partial_\alpha^sD_t\zeta|^2+O(\epsilon^3t^{-1-\delta}\mathcal{E}_s)_{L^1},
    \end{equation}
    for certain $\delta>0$. And thus
    \[
    Re\int\frac{2}{A}D_t\bar{\theta}G_2=O(\epsilon^3t^{-1-\delta}\mathcal{E}_s).
    \]
\end{itemize}
\textbf{Subcase:} For $G_{21}$

\vspace*{1ex}

\noindent Note that it suffices to consider

\[
    \left\|\overline{\frac{1}{\pi i}\int\frac{(D_t\zeta(\alpha)-D_t\zeta(\beta))(\partial_\alpha^sD_t\zeta-\partial_\beta^sD_t\zeta)}{(\zeta(\alpha)-\zeta(\beta))^2}(\zeta_\beta-1)d\beta}\right\|_{L^2}.
\]
Note that $D_t\bar{\zeta},(\bar{\zeta}_\alpha-1)/\zeta_\alpha$, $(I+\mathcal{H})\partial_\alpha^sD_t\bar{\zeta}$ are holomorphic in $\Omega(t)$ and we write
\begin{align*}
    \overline{\frac{\partial_\alpha^sD_t\zeta-\partial_\beta^sD_t\zeta}{(\zeta(\alpha)-\zeta(\beta)^2}(\zeta_\beta-1)}
    &=\frac{\partial_\alpha^sD_t\bar{\zeta}-\partial_\beta^sD_t\bar{\zeta}}{(\zeta(\alpha)-\zeta(\beta)^2}(\bar{\zeta}_\beta-1)\\
    &\ \ \ +(\partial_\alpha^sD_t\bar{\zeta}-\partial_\beta^sD_t\bar{\zeta})\frac{4iRe\{\zeta(\alpha)-\zeta(\beta)\}Im\{\zeta(\alpha)-\zeta(\beta)\}}{|\zeta(\alpha)-\zeta(\beta)|^4}(\bar{\zeta}_\beta-1).
\end{align*}
Observe that
\[
    \begin{aligned}
	\text{Im}\{\zeta(\alpha)-\zeta(\beta)\}
	&=\frac{1}{2i}(\zeta(\alpha)-\zeta(\beta)-\overline{\zeta(\alpha)-\zeta(\beta)}),\\
	&=\frac{1}{2i}\left\{[\zeta(\alpha)-\alpha-(\zeta(\beta)-\beta)]-\overline{[\zeta(\alpha)-\alpha-(\zeta(\beta)-\beta)]}\right\}
    \end{aligned}
\]
and thus, by Lemma \ref{useful lemma},
\begin{align*}
    &\norm{\int(D_t\bar{\zeta}(\alpha)-D_t\bar{\zeta}(\beta))(\partial_\alpha^sD_t\bar{\zeta}-\partial_\beta
    ^sD_t\bar{\zeta})\frac{4iRe\left\{\zeta(\alpha)-\zeta(\beta)\right\}Im\left\{\zeta(\alpha)-\zeta(\beta)\right\}}{|\zeta(\alpha)-\zeta(\beta)|^4}(\bar{\zeta}_\beta-1)d\beta}_{L^2}\\
    \le&\ C\norm{\partial_\alpha^sD_t\zeta}_{L^2}\norm{\partial_\alpha D_t\zeta}_{L^\infty}\norm{\zeta_\alpha}_{L^\infty}\norm{\zeta_\alpha-1}_{W^{1,\infty}}^2\ln t,\\
    \le&\ C\epsilon^4t^{-\frac{3}{2}}\ln t.
\end{align*}
Hence we focus on the first term. Claim that
\[
\norm{(I-\mathcal{H})\partial_\alpha^sD_t\bar{\zeta}}_{L^2}\le C\epsilon^2t^{-\frac{1}{2}}.
\]
Actually, since $(I-\mathcal{H})D_t\bar{\zeta}=0$, there holds
\[
\norm{(I-\mathcal{H})\partial_\alpha^s D_t\bar{\zeta}}_{L^2}
=\norm{[\partial_\alpha^s,\mathcal{H}]D_t\bar{\zeta}}_{L^2}
\le C\epsilon^2t^{-\frac{1}{2}}
\]
as in the proof of Lemma \ref{Hs and order s energy}. Now we write
\begin{align*}
    &\int(D_t\bar{\zeta}(\alpha)-D_t\bar{\zeta}(\beta))\frac{\partial_\alpha^sD_t\bar{\zeta}-\partial_\beta^sD_t\bar{\zeta}}{(\zeta(\alpha)-\zeta(\beta)^2}(\bar{\zeta}_\beta-1)d\beta\\
    =&\ \int\frac{I+\mathcal{H}}{2}\left(\partial_\alpha^sD_t\bar{\zeta}-\partial_\beta^sD_t\bar{\zeta}\right)\frac{D_t\bar{\zeta}(\alpha)-D_t\bar{\zeta}(\beta)}{(\zeta(\alpha)-\zeta(\beta)^2}\frac{\bar{\zeta}_\beta-1}{\zeta_\beta}d\beta\\
    &+\int\frac{I-\mathcal{H}}{2}\left(\partial_\alpha^sD_t\bar{\zeta}-\partial_\beta^sD_t\bar{\zeta}\right)\frac{D_t\bar{\zeta}(\alpha)-D_t\bar{\zeta}(\beta)}{(\zeta(\alpha)-\zeta(\beta)^2}\frac{\bar{\zeta}_\beta-1}{\zeta_\beta}d\beta\\
    =&\ 0+(\epsilon^4t^{-\frac{3}{2}}\ln t)_{L^2}.
\end{align*}
The first integral on the right hand side of the equality is zero due to the Cauchy integral Theorem. Therefore, we conclude that
\[
\norm{G_{21}}_{L^2}\le C\epsilon^4t^{-\frac{3}{2}}\ln t.
\]

\noindent \textbf{Subcase:} For $G_{22}$:

\vspace*{1ex}

\noindent We introduce the frequency decomposition: $f_L=P_{<t^{-2}}f$ and $f_H=P_{\ge t^{-2}}f$. Applying this frequency decomposition to $D_t\zeta$ in $G_{22}$, it is straightforward to verify (one should use Lemma \ref{Bernstein} to obtain $\norm{\partial_\alpha(D_t\zeta)_L}_{L^\infty}\le C\epsilon t^{-1}$) that any term containing at least one low-frequency factor decays faster than $t^{-1}$.

Thus, it's enough to consider
\begin{equation}\label{the main term of G2}
    \tilde{G}_{22}=\frac{1}{\pi i}\int\frac{((D_t\zeta)_H(\alpha)-(D_t\zeta)_H(\beta))(\partial_\alpha^sD_t\zeta-\partial_\beta^sD_t\zeta)}{(\zeta(\alpha)-\zeta(\beta))^2}(\bar{\zeta}_\beta-1)d\beta
\end{equation}
and we shall study it by considering the position of $|\alpha|$. Let $\mu>0$ be a small constant. 

\textbf{Subsubcase:} $|\alpha|$ is far from $t$. 

\vspace*{1ex}

\noindent $\bullet$ For $|\alpha|<t^{1-\mu}$:

\vspace*{1ex}Note that by Lemma \ref{transition-lemma} and \eqref{Omega_0D_t zeta high},
\[
\partial_\alpha(D_t\zeta)_H
=\frac{4\alpha^2}{it^2}\partial_\alpha^2(D_t\zeta)_H-\frac{4\alpha}{it^2}L_0\partial_\alpha(D_t\zeta)_H+\frac{2}{it}\Omega_0\partial_\alpha(D_t\zeta)_H
=\frac{4\alpha^2}{it^2}\partial_\alpha^2(D_t\zeta)_H+O(\epsilon t^{-1+\delta_0}\ln t).
\]
This implies that 
$$\norm{\partial_\alpha(D_t\zeta)_H}_{L^\infty(|\alpha|<Ct^{1-\mu})}\le C\epsilon t^{-\frac{1}{2}-2\mu},$$ 
and by Lemma \ref{small or big alpha}, we obtain
\begin{align*}
    \norm{\tilde{G}_{22}}_{L^2(|\alpha|<t^{1-\mu})}
    &\le C\left(\norm{\partial_{\alpha}(D_t\zeta)_H}_{W^{1,\infty}(|\alpha|\lesssim t^{1-\mu})}+\epsilon t^{-1/2}\norm{\partial_{\alpha}D_t\zeta}_{W^{1,\infty}}\right)\\
    &\ \ \ \times\norm{\partial_{\alpha}^sD_t\zeta}_{L^2}\norm{\zeta_{\alpha}-1}_{W^{1,\infty}}\ln t
    +Ct^{-1+\mu}\norm{D_t\zeta}_{L^\infty}\norm{\partial_\alpha^sD_t\zeta}_{L^2}\norm{\zeta_\alpha-1}_{L^\infty}\\
    &\le C\epsilon^2t^{-1-2\mu}\ln t\sqrt{\mathcal{E}_s}+C\epsilon^3t^{-1-\delta}.
\end{align*}

\noindent $\bullet$ For $|\alpha|>t^{1+\mu}$:

We have
\begin{align*}
    \partial_\alpha D_t\zeta
    =\partial_t(\zeta_\alpha-1)+O(\epsilon^2t^{-1})
    &=-\frac{it}{2\alpha}(\zeta_\alpha-1)+\frac{1}{\alpha}\Omega_0(\zeta_\alpha-1)+O(\epsilon^2t^{-1}),\\
    &=-\frac{it}{2\alpha}(\zeta_\alpha-1)+O(\epsilon t^{-1-\mu+\delta_0}),
\end{align*}
which implies
\[
\norm{\partial_\alpha D_t\zeta}_{L^\infty(|\alpha|>Ct^{1+\mu})}\le C\epsilon t^{-\frac{1}{2}-\mu}.
\]
Hence, applying Lemma \ref{small or big alpha},
\begin{align*}
   \norm{\tilde{G}_{22}}_{L^2(|\alpha|>t^{1+\mu})}\le &\  C\left(\norm{\partial_{\alpha}D_t\zeta}_{W^{1,\infty}(|\alpha|\gtrsim t^{1+\mu})}+\epsilon t^{-1/2}\norm{\partial_{\alpha}D_t\zeta}_{W^{1,\infty}}\right)\\
   &\ \ \ \times\norm{\partial_{\alpha}^sD_t\zeta}_{L^2}\norm{\zeta_{\alpha}-1}_{W^{1,\infty}}\ln t
   +Ct^{-1-\mu}\norm{D_t\zeta}_{L^{\infty}}\norm{\partial_{\alpha}^sD_t\zeta}_{L^2}\norm{\zeta_{\alpha}-1}_{L^{\infty}}\\
    \le &\  C\epsilon^2t^{-1-\mu}\ln t\sqrt{\mathcal{E}_s}+C\epsilon^3t^{-1-\delta}.
\end{align*}

\noindent \textbf{Subsubcase:} $\alpha\in S(t)=\{\alpha\in\mathbb{R}\ |\ t^{1-\mu}\le|\alpha|\le t^{1+\mu}\}$:

\vspace*{1ex} Applying $f=P_{\ge t^{-2}}Q_\alpha$, $g=\partial_{\alpha}^s Q_{\alpha}$, and $h=\bar{\theta}_{\alpha}$ to localization lemma \ref{decouple}, together with Lemma \ref{good-replace}, we obtain:
\begin{align*}
    \frac{1}{\pi i}1_{S(t)}\int\frac{((D_t\zeta)_H(\alpha)-(D_t\zeta)_H(\beta))(\partial_\alpha^sD_t\zeta-\partial_\beta^sD_t\zeta)}{(\zeta(\alpha)-\zeta(\beta))^2}(\bar{\zeta}_\beta-1)d\beta
    =& 2\frac{it^2}{4\alpha^2}1_{S(t)}(Q_\alpha)_H\partial_\alpha^sQ_\alpha\partial_\alpha\bar{\theta}+e
\end{align*}
where $\|e\|_{L^2}\le C\epsilon^3t^{-\frac{5}{4}-\delta'}$ for some $\delta'>0$. We use the formula
\[
\partial_\alpha f=\frac{it^2}{4\alpha^2}f+\frac{1}{\alpha}L_0f-\frac{t}{2\alpha^2}\Omega_0f
\]
and
\[
\partial_t f=-\frac{it}{2\alpha}f+\frac{1}{\alpha}\Omega_0f=\frac{it}{2\alpha}f+\frac{1}{\alpha}\tilde{\Omega}_0f
\]
($f$ represents an arbitrary function here)
to derive
\begin{align*}
    \frac{it^2}{4\alpha^2}1_{S(t)}(Q_\alpha)_H\partial_\alpha^sQ_\alpha\partial_\alpha\bar{\theta}
    &=\left(\frac{it^2}{4\alpha^2}\right)^s1_{S(t)}(Q_\alpha)_H\partial_\alpha Q_\alpha\bar{\theta}_\alpha+O(\epsilon^3t^{-1-\delta})_{L^2}\\
    &=\left(\frac{it^2}{4\alpha^2}\right)^{s-1}\frac{t}{2\alpha}1_{S(t)}(Q_\alpha)_H\partial_\alpha Q_\alpha\partial_\alpha D_t\bar{\theta}+O(\epsilon^3t^{-1-\delta})_{L^2}\\
    &=\frac{4\alpha^2}{it^2}\frac{t}{2\alpha}1_{S(t)}\partial_\alpha^s(Q_\alpha)_H\partial_\alpha Q_\alpha\partial_\alpha D_t\bar{\theta}+O(\epsilon^3t^{-1-\delta})_{L^2}\\
    &=\frac{4\alpha}{it}\partial_\alpha^sD_t\zeta|\partial_\alpha D_t\zeta|^2+O(\epsilon^3t^{-1-\delta})_{L^2}.
\end{align*}
Here we have used Lemma \ref{remainder of Dt theta} and Lemma \ref{good-replace}. Also,
\[
D_t\bar{\theta}_s=[D_t,\partial_\alpha^s]\bar{\theta}+\partial_\alpha^sD_t\bar{\theta}
=2\partial_\alpha^sD_t\bar{\zeta}+O(\epsilon^2t^{-\frac{1}{2}})_{L^2},
\]
    hence
\begin{align*}
    \frac{2}{A}D_t\bar{\theta}_s\tilde{G}_{22}=\frac{16\alpha}{iAt}|\partial_\alpha^sD_t\bar{\zeta}|^2|\partial_\alpha D_t\zeta|^2+O(\epsilon^4t^{-1-\delta})_{L^1}.
\end{align*}
Note that the first term is pure imaginary, which vanishes after taking the real part. Now we finally get
\begin{equation}
    \left|Re\int\frac{2}{A}D_t\bar{\theta}_s G_{22} d\alpha\right|\leq C\epsilon^3 t^{-1-\delta}\mathcal{E}_s+C\epsilon^4t^{-1-\delta}.
\end{equation}

\begin{enumerate}
    \item [(2)] We aim to prove
    \begin{equation}\label{decay of (I-H)G1}
        (I-\mathcal{H})(G_{11}+G_{12})=O(\epsilon^3t^{-1-\delta})_{L^2}
    \end{equation}
    for certain $\delta>0$. Once (\ref{decay of (I-H)G1}) is established, it directly follows that (note the trivial bound $\norm{G_{11}}_{L^2}+\norm{G_{22}}_{L^2}\le C\epsilon^3t^{-1}\ln t$)
    \begin{align*}
        \left|\int D_t\bar{\theta}_s(G_{11}+G_{12})\right|
        &\le \left|\int\partial_\alpha^sD_t\bar{\theta}G_1\right|+C\epsilon^4t^{-1-\delta}\\
        &\le 2\left|\int\partial_\alpha^sD_t\bar{\zeta}(G_{11}+G_{12})\right|+C\epsilon^4t^{-1-\delta}\\
        &\le 2\left|\int \frac{I+\mathcal{H}}{2}\partial_\alpha^sD_t\bar{\zeta}(G_{11}+G_{12})\right|+C\epsilon^4t^{-1-\delta}\\
        &=2\left|\int \frac{I+\mathcal{H}}{2}\partial_\alpha^sD_t\bar{\zeta}\frac{I-\mathcal{H}}{2}(G_{11}+G_{12})\right|+C\epsilon^4t^{-1-\delta}\\
        &\le C\epsilon^4t^{-1-\delta}.
    \end{align*}
\end{enumerate}
\vspace*{2ex}

\noindent $\bullet$ Verification of (\ref{decay of (I-H)G1}). 

\vspace*{2ex}

\noindent \textbf{Subcase:} For $G_{11}$:

Note that in this part, $f(\alpha)=\zeta-\alpha$. 
Let us define the frequency decomposition as $f_L=P_{<t^{-1/2-\delta}}f,\ f_H=P_{\ge t^{-1/2-\delta}}f$ for some $\delta>0$. First, we use the bound
    \[
    \norm{\partial_\alpha f_L}_{L^\infty}\le Ct^{-\frac{1}{2}-\delta}\norm{\Lambda f_L}_{L^2}\le C\epsilon t^{-\frac{1}{2}-\delta},
    \]
to handle terms involving the low-frequency component: 
\begin{align*}
    \norm{\int\frac{(f_L(\alpha)-f_L(\beta))(\partial_\alpha^sD_t\bar{\zeta}-\partial_\beta^sD_t\bar{\zeta})}{(\zeta(\alpha)-\zeta(\beta))^2}\partial_\beta D_t\zeta d\beta}_{L^2}\leq &\  C\norm{\partial_{\alpha}f_L}_{L^{\infty}}\norm{\partial_{\alpha}^sD_t\zeta}_{L^2}\norm{\partial_{\alpha}D_t\zeta}_{W^{1,\infty}}\\
    \leq &\  C\epsilon^2t^{-1-\delta}\sqrt{\mathcal{E}_s}.
\end{align*}
For the high-frequency component:
\[
\frac{2}{\pi i}\int\frac{(f_H(\alpha)-f_H(\beta))(\partial_\alpha^sD_t\bar{\zeta}-\partial_\beta^sD_t\bar{\zeta})}{(\zeta(\alpha)-\zeta(\beta))^2}\partial_\beta D_t\zeta d\beta,
\]
we apply Lemma \ref{decouple-2} to establish the desired decay \footnote{In practice, we need to raise the derivative of $f$ before applying Lemma \ref{decouple-2}; this process will be detailed in the treatment of $G_{12}$.}. Thus, we conclude the estimate for $\norm{G_{11}}_{L^2}$. 

\noindent \textbf{Subcase:} For $G_{12}$:

We first handle the low frequency part as above. Afterward, the remaining contribution from $G_{12}$ is
\begin{equation}\label{the main term of G1}
    \tilde{G}_{12}=-\frac{2}{\pi i}\int\frac{(\bar{f}_H(\alpha)-\bar{f}_H(\beta))(\partial_\alpha^sD_t\bar{\zeta}-\partial_\beta^sD_t\bar{\zeta})}{(\zeta(\alpha)-\zeta(\beta))^2}\partial_\beta D_t\zeta d\beta.
\end{equation}
\noindent $\bullet$ When $|\alpha|$ is far from $t$. 

\noindent As in the treatment of $G_{22}$, we apply Lemma \ref{small or big alpha} to $\tilde{G}_{12}$ to get the desired decay. One should notice that
\begin{align*}
    \norm{f_H}_{L^\infty}
    &\le \norm{P_{>t^2}f}_{L^\infty}+\norm{P_{t^{-1/2-\delta}<|\cdot|\le t^2}f}_{L^\infty},\\
    &\le Ct^{-2}\norm{\partial_\alpha P_{>t^2}f}_{L^\infty}+C\norm{\Lambda P_{t^{-1/2-\delta}<|\cdot|\le t^2}f}_{L^2}\ln t,\\
    &\le C\epsilon\ln t.
\end{align*}

\noindent $\bullet$ When $\alpha\in S(t)$. 

\noindent Decompose $1_{S(t)}\tilde{G}_{12}$ as
\begin{align*}
    1_{S(t)}\tilde{G}_{12}
    &=-\frac{2}{\pi i}1_{S(t)}\left(\int_{|\alpha-\beta|<t^{3/4+1/40}}+\int_{|\alpha-\beta|\ge t^{3/4+1/40}}\right)\frac{(\bar{f}_H(\alpha)-\bar{f}_H(\beta))(\partial_\alpha^sD_t\bar{\zeta}-\partial_\beta^sD_t\bar{\zeta})}{(\zeta(\alpha)-\zeta(\beta))^2}\partial_\beta D_t\zeta d\beta,\\
    &=1_{S(t)}\tilde{G}_{12}^1+1_{S(t)}\tilde{G}_{12}^2.
\end{align*}
\textbf{Subsubcase:} For $\tilde{G}_{12}^2$:

\noindent 
Using Young's inequality, we have
\[
\norm{\tilde{G}_{12}^2}_{L^2}\le \norm{\bar{f}_H}_{L^\infty}\norm{\partial_\alpha^sD_t\bar{\zeta}}_{L^2}\norm{\partial_\alpha D_t\zeta}_{L^\infty}\int_{t^{3/4+1/40}}^\infty\frac{dr}{r^2}
\le C\epsilon^3t^{-\frac{5}{4}-\frac{1}{40}}\ln t.
\]
\textbf{Subsubcase:} For $\tilde{G}_{12}^1$, using Lemma \ref{transition-lemma}, write\footnote{Note that in the following proof we need to derive $|\alpha|\sim|\beta|$ from $\alpha\in S(t)$, which implies that we must assume $3/4+1/40<1-\mu$; indeed, one may take $\mu=1/5$.}
\begin{align*}
        \tilde{G}_{12}^1
        &=\frac{1}{\pi i}\int_{|\alpha-\beta|\le t^{3/4+1/40}}\frac{\left(\frac{2\alpha}{it}\partial_t\bar{f}_H-\frac{2\beta}{it}\partial_t\bar{f}_H\right)(\partial_\alpha^sD_t\bar{\zeta}-\partial_\beta^sD_t\bar{\zeta})}{(\zeta(\alpha)-\zeta(\beta))^2}\partial_\beta D_t\zeta d\beta\\
        &\ \ \ +\frac{1}{\pi i}\int_{|\alpha-\beta|\le t^{3/4+1/40}}\frac{\left(\frac{2}{it}\tilde{\Omega}_0\bar{f}_H(\alpha)-\frac{2}{it}\tilde{\Omega}_0\bar{f}_H(\beta)\right)(\partial_\alpha^sD_t\bar{\zeta}-\partial_\beta^sD_t\bar{\zeta})}{(\zeta(\alpha)-\zeta(\beta))^2}\partial_\beta D_t\zeta d\beta.
    \end{align*}
    Recall that
    \[
    \norm{\partial_\alpha\tilde{\Omega}_0\bar{f}_H}_{L^\infty}=\norm{\partial_{\alpha}\Omega_0f_H}_{L^{\infty}}\le C\epsilon t^{\delta_0}\ln t,
    \]
    Hence we can estimate the second term of $\tilde{G}_{12}^1$ as follows:
    \begin{align*}
        &\norm{\frac{1}{\pi i}\int_{|\alpha-\beta|\le t^{3/4+1/40}}\frac{\left(\frac{2}{it}\tilde{\Omega}_0\bar{f}_H(\alpha)-\frac{2}{it}\tilde{\Omega}_0\bar{f}_H(\beta)\right)(\partial_\alpha^sD_t\bar{\zeta}-\partial_\beta^sD_t\bar{\zeta})}{(\zeta(\alpha)-\zeta(\beta))^2}\partial_\beta D_t\zeta d\beta}_{L^2}\\
        &\le\norm{\frac{2}{it}\partial_\alpha\tilde{\Omega}_0\bar{f}_H}_{W^{1,\infty}}\norm{\partial_\alpha^sD_t\bar{\zeta}}_{L^2}\norm{\partial_\alpha D_t\zeta}_{W^{1,\infty}}\ln t,\\
        &\le C\epsilon^3t^{-\frac{3}{2}+\delta_0}(\ln t)^2.
    \end{align*}
    Decompose the first term as
    \begin{align*}
        &\frac{1}{\pi i}\int_{|\alpha-\beta|\le t^{3/4+1/40}}\frac{(\partial_t\bar{f}_H(\alpha)-\partial_t\bar{f}_H(\beta))(\partial_\alpha^sD_t\bar{\zeta}-\partial_\beta^sD_t\bar{\zeta})}{(\zeta(\alpha)-\zeta(\beta))^2}\frac{2\beta}{it}\partial_\beta D_t\zeta d\beta\\
        &\ \ \ +\frac{2}{it}\partial_t\bar{f}_H\frac{1}{\pi i}\int_{|\alpha-\beta|\le t^{3/4+1/40}}\frac{(\alpha-\beta)(\partial_\alpha^sD_t\bar{\zeta}-\partial_\beta^sD_t\bar{\zeta})}{(\zeta(\alpha)-\zeta(\beta))^2}\partial_\beta D_t\zeta d\beta,\\
        &=-\frac{1}{\pi i}\int_{|\alpha-\beta|\le t^{3/4+1/40}}\frac{((D_t\bar{\zeta})_H(\alpha)-(D_t\bar{\zeta})_H(\beta))(\partial_\alpha^sD_t\bar{\zeta}-\partial_\beta^sD_t\bar{\zeta})}{(\zeta(\alpha)-\zeta(\beta))^2}(\zeta_\beta-1)d\beta\\
        &\ \ \ +\frac{2}{it}\partial_t\bar{f}_H\frac{1}{\pi i}\int_{|\alpha-\beta|\le t^{3/4+1/40}}\frac{(\alpha-\beta)(\partial_\alpha^sD_t\bar{\zeta}-\partial_\beta^sD_t\bar{\zeta})}{(\zeta(\alpha)-\zeta(\beta))^2}\partial_\beta D_t\zeta d\beta+O(\epsilon^3t^{-\frac{5}{4}-\delta})_{L^2},\\
        &:=I+II+O(\epsilon^3t^{-\frac{5}{4}-\delta})_{L^2}. 
    \end{align*}
Applying Lemma \ref{useful lemma}, we obtain the following estimate for $II$:
\begin{align*}
\norm{II}_{L^2}\leq \frac{C}{t}\norm{\partial_t\bar{f}_H}_{L^\infty}\norm{\partial_\alpha^sD_t\bar{\zeta}}_{L^2}\norm{\partial_\alpha D_t\zeta}_{W^{1,\infty}}\ln t
\le C\epsilon^3t^{-\frac{7}{4}}\ln t.
\end{align*}
Note that this approach is sufficient to handle the high frequency part of $G_{11}$, as it has no leading-order term according to Lemma \ref{decouple-2}. 
Using the Localization Lemma \ref{decouple}(note that for $\mu\le 1/5$, the Localization Lemma \ref{decouple} applies to $1_{S(t)}I$), we derive
\[
1_{S(t)}I=-1_{S(t)}\frac{it^2}{2\alpha^2}(D_t\bar{\zeta})_H\partial_\alpha^sD_t\bar{\zeta}(\zeta_\alpha-1)+O(\epsilon^3t^{-\frac{5}{4}-\delta})_{L^2},
\]
i.e.
\begin{align*}
    1_{S(t)}\tilde{G}_{12}
    =-1_{S(t)}\frac{it^2}{2\alpha^2}(D_t\bar{\zeta})_H\partial_\alpha^s D_t\bar{\zeta}(\zeta_\alpha-1)+O(\epsilon^3t^{-\frac{5}{4}-\delta})_{L^2}.
\end{align*}
Next, we apply the Transition of Derivatives Lemma \ref{transition-lemma} to rewrite the term $\frac{it^2}{2\alpha^2}(D_t\bar{\zeta})_H\partial_\alpha^sD_t\bar{\zeta}(\zeta_\alpha-1)$ in the following manner:
\begin{itemize}
\item [(1)] We first distribute $\frac{it^2}{4\alpha^2}$ to $\zeta_\alpha-1$. By Lemma \ref{transition-lemma}, for $\alpha\in S(t)$,
$$\frac{it^2}{4\alpha^2}(\zeta_\alpha-1)=\partial_\alpha(\zeta_\alpha-1)+O(\epsilon t^{-\frac{3}{5}+\delta_0}).$$
This yields
$$1_{S(t)}\tilde{G}_{12}=-1_{S(t)}2(D_t\bar{\zeta})_H\partial_\alpha^sD_t\bar{\zeta}\zeta_{\alpha\alpha}+(\epsilon^3t^{-1-\delta})_{L^2}.$$

\item [(2)] Next, we move one derivative from $\partial_{\alpha}^s D_t\bar{\zeta}$ to both $(D_t\bar{\zeta})_H$ and $\zeta_{\alpha\alpha}$, respectively:
\begin{align*}
    (D_t\bar{\zeta})_H\partial_\alpha^sD_t\bar{\zeta}\zeta_{\alpha\alpha}
    =&(D_t\bar{\zeta})_H\left(\frac{it^2}{4\alpha^2}\right)^2\partial_\alpha^{s-2}D_t\bar{\zeta}\zeta_{\alpha\alpha}+(\epsilon^3t^{-1-\delta})_{L^2}\\
    =&\left(\frac{it^2}{4\alpha^2}(D_t\bar{\zeta})_H\right)\partial_\alpha^{s-2}D_t\bar{\zeta}\left(\frac{it^2}{4\alpha^2}\zeta_{\alpha\alpha}\right)+(\epsilon^3t^{-1-\delta})_{L^2}\\
    =&-\partial_{\alpha}D_t\bar{\zeta}\partial_{\alpha}^{s-2}D_t\bar{\zeta}\partial_\alpha^3\zeta+(\epsilon^3t^{-1-\delta})_{L^2}.
\end{align*}

\item [(3)] Finally, we rewrite $\partial_{\alpha}D_t\bar{\zeta}$ as
$$\partial_{\alpha}D_t\bar{\zeta}=\frac{\partial_{\alpha}D_t\bar{\zeta}}{\zeta_{\alpha}}+\partial_{\alpha}D_t\bar{\zeta}\left(1-\frac{1}{\zeta_\alpha}\right).$$
Substituting this expression into the formula, we obtain
\begin{align*}
    1_{S(t)}\tilde{G}_{12}=-2 \frac{\partial_{\alpha}D_t\bar{\zeta}}{\zeta_{\alpha}}\partial_{\alpha}^{s-2}D_t\bar{\zeta}\partial_\alpha^3\zeta+(\epsilon^3t^{-1-\delta})_{L^2}.
\end{align*}
\end{itemize}

Since $(I-\mathcal{H})\frac{\partial_{\alpha}D_t\bar{\zeta}}{\zeta_{\alpha}}=0$, we have 
\begin{align*}
    (I-\mathcal{H})\left(\frac{\partial_{\alpha}D_t\bar{\zeta}}{\zeta_{\alpha}}\partial_{\alpha}^{s-2}D_t\bar{\zeta}\partial_\alpha^3\zeta\right)
    &=[\partial_{\alpha}^{s-2}D_t\bar{\zeta}\partial_\alpha^3\zeta,\mathcal{H}]\frac{\partial_\alpha D_t\bar{\zeta}}{\zeta_\alpha}.
\end{align*}
From the proof of \eqref{W s-2 norm of b_alpha}, we have
\begin{align*}
\norm{\partial_{\alpha}\Big(\partial_{\alpha}^{s-2}D_t\bar{\zeta}\partial_\alpha^3\zeta\Big)}_{L^{\infty}}\leq C\epsilon^2 t^{-7/6+\delta_0}.
\end{align*}
Hence, by Lemma \ref{singular}, we have 
\begin{align*}
    \norm{(I-\mathcal{H})\tilde{G}_{12}}_{L^2}\leq &\  \norm{2 [\partial_{\alpha}^{s-2}D_t\bar{\zeta}\partial_\alpha^3\zeta,\mathcal{H}]\frac{\partial_\alpha D_t\bar{\zeta}}{\zeta_\alpha}}_{L^2}+C\epsilon^3 t^{-1-\delta'}\\
    \leq &\  C\norm{\partial_{\alpha}\Big(\partial_{\alpha}^{s-2}D_t\bar{\zeta}\partial_\alpha^3\zeta\Big)}_{L^{\infty}}\norm{D_t\bar{\zeta}}_{L^2}+ C\epsilon^3 t^{-1-\delta'}\\
    \leq &\   C\epsilon^3 t^{-1-\delta'}.
\end{align*}
This completes the proof.

\end{proof}

\subsection{Estimate $E_s^{\sigma}$}

Recall that
\[
\frac{dE_s^\sigma}{dt}=Re\left\{\int\frac{2D_t\bar{\sigma}_s}{A}G_s^\sigma d\alpha\right\}-\int\frac{|D_t\sigma_s|^2}{A}\left(\frac{a_t}{a}\circ\kappa^{-1}\right)d\alpha,
\]
where
\[
G_s^\sigma=\partial_\alpha^sD_tG+i\partial_\alpha^s\left(\frac{a_t}{a}\circ\kappa^{-1}A\theta_\alpha\right)+[D_t^2-iA\partial_\alpha,\partial_\alpha^s]\sigma,
\]

As for $G_s^\theta$, we have the following estimate.
\begin{lemma}
    There holds
    \[
    \left|\int\frac{D_t\bar{\sigma}_k}{A}[D_t^2-iA\partial_\alpha,\partial_\alpha^k]\sigma d\alpha\right|\le C\epsilon^4t^{-1-\delta},\ \forall 1\le k\le s.
    \]
\end{lemma}

\begin{proof}
    Repeat the proof for Lemma \ref{calculate for k order commutator} yields the result.
\end{proof}

Also, we need to estimate the extra term that doesn't appear in $G_\theta^s$.
\begin{lemma}
    There exists $\delta>0$ such that
    \[
    \norm{\left(\frac{a_t}{a}\circ\kappa^{-1}\right)A\theta_\alpha}_{H^s}\le C\epsilon^3t^{-1-\delta}.
    \]
\end{lemma}

\begin{proof}
    We only verify the $L^2$ norm of $s$-order derivatives.
    Actually, we have (according to Proposition \ref{norm of at/a})
    \[
    \norm{\partial_\alpha^s\left(\frac{a_t}{a}\circ\kappa^{-1}\right)\cdot A\theta_\alpha}_{L^2}\le C\norm{\frac{a_t}{a}\circ\kappa^{-1}}_{H^s}\norm{\theta_\alpha}_{L^\infty}
    \le C\epsilon^3t^{-1-\delta}.
    \]  
    A similar estimate holds when the derivatives apply to $A$. In addition, there holds
    \[
    \norm{\left(\frac{a_t}{a}\circ\kappa^{-1}\right)A\partial_\alpha^{s}\theta_\alpha}_{L^2}\le C\norm{\frac{a_t}{a}\circ\kappa^{-1}}_{L^\infty}\norm{\theta_\alpha}_{H^s}
    \le C\epsilon^3t^{-1-\delta}.
    \]
    This implies the conclusion.
\end{proof}

 \begin{proposition}\label{estimate of dE_{1/2}/dt}
     There holds
     \begin{equation}
         \left|\frac{d}{dt}E_s^\sigma(t)\right|\le C\epsilon^2t^{-1-\delta'}\mathcal{E}_s+C\epsilon^4t^{-1-\delta'}
     \end{equation}
     for some $\delta'>0$.
 \end{proposition}
 \begin{proof}
It suffices to consider
\[
Re\left\{\int\frac{2D_t\bar{\sigma}_s}{A}\partial_\alpha^sD_tG d\alpha\right\}.
\]
Furthermore, by Lemma \ref{structure of Dt sigma}, we have
\[
\int\frac{D_t\bar{\sigma}_s}{A}\partial_\alpha^s D_tG d\alpha
=2\int\frac{\partial_\alpha^sD_t^2\bar{\zeta}}{A}\partial_\alpha^sD_tG d\alpha+O(\epsilon^4t^{-1-\delta}),
\]
where the estimate $\norm{\partial_\alpha^sD_tG}_{L^2}\le C\epsilon^3t^{-1}\ln t$ follows from Lemma \ref{useful lemma}. Next, leveraging the almost-holomorphic nature of $\partial_\alpha^sD_t^2\bar{\zeta}$, we deduce
\[
2\int\frac{\partial_\alpha^sD_t^2\bar{\zeta}}{A}(I-\mathcal{H})\partial_\alpha^sD_tG d\alpha
=4\int\frac{\partial_\alpha^sD_t^2\bar{\zeta}}{A}\partial_\alpha^sD_tG d\alpha+O(\epsilon^4 t^{-1-\delta'}).
\]
This reduces the problem to controlling $\partial_\alpha^sD_tG$ or $(I-\mathcal{H})\partial_\alpha^sD_tG$, depending on our preference. Since these results will be essential in Chapter 7, we detail the treatment of some specific components of $D_tG$ which closely resemble the approach in Proposition \ref{formula of dE_s/dt}. We first consider the derivatives of $G_2$. Observe that
\begin{align*}
    \partial_t\frac{1}{\pi i}\int\left(\frac{z_t(\alpha)-z_t(\beta)}{z(\alpha)-z(\beta)}\right)^2(z-\bar{z})_\beta d\beta
&=\frac{2}{\pi i}\int\frac{z_t(\alpha)-z_t(\beta)}{z(\alpha)-z(\beta)}\partial_t\left(\frac{z_t(\alpha)-z_t(\beta)}{z(\alpha)-z(\beta)}\right)(z-\bar{z})_\beta d\beta\\
&\ \ \ +\frac{1}{\pi i}\int\left(\frac{z_t(\alpha)-z_t(\beta)}{z(\alpha)-z(\beta)}\right)^2(z_t-\bar{z}_t)_\beta d\beta,\\
&=-\frac{2}{\pi i}\int\frac{z_t(\alpha)-z_t(\beta)}{z(\alpha)-z(\beta)}\left(\frac{z_t(\alpha)-z_t(\beta)}{z(\alpha)-z(\beta)}\right)^2(z-\bar{z})_{\beta}d\beta\\
&\ \ \ +\frac{2}{\pi i}\int\frac{(z_t(\alpha)-z_t(\beta))(z_{tt}(\alpha)-z_{tt}(\beta))}{(z(\alpha)-z(\beta))^2}(z-\bar{z})_\beta d\beta\\
&\ \ \ +\frac{1}{\pi i}\int\left(\frac{z_t(\alpha)-z_t(\beta)}{z(\alpha)-z(\beta)}\right)^2(z_t-\bar{z}_t)_\beta d\beta.
\end{align*}
It is clear that the first term is easier to estimate than the others since it is quartic, whereas the remaining terms are cubic. Therefore, for $\partial_\alpha^sD_tG_2$ we focus on analyzing the following expression:
\[
\frac{2}{\pi i}\partial_\alpha^s\int\frac{(z_t(\alpha)-z_t(\beta))(z_{tt}(\alpha)-z_{tt}(\beta))}{(z(\alpha)-z(\beta))^2}(z-\bar{z})_\beta d\beta
+\frac{1}{\pi i}\partial_\alpha^s\int\left(\frac{z_t(\alpha)-z_t(\beta)}{z(\alpha)-z(\beta)}\right)^2(z_t-\bar{z}_t)_\beta d\beta.
\]
Like we did in Proposition \ref{formula of dE_s/dt}, we only provide the estimate for the extreme case:
\begin{equation}\label{first part of formula for DtG}
    \begin{aligned}
        I&=\frac{2}{\pi i}\int\frac{(D_t\zeta(\alpha)-D_t\zeta(\beta))(\partial_\alpha^s D_t^2\zeta-\partial_\beta^s D_t^2\zeta)}{(\zeta(\alpha)-\zeta(\beta))^2}(\zeta-\bar{\zeta})_\beta d\beta\\
         &\overset{\text{def}}{=}I_1+I_2,
    \end{aligned}
\end{equation}
where
\[I_1=\frac{2}{\pi i}\int\frac{(D_t\zeta(\alpha)-D_t\zeta(\beta))(\partial_\alpha^sD_t^2\zeta-\partial_\beta D_t^2\zeta)}{(\zeta(\alpha)-\zeta(\beta))^2}(\zeta_\beta-1)d\beta,
\]
\[I_2=\frac{2}{\pi i}\int\frac{(D_t\zeta(\alpha)-D_t\zeta(\beta))(\partial_\alpha D_t^2\zeta-\partial_\beta D_t^2\zeta)}{(\zeta(\alpha)-\zeta(\beta))^2}(\bar{\zeta}_\beta-1) d\beta.
\]
We claim that
\begin{equation}\label{structure of I}
    1_{S(t)}I=-\frac{it}{\alpha}\partial_\alpha^sD_t^2\zeta|\zeta_\alpha-1|^2+O(\epsilon^3t^{-\frac{5}{4}-\delta})_{L^2}.
\end{equation}

First, it is straightforward to verify that the arguments used for $G_{21}$ apply similarly to $I_1$. Thus, we obtain
\[
\norm{I_1}_{L^2}\le C\epsilon^3t^{-\frac{3}{2}}\ln t.
\]
For $1_{S(t)}I_2$, we use frequency decomposition $f_L=P_{<t^{-2}}f,\ f=f-f_L$, and apply the Localization Lemma, together with equation $(D_t^2-iA\partial_\alpha)\zeta=-i$ to obtain
\begin{align*}
    1_{S(t)}I_2
    &=1_{S(t)}\frac{it^2}{2\alpha^2}(D_t\zeta)_H\partial_\alpha^sD_t^2\zeta(\bar{\zeta}_\alpha-1)+O(\epsilon^3t^{-\frac{5}{4}-\delta})_{L^2},\\
    &=-1_{S(t)}\frac{t}{\alpha}(D_t^2\zeta)\partial_\alpha^sD_t^2\zeta(\bar{\zeta}_\alpha-1)+O(\epsilon^3t^{-\frac{5}{4}-\delta})_{L^2},\\
    &=-\frac{it}{\alpha}\partial_\alpha^sD_t^2\zeta|\zeta_\alpha-1|^2+O(\epsilon^3t^{-\frac{5}{4}-\delta})_{L^2}.
\end{align*}
Hence, (\ref{structure of I}) is established. The term $1_{S(t)^c} I$ is handled using Lemma \ref{small or big alpha}, as in Proposition \ref{formula of dE_s/dt}.
    Finally, we obtain
\begin{align*}
    2Re\int\frac{D_t\bar{\sigma}_s}{A}Id\alpha
    &=4Re\int\frac{\partial_\alpha^sD_t^2\bar{\zeta}}{A}1_{S(t)}Id\alpha+O(\epsilon^4t^{-1-\delta})\\
    &=-4Re\int\frac{it}{A\alpha}|\partial_\alpha^sD_t^2\zeta|^2|\zeta_\alpha-1|^2d\alpha+O(\epsilon^4t^{-1-\delta})\\
    &=O(\epsilon^4t^{-1-\delta}).
\end{align*}
as desired.

Next, we consider the corresponding term of $\tilde{G}_{12}$ in $D_tG$ (recall that $f_H=P_{>t^{-1/2-\delta}}f$ and $f_L=f-f_H$ here):
\[
\tilde{G}_{12}^*=\frac{1}{\pi i}\int\frac{(\bar{f}_H(\alpha)-\bar{f}_H(\beta))(\partial_\alpha^sD_t^2\bar{\zeta}-\partial_\beta^sD_t^2\bar{\zeta})}{(\zeta(\alpha)-\zeta(\beta))^2}\partial_\beta D_t\zeta d\beta.
\]

Actually, it can be treated using similar arguments as those in the proof of Proposition \ref{formula of dE_s/dt}, since we always estimate $\partial_\alpha^sD_t\zeta$ in $L^2$, and there is no difficulty in replacing $\partial_\alpha^sD_t\zeta$ by $\partial_\alpha^sD_t^2\zeta$. From now on, if a proof of an estimate is omitted, we mean that it can be proved by simply repeating the relevant part in Proposition \ref{formula of dE_s/dt}.

We only consider the case $\alpha\in S(t)$. Decompose the integral domain into $|\alpha-\beta|\le t^{3/4+1/40}$ and $|\alpha-\beta|>t^{3/4+1/40}$. For the second part, we have the estimate
\begin{align*}
    \norm{\frac{1}{\pi i}\int_{|\alpha-\beta|>t^{3/4+1/40}}\frac{(\bar{f}_H(\alpha)-\bar{f}_H(\beta))(\partial_\alpha^sD_t^2\bar{\zeta}-\partial_\beta^sD_t^2\bar{\zeta})}{(\zeta(\alpha)-\zeta(\beta))^2}\partial_\beta D_t\zeta d\beta}_{L^2}
    \le C\epsilon^3t^{-\frac{5}{4}-\frac{1}{40}}\ln t.
\end{align*}
Next, we decompose the first part in the same manner as before: 
\begin{align*}
        &\frac{1}{\pi i}\int_{|\alpha-\beta|\le t^{3/4+1/40}}\frac{\left(\frac{2\alpha}{it}\partial_t\bar{f}_H-\frac{2\beta}{it}\partial_t\bar{f}_H\right)(\partial_\alpha^sD_t^2\bar{\zeta}-\partial_\beta^sD_t^2\bar{\zeta})}{(\zeta(\alpha)-\zeta(\beta))^2}\partial_\beta D_t\zeta d\beta\\
        &\ \ \ +\frac{1}{\pi i}\int_{|\alpha-\beta|\le t^{3/4+1/40}}\frac{\left(\frac{2}{it}\tilde{\Omega}_0\bar{f}_H(\alpha)-\frac{2}{it}\tilde{\Omega}_0\bar{f}_H(\beta)\right)(\partial_\alpha^sD_t^2\bar{\zeta}-\partial_\beta^sD_t^2\bar{\zeta})}{(\zeta(\alpha)-\zeta(\beta))^2}\partial_\beta D_t\zeta d\beta
    \end{align*}
  For the second term, we obtain the following bound:
    \begin{align*}
        \norm{\frac{1}{\pi i}\int_{|\alpha-\beta|\le t^{3/4+1/40}}\frac{\left(\frac{2}{it}\tilde{\Omega}_0\bar{f}_H(\alpha)-\frac{2}{it}\tilde{\Omega}_0\bar{f}_H(\beta)\right)(\partial_\alpha^sD_t^2\bar{\zeta}-\partial_\beta^sD_t^2\bar{\zeta})}{(\zeta(\alpha)-\zeta(\beta))^2}\partial_\beta D_t\zeta d\beta}_{L^2}
        \le C\epsilon^3t^{-\frac{3}{2}+\delta_0}(\ln t)^2.
    \end{align*}
    For the first term, we have
    \begin{align*}
        &1_{S(t)}\frac{1}{\pi i}\int_{|\alpha-\beta|\le t^{3/4+1/40}}\frac{\left(\frac{2\alpha}{it}\partial_t\bar{f}_H-\frac{2\beta}{it}\partial_t\bar{f}_H\right)(\partial_\alpha^sD_t^2\bar{\zeta}-\partial_\beta^sD_t^2\bar{\zeta})}{(\zeta(\alpha)-\zeta(\beta))^2}\partial_\beta D_t\zeta d\beta\\
        =&\ -1_{S(t)}\frac{it^2}{2\alpha^2}(D_t\bar{\zeta})_H\partial_\alpha^sD_t^2\bar{\zeta}(\zeta_\alpha-1)+O(\epsilon^3t^{-\frac{5}{4}-\delta})_{L^2}.
    \end{align*}
    
    Now all we need is to transit the derivative and make an estimate similar to the parallel argument in Proposition \ref{formula of dE_s/dt} and obtain
    \begin{equation}
        (I-\mathcal{H})\tilde{G}_{12}^*=O(\epsilon^3t^{-\frac{5}{4}-\delta})_{L^2}
    \end{equation}
    for some $\delta>0$. This completes the proof.
 \end{proof}

\subsection{A priori control of the energy estimates}

In conclusion, we get the energy estimate for $\mathcal{E}_s$ by summing up all of the estimates above and conclude the proof of Theorem \ref{thm:sec 5main}. As a corollary, we have

\begin{cor}\label{sec 5:energy bound}
    There holds
    \begin{equation}
        \mathcal{E}_s(t)\le C\epsilon^2\left(e^{\frac{C\epsilon^2}{\delta}-\frac{C\epsilon^2}{\delta}(1+t)^{-\delta}}-1\right)+e^{\frac{C\epsilon^2}{\delta}-\frac{C\epsilon^2}{\delta}(1+t)^{-\delta}}\mathcal{E}_s(0).
    \end{equation}
\end{cor}

\begin{proof}
    We may rewrite (\ref{energy estimate}) as
    \[
    \left|\frac{d}{dt}\mathcal{E}_s(t)\right|\le C\epsilon^2(1+t)^{-1-\delta}\mathcal{E}_s(t)+O(\epsilon^4(1+t)^{-1-\delta})
    \]
    since each relevant term is bounded. Now we have
    \[
    \frac{d}{dt}e^{\frac{C\epsilon^2}{\delta}(1+t)^{-\delta}}\mathcal{E}_s(t)\le C'\epsilon^4(1+t)^{-1-\delta}e^{\frac{C\epsilon^2}{\delta}(1+t)^{-\delta}}
    \]
    which implies
    \[
    \mathcal{E}_s(t)\le\frac{C'\epsilon^2}{C}\left(e^{\frac{C\epsilon^2}{\delta}-\frac{C\epsilon^2}{\delta}(1+t)^{-\delta}}-1\right)+e^{\frac{C\epsilon^2}{\delta}-\frac{C\epsilon^2}{\delta}(1+t)^{-\delta}}\mathcal{E}_s(0).
    \]
\end{proof}

\begin{cor}[Almost conservation law]\label{cor:almost conservation}
    We have for any $t\in [0,T]$,
    \begin{equation}\label{equ:almost conserved 1}
        |\mathcal{E}_s(t)-\mathcal{E}_s(0)|\leq C\epsilon^4.
    \end{equation}
\end{cor}
\begin{proof}
    Combining Theorem \ref{thm:sec 5main} and Corollary \ref{sec 5:energy bound},
    \begin{align*}
        \frac{d\mathcal{E}_s(t)}{dt}\leq C\epsilon^2 (1+t)^{-1-\delta}\mathcal{E}_s(t)+C\epsilon^4 (1+t)^{-1-\delta}\leq C\epsilon^4 (1+t)^{-1-\delta}.
    \end{align*}
    Therefore,
    \begin{align*}
\left|\mathcal{E}_s(t)-\mathcal{E}_s(0)\right|\leq C\int_0^t \epsilon^4 (1+\tau)^{-1-\delta}d\tau\leq C\epsilon^4.
    \end{align*}
    This establishes \eqref{equ:almost conserved 1}. 
\end{proof}

\section{Vector field estimates}

In this section, we improve the vector field estimate in the bootstrap assumption. Firstly, we estimate $[L_0, D_t^2-iA\partial_{\alpha}]\partial_{\alpha}^{s-1}f$ where $f=\theta$ or $D_t\theta$. Keeping in mind that we can control at most $s-1$ derivatives of $L_0D_t\zeta$ and $s-1$ derivatives of $L_0D_t^2\zeta$ (or other equivalent quantities).

\begin{lemma}\label{s-1 order estimate for [L0,P]}
    Suppose that $f=\theta$ or $D_t\theta$. 
    Then there holds
    \begin{align*}
        &\norm{[L_0,D_t^2-iA\partial_\alpha]\partial_\alpha^{s-1}f}_{L^2(|\alpha|\ge t)}\\
     \le&\  C\epsilon^2t^{-1}(\norm{L_0(\zeta_\alpha-1)}_{L^2}+\norm{L_0D_t\zeta}_{H^1}+\norm{L_0(\zeta-\alpha)}_{\dot{H}^{\frac{1}{2}}})
        +C\epsilon^3t^{-1}.
    \end{align*}
\end{lemma}

\begin{proof}
    Recall that
\begin{equation}
    [L_0, D_t^2-iA\partial_{\alpha}]=-(D_t^2-iA\partial_{\alpha})+\Big\{L_0D_tb-(L_0b-\frac{1}{2}b)b_{\alpha}\Big\}\partial_{\alpha}+\Big(L_0b-\frac{1}{2}b\Big)(D_t\partial_{\alpha}+\partial_{\alpha}D_t)-i(L_0A)\partial_{\alpha}.
\end{equation}

We point out that the most difficult part is to estimate $\|L_0b\cdot D_t\partial_\alpha\partial_\alpha^{s-1}f\|_{L^2}$, since there are $s$ derivatives on $D_tf$ while there are no derivatives on $b$.

We first study $L_0b$ since we only have the formula of $(I-\mathcal{H})b$. Actually, according to Lemma \ref{control f by using (I-H)f} (note that $L_0b$ is real), we have
\[
\norm{L_0b}_{L^\infty}\le C\norm{(I-\mathcal{H})L_0b}_{L^\infty}
\le C\norm{L_0(I-\mathcal{H})b}_{L^\infty}+C\norm{[L_0,\mathcal{H}]b}_{L^\infty}.
\]
Note that
\[
[L_0,\mathcal{H}]b=[L_0\zeta,\mathcal{H}]\frac{\partial_\alpha b}{\zeta_\alpha}
\]
which implies
\[
\norm{[L_0,\mathcal{H}]b}_{L^\infty}\le C\norm{L_0(\zeta_\alpha-1)}_{W^{1,\infty}}\norm{\partial_\alpha b}_{W^{1,\infty}}\ln t
\le C\epsilon^3t^{-\frac{3}{2}+\delta+\delta_0}
\]
for any $\delta>0$. Here we have used Lemma \ref{singular} and Lemma \ref{norms of b}. This gives
\begin{align*}
    \norm{L_0b\cdot D_t\partial_\alpha^{s}f}_{L^2}
    &\le C\norm{L_0(I-\mathcal{H})b\cdot D_t\partial_\alpha^sf}_{L^2}+C\norm{[L_0,\mathcal{H}]b\cdot D_t\partial_\alpha^sf}_{L^2}\\
    &\le C\norm{L_0(I-\mathcal{H})b\cdot D_t\partial_\alpha^sf}_{L^2}+C\epsilon^4t^{-\frac{3}{2}+\delta+\delta_0},\ \forall\delta>0.
\end{align*}
Hence it suffices to consider $\norm{L_0(I-\mathcal{H})bD_t\partial_\alpha^sf}_{L^2}$ in the following proof. Define
\[
(I-\mathcal{H})b=\frac{1}{\pi i}\int\frac{D_t\zeta(\alpha)-D_t\zeta(\beta)}{\zeta(\alpha)-\zeta(\beta)}(\bar{\zeta}_\beta-1)d\beta
\overset{\text{def}}{=}\int K(\alpha,\beta,t)(\bar{\zeta}_\beta-1)d\beta
\overset{\text{def}}{=}K(\bar{\zeta}_\alpha-1)
\]
Note that
\[
\begin{aligned}
	L_0(I-\mathcal{H})b
	&=KL_0(\bar{\zeta}_\alpha-1)+[L_0,K](\bar{\zeta}_\alpha-1),\\
	&=\frac{1}{\pi i}\int\frac{D_t\zeta(\alpha)-D_t\zeta(\beta)}{\zeta(\alpha)-\zeta(\beta)}L_0(\bar{\zeta}_\beta-1)d\beta+\frac{1}{\pi i}\int\left(\alpha\partial_\alpha+\beta\partial_\beta+\frac{1}{2}t\partial_t\right)K(\alpha,\beta,t)(\bar{\zeta}_\beta-1)d\beta\\
        &\ \ \ +K(\bar{\zeta}_\alpha-1),\\
	&=\frac{1}{\pi i}\int\frac{D_t\zeta(\alpha)-D_t\zeta(\beta)}{\zeta(\alpha)-\zeta(\beta)}L_0(\bar{\zeta}_\beta-1)d\beta+\frac{1}{\pi i}\int\frac{L_0D_t\zeta(\alpha)-L_0D_t\zeta(\beta)}{\zeta(\alpha)-\zeta(\beta)}(\bar{\zeta}_\beta-1)d\beta\\
	&\ \ \ -\frac{1}{\pi i}\int\frac{(L_0\zeta(\alpha)-L_0\zeta(\beta))(D_t\zeta(\alpha)-D_t\zeta(\beta))}{(\zeta(\alpha)-\zeta(\beta))^2}(\bar{\zeta}_\beta-1)d\beta+(I-\mathcal{H})b,\\
	&=I+II+III+(I-\mathcal{H})b.
\end{aligned}
\]
We begin with the term $ID_t\partial_{\alpha}^s f$ and rewrite it as
\begin{equation}\label{first term of L0b}
	ID_t\partial_{\alpha}^sf=D_t\partial_{\alpha}^sf(\alpha,t) \Big(\int_{|\alpha-\beta|\leq t/2}+\int_{|\alpha-\beta|\geq t/2}\Big)\frac{D_t\zeta(\alpha)-D_t\zeta(\beta)}{\zeta(\alpha)-\zeta(\beta)}L_0(\bar{\zeta}_\beta-1)d\beta.
\end{equation}
We have
\begin{align*}
   & D_t\partial_{\alpha}^sf(\alpha,t) \int_{|\alpha-\beta|\geq t/2}\frac{D_t\zeta(\alpha)-D_t\zeta(\beta)}{\zeta(\alpha)-\zeta(\beta)}L_0(\bar{\zeta}_\beta-1)d\beta\\
   =&\frac{1}{\pi i}D_t\partial_{\alpha}^sf D_t\zeta(\alpha,t) \int_{|\alpha-\beta|\geq t/2}\frac{L_0(\bar{\zeta}_\beta-1)}{\zeta(\alpha)-\zeta(\beta)}d\beta - \frac{1}{\pi i}D_t\partial_{\alpha}^sf(\alpha,t) \int_{|\alpha-\beta|\geq t/2}\frac{D_t\zeta(\beta)}{\zeta(\alpha)-\zeta(\beta)}L_0(\bar{\zeta}_\beta-1)d\beta
\end{align*}
For the first term, by Lemma \ref{lemma:transit2} ($f=\theta$ or $D_t\theta$), the $L^2$ boundedness of the truncated singular integrals, we have 
\begin{align*}
    \norm{D_t\partial_{\alpha}^sf D_t\zeta(\alpha,t) \int_{|\alpha-\beta|\geq t/2}\frac{L_0(\bar{\zeta}_\beta-1)}{\zeta(\alpha)-\zeta(\beta)}d\beta}_{L^2(|\alpha|\geq t)}
    \leq &\  \norm{D_t\partial_{\alpha}^sf D_t\zeta}_{L^{\infty}(|\alpha|\ge t)}\norm{\int_{|\alpha-\beta|\geq t/2}\frac{L_0(\bar{\zeta}_\beta-1)}{\zeta(\alpha)-\zeta(\beta)}d\beta}_{L^2}\\
    \leq &\  C\epsilon^2 t^{-1}\norm{L_0(\zeta_{\alpha}-1)}_{L^2}.
\end{align*}
For the second term, we have
\begin{align*}
   & \norm{D_t\partial_{\alpha}^sf(\alpha,t) \int_{|\alpha-\beta|\geq t/2}\frac{D_t\zeta(\beta)}{\zeta(\alpha)-\zeta(\beta)}L_0(\bar{\zeta}_\beta-1)d\beta}_{L^2(|\alpha|\geq t)}\\
   \leq &\  \norm{D_t\partial_{\alpha}^s f}_{L^2}\norm{\int_{|\alpha-\beta|\geq t/2}\frac{D_t\zeta(\beta)}{\zeta(\alpha)-\zeta(\beta)}L_0(\bar{\zeta}_\beta-1)d\beta}_{L^{\infty}}\\
   \leq &\   Ct^{-1}\norm{D_t\partial_{\alpha}^s f}_{L^2}\norm{D_t\zeta}_{L^2}\norm{L_0(\zeta_{\alpha}-1)}_{L^2}\\
   \leq &\  C\epsilon^2 t^{-1}\norm{L_0(\zeta_{\alpha}-1)}_{L^2}.
\end{align*}
For $|\alpha|\geq t$ and $|\alpha-\beta|\leq t/2$, we have $\frac{1}{2}|\alpha|\leq |\beta|\leq 2|\alpha|$. Apply Lemma \ref{lemma:transit2} and we have
\begin{align*}
    &\norm{D_t\partial_{\alpha}^sf(\alpha,t) \int_{|\alpha-\beta|\leq t/2}\frac{D_t\zeta(\alpha)-D_t\zeta(\beta)}{\zeta(\alpha)-\zeta(\beta)}L_0(\bar{\zeta}_\beta-1)d\beta}_{L^2(|\alpha|\geq t)}\\
    \leq &\  \norm{D_t\partial_{\alpha}^sfD_t\zeta(\alpha) \int_{|\alpha-\beta|\leq t/2}\frac{L_0(\bar{\zeta}_\beta-1)}{\zeta(\alpha)-\zeta(\beta)}d\beta}_{L^2(|\alpha|\ge t)}+ \norm{ \int_{|\alpha-\beta|\leq t/2}\frac{D_t\partial_{\alpha}^sf(\alpha,t)D_t\zeta(\beta)}{\zeta(\alpha)-\zeta(\beta)}L_0(\bar{\zeta}_\beta-1)d\beta}_{L^2(|\alpha|\geq t)}\\
    \leq &\  \norm{D_t\partial_{\alpha}^sf D_t\zeta}_{L^{\infty}(|\alpha|\ge t)}\norm{L_0(\zeta_{\alpha}-1)}_{L^2}+\sup_{\frac{|\alpha|}{2}\leq |\beta|\leq 2|\alpha|}|D_t\zeta(\beta,t)D_t\partial_{\alpha}^sf|\norm{L_0(\zeta_{\alpha}-1)}_{L^2}\\
    \leq &\  C\epsilon^2 t^{-1}\norm{L_0(\zeta_{\alpha}-1)}_{L^2}.
\end{align*}

Next we consider II, i.e.
\[
\left\|\frac{1}{\pi i}D_t\partial_\alpha^{s}f\int\frac{L_0D_t\zeta(\alpha)-L_0D_t\zeta(\beta)}{\zeta(\alpha)-\zeta(\beta)}(\bar{\zeta}_\beta-1)d\beta\right\|_{L^2(|\alpha|\geq t)}.
\]
At first, note that
\[
\norm{\frac{1}{\pi i}D_t\partial_\alpha^sfL_0D_t\zeta(\alpha)\int\frac{\bar{\zeta}_\beta-1}{\zeta(\alpha)-\zeta(\beta)}d\beta}_{L^2(|\alpha|\ge t)}
=\norm{D_t\partial_\alpha^sf\frac{\bar{\zeta}_\alpha-1}{\zeta_\alpha}L_0D_t\zeta}_{L^2(|\alpha|\ge t)}
\le C\epsilon^2t^{-1}\norm{L_0D_t\zeta}_{L^2}
\]
according to Lemma \ref{lemma:transit2}. Then, since $|\alpha-\beta|\le t/2$ implies $|\alpha|/2\le|\beta|\le 2|\alpha|$, there holds
\[
\begin{aligned}
    \left\|D_t\partial_\alpha^{s}f\int_{|\alpha-\beta|\leq t/2}\frac{L_0D_t\zeta(\beta)}{\zeta(\alpha)-\zeta(\beta)}(\bar{\zeta}_\beta-1)d\beta\right\|_{L^2(|\alpha|\geq t)}
    &\le\sup_{\frac{|\alpha|}{2}\leq |\beta|\leq 2|\alpha|}|(\bar{\zeta}_{\beta}-1)D_t\partial_{\alpha}^sf|\norm{L_0D_t\zeta}_{L^2},\\
    &\le C\epsilon^2 t^{-1}\norm{L_0D_t\zeta}_{L^2}.
\end{aligned}
\]
Finally,
\[
\begin{aligned}
    \left\|D_t\partial_\alpha^{s}f\int_{|\alpha-\beta|\geq t/2}\frac{L_0D_t\zeta(\beta)}{\zeta(\alpha)-\zeta(\beta)}(\bar{\zeta}_\beta-1)d\beta\right\|_{L^2(|\alpha|\geq t)}
    &\le C\epsilon\int_{|\alpha-\beta|\ge t/2}\left|\frac{L_0D_t\zeta(\beta)}{\zeta(\alpha)-\zeta(\beta)}(\bar{\zeta}_\beta-1)\right|d\beta,\\
    &\le C\epsilon t^{-1}\norm{\bar{\zeta}_\alpha-1}_{L^2}\norm{L_0D_t\zeta}_{L^2},\\
    &\le C\epsilon^2t^{-1}\norm{L_0D_t\zeta}_{L^2}.
\end{aligned}
\]
The proof for III is easier than I,II and we shall omit it. Also, it's obvious that the proof for I,II works for $(I-\mathcal{H})b$ and gives
\[
\norm{D_t\partial_\alpha^sf(I-\mathcal{H})b}_{L^2(|\alpha|\ge t)}\le C\epsilon^3t^{-1}.
\]
This completes the proof.
\end{proof}

\begin{rem}\label{s-1 extend to k}
    There is no substantive difficulty in extending this proof to $\norm{[L_0,D_t^2-iA\partial_\alpha]\partial_\alpha^kf}_{L^2(|\alpha|\ge t)}$ for $0\le k\le s-2$, since the expression of the commutator indicates that there is at least 1 derivative on $f$ (except for the term $(D_t^2-iA\partial_\alpha)f$; however, such a term is cubic and therefore decays in $L^2$ at a rate $t^{-1}$).
\end{rem}

\begin{proposition}\label{vector field est for s-1 order on R}
    Under the assumption of Lemma \ref{s-1 order estimate for [L0,P]}, there holds
    \begin{equation}
        \begin{aligned}
            \norm{[L_0,D_t^2-iA\partial_\alpha]\partial_\alpha^{s-1}f}_{L^2(R)}
            &\le C\epsilon^2t^{-1+\delta_0}\ln t(\norm{L_0(\bar{\zeta}_\alpha-1)}_{L^2}+\norm{L_0(\zeta-\alpha)}_{\dot{H}^{\frac{1}{2}}}+\norm{L_0D_t\zeta}_{L^2})\\
            &\ \ \ +C\epsilon^3t^{-1}.
        \end{aligned}
    \end{equation}
\end{proposition}

\begin{proof}
    It's enough for us to estimate $\norm{L_0(I-\mathcal{H})bD_t\partial_\alpha^sf}_{L^2(|\alpha|\le t)}$. Define $f_L=P_{<t^{-1}}f$, $f_H=f-f_L$.

    We first consider I. Note that
    \[
    \norm{(D_t\zeta)_L}_{L^\infty}\le t^{-\frac{1}{2}}\norm{(D_t\zeta)_L}_{L^2}\le C\epsilon t^{-\frac{1}{2}},
    \]
    hence
    \begin{align*}
        \norm{\int_{|\alpha-\beta|\ge t}\frac{(D_t\zeta)_{L}(\alpha)-(D_t\zeta)_{L}(\beta)}{\zeta(\alpha)-\zeta(\beta)}L_0(\bar{\zeta}_\beta-1)d\beta}_{L^\infty}
        &\le Ct^{-\frac{1}{2}}\norm{(D_t\zeta)_L}_{L^\infty}\norm{L_0(\bar{\zeta}_\alpha-1)}_{L^2},\\
        &\le C\epsilon t^{-1}\norm{L_0(\bar{\zeta}_\alpha-1)}_{L^2}.
    \end{align*}
    Also,
    \begin{align*}
        \norm{\int_{|\alpha-\beta|\le t}\frac{(D_t\zeta)_{L}(\alpha)-(D_t\zeta)_{L}(\beta)}{\zeta(\alpha)-\zeta(\beta)}L_0(\bar{\zeta}_\beta-1)d\beta}_{L^\infty}
        &\le C\norm{\partial_\alpha(D_t\zeta)_{L}}_{L^\infty}t^{\frac{1}{2}}\norm{L_0(\bar{\zeta}_\alpha-1)}_{L^2}.
    \end{align*}
    While
    \[
    \norm{\partial_\alpha(D_t\zeta)_L}_{L^\infty}\le t^{-1}\norm{(D_t\zeta)_L}_{L^\infty}\le C\epsilon t^{-\frac{3}{2}},
    \]
    we have
    \begin{equation}\label{estimate for Low frequency}
        \norm{\int \frac{(D_t\zeta)_{L}(\alpha)-(D_t\zeta)_{L}(\beta)}{\zeta(\alpha)-\zeta(\beta)}L_0(\bar{\zeta}_\beta-1)d\beta}_{L^\infty}
        \le C\epsilon t^{-1}\norm{L_0(\bar{\zeta}_\alpha-1)}_{L^2}.
    \end{equation}

    For the high frequency part (denoted by $I_h$), we write
    \[
    \norm{D_t\partial_\alpha^sfI_h}_{L^2(|\alpha|\le t)}\le \max\left\{\norm{D_t\partial_\alpha^sfI_h}_{L^2(t^{3/4}\le|\alpha|\le t)},\ \norm{D_t\partial_\alpha^sfI_h}_{L^2(|\alpha|\le t^{3/4})}\right\}.
    \]
    Note that
    \begin{align*}
        \norm{D_t\partial_\alpha^sfI_h}_{L^2(|\alpha|\le t^{3/4})}
        &\le\norm{D_t\partial_\alpha^sf\int_{|\alpha-\beta|\le 2t^{\frac{3}{4}}}\frac{(D_t\zeta)_H(\alpha)-(D_t\zeta)_H(\beta)}{\zeta(\alpha)-\zeta(\beta)}L_0(\bar{\zeta}_\beta-1)d\beta}_{L^2(|\alpha|\le t^{\frac{3}{4}})}\\
        &\ \ \ +\norm{D_t\partial_\alpha^sf(D_t\zeta)_H\int_{|\alpha-\beta|>2t^{3/4}}\frac{L_0(\bar{\zeta}_\beta-1)}{\zeta(\alpha)-\zeta(\beta)}d\beta}_{L^2(|\alpha|\le t^{3/4})}\\
        &\ \ \ +\norm{D_t\partial_\alpha^sf\int_{|\alpha-\beta|\ge 2t^{3/4}}\frac{L_0(\bar{\zeta}_\beta-1)}{\zeta(\alpha)-\zeta(\beta)}(D_t\zeta)_H(\beta)d\beta}_{L^2(|\alpha|\le t^{3/4})}.
    \end{align*}
    For the first term, we have
    \[
    (D_t\zeta)_H(\alpha)=\frac{4\alpha^2}{it^2}\partial_\alpha(D_t\zeta)_H-\frac{4\alpha}{it^2}L_0(D_t\zeta)_H+\frac{2}{it}\Omega_0(D_t\zeta)_H
    \]
    which implies (use (\ref{Omega_0D_t zeta high}))
    \[
    \norm{(D_t\zeta)_H}_{L^\infty(|\alpha|\le Ct^{\frac{3}{4}})}\le C\epsilon t^{-1+\delta_0}\ln t
    \]
    and we derive (note that $|\alpha-\beta|\le 2t^{\frac{3}{4}}$ implies $|\beta|\le 3t^{\frac{3}{4}}$)
    \begin{align*}
        &\norm{D_t\partial_\alpha^sf\int_{|\alpha-\beta|\le 2t^{\frac{3}{4}}}\frac{(D_t\zeta)_H(\alpha)-(D_t\zeta)_H(\beta)}{\zeta(\alpha)-\zeta(\beta)}L_0(\bar{\zeta}_\beta-1)d\beta}_{L^2(|\alpha|\le t^{\frac{3}{4}})}\\
        \le&\ C\norm{D_t\partial_\alpha^sf}_{L^2}\norm{\int_{|\alpha-\beta|\le 2t^{\frac{3}{4}}}\frac{(D_t\zeta)_H(\alpha)-(D_t\zeta)_H(\beta)}{\zeta(\alpha)-\zeta(\beta)}L_0(\bar{\zeta}_\beta-1)d\beta}_{L^\infty(|\alpha|\le t^{\frac{3}{4}})}\\
        \le&\  C\epsilon\norm{\int_{|\alpha-\beta|\le 1}}_{L^\infty(|\alpha|\le t^{\frac{3}{4}})}+C\epsilon\norm{\int_{1<|\alpha-\beta|\le t^{\frac{3}{4}}}}_{L^\infty(|\alpha|\le t^{\frac{3}{4}})}\\
        \le&\ C\epsilon\norm{\partial_\alpha(D_t\zeta)_H}_{L^\infty(|\alpha|\le 3t^{\frac{3}{4}})}\norm{L_0(\bar{\zeta}_\alpha-1)}_{L^2}+C\epsilon\norm{(D_t\zeta)_H}_{L^\infty(|\alpha|\le 3t^{\frac{3}{4}})}\int_{|\alpha-\beta|\ge 1}\left|\frac{L_0(\bar{\zeta}_\alpha-1)}{\zeta(\alpha)-\zeta(\beta)}\right|d\beta\\
        \le&\ C\epsilon^2t^{-1+\delta_0}\ln t\norm{L_0(\bar{\zeta}_\alpha-1)}_{L^2}.
    \end{align*}
    For the second term, there holds
    \begin{align*}
        \norm{D_t\partial_\alpha^sf(D_t\zeta)_H\int_{|\alpha-\beta|\ge 2t^{\frac{3}{4}}}\frac{L_0(\bar{\zeta}_\beta-1)}{\zeta(\alpha)-\zeta(\beta)}d\beta}_{L^2(|\alpha|\le t^{\frac{3}{4}})}
        &\le C\epsilon^2t^{-1+\delta_0}\ln t \times t^{-\frac{3}{8}}\norm{L_0(\bar{\zeta}_\alpha-1)}_{L^2},\\
        &\le C\epsilon^3t^{-1-\delta}.
    \end{align*}
    For the third term, note that $|\alpha-\beta|\ge 2t^{\frac{3}{4}}\ge 2|\alpha|$, which implies $C|\alpha-\beta|\ge|\beta|\ge t^{\frac{3}{4}}$ and thus
    \begin{equation}\label{useful argument}
        \begin{aligned}
            &\norm{D_t\partial_\alpha^sf\int_{|\alpha-\beta|\ge 2t^{\frac{3}{4}}}\frac{L_0(\bar{\zeta}_\beta-1)}{\zeta(\alpha)-\zeta(\beta)}(D_t\zeta)_H(\beta)d\beta}_{L^2(|\alpha|\le t^{\frac{3}{4}})}\\
     \le&\ C\epsilon\norm{\int_{\substack{|\alpha-\beta|\ge 2t^{\frac{3}{4}},\\|\beta|\ge t^{\frac{3}{4}}}}\left|\frac{L_0(\bar{\zeta}_\beta-1)}{\zeta(\alpha)-\zeta(\beta)}(D_t\zeta)_H(\beta)\right|d\beta}_{L^\infty(|\alpha|\le t^{\frac{3}{4}})},\\
     \le&\ C\epsilon t^{-1}\int_{|\beta|\ge t^{\frac{3}{4}}}\left|\frac{t}{\beta}(D_t\zeta)_H(\beta)\right||L_0(\bar{\zeta}_\beta-1)|d\beta,\\
     \le&\ C\epsilon t^{-1}\int_{|\beta|\ge t^{\frac{3}{4}}}\left(|\partial_t(D_t\zeta)_H|+\frac{1}{|\beta|}|\Omega_0(D_t\zeta)_H|\right)|L_0(\bar{\zeta}_\beta-1)|d\beta,\\
     \le&\ C\epsilon^2 t^{-1}\norm{L_0(\bar{\zeta}_\alpha-1)}_{L^2}.
        \end{aligned}
    \end{equation}
    So we conclude that
    \begin{equation}\label{estimate of IM for very small alpha}
        \norm{D_t\partial_\alpha^sfI_h}_{L^2(|\alpha|\le t^{\frac{3}{4}})}\le C\epsilon^2t^{-1+\delta_0}\ln t\norm{L_0(\bar{\zeta}_\alpha-1)}_{L^2}.
    \end{equation}
    For $t^{3/4}<|\alpha|\le t$, define $E_t=\left\{\alpha\in \mathbb{R}\ |\ t^{3/4}<|\alpha|\le t\right\}$. Still, we decompose the integral as
    \begin{align*}
        \norm{D_t\partial_\alpha^sfI_h}_{L^2(E_t)}
        &\le\norm{D_t\partial_\alpha^sf\int_{|\alpha-\beta|\le \frac{1}{2}t^{3/4}}\frac{(D_t\zeta)_H(\alpha)-(D_t\zeta)_H(\beta)}{\zeta(\alpha)-\zeta(\beta)}L_0(\bar{\zeta}_\beta-1)d\beta}_{L^2(E_t)}\\
        &\ \ \ +\norm{D_t\partial_\alpha^sf(D_t\zeta)_H\int_{|\alpha-\beta|\ge \frac{1}{2}t^{3/4}}\frac{L_0(\bar{\zeta}_\beta-1)}{\zeta(\alpha)-\zeta(\beta)}d\beta}_{L^2(E_t)}\\
        &\ \ \ +\norm{D_t\partial_\alpha^sf\int_{|\alpha-\beta|\ge \frac{1}{2}t^{3/4}}\frac{L_0(\bar{\zeta}_\beta-1)}{\zeta(\alpha)-\zeta(\beta)}(D_t\zeta)_H(\beta)d\beta}_{L^2(E_t)}.
    \end{align*}
    For the first term, observe that $|\alpha-\beta|\le \frac{1}{2}t^{3/4}$ and $\alpha\in E(t)$ implies $|\alpha|/2\le|\beta|\le 2|\alpha|$. So applying Lemma \ref{lemma:transit2} yields the result (note that we can transit the derivative on $E_t$). Also, the second term behaves well for a similar reason. For the third term, we decompose it as
    \begin{align*}
        &\int_{|\alpha-\beta|\ge \frac{1}{2}t^{3/4}}\frac{L_0(\bar{\zeta}_\beta-1)}{\zeta(\alpha)-\zeta(\beta)}(D_t\zeta)_H(\beta)d\beta\\
       =&\left(\int_{\substack{|\alpha-\beta|\ge \frac{1}{2}t^{3/4}\\|\beta|\le |\alpha|/2}}+\int_{\substack{|\alpha-\beta|\ge \frac{1}{2}t^{3/4}\\|\alpha|/2<|\beta|\le 2|\alpha|}}+\int_{\substack{|\alpha-\beta|\ge \frac{1}{2}t^{3/4}\\|\beta|>2|\alpha|}}\right)\frac{L_0(\bar{\zeta}_\beta-1)}{\zeta(\alpha)-\zeta(\beta)}(D_t\zeta)_H(\beta)d\beta,\\
        \overset{\text{def}}{=}&I_1+I_2+I_3.
    \end{align*}
    For $I_1$, we have
    \begin{align*}
        \norm{I_1}_{L^\infty(E_t)}
        &\le\norm{\int_{|\beta|\le |\alpha|/2}\left|\frac{L_0(\bar{\zeta}_\beta-1)}{\zeta(\alpha)-\zeta(\beta)}(D_t\zeta)_H(\beta)\right|d\beta}_{L^\infty(E_t)},\\
        &\le C\norm{\int_{|\beta|\le |\alpha|/2}\left|\frac{L_0(\bar{\zeta}_\beta-1)}{|\alpha|}\left(-\frac{2\beta}{it}\partial_t(D_t\zeta)_H+\frac{2}{it}\Omega_0(D_t\zeta)_H\right)\right|d\beta}_{L^\infty(E_t)},\\
        &\le Ct^{-1}\int_{|\beta|\le t/2}|L_0(\bar{\zeta}_\beta-1)||\partial_t(D_t\zeta)_H|d\beta\\
        &\ \ \ +C\epsilon t^{-1+\delta_0}\ln t\norm{\frac{1}{|\alpha|}\int_{|\beta|\le|\alpha|/2}|L_0(\bar{\zeta}_\beta-1)|d\beta}_{L^\infty(E_t)},\\
        &\le C\epsilon t^{-1}\norm{L_0(\bar{\zeta}_\alpha-1)}_{L^2}+C\epsilon^2t^{-1-\frac{3}{8}+2\delta_0}\ln t
    \end{align*}
    which is enough. For $I_2$, note that $|\beta|<2|\alpha|$ implies $|\alpha-\beta|<3|\alpha|\le 3t$ and we have
    \begin{align*}
        \norm{D_t\partial_\alpha^sfI_2}_{L^2(E_t)}
        &\le\sup_{|\alpha|/2\le|\beta|\le 2|\alpha|}|(D_t\zeta)_H(\beta)D_t\partial_\alpha^sf|\norm{\int_{\substack{|\alpha-\beta|>\frac{1}{2}t^{3/4}\\|\alpha|/2<|\beta|\le 2|\alpha|}}\left|\frac{L_0(\bar{\zeta}_\beta-1)}{\alpha-\beta}\right|\left|\frac{\alpha-\beta}{\zeta(\alpha)-\zeta(\beta)}\right|d\beta}_{L^2(E_t)},\\
        &\le C\epsilon^2t^{-1}\norm{\int_{\frac{1}{2}t^{3/4}<|\alpha-\beta|\le 3t}\left|\frac{L_0(\bar{\zeta}_\beta-1)}{\alpha-\beta}\right|d\beta}_{L^2},\\
        &\le C\epsilon^2t^{-1}\norm{\chi_{\frac{1}{2}t^{3/4}\le|\alpha|\le 3t}\frac{1}{|\alpha|}}_{L^1}\norm{L_0(\bar{\zeta}_\alpha-1)}_{L^2},\\
        &\le C\epsilon^2t^{-1}\ln t\norm{L_0(\bar{\zeta}_\alpha-1)}_{L^2}.
    \end{align*}
    For $I_3$, note that $|\alpha-\beta|\ge|\beta|/2$ and we thus have
    \begin{align*}
        \norm{I_3}_{L^\infty(E_t)}
        \le C\int_{|\beta|>2|\alpha|>2t^{3/4}}\left|\frac{L_0(\bar{\zeta}_\beta-1)}{\beta}(D_t\zeta)_H(\beta)\right|d\beta
        \le C\epsilon t^{-1}\norm{L_0(\bar{\zeta}_\alpha-1)}_{L^2}.
    \end{align*}
    Here we have used the same estimate in \eqref{useful argument}. So we conclude that
    \[
    \norm{D_t\partial_\alpha^sfI_h}_{L^2(|\alpha|\le t)}\le C\epsilon^2 t^{-1+\delta_0}\ln t\norm{L_0(\bar{\zeta}_\alpha-1)}_{L^2}+C\epsilon^3t^{-1-\delta}.
    \]
    
    Then we consider II. For $\alpha\in E_t$, write
    \begin{align*}
        \norm{D_t\partial_\alpha^sfII}_{L^2(E_t)}
        &\le\norm{D_t\partial_\alpha^sfL_0D_t\zeta(\alpha)\frac{\bar{\zeta}_\alpha-1}{\zeta_\alpha}}_{L^2(E_t)}
        +\norm{D_t\partial_\alpha^sf\int\frac{L_0D_t\zeta(\beta)}{\zeta(\alpha)-\zeta(\beta)}(\bar{\zeta}_\beta-1)d\beta}_{L^2(E_t)},\\
        &\le C\epsilon^2t^{-1}\norm{L_0D_t\zeta}_{L^2}+\norm{D_t\partial_\alpha^sf\int\frac{L_0D_t\zeta(\beta)}{\zeta(\alpha)-\zeta(\beta)}(\bar{\zeta}_\beta-1)d\beta}_{L^2(E_t)}
    \end{align*}
    and the second term is treated by the same method which is applied to $\norm{D_t\partial_\alpha^sfI_h}_{L^2(E_t)}$\footnote{Namely, decompose the integral into two parts: $|\alpha-\beta|\le\frac{1}{2}t^{3/4}$ and $|\alpha-\beta|>\frac{1}{2}t^{3/4}$. For the former, apply Lemma \ref{lemma:transit2}; for the latter, apply the same argument in \eqref{useful argument}.}

    For $|\alpha|\le t^{\frac{3}{4}}$, we have
    \begin{align*}
        \norm{D_t\partial_\alpha^sfII}_{L^2(|\alpha|\le t^{\frac{3}{4}})}
        &\le\norm{D_t\partial_\alpha^sf\int_{|\zeta(\alpha)-\zeta(\beta)|\le 1}\frac{L_0D_t\zeta(\alpha)-L_0D_t\zeta(\beta)}{\zeta(\alpha)-\zeta(\beta)}(\bar{\zeta}_\beta-1)d\beta}_{L^2(|\alpha|\le t^{\frac{3}{4}})}\\
        &\ \ \ +\norm{D_t\partial_\alpha^sf\int_{1<|\zeta(\alpha)-\zeta(\beta)|\le 2t^{3/4}}\frac{L_0D_t\zeta(\alpha)-L_0D_t\zeta(\beta)}{\zeta(\alpha)-\zeta(\beta)}(\bar{\zeta}_\beta-1)d\beta}_{L^2(|\alpha|\le t^{\frac{3}{4}})}\\
        &\ \ \ +\norm{D_t\partial_\alpha^sf\int_{|\zeta(\alpha)-\zeta(\beta)|\ge 2t^{3/4}}\frac{L_0D_t\zeta(\alpha)-L_0D_t\zeta(\beta)}{\zeta(\alpha)-\zeta(\beta)}(\bar{\zeta}_\beta-1)d\beta}_{L^2(|\alpha|\le t^{\frac{3}{4}})},\\
        &\overset{\text{def}}{=}II_1+II_2+II_3.
    \end{align*}
    Note that
    \[
    |\zeta(\alpha)-\zeta(\beta)|\le Ct^{3/4},|\alpha|\le t^{3/4}\Rightarrow|\beta|\le|\alpha|+\left|\frac{\alpha-\beta}{\zeta(\alpha)-\zeta(\beta)}\right||\zeta(\alpha)-\zeta(\beta)|
    \le Ct^{3/4}
    \]
    and
    \[
    \norm{\bar{\zeta}_\alpha-1}_{L^\infty(|\alpha|\le t^{\frac{3}{4}})}\le Ct^{-1}\norm{\tilde{\Omega}_0(\bar{\zeta}_\alpha-1)}_{L^\infty}+O(\epsilon t^{-1})
    \le C\epsilon t^{-1+\delta_0}.
    \]
    At first, we use Hardy's inequality to get
    \[
    II_1\le C\epsilon \norm{\partial_\alpha L_0D_t\zeta}_{L^2}\norm{\bar{\zeta}_\alpha-1}_{L^\infty(|\alpha|\le Ct^{\frac{3}{4}})}
    \le C\epsilon^2t^{-1+\delta_0}\norm{\partial_\alpha L_0D_t\zeta}_{L^2};
    \]
    then, observe that
    \[
    II_2\le C\epsilon\norm{L_0D_t\zeta}_{L^\infty}\norm{\bar{\zeta}_\alpha-1}_{L^\infty(|\alpha|\le Ct^{\frac{3}{4}})}\ln t\le C\epsilon^2t^{-1+\delta_0}\ln t\norm{L_0D_t\zeta}_{H^1}.
    \]
    For $II_3$, we split $L_0D_t\zeta(\alpha)-L_0D_t\zeta(\beta)$ into two parts, and
    \[
    D_t\partial_\alpha^sf\int_{|\zeta(\alpha)-\zeta(\beta)|\ge 2t^{3/4}}\frac{L_0D_t\zeta(\beta)}{\zeta(\alpha)-\zeta(\beta)}(\bar{\zeta}_\beta-1)d\beta
    \]
    is treated by the argument for \eqref{useful argument}. Finally, we deduce that
    \begin{align*}
        &\norm{D_t\partial_\alpha^sfL_0D_t\zeta\int_{|\zeta(\alpha)-\zeta(\beta)|\ge 2t^{3/4}}\frac{\bar{\zeta}_\beta-1}{\zeta(\alpha)-\zeta(\beta)}d\beta}_{L^2(|\alpha|\le t^{\frac{3}{4}})}\\
        &\le C\epsilon\norm{L_0D_t\zeta}_{H^1}\norm{\int_{|\zeta(\alpha)-\zeta(\beta)|\ge 2t^{3/4}}\frac{\bar{\zeta}_\beta-1}{\zeta(\alpha)-\zeta(\beta)}d\beta}_{L^\infty(|\alpha|\le t^{\frac{3}{4}})},\\
        &\le C\epsilon\norm{L_0D_t\zeta}_{H^1}\norm{\frac{\bar{\zeta}_\alpha-1}{\zeta_\alpha}}_{L^\infty(|\alpha|\le t^{\frac{3}{4}})}
        +C\epsilon\norm{L_0D_t\zeta}_{H^1}\norm{\int_{|\zeta(\alpha)-\zeta(\beta)|\le 2t^{3/4}}\frac{\bar{\zeta}_\beta-1}{\zeta(\alpha)-\zeta(\beta)}d\beta}_{L^\infty(|\alpha|\le t^{\frac{3}{4}})},\\
        &\le C\epsilon^2t^{-1+\delta_0}\ln t\norm{L_0D_t\zeta}_{H^1}.
    \end{align*}
    
    Therefore, we conclude that $\norm{D_t\partial_\alpha^sfII}_{L^2(|\alpha|\le t)}\le C\epsilon^2t^{-1+\delta_0}\ln t\norm{L_0D_t\zeta}_{H^1}+O(\epsilon^3t^{-1-\delta})$. Due to the same reason that is stated in the proof for Lemma \ref{s-1 order estimate for [L0,P]}, we omit the estimate for III and point out that
    \[
    \norm{(I-\mathcal{H})bD_t\partial_\alpha^sf}_{L^2(R)}\le C\epsilon^4t^{-1}
    \]
    without proof. This completes the proof.
\end{proof}

\begin{cor}\label{k order estimate of [L0,P]}
    Under the assumption of Lemma \ref{s-1 order estimate for [L0,P]}, there holds
    \[
    \norm{[L_0,D_t^2-iA\partial_\alpha]\partial_\alpha^kf}_{L^2(R)}\le C\epsilon^2t^{-1}(\norm{L_0(\zeta_\alpha-1)}_{L^2}+\norm{L_0D_t\zeta}_{L^2}+\norm{L_0(\zeta-\alpha)}_{\dot{H}^{\frac{1}{2}}})+C\epsilon^3t^{-1}
    \]
    for $0\le k\le s-2$.
\end{cor}

\begin{proof}
    According to Remark \ref{s-1 extend to k}, we shall focus on $|\alpha|\le t$. Still, we only consider the term $L_0bD_t\partial_\alpha\partial_\alpha^k f$. At first, thanks to Lemma \ref{decay of D_tzeta}, we have $\norm{D_t\partial_\alpha^{k+1}f}_{L^\infty}\le C\epsilon t^{-1/6}\ln t$.
    Then, we list out the estimate for each term of $D_t\partial_\alpha^{k+1}fI$ which does not decay with speed $t^{-1}$ in $L^2$ in the proof of Proposition \ref{vector field est for s-1 order on R}.

    \begin{enumerate}
        \item 
        \begin{align*}
            &\norm{D_t\partial_\alpha^{k+1}f\int_{|\alpha-\beta|\le 2t^{\frac{3}{4}}}\frac{(D_t\zeta)_H(\alpha)-(D_t\zeta)_H(\beta)}{\zeta(\alpha)-\zeta(\beta)}L_0(\bar{\zeta}_\beta-1)d\beta}_{L^2(|\alpha|\le t^{\frac{3}{4}})}\\
            &\le C\norm{D_t\partial_\alpha^{k+1}f}_{L^\infty}\norm{(D_t\zeta)_H}_{L^\infty(|\alpha|\le 3t^{3/4})}\norm{L_0(\bar{\zeta}_\alpha-1)}_{L^2},\\
            &\le C\epsilon^3 t^{-\frac{1}{6}-1+\delta_0}(\ln t)^2,\\
            &\le C\epsilon^3t^{-1-\delta}.
        \end{align*}
        \item 
        \begin{align*}
            &\norm{D_t\partial_\alpha^{k+1}f\int_{\substack{|\alpha-\beta|>\frac{1}{2}t^{3/4}\\|\alpha|/2<|\beta|\le 2|\alpha|}}\frac{L_0(\bar{\zeta}_\beta-1)}{\zeta(\alpha)-\zeta(\beta)}(D_t\zeta)_H(\beta)d\beta}_{L^2(E_t)}\\
            &\le\norm{D_t\partial_\alpha^{k+1}f\int_{\substack{\frac{1}{2}t^{3/4}<|\alpha-\beta|\le\frac{1}{2}t^{6/7}\\|\alpha|/2<|\beta|\le 2|\alpha|}}\frac{L_0(\bar{\zeta}_\beta-1)}{\zeta(\alpha)-\zeta(\beta)}(D_t\zeta)_H(\beta)d\beta}_{L^2(E_t)}\\
            &\ \ \ +\norm{D_t\partial_\alpha^{k+1}f\int_{\substack{|\alpha-\beta|>\frac{1}{2}t^{6/7}\\|\alpha|/2<|\beta|\le 2|\alpha|}}\frac{L_0(\bar{\zeta}_\beta-1)}{\zeta(\alpha)-\zeta(\beta)}(D_t\zeta)_H(\beta)d\beta}_{L^2(E_t)},\\
            &\le \norm{D_t\partial_\alpha^{k+1}f\int_{\substack{\frac{1}{2}t^{3/4}<|\alpha-\beta|\le\frac{1}{2}t^{6/7}\\|\alpha|/2<|\beta|\le 2|\alpha|}}\frac{L_0(\bar{\zeta}_\beta-1)}{\zeta(\alpha)-\zeta(\beta)}(D_t\zeta)_H(\beta)d\beta}_{L^2(E_t)}\\
            &\ \ \ +C\norm{D_t\partial_\alpha^{k+1}f}_{L^\infty}t^{-6/7}\norm{(D_t\zeta)_H}_{L^2}\norm{L_0(\bar{\zeta}_\alpha-1)}_{L^2},\\
            &\le\norm{D_t\partial_\alpha^{k+1}f\int_{\substack{\frac{1}{2}t^{3/4}<|\alpha-\beta|\le\frac{1}{2}t^{6/7}\\|\alpha|/2<|\beta|\le 2|\alpha|}}\frac{L_0(\bar{\zeta}_\beta-1)}{\zeta(\alpha)-\zeta(\beta)}(D_t\zeta)_H(\beta)d\beta}_{L^2(E_t)}+C\epsilon^3t^{-1-\frac{1}{42}+\delta_0}\ln t.
        \end{align*}
        Denote the integral by $\tilde{I}$ and we have
        \begin{align*}
            \norm{D_t\partial_\alpha^{k+1}f\tilde{I}}_{L^2(t^{3/4}<|\alpha|\le t^{6/7})}
            &\le \norm{D_t\partial_\alpha^{k+1}f}_{L^2}\norm{(D_t\zeta)_H}_{L^\infty(|\alpha|\le 2t^{6/7})}\int_{|\alpha-\beta|\ge\frac{1}{2}t^{3/4}}\left|\frac{L_0(\bar{\zeta}_\beta-1)}{\zeta(\alpha)-\zeta(\beta)}\right|d\beta,\\
            &\le C\epsilon^2t^{-\frac{1}{2}-\frac{1}{7}}t^{-\frac{3}{8}}\norm{L_0(\bar{\zeta}_\alpha-1)}_{L^2},\\
            &\le C\epsilon^3t^{-1-\frac{1}{42}+\delta_0}.
        \end{align*}
        Finally, since
        \[
        |\alpha|>t^{6/7},|\alpha-\beta|\le\frac{1}{2}t^{6/7}\Rightarrow |\alpha-\beta|<\frac{|\alpha|}{2}\Rightarrow \frac{|\alpha|}{2}\le|\beta|\le 2|\alpha|,
        \]
        we have
        \[
        1_{|\alpha|>t^{6/7}}\tilde{I}=\int_{\frac{1}{2}t^{3/4}<|\alpha-\beta|\le\frac{1}{2}t^{6/7}}\frac{L_0(\bar{\zeta}_\beta-1)}{\zeta(\alpha)-\zeta(\beta)}(D_t\zeta)_H(\beta)d\beta
        \]
        which allows us to apply T1 Theorem:
        \begin{align*}
            \norm{D_t\partial_\alpha^{k+1}f\tilde{I}}_{L^2(t^{6/7}<|\alpha|\le t)}
            &\le\sup_{|\alpha|/2\le|\beta|\le 2|\alpha|}|D_t\partial_\alpha^{k+1}f(D_t\zeta)_H(\beta)|\norm{L_0(\bar{\zeta}_\alpha-1)}_{L^2},\\
            &\le C\epsilon^2t^{-1}\norm{L_0(\bar{\zeta}_\alpha-1)}_{L^2}.
        \end{align*}
        And we conclude that
        \[
        \norm{D_t\partial_\alpha^{k+1}f\int_{\substack{|\alpha-\beta|>\frac{1}{2}t^{3/4}\\|\alpha|/2<|\beta|\le 2|\alpha|}}\frac{L_0(\bar{\zeta}_\beta-1)}{\zeta(\alpha)-\zeta(\beta)}(D_t\zeta)_H(\beta)d\beta}_{L^2(E_t)}
        \le C\epsilon^2t^{-1}\norm{L_0(\bar{\zeta}_\alpha-1)}_{L^2}+C\epsilon^3t^{-1-\delta}.
        \]
        \item 
        \begin{align*}
            II_1,II_2,II_3\le \norm{D_t\partial_\alpha^{k+1}f}_{L^\infty}\norm{L_0D_t\zeta}_{L^2}\norm{\bar{\zeta}_\alpha-1}_{W^{1,\infty}(|\alpha|\le t^{\frac{3}{4}})}\ln t
            \le C\epsilon^3t^{-1-\delta}.
        \end{align*}
    \end{enumerate}
    This completes the proof.
\end{proof}

\begin{thm}\label{energy estimate for vec}
    Set
    \[
    E_{L_0\theta}^k(t)=\int\frac{|D_tL_0\theta_k|^2}{A}d\alpha+i\int L_0\theta_k\partial_\alpha L_0\bar{\theta}_kd\alpha,
    \]
    \[
    E_{L_0\sigma}^k(t)=\int\frac{|D_tL_0\sigma_k|^2}{A}d\alpha+i\int L_0\sigma_k\partial_\alpha L_0\bar{\sigma}_kd\alpha
    \]
    and
    \[
    \mathcal{E}_k^{L_0}(t)=\sum_{j=0}^{k}(E_{L_0\theta}^j(t)+E_{L_0\sigma}^j(t))
    \]
    for $0\le k\le s-1$,
    then
    \begin{enumerate}
        \item
        \begin{align*}
            &\norm{L_0(\zeta_\alpha-1)}_{H^k}+\norm{L_0D_t\zeta}_{H^k}+\norm{L_0\Lambda(\zeta-\alpha)}_{H^k}+\norm{L_0\Lambda D_t\zeta}_{H^k}+\norm{L_0D_t^2\zeta}_{H^k}\\
         \le&\ C\sqrt{\mathcal{E}_k^{L_0}(t)}+C\epsilon^{\frac{3}{2}}t^{-\frac{1}{4}}\ln t,\ \forall\ 0\le k\le s-1.
        \end{align*}
        \item
        \begin{equation}\label{vec field estimate for k}
            \left|\frac{d}{dt}(E_{L_0\theta}^k(t)+E_{L_0\sigma}^k(t))\right|\le C\epsilon^2t^{-1}\left(\mathcal{E}_{k}^{L_0}+\epsilon\sqrt{\mathcal{E}_k^{L_0}}\right),\ \forall 0\le k\le s-2.
        \end{equation}
        \item 
        \begin{equation}\label{vec field estimate for s-1}
            \left|\frac{d}{dt}(E_{L_0\theta}^{s-1}+E_{L_0\sigma}^{s-1})(t)\right|\le C\epsilon^2t^{-1}\left(\mathcal{E}_{s-1}^{L_0}+\epsilon\sqrt{\mathcal{E}_{s-1}^{L_0}}\right)+C\epsilon^3t^{-1+3\delta_0}\ln t\sqrt{\mathcal{E}_{0}^{L_0}}.
        \end{equation}
    \end{enumerate}
\end{thm}

\begin{proof}
    (1) is verified by applying an argument similar to Lemma \ref{Hs and order s energy}.
    
    For (\ref{vec field estimate for k}) we only prove the case $k=0$, since others are similar to it (except the treatment of $L_0[D_t^2-iA\partial_\alpha,\partial_\alpha^k]\theta$, which will be shown in the proof for (\ref{vec field estimate for s-1})).
    
    We have
    \[
    (D_t^2-iA\partial_\alpha)\theta=G
    \]
    and thus
    \[
    (D_t^2-iA\partial_\alpha)L_0\theta=L_0G+[D_t^2-iA\partial_\alpha,L_0]\theta
    \]
    which leads to
    \begin{align*}
        \frac{d}{dt}E_{L_0\theta}^0(t)
        &=2Re\left\{\int\frac{D_tL_0\bar{\theta}}{A}L_0Gd\alpha\right\}-\int\frac{|D_tL_0\theta|^2}{A}\left(\frac{a_t}{a}\circ\kappa^{-1}\right)d\alpha\\
        &\ \ \ +2Re\left\{\int\frac{D_tL_0\bar{\theta}}{A}[D_t^2-iA\partial_\alpha,L_0]\theta d\alpha\right\}.
    \end{align*}
    According to Corollary \ref{k order estimate of [L0,P]}, we have
    \begin{align*}
        &\left|-\int\frac{|D_tL_0\theta|^2}{A}\left(\frac{a_t}{a}\circ\kappa^{-1}\right)d\alpha
        +2Re\left\{\int\frac{D_tL_0\bar{\theta}}{A}[D_t^2-iA\partial_\alpha,L_0]\theta d\alpha\right\}\right|\\
        \le&\ C\epsilon^2t^{-1-\delta}\norm{D_tL_0\theta}_{L^2}^2+2\norm{[D_t^2-iA\partial_\alpha,L_0]\theta}_{L^2}\norm{D_tL_0\bar{\theta}}_{L^2},\\
        \le&\ C\epsilon^2t^{-1}E_{L_0\theta}^0(t)+C\epsilon^2t^{-1}\mathcal{E}_0^{L_0}(t)+C\epsilon^3t^{-1}\sqrt{\mathcal{E}_0^{L_0}(t)},\\
        \le&\ C\epsilon^2t^{-1}\left(\mathcal{E}_0^{L_0}(t)+\epsilon\sqrt{\mathcal{E}_0^{L_0}(t)}\right).
    \end{align*}
    In addition, since each term of $G$ is cubic, one may exploit the technique of calculating $L_0b$ and the method used in Corollary \ref{k order estimate of [L0,P]} (together with Lemma \ref{s-1 order estimate for [L0,P]}, since we also need to control for large $|\alpha|$) to obtain $\norm{L_0G}_{L^2}\le C\epsilon^2t^{-1}\sqrt{\mathcal{E}_0^{L_0}(t)}+O(\epsilon^3t^{-1})$; now we get
    \begin{align*}
        \left|\frac{d}{dt}E_{L_0\theta}^0(t)\right|
        &\le C\norm{D_tL_0\theta}_{L^2}\norm{L_0G}_{L^2}+C\epsilon^2t^{-1}\left(\mathcal{E}_0^{L_0}(t)+\epsilon\sqrt{\mathcal{E}_0^{L_0}(t)}\right)),\\
        &\le C'\epsilon^2t^{-1}\left(\mathcal{E}_0^{L_0}(t)+\epsilon\sqrt{\mathcal{E}_0^{L_0}(t)}\right).
    \end{align*}
    The treatment of $\frac{d}{dt}E^0_{L_0\sigma}(t)$ is similar.
    
     For (\ref{vec field estimate for s-1}), simply use Proposition \ref{vector field est for s-1 order on R} to get
    \begin{align*}
        \norm{[D_t^2-iA\partial_\alpha,L_0]\theta_{s-1}}_{L^2}
        &\le C\epsilon^2t^{-1+\delta_0}\ln t\sqrt{\mathcal{E}_{0}^{L_0}(t)}+C\epsilon^3t^{-1}
    \end{align*}
    and use $\norm{D_tL_0\bar{\theta}_{s-1}}_{L^2}\le C\epsilon t^{2\delta_0}$ to derive
    \[
    \left|\int\frac{D_tL_0\bar{\theta}_{s-1}}{A}[D_t^2-iA\partial_\alpha,L_0]\theta_{s-1}d\alpha\right|
    \le C\epsilon^3t^{-1+3\delta_0}\ln t\sqrt{\mathcal{E}_0^{L_0}(t)}+C\epsilon^3t^{-1}\sqrt{\mathcal{E}_{s-1}^{L_0}(t)}.
    \]
    
    Also, we may apply its technique to derive the same bound for $\norm{L_0\partial_\alpha^{s-1}G}_{L^2}$. Hence we only need to consider
    \[
    \int\frac{D_tL_0\bar{\theta}_{s-1}}{A}L_0[D_t^2-iA\partial_\alpha,\partial_\alpha^{s-1}]\theta d\alpha.
    \]
    Recall that
    \[
    [A\partial_\alpha,\partial_\alpha^{s-1}]=-\sum_{j=1}^{s-1}C_{j,k}\partial_\alpha^jA\partial_\alpha^{s-j},
    \]
    where $C_{j,k}=k!/j!(k-j)!$, and thus
    \[
    L_0[A\partial_\alpha,\partial_\alpha^{s-1}]\theta=-\sum_{j=1}^{s-1}C_{j,k}\left((L_0\partial_\alpha^jA)\partial_\alpha^{s-j}\theta+\partial_\alpha^jA(L_0\partial_\alpha^{s-j}\theta)\right).
    \]
    Combining with $\norm{L_0A}_{H^s}\le C\epsilon^2t^{-\frac{1}{2}}\sqrt{\mathcal{E}_{s-1}^{L_0}(t)}$ (also, this requires some techniques to avoid the overflow, like we did before),$\norm{\theta_\alpha}_{W^{s-2,\infty}}\le C\epsilon t^{-\frac{1}{2}}$ and $\norm{A-1}_{W^{s-1,\infty}}\le\norm{A-1}_{H^s}\le C\epsilon^2t^{-1+\delta_0}\ln t$ (\textbf{Note that for $0\le k\le s-2$, there holds $\norm{A-1}_{W^{k,\infty}}\le C\epsilon^2t^{-1-\delta}$ for certain $\delta>0$, which makes the estimate for $k$ order vector field energy different from $s-1$ order}) yields
    \[
    \left|\int\frac{D_tL_0\bar{\theta}_k}{A}L_0[-iA\partial_\alpha,\partial_\alpha^{s-1}]\theta d\alpha\right|\le C\epsilon^2t^{-1}\mathcal{E}_{s-1}^{L_0}(t)+C\epsilon^3t^{-1+3\delta_0}\ln t\sqrt{\mathcal{E}^{L_0}_{0}}.
    \]
    The last term comes from the estimate for the extreme term:
    \begin{align*}
        \norm{\partial_\alpha^{s-1}A(L_0\partial_\alpha\theta)}_{L^2}\norm{D_tL_0\bar{\theta}_{s-1}}_{L^2}
        &\le\norm{A-1}_{H^s}\norm{D_tL_0\bar{\theta}_{s-1}}_{L^2}\sqrt{\mathcal{E}_0^{L_0}}\\
        &\le C\epsilon^3t^{-1+3\delta_0}\ln t\sqrt{\mathcal{E}_0^{L_0}}.
    \end{align*}
    A similar estimate works for the $[D_t^2,\partial_\alpha^{s-1}]$ part (recall that $\norm{b_\alpha}_{W^{s-2,\infty}}\le Ct^{-1-\delta}$), and this completes the proof.
\end{proof}

\begin{cor}\label{sec 6:L0bound}
    There holds
    \begin{equation}\label{vec energy for k}
        \mathcal{E}_k^{L_0}(t)\le\left[\left(\sqrt{\mathcal{E}_k^{L_0}(0)}+\epsilon\right)(1+t)^{\frac{C\epsilon^2}{2}}-\epsilon\right]^2\ \forall 0\le k\le s-2
    \end{equation}
    and
    \begin{equation}\label{vec energy for s-1}
        \mathcal{E}_{s-1}^{L_0}(t)\le (1+t)^{C\epsilon^2}\mathcal{E}_{s-1}^{L_0}(0)+\frac{C\epsilon^4}{3\delta_0-\frac{C\epsilon^2}{2}}(1+t)^{3\delta_0+\frac{C\epsilon^2}{2}}\ln(1+t)
    \end{equation}
\end{cor}

\begin{proof}
    Note that for $t\in[0,1]$, there holds
    \[
    \left|\frac{d}{dt}\mathcal{E}_k^{L_0}\right|\le C\epsilon^2\mathcal{E}_k^{L_0},\ \forall k\le s-2
    \]
    since each relevant term is bounded. So we rewrite (\ref{vec field estimate for k}) as
    \[
    \left|\frac{d}{dt}\mathcal{E}_k^{L_0}(t)\right|\le C\epsilon^2(1+t)^{-1}\left(\mathcal{E}_k^{L_0}+\epsilon\sqrt{\mathcal{E}_k^{L_0}}\right)
    \]
    to obtain
    \[
    \frac{d\mathcal{E}_k^{L_0}}{\mathcal{E}_k^{L_0}+\epsilon\sqrt{\mathcal{E}_k^{L_0}}}\le C\epsilon^2(1+t)^{-1}dt
    \]
    and we integrate it over $[0,t]$ to get
    \[
    2\ln\frac{\sqrt{\mathcal{E}_k^{L_0}(t)}+\epsilon}{\sqrt{\mathcal{E}_k^{L_0}(0)}+\epsilon}\le C\epsilon^2\ln(1+t)
    \]
    which simplifies to (\ref{vec energy for k}). Plugging it into (\ref{vec field estimate for s-1}) (also, we replace $t$ with $1+t$) yields
    \begin{align*}
        \frac{d}{dt}\mathcal{E}_{s-1}^{L_0}(t)
        &\le C\epsilon^2(1+t)^{-1}\left(\mathcal{E}_{s-1}^{L_0}+\epsilon\sqrt{\mathcal{E}_{s-1}^{L_0}}\right)\\
        &\ \ \ +C\epsilon^3(1+t)^{-1+3\delta_0}\left[\left(\sqrt{\mathcal{E}_k^{L_0}(0)}+\epsilon\right)(1+t)^{\frac{C\epsilon^2}{2}}-\epsilon\right]\ln(1+t)\\
        &\le C\epsilon^2(1+t)^{-1}\left(\mathcal{E}_{s-1}^{L_0}+\epsilon\sqrt{\mathcal{E}_{s-1}^{L_0}}\right)+C\epsilon^4(1+t)^{-1+3\delta_0+\frac{C\epsilon^2}{2}}\ln(1+t).
    \end{align*}
    According to the bootstrap assumption, $\sqrt{E_{s-1}^{L_0}(t)}\le C\epsilon t^{2\delta_0}$ and the inequality simplifies to
    \[
    \frac{d}{dt}\mathcal{E}_{s-1}^{L_0}(t)\le C\epsilon^2(1+t)^{-1}\mathcal{E}_{s-1}^{L_0}(t)+C\epsilon^4(1+t)^{-1+3\delta_0+\frac{C\epsilon^2}{2}}\ln(1+t)
    \]
    which is solvable. Indeed, set $s=-1+3\delta_0-\frac{C\epsilon^2}{2}$ and one gets
    \begin{align*}
        \mathcal{E}_{s-1}^{L_0}(t)
        &\le(1+t)^{C\epsilon^2}\mathcal{E}_{s-1}^{L_0}(0)+C\epsilon^4(1+t)^{C\epsilon^2}\int_0^t(1+t)^s\ln(1+t)dt\\
        &=(1+t)^{C\epsilon^2}\mathcal{E}_{s-1}^{L_0}(0)+C\epsilon^4(1+t)^{C\epsilon^2}\frac{(1+t)^{s+1}\ln(1+t)}{s+1}\\
        &\ \ \ -C\epsilon^4(1+t)^{C\epsilon^2}\left(\frac{(1+t)^{s+1}}{(s+1)^2}+\frac{1}{(s+1)^2}\right)\\
        &\le(1+t)^{C\epsilon^2}\mathcal{E}_{s-1}^{L_0}(0)+\frac{C\epsilon^4}{3\delta_0-\frac{C\epsilon^2}{2}}(1+t)^{3\delta_0+\frac{C\epsilon^2}{2}}\ln(1+t).
    \end{align*}
    This completes the proof.
\end{proof}

\begin{cor}\label{cor:Omega_0}
    There holds
    \begin{equation}
        \norm{\Omega_0\partial_\alpha\theta}_{H^{k}}+\norm{\Lambda\Omega_0 Q_\alpha}_{H^k}\le C\epsilon (1+t)^{\frac{C\epsilon^2}{2}},\ \forall 0\le k\le s-2
    \end{equation}
    and
    \begin{equation}
        \norm{\Omega_0\partial_\alpha\theta}_{H^{s-1}}+\norm{\Lambda\Omega_0 Q_\alpha}_{H^{s-1}}\le C\epsilon^2\sqrt{\frac{1}{3\delta_0-\frac{C\epsilon^2}{2}}}(1+t)^{\frac{3}{2}\delta_0+\frac{C\epsilon^2}{4}}\ln^{\frac{1}{2}}(1+t).
    \end{equation}
\end{cor}

\begin{proof}
    Note that
    \begin{equation}\label{Omega-to-Dt}
        \Omega_0\partial_\alpha=L_0\partial_t-\frac{t}{2}(\partial_t^2-i\partial_\alpha)
    \end{equation}
    which implies that for any $0\le k\le s-1$,
    \[
    \norm{\Omega_0\partial_\alpha\theta}_{H^{s-1}}\le C\norm{L_0\partial_t\theta}_{H^{s-1}}+\frac{t}{2}\norm{(\partial_t^2-i\partial_\alpha)\theta}_{H^{s-1}}.
    \]
    For the second term, recall that $(D_t^2-iA\partial_\alpha)\theta=G$ is cubic, so
    \begin{align*}
        \norm{(\partial_t^2-i\partial_\alpha)\theta}_{H^{k}}
        \le\norm{G}_{H^{k}}+C\epsilon^3t^{-1}
        \le C\epsilon^3t^{-1}.
    \end{align*}
    
    For the first term, it's enough for us to note that $L_0\partial_t=L_0(D_t-b\partial_\alpha)$. Hence we have
    \[
    \norm{\Omega_0\partial_\alpha\theta}_{H^{k}}\le C\sqrt{\mathcal{E}_{k}^{L_0}}(t)+C\epsilon^3
    \le C\sqrt{\mathcal{E}_k^{L_0}(t)}
    \]
    i.e.
    \[
    \norm{\Omega_0\partial_\alpha\theta}_{H^k}
    \le
    \begin{cases}
        C\epsilon(1+t)^{\frac{C\epsilon^2}{2}},\ 0\le k\le s-2\\
        \displaystyle C\epsilon^2\sqrt{\frac{1}{3\delta_0-\frac{C\epsilon^2}{2}}}(1+t)^{\frac{3}{2}\delta_0+\frac{C\epsilon^2}{4}}\ln^{\frac{1}{2}}(1+t),\ k=s-1.
    \end{cases}
    \]
    Also, we may use formula (\ref{Omega-to-Dt}) to derive
    \[
    \norm{\Lambda\Omega_0 Q_\alpha}_{H^{k}}
    \le\norm{\Lambda L_0\partial_t Q}_{H^{k}}+\frac{t}{2}\norm{\Lambda(\partial_t^2-i\partial_\alpha)Q}_{H^{k}}.
    \]
    For the first term, we write
    \[
    \partial_t Q=D_tQ-b\partial_\alpha Q
    \]
    and we shall only verify that $\norm{\Lambda L_0 D_tQ}_{L^2}+\frac{t}{2}\norm{\Lambda(\partial_t^2-i\partial_\alpha)Q}_{L^2}\le C\sqrt{\mathcal{E}_0^{L_0}}$. Other terms or cases are similar to it. By utilizing (\ref{formula of Dt Q}), we have
    \begin{equation}\label{L0Dt Q}
        L_0D_tQ=-L_0[D_t\zeta,\mathcal{H}]\frac{Re\left\{\bar{\zeta}_\alpha D_t\zeta\right\}}{\zeta_\alpha}+\frac{1}{2}L_0(I-\mathcal{H})|D_t\zeta|^2-\frac{1}{2i}L_0\theta.
    \end{equation}
    For the first term, recall that
    \[
    L_0[f,\mathcal{H}]\frac{g}{\zeta_\alpha}=[f,\mathcal{H}]\frac{L_0g}{\zeta_\alpha}+[L_0f,\mathcal{H}]\frac{g}{\zeta_\alpha}-\frac{1}{\pi i}\int\frac{(L_0\zeta(\alpha)-L_0\zeta(\beta))(f(\alpha)-f(\beta))}{(\zeta(\alpha)-\zeta(\beta))^2}gd\beta+[f,\mathcal{H}]\frac{g}{\zeta_\alpha}
    \]
    and we have
    \begin{align*}
        \norm{L_0[D_t\zeta,\mathcal{H}]\frac{Re\left\{\bar{\zeta}_\alpha D_t\zeta\right\}}{\zeta_\alpha}}_{L^2}
        &\le C\norm{D_t\zeta}_{L^\infty}\norm{L_0Re\{\bar{\zeta}_\alpha D_t\zeta\}}_{L^2}+C\norm{L_0D_t\zeta}_{L^2}\norm{\mathcal{H}(\bar{\zeta}_\alpha D_t\zeta)}_{L^\infty}\\
        &\ \ \ +C\norm{D_t\zeta}_{L^\infty}\norm{L_0(\zeta_\alpha-1)}_{L^2}\norm{\bar{\zeta}_\alpha D_t\zeta}_{L^\infty}+\norm{D_t\zeta}_{L^\infty}\norm{\bar{\zeta}_\alpha D_t\zeta}_{L^2},\\
        &\le C\epsilon\sqrt{\mathcal{E}_{0}^{L_0}}+C\epsilon^2.
    \end{align*}
    Since
    \[
    (I-\mathcal{H})|D_t\zeta|^2
    =(I-\mathcal{H})(D_t\bar{\zeta}D_t\zeta)
    =[\mathcal{H},D_t\zeta]D_t\bar{\zeta},
    \]
    similar estimate works for the second term in (\ref{L0Dt Q}) and their derivatives, hence we are allowed to use the embedding $\norm{\Lambda L_0f}_{L^2}\le\sqrt{\norm{L_0f}_{L^2}\norm{\partial_\alpha L_0f}_{L^2}}$ while estimating the first and the second term of (\ref{L0Dt Q})\footnote{Such method doesn't cause the overflow of the derivative. First, if all derivatives apply to a term that doesn't involve vector field $L_0$, then the worst case is that there are $s+1/2$ derivatives on $D_t\zeta$ (after embedding), which is under control; second, if there are $s-1$ derivatives apply to $L_0D_t\zeta$, then we shall use the bootstrap assumption $\norm{\Lambda L_0D_t\zeta}_{H^{s-1}}\le C\epsilon t^{2\delta_0}$ instead of embedding. Note that each term here is at least quadratic, so the other factor will provide sufficient decay that recovers the $t^{2\delta_0}$ increase.}. For the last term, note that
    \[
    \norm{\Lambda L_0\theta}_{L^2}
    \le C\sqrt{\mathcal{E}_{0}^{L_0}}+C\epsilon.
    \]
    And we conclude that
    \[
    \norm{\Lambda L_0D_tQ}_{L^2}\le C\sqrt{\mathcal{E}_{0}^{L_0}}+C\epsilon\le C\sqrt{\mathcal{E}_0^{L_0}}.
    \]

Claim that
\begin{equation}\label{control of equation of Q}
   \frac{t}{2}\norm{\Lambda(D_t^2-iA\partial_\alpha)Q}_{L^2}\le C\epsilon^3\ln(1+t),
\end{equation}
which completes the proof. We have, according to \cite{Wu2009},
\begin{equation}
    \begin{split}
        (D_t^2-iA\partial_{\alpha})Q=& -\left[D_t\zeta, \mathcal{H}\frac{1}{\zeta_{\alpha}}+\bar{\mathcal{H}}\frac{1}{\bar{\zeta}_{\alpha}}\right](\bar{\zeta}_{\alpha}D_t^2\zeta)+[D_t\zeta, \bar{\mathcal{H}}]\Big(D_t\bar{\zeta} \frac{\partial_{\alpha}D_t\zeta}{\bar{\zeta}_{\alpha}}\Big)+D_t\zeta[D_t\zeta, \mathcal{H}]\frac{\partial_{\alpha}D_t\bar{\zeta}}{\zeta_{\alpha}}\\
        & -2[D_t\zeta, \mathcal{H}]\frac{D_t\zeta\cdot \partial_{\alpha}D_t\zeta}{\zeta_{\alpha}}+\frac{1}{\pi i}\int\Big(\frac{D_t\zeta(\alpha,t)-D_t\zeta(\beta,t)}{\zeta(\alpha,t)-\zeta(\beta,t)}\Big)^2 D_t\zeta \cdot \partial_{\beta}\zeta d\beta\\
        :=& N_1+N_2+N_3+N_4+N_5.
    \end{split}
\end{equation}
The symbol '$\cdot$' here means the inner product of 2d vector. i.e. for $z,w\in\mathbb{C}$, $z\cdot w=Re\{\bar{z}w\}=Re\{z\bar{w}\}$. Since $\norm{\Lambda f}_{L^2}\le\sqrt{\norm{f}_{L^2}\norm{\partial_\alpha f}_{L^2}}$, we only need to control $\norm{N_j}_{H^1}$. First, note that
\[
N_1=\frac{2}{\pi}\int\frac{(D_t\zeta(\alpha)-D_t\zeta(\beta))(\eta(\alpha)-\eta(\beta))}{|\zeta(\alpha)-\zeta(\beta)|^2}(\bar{\zeta}_\beta D_t^2\zeta)d\beta
\]
and we apply Lemma \ref{useful lemma} to derive $\norm{N_1}_{L^2}\le C\epsilon^3t^{-1}\ln t$. A similar estimate applies to $N_2$ and $N_4$, except that one should use Lemma \ref{lemma:transit2} to estimate $\norm{D_t\zeta\partial_\alpha D_t\zeta}_{L^\infty}$\footnote{Such lemma faces some difficulties when it comes to high order derivative. Namely, when we consider the $s-1$ order and use the embedding, we must deal with $[D_t\zeta,\bar{\mathcal{H}}]\partial_\alpha^s(D_t\bar{\zeta}\partial_\alpha D_t\zeta)$. For such case, we integrate by parts to move 1 derivative to the core and divide into several cases: $|\alpha|\le t$, $|\alpha|>t$ with $|\alpha-\beta|\le t/2$ and $|\alpha|>t$ with $|\alpha-\beta|>t/2$. Using the techniques in energy estimate and vector field estimate, one can handle each cases.}. For $N_3$, note that
\[
N_3=D_t\zeta(I-\mathcal{H})\left(D_t\zeta\frac{\partial_\alpha D_t\bar{\zeta}}{\zeta_\alpha}\right)
\]
which implies
\[
\norm{N_3}_{L^2}
\le\norm{D_t\zeta}_{L^2}\norm{(I-\mathcal{H})\left(D_t\zeta\frac{\partial_\alpha D_t\bar{\zeta}}{\zeta_\alpha}\right)}_{L^\infty}
\le C\epsilon^3t^{-1}\ln t.
\]
Here we have used Lemma \ref{L infty of hilbert}\footnote{It's obvious that such argument doesn't work for high order derivative case. In such case, one should keep the commutator expression of $N_3$ and use the method shown before.}. For $N_5$, write
\begin{align*}
    N_5
    &=\frac{1}{\pi i}\int\left(\frac{D_t\zeta(\alpha)-D_t\zeta(\beta)}{\zeta(\alpha)-\zeta(\beta)}\right)^2Re\{\bar{\zeta}_\beta D_t\zeta\}d\beta,\\
    &=\frac{1}{\pi i}\int\left(\frac{D_t\zeta(\alpha)-D_t\zeta(\beta)}{\zeta(\alpha)-\zeta(\beta)}\right)^2Re\{D_t\zeta\} d\beta+O(\epsilon^4t^{-\frac{5}{4}})_{L^2},\\
    &=\frac{1}{2\pi i}\int\left(\frac{D_t\zeta(\alpha)-D_t\zeta(\beta)}{\zeta(\alpha)-\zeta(\beta)}\right)^2D_t\zeta d\beta+\frac{1}{2\pi i}\int\left(\frac{D_t\zeta(\alpha)-D_t\zeta(\beta)}{\zeta(\alpha)-\zeta(\beta)}\right)^2D_t\bar{\zeta} d\beta\\
    &\ \ \ +O(\epsilon^4t^{-\frac{5}{4}})_{L^2},\\
    &=N_{51}+N_{52}+O(\epsilon^4t^{-\frac{5}{4}})_{L^2}.
\end{align*}
$N_{51}$ is actually quartic (refer to our treatment for $G_{21}$, which appeared in Chapter 5) and we get $\norm{N_{51}}_{L^2}\le C\epsilon^4t^{-\frac{5}{4}}$; for $N_{52}$, write
\[
N_{52}=\frac{1}{2\pi i}\int\left(\frac{D_t\zeta(\alpha)-D_t\zeta(\beta)}{\zeta(\alpha)-\zeta(\beta)}\right)^2\zeta_\beta D_t\bar{\zeta} d\beta+O(\epsilon^4t^{-\frac{5}{4}})_{L^2}
\]
and we integrate by parts to derive
\begin{align*}
    &\frac{1}{2\pi i}\int\left(\frac{D_t\zeta(\alpha)-D_t\zeta(\beta)}{\zeta(\alpha)-\zeta(\beta)}\right)^2\zeta_\beta D_t\bar{\zeta}d\beta\\
    =&\frac{1}{2\pi i}\int(D_t\zeta(\alpha)-D_t\zeta(\beta))^2D_t\bar{\zeta}d\frac{1}{\zeta(\alpha)-\zeta(\beta)},\\
    =&\frac{1}{\pi i}\int\frac{D_t\zeta(\alpha)-D_t\zeta(\beta)}{\zeta(\alpha)-\zeta(\beta)}D_t\bar{\zeta}\partial_\beta D_t\zeta d\beta
    -\frac{1}{2\pi i}\int\frac{(D_t\zeta(\alpha)-D_t\zeta(\beta))^2}{\zeta(\alpha)-\zeta(\beta)}\partial_\beta D_t\bar{\zeta}d\beta,\\
    =&[D_t\zeta,\mathcal{H}]\frac{D_t\bar{\zeta}\partial_\alpha D_t\zeta}{\zeta_\alpha}-\frac{1}{2}(D_t\zeta(\alpha))^2\frac{\partial_\alpha D_t\bar{\zeta}}{\zeta_\alpha}+D_t\zeta(\alpha)\mathcal{H}\left(D_t\zeta\frac{\partial_\alpha D_t\bar{\zeta}}{\zeta_\alpha}\right)\\
    &-\frac{1}{2}\mathcal{H}\left((D_t\zeta)^2\frac{\partial_\alpha D_t\bar{\zeta}}{\zeta_\alpha}\right),\\
    =&[D_t\zeta,\mathcal{H}]\frac{D_t\bar{\zeta}\partial_\alpha D_t\zeta}{\zeta_\alpha}-\frac{1}{2}D_t\zeta(\alpha)(I-\mathcal{H})\left(D_t\zeta\frac{\partial_\alpha D_t\bar{\zeta}}{\zeta_\alpha}\right)+\frac{1}{2}[D_t\zeta,\mathcal{H}]\frac{D_t\zeta\partial_\alpha D_t\bar{\zeta}}{\zeta_\alpha}\\
    =&[D_t\zeta,\mathcal{H}]\frac{Re\{D_t\zeta\partial_\alpha D_t\bar{\zeta}\}}{\zeta_\alpha}+\frac{1}{2}[D_t\zeta,\mathcal{H}]\frac{D_t\bar{\zeta}\partial_\alpha D_t\zeta}{\zeta_\alpha}-\frac{1}{2}N_3\\
    =&-\frac{1}{2}N_4-\frac{1}{2}N_3+\frac{1}{2}[D_t\zeta,\mathcal{H}]\frac{D_t\bar{\zeta}\partial_\alpha D_t\zeta}{\zeta_\alpha}
\end{align*}
Note that the commutator can be handled by the method shown above, so we conclude that $\norm{N_5}_{L^2}\le C\epsilon^3t^{-1}\ln t$. This completes the proof.
\end{proof}

\section{Decay estimates}
Due to Lemma \ref{bouns decay for small or big alpha}, we only need to make the decay estimate for $t^{4/5}\le|\alpha|\le t^{6/5}$.

Note that
\[
\begin{aligned}
	(D_t^2-iA\partial_\alpha)D_t\theta
    &=[D_t^2-iA\partial_\alpha,D_t]\theta+D_tG,\\
    &=-i\left(\frac{a_t}{a}\circ \kappa^{-1}\right)A\partial_\alpha\theta+D_tG
\end{aligned}
\]
and we have
\[
(D_t^2-i\partial_\alpha)D_t\theta
=D_tG-i\left(\frac{a_t}{a}\circ\kappa^{-1}\right)A\partial_\alpha\theta+i(A-1)\partial_\alpha D_t\theta
\]
Define $w=D_t^2\theta$ to get
\[
D_tw-i\partial_\alpha D_t\theta=D_tG+R'.
\]
Let $\chi(x)\in C_c^\infty(\mathbb{R})$ be a real smooth bump function such that $\int \chi(\alpha)d\alpha=1$. Fix $\nu>0$, then for $(\alpha,t)\in \mathbb{R}\times(0,\infty)$, define
\begin{equation}
	\tilde{E}(t)=\int w\bar{u}-D_t\theta \partial_t\bar{u}
\end{equation}
where $u(\alpha,t)=\chi\left(\frac{\alpha-\nu t}{\sqrt{t}}\right)e^{-\frac{it^2}{4\alpha}}$. Note that
\[
\tilde{E}(t)=2\int w\bar{u}-\frac{d}{dt}\int D_t\theta \bar{u}-\int b\partial_\alpha D_t\theta\bar{u}.
\]
\begin{rem}
    Similar construction of $\tilde{E}(t)$ was also employed in \cite{ai2022two}. \cite{ai2022two} takes $u(\alpha,t)=\nu^{3/2}\chi(\frac{\alpha-\nu t}{\nu^{3/2}\sqrt{t}})$. We note that, as we have formulated the water wave equations in a simpler manner, our analysis is considerably more straightforward.
\end{rem}
\begin{thm}\label{main theorem of decay estimate}
For $\nu\in [t^{-1/5}, t^{1/5}]$, 
	\[
	\tilde{E}(t)=C\sqrt{t}W(\nu t,t)+O(\epsilon t^{-\frac{1}{20}+\delta_0})
	\]
	where $W(\alpha,t)=e^{\frac{it^2}{4\alpha}}w(\alpha,t)$ and $C$ is independent of $t$.
\end{thm}

\begin{proof}
        For the third term of $\tilde{E}(t)$, it's obvious that
        \[
        \left|\int b\partial_\alpha D_t\theta\bar{u}\right|\le\norm{b}_{L^\infty}\norm{\partial_\alpha D_t\theta}_{L^2}\norm{u}_{L^2}
        \le C\epsilon^3t^{-\frac{1}{4}}.
        \]
	For the second term, define $U=e^{\frac{it^2}{4\alpha}}D_t\theta$. We have
	\[
	\begin{aligned}
		\int\partial_t(D_t\theta\bar{u})
		=\int\partial_t(\chi U)
		&=\int U\partial_t\chi+\int\chi\partial_tU,\\
		&=-\frac{1}{2\sqrt{t}}\int\left(\nu+\frac{\alpha}{t}\right)\chi'\left(\frac{\alpha-\nu t}{\sqrt{t}}\right)U
            +\int\chi\partial_tU.
	\end{aligned}
	\]
	and
	\[
	\left(\int\chi\left(\frac{\alpha-\nu t}{\sqrt{t}}\right)^2d\alpha\right)^{\frac{1}{2}}=t^{\frac{1}{4}};
	\]
	similar calculation works for $\chi'$, so we get
	\[
	\frac{d}{dt}\int(D_t\theta\bar{u})=\int\chi\partial_tU+O(\epsilon t^{-\frac{1}{20}}).
	\]
 We invoke the frequency decomposition technique in Proposition \ref{norms of A-1} here. Write (we use Proposition \ref{good-replace} here, by noticing $D_t\theta=2D_t\zeta+e=2Q_\alpha+2R+e$).
 \[
 D_t\theta=2P_{>t^{-2}}Q_\alpha+2P_{\le t^{-2}}Q_\alpha+e+2R
 \]
 and we get
 \[
 \norm{P_{\le t^{-2}}Q_\alpha}_{L^\infty}\le Ct^{-1}\norm{P_{\le t^{-2}}Q_\alpha}_{L^2}\le C\epsilon t^{-1},\ 
 \norm{\partial_tP_{\le t^{-2}}Q_\alpha}_{L^\infty}\le C\epsilon t^{-1}.
 \]
 Hence (recall that $\norm{e}_{L^2}+\norm{R}_{L^2}\le C\epsilon^2 t^{-\frac{1}{2}}$)
 \[
 \int\chi\partial_tU=\int \frac{it}{2\alpha}e^{\frac{it^2}{4\alpha}}\bar{u}Q_\alpha+\int e^{\frac{it^2}{4\alpha}}\bar{u}\partial_tQ_\alpha
 =2\int\chi\partial_tU_H+O(\epsilon^2t^{-\frac{1}{4}})+O(\epsilon t^{-\frac{3}{4}})
 \]
 where $U_H=e^{\frac{it^2}{4\alpha}}P_{>t^{-2}}Q_\alpha$. Now we have
 \[
 \norm{\partial_tU_H}_{L^\infty}=\norm{\frac{1}{\alpha}\Omega_0(P_{>t^{-2}}Q_\alpha)}_{L^\infty}\le C\epsilon t^{-\frac{4}{5}+\delta_0}\ln t
 \]
 which implies
 \[
 \left|\int\chi\partial_tU\right|\le C\epsilon t^{\frac{1}{2}-\frac{4}{5}+\delta_0}\ln t\le C\epsilon t^{-\frac{1}{8}}.
 \]
Now it remains to consider the first term of $\tilde{E}(t)$. Note that
\begin{align*}
    \int w\bar{u}
    =\int D_t^2\theta\bar{u}
    &=\int i\partial_\alpha\theta\bar{u}+\int(i(A-1)\partial_\alpha\theta+G)\bar{u},\\
    &=\int i\left(\frac{4\alpha^2}{it^2}\partial_\alpha^2\theta-\frac{4\alpha}{it^2}L_0\partial_\alpha\theta+\frac{2}{it}\Omega_0\partial_\alpha\theta\right)\bar{u}+O(\epsilon^3t^{-\frac{1}{2}}),\\
    &=\int i\frac{4\alpha^2}{it^2}\partial_\alpha^2\theta\bar{u}+O(\epsilon t^{-\frac{1}{2}}),\\
    &=\int\frac{4\alpha^2}{it^2}\Theta\chi+O(\epsilon t^{-\frac{1}{2}}),\\
    &=\frac{4(\nu t)^2}{it^2}\Theta(\nu t,t)\int\chi+\int\left(\frac{4\alpha^2}{it^2}\Theta(\alpha,t)-\frac{4(\nu t)^2}{it^2}\Theta(\nu t,t)\right)\chi+O(\epsilon t^{-\frac{1}{2}}),\\
    &=\sqrt{t}\frac{4(\nu t)^2}{it^2}\Theta(\nu t,t)+\int\left(\frac{4\alpha^2}{it^2}\Theta(\alpha,t)-\frac{4(\nu t)^2}{it^2}\Theta(\nu t,t)\right)\chi+O(\epsilon t^{-\frac{1}{2}})
\end{align*}
where $\Theta=ie^{\frac{it^2}{4\alpha}}\partial_\alpha^2\theta$. Note that according to Lemma \ref{transition-lemma},
\[
\frac{4(\nu t)^2}{it^2}\Theta(\nu t,t)=ie^{\frac{it^2}{4(\nu t)}}\partial_\alpha\theta(\nu t,t)+O(\epsilon t^{-\frac{4}{5}+\delta_0})
=W+O(\epsilon t^{-\frac{1}{2}}).
\]
	Therefore, it suffices to notice (use Hardy's inequality)
	\[
	\begin{aligned}
		\left|\int\left(\frac{4\alpha^2}{it^2}\Theta(\alpha,t)-\frac{4(\nu t)^2}{it^2}\Theta(\nu t,t)\right)\chi\right|
		&=\left|\int\left(\frac{4\alpha^2}{it^2}\Theta(\alpha,t)-\frac{4(\nu t)^2}{it^2}\Theta(\nu t,t)\right)\frac{\alpha-\nu t}{\alpha-\nu t}\chi d\alpha\right|,\\
		&\le\left(\int\left|\frac{\frac{4\alpha^2}{it^2}\Theta(\alpha,t)-\frac{4(\nu t)^2}{it^2}\Theta(\nu t,t)}{\alpha-\nu t}\right|^2d\alpha\right)^{\frac{1}{2}}\left(\int|(\alpha-\nu t)\chi|^2 d\alpha\right)^{\frac{1}{2}},\\
		&\le C\norm{\partial_\alpha\left(\frac{4\alpha^2}{it^2}\Theta\right)}_{L^2(t^{4/5}\lesssim|\alpha|\lesssim t^{6/5})}t^{\frac{3}{4}},\\
		&\le C\norm{\frac{1}{t^2}e^{\frac{it^2}{4\alpha}}\left(\alpha L_0+\frac{t}{2}\Omega_0\right)\partial_\alpha^2\theta}_{L^2(t^{4/5}\lesssim|\alpha|\lesssim t^{6/5})}t^{\frac{3}{4}}\\
            &\ \ \ +C\norm{\frac{\alpha}{t^2}\Theta}_{L^2(t^{4/5}\lesssim|\alpha|\lesssim t^{6/5})}t^{\frac{3}{4}},\\
		&\le C\epsilon t^{-\frac{1}{4}+\mu+\delta_0},\\
            &\le C\epsilon t^{-\frac{1}{20}+\delta_0}.
	\end{aligned}
	\]
    Here we have used the fact that
    \[
    \norm{(\alpha-\nu t)\chi\left(\frac{\alpha-\nu t}{\sqrt{t}}\right)}_{L^2}=t^{\frac{3}{4}}\norm{\alpha\chi(\alpha)}_{L^2}
    =Ct^{\frac{3}{4}}.
    \]
	This completes the proof.
\end{proof}

\begin{rem}
    One may infer from the proof that
    \[
    \tilde{E}(t)=\int w\bar{u}+O(\epsilon t^{-\frac{1}{20}+\delta_0}).
    \]
\end{rem}

According to Theorem \ref{main theorem of decay estimate}, if $\tilde{E}(t)\le C$ for some constant $C$ which is independent of $\nu$, then we may take $\nu\in [t^{-1/5},t^{1/5}]$ to derive $W(\alpha,t)\le Ct^{-\frac{1}{2}}$ for any $t^{4/5}\le|\alpha|\le t^{6/5}$.

We now try to control $\|\tilde{E}\|_{L^\infty}$. Straight computation yields
\[
\begin{aligned}
	\frac{d}{dt}\tilde{E}(t)
	&=\int \bar{u}\partial_tw+w\partial_t\bar{u}-D_t\theta\partial_t^2\bar{u}-\partial_tD_t\theta\partial_t\bar{u},\\
	&=\int\bar{u}(D_tw-b\partial_\alpha w)+b\partial_\alpha D_t\theta\partial_t\bar{u}-D_t\theta\partial_t^2\bar{u},\\
    &=\int(i\partial_\alpha D_t\theta+D_tG+R')\bar{u}-b\partial_\alpha D_t^2\theta\bar{u}+b\partial_\alpha D_t\theta\partial_t\bar{u}-D_t\theta\partial_t^2\bar{u},\\
    &=\int(D_tG+R')\bar{u}-\int(\partial_t^2+i\partial_\alpha)\bar{u}D_t\theta+\int b(\partial_\alpha D_t\theta\partial_t\bar{u}-\partial_\alpha D_t^2\theta\bar{u}),\\
    &\overset{\text{def}}{=}I_1+I_2+I_3.
\end{aligned}
\]
And
\[
\begin{aligned}
	(-\partial_t^2+i\partial_\alpha)u
	&=-\partial_t\left(e^{-\frac{it^2}{4\alpha}}\partial_t\chi-\frac{it}{2\alpha}e^{-\frac{it^2}{4\alpha}}\chi\right)+i\left(e^{-\frac{it^2}{4\alpha}}\partial_\alpha\chi+\frac{it^2}{4\alpha^2}e^{-\frac{it^2}{4\alpha}}\chi\right),\\
	&=e^{-\frac{it^2}{4\alpha}}\left(\frac{it}{2\alpha}\partial_t\chi-\partial_t^2\chi+\frac{i}{2\alpha}\chi+\frac{t^2}{4\alpha^2}\chi+\frac{it}{2\alpha}\partial_t\chi+i\partial_\alpha\chi-\frac{t^2}{4\alpha^2}\chi\right),\\
	&=e^{-\frac{it^2}{4\alpha}}\left(-\partial_t^2\chi+\frac{it}{\alpha}\partial_t\chi+\frac{i}{2\alpha}\chi+i\partial_\alpha\chi\right).
\end{aligned}
\]
We now analyze $I_1,I_2,I_3$ respectively.
\begin{lemma}\label{estimate for I_2}
	We have 
	\begin{equation}
		I_2=\int\overline{(-\partial_t^2+i\partial_\alpha)u}D_t\theta=O(\epsilon t^{-1-\delta})
	\end{equation}
	for some $\delta>0$.	
\end{lemma}
\begin{proof}
	Actually, we have
	\[
	\partial_\alpha\chi=\frac{1}{\sqrt{t}}\chi'\left(\frac{\alpha-\nu t}{\sqrt{t}}\right)=-\frac{2\partial_t\chi}{\nu+\frac{\alpha}{t}}
	\]
	and thus
	\[
	\partial_t^2\chi
	=-\partial_t\left[\frac{1}{2}\left(\frac{\alpha}{t}+\nu\right)\partial_\alpha\chi\right]
	=-\frac{1}{2}\left(\frac{\alpha}{t}+\nu\right)\partial_t\partial_\alpha\chi+\frac{\alpha}{2t^2}\partial_\alpha\chi.
	\]
	This implies (we still use $U=e^{\frac{it^2}{4\alpha}}D_t\theta$)
	\begin{equation}\label{Dt2 chi}
		\begin{aligned}
			\int\overline{\partial_t^2\chi e^{-\frac{it^2}{4\alpha}}}D_t\theta
			&=-\frac{1}{2}\int\left(\frac{\alpha}{t}+\nu\right)\partial_\alpha\partial_t\chi 	U+\int\frac{\alpha}{2t^2}\partial_\alpha\chi U,\\
			&=\frac{1}{2}\int\frac{1}{t}\partial_t\chi U+\frac{1}{2}\int\left(\frac{\alpha}{t}+\nu\right)\partial_t\chi\partial_\alpha U-\frac{1}{2t^2}\int\chi U-\int\frac{\alpha}{2t^2}\chi\partial_\alpha U,\\
			&=-\frac{1}{4t}\int\left(\nu+\frac{\alpha}{t}\right)\partial_\alpha\chi U+\frac{1}{2}\int\left(\frac{\alpha}{t}+\nu\right)\partial_t\chi\partial_\alpha U-\frac{1}{2t^2}\int\chi U-\int\frac{\alpha}{2t^2}\chi\partial_\alpha U,\\
			&=-\frac{1}{4t^2}\int\chi U+\frac{1}{4t}\int\left(\nu+\frac{\alpha}{t}\right)\chi\partial_\alpha U-\frac{1}{4}\int\frac{\left(\frac{\alpha}{t}+\nu\right)^2}{\sqrt{t}}\chi'\left(\frac{\alpha-\nu t}{\sqrt{t}}\right)\partial_\alpha U\\
                &\ \ \ -\int\frac{\alpha}{2t^2}\chi\partial_\alpha U.
		\end{aligned}
	\end{equation}
	Recall that
	\[
	\partial_\alpha U=\partial_\alpha(e^{\frac{it^2}{4\alpha}}D_t\theta)=\frac{1}{\alpha^2}e^{\frac{it^2}{4\alpha}}\left(\alpha L_0+\frac{t}{2}\Omega_0\right)D_t\theta,
	\]
        and, for $\alpha$ in the support of $u(\cdot,t)$, we have
        \[
        \left|\frac{\alpha}{t}-\nu\right|\le\frac{C}{\sqrt{t}},
        \]
        and we obtain
        \begin{align*}
            \left|\frac{1}{4}\int\frac{\left(\frac{\alpha}{t}+\nu\right)^2}{\sqrt{t}}\chi'\left(\frac{\alpha-\nu t}{\sqrt{t}}\right)\partial_\alpha U\right|
            &\le \frac{C}{\sqrt{t}}\int\left(\left|\frac{2\alpha}{t}\right|+\frac{1}{\sqrt{t}}\right)^2\left|\chi'\left(\frac{\alpha-\nu t}{\sqrt{t}}\right)\right|\frac{1}{\alpha^2}\left|\alpha L_0D_t\theta+\frac{t}{2}\Omega_0D_t\theta\right|,\\
            &\le\frac{C}{t^2\sqrt{t}}\int\left|\chi'\left(\frac{\alpha-\nu t}{\sqrt{t}}\right)\right|\left|\alpha L_0D_t\theta+\frac{t}{2}\Omega_0D_t\theta\right|+O(\epsilon t^{-1-\delta}),\\
            &\le C\epsilon t^{-1-\frac{1}{20}+\delta_0}.
        \end{align*}
        (one should use the frequency decomposition method in Theorem \ref{main theorem of decay estimate} to replace $D_t\theta$ by its high frequency part). Other terms are treated similarly. Now, combining with (\ref{Dt2 chi}) yields
	\begin{equation}
		\left|\int\overline{\partial_t^2\chi e^{-\frac{it^2}{4\alpha}}}D_t\theta\right|\le C\epsilon t^{-1-\delta}.
	\end{equation}
	Also, observe that
	\[
	\frac{it}{\alpha}\partial_t\chi+i\partial_\alpha\chi
	=-\frac{it(\nu+\frac{\alpha}{t})}{2\alpha}\partial_\alpha\chi+i\partial_\alpha\chi
	=\frac{i(\alpha-\nu t)}{2\alpha}\partial_\alpha\chi,
	\]
	hence
	\[
	\begin{aligned}
		\int\overline{\left(\frac{it}{\alpha}\partial_t\chi+i\partial_\alpha\chi+\frac{i}{2\alpha}\chi\right)e^{-\frac{it^2}{4\alpha}}}D_t\theta
		&=-i\int\frac{\alpha-\nu t}{2\alpha}\partial_\alpha\chi e^{\frac{it^2}{4\alpha}}D_t\theta-i\int\frac{\chi}{2\alpha}e^{\frac{it^2}{4\alpha}}D_t\theta.
	\end{aligned}
	\]
    Again we use
    \[
    D_t\theta=\frac{4\alpha^2}{it^2}\partial_\alpha D_t\theta-\frac{4\alpha}{it^2}L_0D_t\theta+\frac{2}{it}\Omega_0D_t\theta,
    \]
    \[
    \partial_\alpha D_t\theta=-\frac{2\alpha}{it}\partial_t\partial_\alpha D_t\theta+\frac{2}{it}\Omega_0\partial_\alpha D_t\theta
    \]
    to obtain
    \[
    \int\overline{\left(\frac{it}{\alpha}\partial_t\chi+i\partial_\alpha\chi+\frac{i}{2\alpha}\chi\right)e^{-\frac{it^2}{4\alpha}}}D_t\theta
    =\int\frac{\alpha-\nu t}{2\alpha}\frac{8\alpha^3}{it^3}\partial_\alpha\chi U^*+\int\frac{\chi}{2\alpha}\frac{8\alpha^3}{it^3}U^*+O(\epsilon t^{-1-\delta})
    \]
    where $U^*=e^{\frac{it^2}{4\alpha}}\partial_\alpha D_t^2\theta$.

    Further computation yields
    \begin{align*}
        &\int\frac{\alpha-\nu t}{2\alpha}\frac{8\alpha^3}{it^3}\partial_\alpha\chi U^*+\int\frac{\chi}{2\alpha}\frac{8\alpha^3}{it^3}U^*\\
        &=-\int\partial_\alpha\left(\frac{\alpha-\nu t}{2\alpha}\frac{8\alpha^3}{it^3}\right)\chi U^*-\int\frac{\alpha-\nu t}{2\alpha}\frac{8\alpha^3}{it^3}\chi\partial_\alpha U^*
        +\int\frac{\chi}{2\alpha}\frac{8\alpha^3}{it^3}U^*,\\
        &=\int\frac{\alpha-\nu t}{2\alpha^2}\frac{8\alpha^3}{it^3}\chi U^*-\int\frac{\alpha-\nu t}{2\alpha}\frac{8\alpha^3}{it^3}\chi\partial_\alpha U^*+O(\epsilon t^{-1-\frac{1}{20}})
    \end{align*}
    which implies
    \begin{align*}
        \left|\int\overline{\left(\frac{it}{\alpha}\partial_t\chi+i\partial_\alpha\chi+\frac{i}{2\alpha}\chi\right)e^{-\frac{it^2}{4\alpha}}}D_t\theta\right|
        &\le C\epsilon t^{-1-\frac{1}{20}}+Ct^{\frac{1}{4}-\frac{1}{2}}\norm{\frac{4\alpha^2}{it^2}\partial_\alpha U^*}_{L^2(t^{4/5}\le|\alpha|\le t^{6/5})},\\
        &\le C\epsilon t^{-1-\frac{1}{20}+\delta_0}.
    \end{align*}
    
    This completes the proof.
 \end{proof}
 
\begin{lemma}\label{estimate for I_3}
    \begin{equation}
        I_3=\int b(\partial_\alpha D_t\theta\partial_t\bar{u}-\partial_\alpha D_t^2\theta\bar{u})=\frac{i|D_t^2\zeta(\nu t,t)|^2}{\nu}\tilde{E}(t)+O(\epsilon^3t^{-1-\delta}).
    \end{equation}
\end{lemma}
\begin{proof}
    We have
    \[
    \partial_t\bar{u}=e^{\frac{it^2}{4\alpha}}\partial_t\chi+\frac{it}{2\alpha}\bar{u}
    \]
    which implies
    \begin{align*}
        \partial_\alpha D_t\theta\partial_t\bar{u}
        &=e^{\frac{it^2}{4\alpha}}\partial_t\chi\partial_\alpha D_t\theta
        +\frac{it}{2\alpha}\partial_\alpha D_t\theta\bar{u},\\
        &=-\frac{\nu+\frac{\alpha}{t}}{2\sqrt{t}}e^{\frac{it^2}{4\alpha}}\chi'\partial_\alpha D_t\theta-\partial_t\partial_\alpha D_t\theta\bar{u}+\frac{1}{\alpha}\Omega_0(\partial_\alpha D_t\theta)\bar{u}.
    \end{align*}
    Note that 
    \[
    \partial_t\partial_\alpha D_t\theta=\partial_\alpha D_t^2\theta-\partial_\alpha(b\partial_\alpha D_t\theta)
    \]
    and we obtain
    \begin{align*}
        \partial_\alpha D_t\theta\partial_t\bar{u}-\partial_\alpha D_t^2\theta\bar{u}
        &=-2\partial_\alpha D_t^2\theta\bar{u}+\partial_\alpha(b\partial_\alpha D_t\theta)\bar{u}+\frac{1}{\alpha}\Omega_0(\partial_\alpha D_t\theta)\bar{u}-e^{\frac{it^2}{4\alpha}}\partial_t\chi\partial_\alpha D_t\theta,\\
        &=-2\partial_\alpha D_t^2\theta\bar{u}-\frac{it}{2\alpha}e^{\frac{it^2}{4\alpha}}\partial_t\chi\partial_\alpha\theta+O(\epsilon t^{-\frac{1}{2}-\delta})_{L^2}
    \end{align*}
    which implies
    \begin{align*}
        I_3
        &=-2\int b\partial_\alpha D_t^2\theta\bar{u}-\int b\frac{it}{2\alpha}e^{\frac{it^2}{4\alpha}}\partial_t\chi\partial_\alpha\theta+O(\epsilon^3t^{-1-\delta}),\\
        &=-2\int b\partial_\alpha D_t^2\theta\bar{u}+i\int|D_t^2\zeta|^2e^{\frac{it^2}{4\alpha}}\partial_t\chi\partial_\alpha\theta+O(\epsilon^3t^{-1-\delta}),\\
        &=-2\int b\partial_\alpha D_t^2\theta\bar{u}+O(\epsilon^3t^{-1-\delta}).
    \end{align*}
    We have used Corollary \ref{structure of b} and
    \[
    \norm{\partial_t\chi}_{L^2}\le C\norm{\left(\nu+\frac{\alpha}{t}\right)\partial_\alpha\chi}_{L^2}\le Ct^{-\frac{1}{20}}
    \]
    here. Integration by parts and observing that
    \[
    \partial_\alpha\bar{u}=\partial_\alpha(e^{\frac{it^2}{4\alpha}}\chi)=e^{\frac{it^2}{4\alpha}}\partial_\alpha\chi-\frac{it^2}{4\alpha^2}\bar{u}
    \]
    yields
    \begin{align*}
         I_3
         &=2\int b_\alpha D_t^2\theta\bar{u}+2\int
         bD_t^2\theta\partial_\alpha\bar{u}+O(\epsilon^3t^{-1-\delta}),\\
         &=-2\int \frac{it^2}{4\alpha^2}bw\bar{u}+2\int
         bD_t^2\theta e^{\frac{it^2}{4\alpha}}\partial_\alpha\chi+O(\epsilon^3t^{-1-\delta}),\\
         &=-2\int\frac{it^2}{4\alpha^2}bw\bar{u}+O(\epsilon^3t^{-1-\delta})
    \end{align*}
    since $\norm{b_\alpha}_{L^\infty}\le C\epsilon^2t^{-1-\delta}$. Apply Corollary \ref{structure of b} and define $F=e^{\frac{it^2}{4\alpha}}\partial_\alpha D_t^2\zeta$ to get
    \begin{align*}
        \int\frac{it^2}{4\alpha^2}bw\bar{u}
        &=-\int\frac{it}{2\alpha}|D_t^2\zeta|^2w\bar{u}+O(\epsilon^3t^{-1-\delta}),\\
        &=-i\int e^{-\frac{it^2}{4\alpha}}\partial_\alpha D_t\bar{\zeta}\frac{4\alpha^2}{it^2}Fw\bar{u}+O(\epsilon^3t^{-1-\delta}),\\
        &=-i\int\left(\frac{4\alpha^2}{it^2}F(\alpha)-\frac{4(\nu t)^2}{it^2}F(\nu t)\right)e^{-\frac{it^2}{4\alpha}}\partial_\alpha D_t\bar{\zeta}w\bar{u}\\
        &\ \ \ -i\frac{4(\nu t)^2}{it^2}F(\nu t,t)\int e^{-\frac{it^2}{4\alpha}}\partial_\alpha D_t\bar{\zeta}w\bar{u}+O(\epsilon^3t^{-1-\delta}),\\
        &=I_{31}-ie^{\frac{it^2}{4(\nu t)}}D_t^2\zeta(\nu t,t)\int e^{-\frac{it^2}{4\alpha}}\partial_\alpha D_t\bar{\zeta}w\bar{u}.
    \end{align*}
    It's clear that
    \[
    |I_{31}|\le C\norm{\sqrt{t}\partial_\alpha \left(\frac{4\alpha^2}{it^2}F\right)}_{L^2}\norm{\partial_\alpha D_t\bar{\zeta}}_{L^\infty}\norm{w}_{L^\infty}\norm{\bar{u}}_{L^2}
    \le C\epsilon^3t^{\frac{1}{2}-\frac{4}{5}-1+\frac{1}{4}+\delta_0}
    \le C\epsilon^3t^{-1-\frac{1}{20}+\delta_0}.
    \]
    In addition, note that
    \begin{align*}
        ie^{\frac{it^2}{4(\nu t)}}D_t^2\zeta(\nu t,t)\int e^{-\frac{it^2}{4\alpha}}\partial_\alpha D_t\bar{\zeta}w\bar{u}
        &=ie^{\frac{it^2}{4(\nu t)}}D_t^2\zeta(\nu t,t)\int\frac{2\alpha}{it}\bar{F}w\bar{u}+O(\epsilon^3 t^{-1-\delta}),\\
        &=ie^{\frac{it^2}{4(\nu t)}}D_t^2\zeta(\nu t,t)\frac{2(\nu t)}{it}\bar{F}(\nu t)\int w\bar{u}\\
        &\ \ \ +ie^{\frac{it^2}{4(\nu t)}}D_t^2\zeta(\nu t,t)\int\left(\frac{2\alpha}{it}\bar{F}-\frac{2(\nu t)}{it}\bar{F}(\nu t)\right)w\bar{u}+O(\epsilon^3t^{-1-\delta}),\\
        &=\frac{i|D_t^2\zeta|^2}{2\nu}\int w\bar{u}+I_{32}+O(\epsilon^3t^{-1-\delta}).
    \end{align*}
    Still, there holds
    \[
    |I_{32}|\le C\norm{D_t^2\zeta}_{L^\infty}\norm{\sqrt{t}\partial_\alpha\left(\frac{2\alpha}{it}F\right)}_{L^2}\norm{w}_{L^\infty}\norm{u}_{L^2}
    \le C\epsilon^3t^{-1-\frac{1}{20}+\delta_0}.
    \]
    Finally, we have
    \begin{align*}
        I_3&=2iD_t^2\zeta(\nu t,t)\frac{2(\nu t)}{it}\partial_\alpha D_t^2\bar{\zeta}(\nu t,t)\int w\bar{u}+O(\epsilon^3t^{-1-\delta})\\
           &=2iD_t^2\zeta(\nu t,t)\partial_\alpha D_t\bar{\zeta}(\nu t,t)\int w\bar{u}+O(\epsilon^3t^{-1-\delta})\\
           &=2iD_t^2\zeta(\nu t,t)\partial_\alpha D_t\bar{\zeta}(\nu t,t)\tilde{E}(t)+O(\epsilon^3t^{-1-\delta})\\
           &=2i\frac{|D_t^2\zeta|^2}{\nu}\tilde{E}(t)+O(\epsilon^3t^{-1-\delta}).
    \end{align*}
    This completes the proof.
\end{proof}

\begin{rem}
    The last part of the proof aims to avoid the appearance of an uncontrollable term. To be specific, if we write
    \[
    I_3=\frac{2i|D_t^2\zeta|^2}{\nu}\int w\bar{u}+O(\epsilon^3t^{-1-\delta})
    \]
    and substitute
    \[
    \tilde{E}(t)=\int w\bar{u}+O(\epsilon t^{-\frac{1}{20}+\delta_0})
    \]
    into it, then we must deal with the term $\nu^{-1}|D_t^2\zeta|^2O(\epsilon t^{-\frac{1}{20}+\delta_0})$, which is hard to analyze since we only know that $\nu^{-1}\le t^{1/5}$. By absorbing $2\nu t/it$ into $\partial_\alpha D_t^2\bar{\zeta}(\nu t,t)$, one can avoid this discussion. Such a trick will be applied in the proof of Lemma \ref{estimate for I_1} too, and we will not write it out again. Instead, we will simply use the following argument:
    \[
    \frac{i|D_t^2\zeta|^2}{\nu}\int w\bar{u}+O(\epsilon^3t^{-1-\delta})
    =\frac{i|D_t^2\zeta|^2}{\nu}\tilde{E(t)}+O(\epsilon^3t^{-1-\delta}).
    \]
\end{rem}

\begin{lemma}\label{estimate for I_1}
	\begin{equation}\label{structure of I_1}
		I_1=\int(D_tG+R')\bar{u}=\frac{2i|D_t^2\zeta(\nu t,t)|^2}{\nu}\tilde{E}(t)+O(\epsilon^3t^{-1-\delta})
	\end{equation}
	for some $\delta>0$.
\end{lemma}

\begin{proof}
	To begin with, we assert that
	\begin{equation}\label{estimate for R'}
		\left|\int R'\bar{u}\right|\le C\epsilon^3t^{-1-\delta}.
	\end{equation}
	To prove (\ref{estimate for R'}), first recall that
	\[
	R'=-i\left(\frac{a_t}{a}\circ\kappa^{-1}\right)A\partial_\alpha\theta+i(A-1)\partial_\alpha D_t\theta
	\]
	and $\norm{\left(\frac{a_t}{a}\circ\kappa^{-1}\right)}_{L^\infty}\le C\epsilon^2t^{-1-\delta}$, which implies
    \[
    \left|-i\int\left(\frac{a_t}{a}\circ\kappa^{-1}\right)A\partial_\alpha\theta\bar{u}\right|\le C\sqrt{t}\norm{\frac{a_t}{a}\circ\kappa^{-1}}_{L^\infty}\norm{A\partial_\alpha\theta}_{L^\infty}\le C\epsilon^3t^{-1-\delta}.
    \]
    Also, we have (use Remark \ref{decay for A-1 near t})
    \[
    \left|\int i(A-1)\partial_\alpha D_t\theta\bar{u}\right|
    \le C\epsilon^3t^{-\frac{5}{4}-\frac{1}{20}+\delta_0}\ln t\norm{\bar{u}}_{L^2}
    \le C\epsilon^3t^{-1-\frac{1}{20}+\delta_0}\ln t.
    \]
 
	Next we consider the term
	\[
	\int D_tG\bar{u}.
	\]
	We begin with a part of $D_tG$, namely,
 \begin{align*}
      I
      &=\frac{2}{\pi i}\int\frac{(D_t\zeta(\alpha)-D_t\zeta(\beta))(D_t^2\zeta(\alpha)-D_t^2\zeta(\beta))}{(\zeta(\alpha)-\zeta(\beta))^2}(\zeta-\bar{\zeta})_\beta d\beta\\
      &\ \  \ +\frac{1}{\pi i}\int\left(\frac{D_t\zeta(\alpha)-D_t\zeta(\beta)}{\zeta(\alpha)-\zeta(\beta)}\right)^2(D_t\zeta-D_t\bar{\zeta})_\beta d\beta.
 \end{align*}
 Recall (\ref{structure of I}) and we get
        \[
        \int I\bar{u}=-\int\frac{t^2}{\alpha^2}(D_t\zeta)|D_t^2\zeta|^2\bar{u}+O(\epsilon^3t^{-1-\delta}).
        \]
        Repeat the technique that we used in Lemma \ref{estimate for I_3} (namely, write $\frac{t^2}{\alpha^2}|D_t^2\zeta|^2=4|D_t^3\zeta|^2+error=4|\partial_\alpha D_t\zeta|^2+error$, then multiply $\frac{2\alpha}{it}$ to each of them) to obtain
        \begin{align*}
            \int I\bar{u}
            &=-\frac{|D_t^2\zeta|^2}{i\nu}\int w\bar{u}+O(\epsilon^3t^{-1-\delta}),\\
            &=\frac{i|D_t^2\zeta|^2}{\nu}\tilde{E}(t)+O(\epsilon^2t^{-1-\delta}).
        \end{align*}

    We still need to consider another part of $D_tG$, i.e.
    \begin{equation}\label{another part of DtG}
        \begin{aligned}
            &\frac{1}{\pi i}D_t\int\frac{(f(\alpha)-f(\beta))(D_t\bar{\zeta}(\alpha)-D_t\bar{\zeta}(\beta))}{(\zeta(\alpha)-\zeta(\beta))^2}\partial_\beta D_t\zeta d\beta\\
            &\ \ \ -\frac{1}{\pi i}D_t\int\frac{(\bar{f}(\alpha)-\bar{f}(\beta))(D_t\bar{\zeta}(\alpha)-D_t\bar{\zeta}(\beta))}{(\zeta(\alpha)-\zeta(\beta))^2}\partial_\beta D_t\zeta d\beta
        \end{aligned}
    \end{equation}
    where $f(\alpha)=\zeta-\alpha$.
    
    We only consider
    \[
    \tilde{G}_{12}^*=\frac{1}{\pi i}\int\frac{(\bar{f}(\alpha)-\bar{f}(\beta))(D_t^2\bar{\zeta}(\alpha)-D_t^2\bar{\zeta}(\beta))}{(\zeta(\alpha)-\zeta(\beta))^2}\partial_\beta D_t\zeta d\beta
    \]
    since others are similar or easier. 
    
    Claim that
    \begin{equation}
        \left|\int\tilde{G}_{12}^*\bar{u}\right|\le C\epsilon^3t^{-1-\delta}
    \end{equation}
    for some $\delta>0$.
    
    For $\delta>0$, define
    \[
    f_L=P_{<t^{-1/2-\delta}}f,\ f_H=f-f_L.
    \]
    According to the relevant conclusions in Chapter 5, we get
    \begin{align*}
        \int \tilde{G}_{12}^*\bar{u}
        &=-\int\frac{it^2}{2\alpha^2}(D_t\bar{\zeta})_HD_t^2\bar{\zeta}(\zeta_\alpha-1) e^{\frac{it^2}{4\alpha}}\chi+O(\epsilon^3t^{-1-\delta}),\\
        &=-\int\frac{t}{\alpha}(D_t^2\bar{\zeta})^2(\zeta_\alpha-1)e^{\frac{it^2}{4\alpha}}\chi+O(\epsilon^3t^{-1-\delta}),\\
        &=-\int\frac{t}{\alpha}e^{\frac{it^2}{2\alpha}}U^2V\chi+O(\epsilon^3t^{-1-\delta}),\\
        &=C\int e^{\frac{it^2}{2\alpha}}\partial_\alpha\left(\frac{\alpha}{t}U^2V\chi\right)+O(\epsilon^3t^{-1-\delta}),\\
        &=O(\epsilon^3t^{-1-\delta})
    \end{align*}
    as desired. Here
    \[
    U=e^{-\frac{it^2}{4\alpha}}D_t^2\bar{\zeta},\ V=e^{\frac{it^2}{4\alpha}}(\zeta_\alpha-1).
    \]
    This completes the proof of (\ref{structure of I_1}).
\end{proof}

Combining Lemma \ref{estimate for I_2}, \ref{estimate for I_3} and \ref{estimate for I_1} yields
\begin{thm}
    There exists $\delta>0$ such that
    \begin{equation}
        \frac{d}{dt}\tilde{E}(t)=\frac{3i|D_t^2\zeta(\nu t,t)|^2}{\nu}\tilde{E}(t)+O(\epsilon^3(1+t)^{-1-\delta}).
    \end{equation}
\end{thm}

\begin{cor}\label{Decay}
    \begin{equation}\label{decay}
        \norm{\tilde{E}(t)}_{L^\infty}\le\norm{\tilde{E}(0)}_{L^\infty}+\frac{C\epsilon^3}{\delta}(1-(1+t)^{-\delta}).
    \end{equation}
\end{cor}

\begin{proof}
    Set $f(t)=\frac{3|D_t^2\zeta(\nu t,t)|^2}{\nu}\ge 0$ and we get (we still replace $t$ by $(1+t)$ here)
    \[
    \frac{d}{dt}\tilde{E}(t)=if(t)\tilde{E}(t)+O(\epsilon^3(1+t)^{-1-\delta})
    \]
    i.e.
    \[
    \left|\frac{d}{dt}\left(e^{-i\int_{0}^tf(s)ds}\tilde{E}(t)\right)\right|\le C\epsilon^3(1+t)^{-1-\delta}
    \]
    which implies (\ref{decay}).
\end{proof}

\begin{cor}
    For $t^{4/5}\le|\alpha|\le t^{6/5}$, there holds
    \begin{equation}
        \norm{D_t^2\theta(\alpha,t)}_{L^\infty}\le 4\epsilon t^{-\frac{1}{2}}.
    \end{equation}
\end{cor}

\begin{proof}
    Combining Corollary \ref{Decay} with Theorem \ref{main theorem of decay estimate} yields
    \[
    \norm{D_t^2\theta(\nu t,t)}_{L^\infty}\le 4\epsilon t^{-\frac{1}{2}};
    \]
    note that the right-hand side is independent of $\nu$ and we are done.
\end{proof}

\begin{cor}\label{sec 7:decay bound}
    For $t^{4/5}\le|\alpha|\le t^{6/5}$, there holds
    \begin{equation}
        \norm{(D_t^2\theta,\partial_\alpha D_t\theta,\theta_\alpha)}_{W^{s-2,\infty}\times W^{s-3,\infty}\times W^{s-2,\infty}}\le C\epsilon^3t^{-\frac{1}{2}}.
    \end{equation}
\end{cor}

\begin{proof}
    One simply repeats the process for $D_t^2\theta$ as above (note that we didn't lose any derivatives while taking $L^\infty$ norm, so the process closes). One should pay attention that when it comes to higher order derivatives of the equation, there is an extra term
    \[
    \int[D_t^2-iA\partial_\alpha,\partial_\alpha^k]\theta\bar{u}\ (1\le k\le s-2).
    \]
    For such a term, we explicitly compute the commutator (refer to Lemma \ref{calculate for k order commutator}) and treat it by the techniques that were shown in the proof of Lemma \ref{estimate for I_3} and Lemma \ref{estimate for I_1}.
\end{proof}


\appendix

\section{Singular integrals and commutators}

Let $m\geq 1$ be an integer. Denote
\begin{equation}
S_1(A,f)=p.v.\int \prod_{j=1}^m \frac{A_j(\alpha)-A_j(\beta)}{\gamma_j(\alpha)-\gamma_j(\beta)} \frac{f(\beta)}{\gamma_0(\alpha)-\gamma_0(\beta)}d\beta.
\end{equation}

\begin{equation}
S_2(A,f)=\int \prod_{j=1}^m \frac{A_j(\alpha)-A_j(\beta)}{\gamma_j(\alpha)-\gamma_j(\beta)} f_{\beta}(\beta)d\beta.
\end{equation}
\noindent We have the following commutator estimates, which can be found in \cite{Totz2012}, \cite{Wu2009}.
\begin{lemma}\label{singular}
We have the following results:

(1) Assume each $\gamma_j (j=0,...,m)$ satisfies 
\begin{equation}\label{gammaj}
C_{0,j}|\alpha-\beta|\leq |\gamma_j(\alpha)-\gamma_j(\beta)|\leq C_{1,j}|\alpha-\beta|.
\end{equation}
Then both $\|S_1(A,f)\|_{L^2}$ and $\|S_2(A,f)\|_{L^2}$ are bounded by
$$C\prod_{j=1}^m \|A_j'\|_{X_j}\|f\|_{X_0},$$
where one of the $X_0, X_1, ...X_m$ is equal to $L^2$ and the rest are $L^{\infty}$. The constant $C$ depends on $\|\gamma_j'\|_{L^{\infty}}^{-1}, j=0,1,..,m$.

(2) Let $s\geq 3$ be given, and assume (\ref{gammaj}) for each $\gamma_j$, then 
$$\|S_2(A,f)\|_{H^s}\leq C\prod_{j=1}^m \|A_j'\|_{Y_j}\|f\|_Z,$$
where for all $j=1,...,m$, $Y_j=H^{s-1}$ or $W^{[s-1]+1,\infty}$ and $Z=H^s$ or $W^{[s]+1,\infty}$. At most one of these $Y_j, Z$ norms is in $H^k$ ($k=s-1$ for $Y_j$ or $s$ for $Z$). The constant $C$ depends on $\norm{\partial_\alpha\gamma_j-1}_{H^{s-1}},\ j=1,...,m$.

(3) Assume (\ref{gammaj}) for each $\gamma_j$, then
\[
\norm{S_1(A,f)}_{L^\infty}\le C\left(\prod_{j=1}^m\norm{A'_j}_{W^{1,\infty}}\norm{f}_{W^{1,\infty}}+\prod_{j=1}^m\norm{A_j'}_{L^\infty}\norm{f}_{L^\infty}\ln r+\prod_{j=1}^m\norm{A_j'}_{L^\infty}\norm{f}_{L^2}r^{-\frac{1}{2}}\right)
\]
for any $r>1$. The constant $C$ depends on $\|\gamma_j'\|_{L^{\infty}}^{-1}, j=0,1,..,m$.
\end{lemma}

\begin{lemma}\label{derivative_K_com}
    Let $K$ be such that $K$ or $(\alpha-\beta)K(\alpha,\beta;t)$ is continuous and bounded, and $K$ is smooth away from the diagonal $\Delta=\{(\alpha,\beta)| \alpha=\beta\}$. Denote
    \begin{equation}
        \boldsymbol{K}f(\alpha,t)=p.v. \int K(\alpha,\beta;t)f(\beta,t)d\beta
    \end{equation}
    Then we have 
    \begin{equation}
        \begin{split}
            [\partial_t, \boldsymbol{K}]f(\alpha,t)=&\int \partial_t K(\alpha,\beta;t)f(\beta,t)d\beta.\\
            [\partial_{\alpha},\boldsymbol{K}]f(\alpha,t)=& \int (\partial_{\alpha}+\partial_{\beta})K(\alpha,\beta;t)f(\beta,t)d\beta.\\
            [L_0, \boldsymbol{K}]f(\alpha,t)=& \int (\alpha\partial_{\alpha}+\beta\partial_{\beta}+\frac{1}{2}t\partial_t)K(\alpha,\beta;t)f(\beta,t)d\beta+\boldsymbol{K}f(\alpha,t)
        \end{split}
    \end{equation}
    for $f\in C^1(\mathbb{R};\mathbb{R}^d)$ vanisheses as $|\alpha|\rightarrow\infty$.
\end{lemma}

\begin{lemma}\label{useful lemma}
			Suppose that $A_1\in L^2, f\in L^2\cap W^{1,\infty},\ \partial_\alpha A_i\in W^{1,\infty}\ (2\le i\le m)$, then for any $t>e$, there holds
			\[
			\begin{aligned}
				&\left\|\int\frac{\prod_{i=1}^m(A_i(\alpha)-A_i(\beta))}{(\zeta(\alpha)-\zeta(\beta))^m}fd\beta\right\|_{L^2}\\
				&\le C\norm{A_1}_{L^2}\left(\prod_{i\ge 2}\norm{\partial_\alpha A_i}_{W^{1,\infty}}\norm{f}_{W^{1,\infty}}+\prod_{i\ge 2}\norm{\partial_\alpha A_i}_{L^\infty}\norm{f}_{L^\infty}\ln t+\prod_{i\ge 2}\norm{\partial_\alpha A_i}_{L^\infty}\norm{f}_{L^2}t^{-\frac{1}{2}}\right).
			\end{aligned}
			\]
		\end{lemma}
		
\begin{proof}
We have
\[
\begin{aligned}
    &\left\|\int\frac{\prod_{i=1}^m(A_i(\alpha)-A_i(\beta))}{(\zeta(\alpha)-\zeta(\beta))^m}fd\beta\right\|_{L^2}\\
    &\le\|A_1\|_{L^2}\left\|\int\frac{\prod_{i\ge 2}(A_i(\alpha)-A_i(\beta))}{(\zeta(\alpha)-\zeta(\beta))^m}fd\beta\right\|_{L^\infty}
    +\left\|\int\frac{\prod_{i\ge 2}(A_i(\alpha)-A_i(\beta))}{(\zeta(\alpha)-\zeta(\beta))^m}A_1fd\beta\right\|_{L^2}\\
    &\le C\norm{A_1}_{L^2}\left(\prod_{i\ge 2}\norm{\partial_\alpha A_i}_{W^{1,\infty}}\norm{f}_{W^{1,\infty}}+\prod_{i\ge 2}\norm{\partial_\alpha A_i}_{L^\infty}\norm{f}_{L^\infty}\ln t+\prod_{i\ge 2}\norm{\partial_\alpha A_i}_{L^\infty}\norm{f}_{L^2}t^{-\frac{1}{2}}\right)\\
    &\ \ \ +C\prod_{i\ge 2}\|\partial_\alpha A_i\|_{L^\infty}\|A_1f\|_{L^2}\\
    &\le C\norm{A_1}_{L^2}\left(\prod_{i\ge 2}\norm{\partial_\alpha A_i}_{W^{1,\infty}}\norm{f}_{W^{1,\infty}}+\prod_{i\ge 2}\norm{\partial_\alpha A_i}_{L^\infty}\norm{f}_{L^\infty}\ln t+\prod_{i\ge 2}\norm{\partial_\alpha A_i}_{L^\infty}\norm{f}_{L^2}t^{-\frac{1}{2}}\right).
\end{aligned}
\]
Note that we have invoked Lemma \ref{singular}.
\end{proof}

\section{Frequency localization and Bernstein inequality}\label{sec-bernstein}
Let $d$ be a positive integer. Fix a real-valued radially symmetric bump function $\phi$ in $\mathbb{R}^d$ adapted to the interval $\{\xi: |\xi|\leq 2\}$ such that $\phi(\xi)\equiv 1$ for $|\xi|\leq 1$ and equals $0$ for $|\xi|\geq 2$. Let $N$ be a dyadic number. Define the Fourier multipliers
\begin{equation}
    \begin{split}
        &\widehat{P_{<N}f}(\xi):=\phi(\frac{\xi}{N})\hat{f}(\xi),\\
        &\widehat{P_{>N}f}(\xi):=\Big(1-\phi(\frac{\xi}{N})\Big)\hat{f}(\xi),\\
        &\widehat{P_{N}f}(\xi):=\Big(\phi(\frac{\xi}{N})-\phi(\frac{\xi}{2N})\Big)\hat{f}(\xi),\\
    \end{split}
\end{equation}
We record the Bernstein inequality as follows, which can be found, for example, on page 333 of \cite{tao2006nonlinear}.
\begin{lemma}[Bernstein inequality]\label{Bernstein}
    Let $s\geq 0$ and $1\leq p\leq q\leq \infty$. We have
    \begin{equation}
        \norm{P_{\geq N}f}_{L^p(\mathbb{R}^d)}\lesssim_{p,s,d} N^{-s}\norm{|\nabla|^s P_{\geq N}f}_{L^p(\mathbb{R}^d)},
    \end{equation}
        \begin{equation}
        \norm{P_{N}|\nabla|^{\pm s}f}_{L^p(\mathbb{R}^d)}\sim_{p,s,d} N^{\pm s}\norm{|\nabla|^s P_{N}f}_{L^p(\mathbb{R}^d)},
    \end{equation}

            \begin{equation}
        \norm{P_{\leq N}|\nabla|^{s}f}_{L^p(\mathbb{R}^d)}\sim_{p,s,d} N^{s}\norm{|\nabla|^s P_{\leq N}f}_{L^p(\mathbb{R}^d)},
    \end{equation}

       \begin{equation}
        \norm{P_{\leq N}f}_{L^q(\mathbb{R}^d)}\lesssim_{p,s,d} N^{\frac{d}{p}-\frac{d}{q}}\norm{P_{\geq N}f}_{L^p(\mathbb{R}^d)},
    \end{equation}

           \begin{equation}
        \norm{P_{N}f}_{L^q(\mathbb{R}^d)}\lesssim_{p,s,d} N^{\frac{d}{p}-\frac{d}{q}}\norm{P_{N}f}_{L^p(\mathbb{R}^d)},
    \end{equation}
\end{lemma}

\newpage
\section{Estimate for half derivative}
\begin{lemma}\label{estimate for half derivative}
    We have
    \begin{equation}
\norm{\Lambda \mathcal{H}f}_{L^2}\leq C\norm{f}_{H^{1/2}}.
    \end{equation}
\end{lemma}
\begin{proof}
The proof is fairly standard, but we provide the details here for completeness. Let $s\in [0,1]$ and $y\in \mathbb{R}$. Define a family of operators  $T_{s+iy}$ as follows:
$$T_{s+iy}:=|\partial_{\alpha}|^{s+iy}\mathcal{H}(I+|\partial_{\alpha}|^{s+iy})^{-1}.$$
Clearly, $T_{s+iy}$ is analytic on the strip $[0,1]\times \mathbb{R}$. We compute the following:
\begin{itemize}
\item [(1)] For any $T_{0+iy}$:
\begin{align*}
\norm{T_{0+iy}f}_{L^2}=&\norm{|\partial_{\alpha}|^{iy}\mathcal{H}(I+|\partial_{\alpha}|^{iy})^{-1}f}_{L^2}\\
    =& \norm{\mathcal{H}(I+|\partial_{\alpha}|^{iy})^{-1}f}_{L^2}\\
    \leq &\  C\norm{(I+|\partial_{\alpha}|^{iy})^{-1}f}_{L^2}\leq C\norm{f}_{L^2}.
\end{align*}
\item [(2)] For any $T_{1+iy}$:
\begin{align*}
\norm{T_{1+iy}f}_{L^2}=&\norm{|\partial_{\alpha}|^{1+iy}\mathcal{H}(I+|\partial_{\alpha}|^{1+iy})^{-1}f}_{L^2}\\
    =& \norm{\partial_{\alpha}\mathcal{H}(I+|\partial_{\alpha}|^{1+iy})^{-1}f}_{L^2}\\
    \leq &\  C\norm{(I+|\partial_{\alpha}|^{iy})^{-1}f}_{H^1}\leq C\norm{f}_{L^2},
\end{align*}
\end{itemize}
By Stein's analytic interpolation theorem, $T_{s+iy}$ is bounded on $L^2$, for any $(s,y)\in [0,1]\times\mathbb{R}$:
$$\norm{|\partial_{\alpha}|^{s+iy}\mathcal{H}(I+|\partial_{\alpha}|^{s+iy})^{-1}f}_{L^2}\leq C\norm{f}_{L^2}.$$
In particular, take $y=0$ and $s=1/2$ and let $g:=(I+|\partial_{\alpha}|^{1/2})^{-1}f$ or equivalently, $f=(I+|\partial_{\alpha}|^{1/2})g$, we obtain
$$\norm{\Lambda \mathcal{H}g}_{L^2}\leq C\norm{g}_{H^{1/2}},$$
as desired.
\end{proof}

\bibliography{qingtang}{}
\bibliographystyle{plain}

\end{document}